\documentclass[<3p>]{elsarticle}

\usepackage[utf8]{inputenc}% or latin1 with Windows editors
\usepackage[T1]{fontenc}
\usepackage{appendix}
\usepackage{lmodern}
\usepackage{newunicodechar}
\usepackage{amsmath, amsthm, amsfonts,amssymb}
\usepackage{xcolor,mhsetup}
\usepackage[extra,safe]{tipa}
\usepackage{mathtools}
\usepackage{geometry}
\usepackage{mathrsfs}
\usepackage{hyperref}
\usepackage{url}
\usepackage{fancyhdr}
\theoremstyle{definition}
\newtheorem{theorem}{Theorem}[section]
\newtheorem{prop}{Proposition}[section]
\newtheorem{definition}[theorem]{Definition}

\newtheorem{lem}{Lemma}[section]

\theoremstyle{remark}
\newtheorem{remark}[theorem]{Remark}
\theoremstyle{cor}
\newtheorem{cor}[theorem]{Corollary}

\numberwithin{equation}{section}

\def \HR{(\textbf{HR})}
\def \HRp{(\textbf{HR$_+$})}
\def \HRpp{(\textbf{HR$_{++}$})}

\def \HE{(\textbf{HE})}

\def \tr{\mbox{trace}}
\newcommand \A[1]{{\bf (#1)}}

\def \mH{\mathcal{H}}

\def\my_c{c_\infty}

\def \l{{\mathbf{l}}}

\newcommand{\mynewtheorem}[2]{
  \newaliascnt{#1}{dummy}
  \newtheorem{#1}[#1]{#2}
  \aliascntresetthe{#1}
  \expandafter\def\csname #1autorefname\endcsname{#2}
}

\newcommand{\seq}[1]{{\lbrace #1 \rbrace}}

\newcommand{\dcb}{\begin{array}{lll}}
\newcommand{\dce}{\end{array}}
\newcommand{\ebe}{\begin{enumerate}\setlength{\baselineskip}{13pt}\setlength{\parskip}{0pt}}
\newcommand{\dbe}{\end{enumerate}}

\newcommand{\E}{\mathcal{E}}

\def\pp{{\cal P}}

\def\rr{{\mathbb R}}

\def\P{{\mathbb P}}

\def\E{\mathbb{E}}
\def \pp{{\mathcal{P}_2(\mathbb{R}^d)}}
\def\I{\mathsf{1}}

\def \supp{\mbox{Supp} }

\def\0{{\mathbf{0}}}

\def\d{{{\rm det}}}

\def\x{{\mathbf{x}}}

\def\gv{{\mathbf{v}}}

\begin{document}

\setcounter{tocdepth}{1}

\title{From the backward Kolmogorov PDE on the Wasserstein space to propagation of chaos for McKean-Vlasov SDEs}

 \author[1]{Paul-Eric Chaudru de Raynal
 \fnref{fn1}}
\ead{pe.deraynal@univ-smb.fr}
%\address{Paul-Eric Chaudru de Raynal,
% LAMA, UMR 5127, Universit\'e Savoie Mont Blanc, Campus Scientifique, 73376 Le Bourget-du-Lac Cedex}
  \address[1]{Universit\'e Savoie Mont Blanc, CNRS, LAMA, 73000 Chamb\'ery, France}
%\curraddr{}
%\email{pe.deraynal@univ-smb.fr}

\author[2]{Noufel Frikha \corref{cor1}}
\address[2]{Universit\'e de Paris, Laboratoire de Probabilit\'es, Statistique et Mod\'elisation (LPSM), F-75013 Paris, France}
%\curraddr{}
\ead{frikha@lpsm.paris}
\cortext[cor1]{Corresponding author}

\fntext[fn1]{For the first Author, this work has been partially supported by the ANR project ANR-15-IDEX-02.}

%\curraddr{}

%\author{Arturo Kohatsu-Higa}
%\address{Arturo Kohatsu-Higa,
%Department of Mathematical Sciences
%Ritsumeikan University
%1-1-1 Nojihigashi, Kusatsu, Shiga, 525-8577, Japan }
%\curraddr{}
%\email{khts00@fc.ritsumei.ac.jp}

%\author{Libo Li}
%\address{Libo Li, Department of Mathematics and Statistics, University of New South Wales, Sydney, Australia}
%\curraddr{}
%\email{libo.li@unsw.edu.au }

%    General info
%\MSC{Primary 60H10, 93E03; Secondary 60H30, 35K40}
\journal{Journal de Math\'ematiques Pures et Appliqu\'ees.}
\date{\today}


\begin{abstract} This article is a continuation of our first work \cite{chaudruraynal:frikha}. We here establish some new quantitative estimates for propagation of chaos of non-linear stochastic differential equations in the sense of McKean-Vlasov. We obtain explicit error estimates, at the level of the trajectories, at the level of the semi-group and at the level of the densities, for the mean-field approximation by systems of interacting particles under mild regularity assumptions on the coefficients. A first order expansion for the difference between the densities of one particle and its mean-field limit is also established. Our analysis relies on the well-posedness of classical solutions to the backward Kolmogorov partial differential equations defined on the strip $[0,T] \times \mathbb{R}^d \times \mathcal{P}_2(\mathbb{R}^d)$, $\mathcal{P}_2(\mathbb{R}^d)$ being the Wasserstein space, that is, the space of probability measures on $\mathbb{R}^d$ with a finite second-order moment and also on the existence and uniqueness of a fundamental solution for the related parabolic linear operator here stated on $[0,T]\times \mathcal{P}_2(\mathbb{R}^d)$. 
\end{abstract}

\maketitle

\section*{R\'esum\'e}
Cet article est la suite de notre premier travail \cite{chaudruraynal:frikha}. Nous \'etablissons ici de nouvelles estimations quantitatives pour la propagation du chaos des \'equations diff\'erentielles stochastiques non-lin\'eaires au sens de McKean-Vlasov. Nous obtenons des estimations d'erreurs explicites, au niveau des trajectoires, au niveau du semi-groupe et au niveau des densit\'es de transition, pour l'approximation champ moyen par des syst\`emes de particules en interaction sous de faibles hypoth\`eses de r\'egularit\'e sur les coefficients. Un d\'eveloppement \`a l'ordre un pour la diff\'erence entre les densit\'es d'une particule et celle de sa limite champ moyen est \'egalement \'etabli. Notre analyse repose sur le caract\`ere bien pos\'e de solutions classiques aux \'equations aux d\'eriv\'ees partielles de Kolmogorov r\'etrogrades d\'efinies sur la bande $[0,T] \times \mathbb{R}^d \times \mathcal{P}_2(\mathbb{R}^d)$, $\mathcal{P}_2(\mathbb{R}^d)$ \'etant l'espace de Wasserstein, c'est-\`a-dire l'espace des mesures de probabilit\'es sur $\mathbb{R}^d$ de moment d'ordre deux fini et aussi sur l'existence et l'unicit\'e d'une solution fondamentale pour l'op\'erateur parabolique lin\'eaire associ\'e \'enonc\'e ici sur $[0,T]\times \mathcal{P}_2(\mathbb{R}^d)$.
%\tableofcontents

\section{Introduction}\label{introduction}
In this work, we are interested in some non-linear stochastic differential equations (SDEs for short) in the sense of McKean-Vlasov with dynamics:
\begin{equation}
\label{SDE:MCKEAN}
X^{\xi}_t = \xi + \int_0^t b(s, X^{\xi}_s, [X^{\xi}_s]) ds + \int_0^t \sigma(s, X^{\xi}_s, [X^{\xi}_s]) dW_s, \quad [\xi]=\mu \in \mathcal{P}(\mathbb{R}^d),
\end{equation}

\noindent where $\xi$ is an $\mathbb{R}^d$-valued random variable which is independent of the $q$-dimensional Brownian motion $W=(W^1,\cdots, W^q)$ and with coefficients $b: \mathbb{R}_+ \times \mathbb{R}^d \times \mathcal{P}(\mathbb{R}^d)\rightarrow \mathbb{R}^d$ and $\sigma: \mathbb{R}_+ \times \mathbb{R}^d \times \mathcal{P}(\mathbb{R}^d) \rightarrow \mathbb{R}^{d\times q}$, $[\theta]$ denoting the law of the random variable $\theta$ and its approximation by the associated system of $N$ particles $\left\{(X^{i}_t)_{t\in [0,T]}, 1\leq i \leq N\right\}$ interacting through its empirical measure
%, as the size of the population $N$ grows to infinity
\begin{equation}
\label{SDE:particle:system}
X^{i}_t = \xi^{i} + \int_0^t b(s, X^{i}_s, \mu^{N}_s) ds + \int_0^t \sigma(s, X^{i}_s, \mu^{N}_s) dW^{i}_s, \quad \mu^{N}_t := \frac{1}{N} \sum_{i=1}^{N}\delta_{X^{i}_t}, \quad i=1, \cdots, N,
\end{equation}

\noindent where $\left\{(\xi^{i}, (W^{i}_t)_{t\in [0,T]}), 1\leq i \leq N\right\}$ are i.i.d. copies of $(\xi, W)$. The connection between the two above systems of SDEs comes from fact that the dynamics \eqref{SDE:MCKEAN} describes the limiting behaviour of an individual particle in \eqref{SDE:particle:system} when the size of the population $N$ grows to infinity as stated by the so-called propagation of chaos phenomenon, originally studied by McKean \cite{mckean1967propagation} and then investigated by Sznitman \cite{Sznitman}. Roughly speaking, it is expected that the dynamics of $k$ particles among $N$, say $(X^{1}, \cdots, X^{k})$, $k$ being a fixed positive integer, consists of $k$ independent copies $(\bar{X}^1, \cdots, \bar{X}^k)$ of a process following the law of the unique solution to the limiting equation \eqref{SDE:MCKEAN} as $N$ goes to infinity. Since the original works of Kac \cite{kac1956} in kinetic theory and of McKean \cite{McKean:1966} in non-linear parabolic partial differential equations (PDEs for short), theoretical and numerical aspects of McKean-Vlasov SDEs have been an active research area in several directions during the last decades such as the well-posedness of the related martingale problem, the propagation of chaos and other limit theorems, probabilistic representations to non-linear parabolic PDEs and their numerical approximation schemes. We refer to Tanaka \cite{tanaka1978probabilistic}, Funaki \cite{Funaki1984}, Oelschl\"ager \cite{oelschlager1984}, G\"artner \cite{gartner}, \cite{Sznitman}, Mishura and Veretenikov \cite{mishura:veretenikov}, Chaudru de Raynal \cite{CHAUDRUDERAYNAL2019}, Lacker \cite{la:18} for a small sample among others. \\

As a continuation of our first work \cite{chaudruraynal:frikha}, our main objective here consists in revisiting and rigorously justify the mean-field approximation of \eqref{SDE:MCKEAN} by its system of particles \eqref{SDE:particle:system} under mild assumptions on the coefficients. In particular, both the drift and diffusion coefficients are assumed to be uniformly H\"older continuous with respect to the space variable and less than Lipschitz continuous with respect to the first order Wasserstein distance with respect to the measure argument. We refer the reader to Subsections \ref{sec:some:remainders} and \ref{sec:additional:regularity} for the precise statement of our regularity assumptions on the coefficients. To do so, our analysis strongly relies on the smoothing properties of the McKean-Vlasov SDE under the assumption that $a=\sigma \sigma^{*}$ is uniformly elliptic. We achieve this goal by bringing to light some new quantitative estimates of propagation of chaos for the system of particles \eqref{SDE:particle:system} at three different levels. Namely, we prove the $L^{2}(\P)$-convergence of the trajectories of $(X^{i}_t)_{t\in [0,T]}$ to its McKean-Vlasov limit dynamics. We also establish an explicit error estimate and a first order expansion for the difference between the transition densities of one particle and its limit. Eventually, we provide some convergence rate for the difference between the flow of empirical measures $(\mu^{N}_t)_{t\in [0,T]}$ of the system of particles and its limit given by the flow of probability measures $(\mu_t)_{t\in [0,T]}$ associated to the dynamics \eqref{SDE:MCKEAN} when they both act on some irregular map defined on $\pp$.\\

A natural question to be addressed before investigating the convergence problem for the system of particles \eqref{SDE:particle:system} is the well-posedness in the weak or strong sense of its mean-field limit \eqref{SDE:MCKEAN}. This problem has been intensively investigated under various settings by many authors. We refer e.g. to \cite{gartner}, \cite{Sznitman}, Jourdain \cite{jourdain:1997}, and more recently, Li and Min \cite{Li:min:2},  \cite{mishura:veretenikov} and Hammersley et al. \cite{HSSzpruch:18} for a short sample. 

In our recent contribution \cite{chaudruraynal:frikha}, we revisited the problem of the unique solvability by tackling the corresponding formulation of the martingale problem under mild regularity assumptions on the coefficients. Namely, if $a= \sigma \sigma^{*}$ is uniformly elliptic, $b$ is bounded, measurable and Lipschitz continuous in its measure variable with respect to the total variation metric, $a$ is bounded, continuous, $\eta$-H\"older continuous in space uniformly with respect to the time and measure variables and admits a bounded and $\eta$-H\"older continuous linear functional derivative then the martingale problem associated to \eqref{SDE:MCKEAN} is well-posed. Under an additional regularity assumption, namely, if $\mathbb{R}^d \times \mathcal{P}(\mathbb{R}^d) \ni (x, m)\mapsto b(t, x, m), \, a(t, x, m)$ are uniformly H\"older continuous with respect to the space variable $x$ and admit two bounded and uniformly H\"older continuous linear functional derivatives, it then turns out that the transition density $p(\mu, s, t, z)$ of the SDE \eqref{SDE:MCKEAN} at time $t$ starting from the initial distribution $\mu$ at time $s$ exists and is smooth with respect to the variables $s$ and $\mu$, the derivatives in the measure argument being understood for a stronger notion of differentiation, namely in the sense of Lions. More precisely, the map $(s, \mu) \mapsto p(\mu, s, t, z) \in \mathcal{C}^{1, 2}([0,t) \times \pp)$ (see Section \ref{wasserstein:differentiation} for a precise definition of this space). The previous regularity properties of the density finally allows to establish the existence and uniqueness of classical solutions for a class of linear parabolic PDEs on the Wasserstein space, namely
\begin{align}
\begin{cases}
(\partial_t + \mathcal{L}_t) U(t, x, \mu)  = f(t, x, \mu) & \quad \mbox{ for } (t, x , \mu) \in [0,T) \times \mathbb{R}^d \times \pp,  \\
U(T, x, \mu)   = h(x, \mu) & \quad \mbox{ for } (x , \mu) \in \mathbb{R}^d \times \pp, 
\end{cases}
 \label{pde:wasserstein:space}
\end{align}

\noindent where the source term $f : \mathbb{R}_+ \times \mathbb{R}^d \times \pp \rightarrow \mathbb{R}$ and the terminal condition $h: \mathbb{R}^d \times \pp \rightarrow \mathbb{R}$ are some given functions and the operator $\mathcal{L}_t$ acts on sufficiently smooth test functions $g: \mathbb{R}^d \times \pp \rightarrow \mathbb{R}$ and is defined by
\begin{align}
\mathcal{L}_t g(x,  \mu) & = \sum_{i=1}^d b_i(t, x, \mu) \partial_{x_i}g(x, \mu) + \frac12 \sum_{i, j=1}^d a_{i, j}(t, x, \mu) \partial^2_{x_i, x_j}g( x, \mu) \nonumber \\
& \quad + \int_{\mathbb{R}^d} \left\{ \sum_{i=1}^{d} b_i(t, v, \mu) [\partial_{\mu}g(x,\mu)]_{i}(v) + \frac12 \sum_{i, j=1}^d a_{i, j}(t, z, \mu) \partial_{v_j}[\partial_{\mu} g(x, \mu)]_i(v)  \right\} \mu(dv) \label{inf:generator:mckean:vlasov}
\end{align}

\noindent where we recall that $a(t, x, \mu) = (\sigma \sigma^{*})(t, x, \mu)$. The aforementioned well-posedness and smoothing property results for the dynamics \eqref{SDE:MCKEAN} and the PDE \eqref{pde:wasserstein:space} allow us to investigate in turn the convergence problem of the particle system \eqref{SDE:particle:system} at the three levels mentioned above within the same framework.\\

The former convergence problem of the trajectories has been thoroughly investigated under the standard framework of Lipschitz continuous coefficients $b$ and $\sigma$ over $\rr^d \times \mathcal{P}_p(\rr^d)$, $\mathcal{P}_p(\rr^d)$ being the space of probability measures with finite moment of order $p$ equipped with the Wasserstein distance $W_p$, by using the very effective and now well-known coupling argument between the solution of the system of particles \eqref{SDE:particle:system} and $N$ independent copies of the unique strong solution of the nonlinear SDE \eqref{SDE:MCKEAN} taken with the same input $(\xi^i, W^i)_{1\leq i \leq N}$. We refer to \cite{Sznitman}, L\'eonard \cite{Leonard}, M\'el\'eard \cite{meleard} for a presentation of this argument and also to Jourdain and M\'el\'eard \cite{jourdain:meleard} and Malrieu \cite{malrieu2003} for some extensions to non-linear SDEs with coefficients depending locally on its density and to granular media equations respectively.

It actually turns out to be a challenging question to go beyond the aforementioned framework by weakening the Lipschitz regularity assumption on the coefficients. Let us however mention the recent work of Holding \cite{holding} in which some quantitative propagation of chaos estimates are established for systems of interacting particles with a constant diffusion coefficient and a drift coefficient with an H\"older continuous interacting kernel of first order type, that is, $b(t, x, m) = \int_{\mathbb{R}^d} K(x, y) \, m(dy)$ with $K \in \mathcal{C}^{0,\alpha}(\mathbb{R}^d \times \mathbb{R}^d; \mathbb{R}^d)$ or $K(x, y) = W(x-y)$ with $W\in W^{s, q}(\mathbb{R}^d; \mathbb{R}^d)$ for $s$ and $q$ such that $(2+d)/q < s \leq 1$ and $q\geq 2$. Therein, an error bound for the Wasserstein distance of order $1$ between the empirical measure $(\mu^{N}_t)_{t\in [0,T]}$ of the system of particles and its mean-field limit is obtained with a convergence rate depending on the H\"older exponent $\alpha$ of the interacting kernel.\\

Our first contribution is a general rate of convergence for the $L^{2}(\P)$-error on the trajectories of the solution of the system of particles \eqref{SDE:particle:system} and $N$ independent copies of its mean-field dynamics \eqref{SDE:MCKEAN} as well as for the Wasserstein distance of order $2$ between $\mu^{N}_t$ and its corresponding limit. The main novelty here compared to the aforementioned references is that we make the approach as systematic as possible by connecting the above convergence problem to the well-posedness and the regularity properties of the solution $U$ of the backward Kolmogorov PDE \eqref{pde:wasserstein:space} with source term $f(t, x, \mu) = b(t, x, \mu)$ and terminal condition $U(T, x, \mu) \equiv h \equiv 0$. This strategy is reminiscent of Zvonkin's method for solving SDEs driven by a bounded and measurable drift \cite{zvonkin_transformation_1974}. Indeed, testing the solution $U$ on the dynamics of the system of particles notably allows to remove the drift from the convergence analysis and to achieve the expected convergence rate of the framework of Lipschitz coefficients but with weaker conditions on the drift coefficient, namely the drift is assumed to be bounded, H\"older continuous in space and with two bounded and H\"older continuous linear functional derivatives. 

Our second contribution is an error bound together with a first order expansion for the difference between the densities of the one-dimensional marginal of the system of particles and its corresponding limit. Here again, the technique of proof is based on the well-posedness of the backward Kolmogorov PDE here stated on the strip $[0,T] \times \pp$ for which we introduce and study a notion of fundamental solution. The natural candidate for being the unique classical solution is the transition density of the McKean-Vlasov SDE \eqref{SDE:MCKEAN} with initial distribution $\mu$ at time $s$, namely the map $[0,t) \times \pp \ni (s, \mu) \mapsto p(\mu, s, t, z)$, $(t, z) \in (0,T] \times \rr^d$ being fixed. By taking advantage of its regularity properties, the key idea then consists in testing the fundamental solution along the empirical measure $(\mu^{N}_s)_{s\in [0,t]}$ of the system of particles. On the one hand, the proxies $\left\{[0,t] \ni s\mapsto p(\mu^{N}_s, s, t, z), N\geq1\right\}$ should get closer and closer in average to the (constant) map $s\mapsto p(\mu_s, s, t, z) = p(\mu, 0, t, z)$ up to a remainder term that vanishes as $N$ goes to infinity. On the other hand, as $(s, \mu) \mapsto p(\mu, s, t, z)$ is the fundamental solution of the backward Kolmogorov PDE and by symmetry of the dynamics \eqref{SDE:particle:system}, $s \mapsto p(\mu^{N}_s, s, t, z)$ converges weakly to the one-dimensional marginal density function of the system of particles as $s$ goes to $t$. Combining these two facts yields our results. 

Our third contribution consists in an analysis of a weak form of propagation of chaos. Inspired by Remark 5.110 in \cite{carmona2018probabilistic}, we quantify the distance between the empirical measure of the particle system and the law of the solution of the equation both acting on a large class of irregular functions of $\mathcal{P}(\mathbb{R}^d)$. We provide an explicit error estimate for the difference between the semigroup generated by the mean field system and its approximation by the system of particles. The key tool to prove such result is very closed to the one developed to handle the previous estimates on the densities. Namely, it first consists in investigating the regularity properties of the solution to the Cauchy problem related to the backward Kolmogorov PDE stated on the strip $[0,T] \times \pp$, without source term and with a terminal condition $h:\mathcal{P}(\mathbb{R}^d) \to \mathbb{R}$ admiting two bounded and H\"older continuous linear functional derivatives and then to test such a solution along the empirical measure and the limiting law. Although we refrain to go further in that direction, we are convinced that repeating the previous strategy in order to obtain a first order expansion for the difference between the densities would lead to a first order expansion for the semigroups.\\

Taking benefit of the well-posedness of classical solutions to the backward Kolmogorov PDE on the Wasserstein space to prove the aforementioned quantitative estimates of propagation of chaos for the system of particles thus plays a central role in our analysis. Let us however mention that the strategy developed here is clearly reminiscent of the point of view taken by Cardaliaguet \& al. \cite{cardaliaguet:delarue:lasry:lions}, Mischler and Mouhot \cite{mischler:mouhot} and by Mischler, Mouhot and Wennberg in their subsequent work \cite{mischler:mouhot:wennberg}. In \cite{cardaliaguet:delarue:lasry:lions}, the authors study the convergence problem, as $N\uparrow \infty$, of the $N$-Nash system consisting of a system of $N$ coupled Hamilton-Jacobi equations. The limit equation is no longer a linear backward Kolmogorov equation but a non-linear PDE of second order type also stated on the space of probability measures, the so-called master equation of mean-field games. The strategy developed by the authors to establish their estimates of the rate of convergence consists exactly in testing the solution of the master equation as an approximate solution to the $N$-Nash system. Obviously, the very nature of our approach is the same, except that, in our case we work with a linear PDE and its fundamental solution under mild regularity conditions while in \cite{cardaliaguet:delarue:lasry:lions}, the PDE is non-linear but has smooth coefficients. The point of view expressed to establish propagation of chaos estimates for systems of particles undergoing collisions in \cite{mischler:mouhot} and for mean-field systems undergoing jumps and/or diffusions in the subsequent work \cite{mischler:mouhot:wennberg} is also very close to ours. One of the main difference being that in \cite{mischler:mouhot} the quantitative estimates are uniform in time while ours are established on a finite time horizon. Moreover, in \cite{mischler:mouhot:wennberg}, the authors directly compares the semigroup generated by the system of particles and the \emph{lifted} one, that is, the one generated by the mean-field limit both acting on symmetric functions on $(\rr^d)^N$ while in our case we work at the level of the densities. An error bound of order $N^{-1/2}$ for the total variation distance between $k$ particles and $k$ independent copies of the mean-field limit for non-linear SDEs with a constant diffusion coefficient and a drift with general and singular interacting kernel of first order type has been established in Jabin and Wang \cite{jabin}. We also refer to the book of Kolokolstov \cite{Kolokoltsov} and to the work by Kolokoltsov, Troeva and Yang \cite{kolokolstov:troeva:yang} for a point of view based on measure-valued Markov processes and some quantitative estimates for mean-field games approximation. Let us finally mention the recent work of Chassagneux, Szpruch and Tse \cite{chassagneux:szpruch:tse} where an expansion for the difference $\E[h(\mu^{N}_t)] - h(\mu_t)$, $t\in  [0,T]$, is established by exploiting the well-posedness and the regularity properties of the backward Kolmogorov PDE \eqref{pde:wasserstein:space} (with $f \equiv 0$) stated on $[0,T] \times \pp$ in the spirit of Buckdhan \& al. \cite{buckdahn2017}, under the assumptions that $h$, $b$ and $\sigma$ are smooth functions of the space and measure variables.\\

The article is organized as follows. The basic notions of differentiation on the Wasserstein space with an emphasis on the smoothing properties of McKean-Vlasov SDEs that will play a key role in our analysis are presented in Section \ref{diff:wasserstein:structural:class}. The general set-up together with the assumptions and the main results are described in Section \ref{assumptions:results}. The proof of the existence and uniqueness of the fundamental solution of the backward Kolmogorov PDE on the Wasserstein space together with some additional regularity properties of the transition density associated to \eqref{SDE:MCKEAN} are addressed in Section \ref{fully:c2:regularity:section}. The propagation of chaos estimates are established in Section \ref{propagation:of:chaos}. The proof of some useful but auxiliary technical results are given in Appendix.

\subsection*{Notations:}
In the following we will denote by $C$ and $K$ some generic positive constants that may depend on the coefficients $b$ and $\sigma$. We reserve the notation $c$ for constants depending on $|\sigma|_\infty$, $\lambda$ (see assumption \HE\, in Section \ref{assumptions:results}) and possibly on $N$ in which case we write $c(N)$ but not on the time horizon $T$. Moreover, the value of both $C,\, K$ or $c$ may change from line to line.  

We will denote by $\mathcal{P}(\rr^d)$ the space of probability measures on $\rr^d$ and by $\pp \subset \mathcal{P}(\rr^d)$ the space of probability measures with finite second moment. For $\mu \in \mathcal{P}(\rr^d)$ and $q>0$, we set $M_q(\mu) := (\int_{\mathbb{R}^d} |x|^q \mu(dx))^{1/q}$ if $\int_{\mathbb{R}^d} |x|^q \mu(dx) < +\infty$ and $M_q(\mu) = + \infty$ otherwise.

For a positive variance-covariance matrix $\Sigma$, the function $y\mapsto g(\Sigma , y)$ stands for the $d$-dimensional Gaussian kernel with $\Sigma$ as covariance matrix $g(\Sigma, x) = (2\pi)^{-\frac{d}{2}} (\d \, \Sigma)^{-\frac12} \exp(-\frac12 \langle \Sigma^{-1} x, x \rangle)$. We also define the first and second order Hermite polynomials: $H^{i}_1(\Sigma, x) := -(\Sigma^{-1} x)_{i}$ and $H^{i, j}_2(\Sigma, x) := (\Sigma^{-1} x)_i (\Sigma^{-1} x)_j - (\Sigma^{-1} )_{i, j}$, $1\leq i, j \leq d$ which are related to the previous Gaussian density as follows $\partial_{x_i} g(\Sigma, x) = H^{i}_1(\Sigma, x) g(\Sigma, x)$, $\partial^2_{x_i, x_j} g(\Sigma, x) = H^{i, j}_2(\Sigma, x) g(\Sigma, x)$. 
%We will sometimes use the following relations: $\partial_{\Sigma} g(\Sigma, x) = -\frac12 (H_2.g)(\Sigma, x)$, $\partial^2_{\Sigma} g(\Sigma, x) = \frac14 (H_4.g)(\Sigma, x)$ where $H_2(\Sigma, x) = (H^{i, j}_2(\Sigma, x))_{1\leq i, j \leq d}$ and $H_4(\Sigma, x) = (H^{i, j, k, l}_4(\Sigma, x))_{1\leq i, j, k, l \leq d}$ satisfies $\partial^4_{x_i, x_j, x_k, x_l} g(\Sigma, x) = H^{i, j, k, l}_4(\Sigma, x) g(\Sigma, x)$. 
Also, when $\Sigma= c I_d$, for some positive constant $c$, the latter notation is simplified to $g(c, x) := (1/(2\pi c))^{d/2} \exp(-|x|^2/(2c))$. 

One of the key inequality that will be used intensively in this work is the following: for any $p,q>0$ and any $x\in \mathbb{R}$, $|x|^p e^{-q x^2} \leq (p/(2qe))^{p/2}$. 
As a direct consequence, we obtain the {\it space-time inequality}, for any $p,\, c, \, t>0$ and any $c'>c$ there exists $C>0$ such that
\begin{gather}
 |x|^p g(c t, x)\leq C t^{p/2} g(c' t, x) \label{space:time:inequality}
\end{gather}
\noindent which in turn gives {the \it standard Gaussian estimates} for the following derivatives of Gaussian density
\begin{align}
 |H^{i}_1(c t, x)| g(c t, x) \leq \frac{C}{t^{\frac12}} g(c' t,x), \quad |H^{i, j}_2(c t, x)| g(c t, x) \leq \frac{C}{t}g(c' t, x). \label{standard}
\end{align}

Since we will employ it quite frequently, we will often omit to mention it explicitly at some places. We finally define the Mittag-Leffler function $E_{\alpha,\beta}(z) := \sum_{n\geq 0} z^{n}/\Gamma(\alpha n +\beta)$, $z\in \mathbb{R}$, $\alpha,\ \beta>0$.

\section{Preliminaries: Differentiation on the Wasserstein space and smoothing properties}\label{diff:wasserstein:structural:class}

\subsection{Differentiation on the Wasserstein space}\label{wasserstein:differentiation}
In this section, we briefly present the regularity notions we will use when working with mappings defined on $\mathcal{P}_2(\rr^d)$. We refer the reader to Lions' seminal lectures \cite{lecture:lions:college}, to Cardaliaguet's lectures notes \cite{cardaliaguet}, to the recent work Cardaliaguet et al. \cite{cardaliaguet:delarue:lasry:lions} or to Chapter 5 of Carmona and Delarue's monograph \cite{carmona2018probabilistic} for a more complete and detailed exposition. Unless otherwise specified, we equip the space $\mathcal{P}(\rr^d)$ with the topology induced by the total variation metric $d_{{\rm TV}}$ defined by
$$
d_{{\rm TV}}(\mu, \nu) = \sup_{A \in \mathcal{B}(\rr^d)}\int_{A} (\mu-\nu)(dx) .
$$

The space $\pp$ is equipped with the 2-Wasserstein metric
$$
W_2(\mu, \nu) = \inf_{\pi \in \mathcal{P}(\mu, \nu)} \left( \int_{\rr^d \times \rr^d} |x-y|^2 \, \pi(dx, dy) \right)^{\frac12}
$$

\noindent where, for given $\mu, \nu \in \pp$, $\mathcal{P}(\mu, \nu)$ denotes the set of measures on $\rr^d \times \rr^d$ with marginals $\mu$ and $\nu$.

Following our recent work \cite{chaudruraynal:frikha}, we will employ two notions of differentiation of a continuous map $U$ defined on $\mathcal{P}(\mathbb{R}^d)$. The first one, called the \emph{linear functional derivative} and denoted by $[\delta U/\delta m]$, will play an important role in our linearization procedure to strengthen the regularity properties of the transition densities of the McKean-Vlasov SDE \eqref{SDE:MCKEAN} and its corresponding decoupling field. \\

\noindent\textbf{Linear functional derivative.} 
\begin{definition}\label{continuous:L:derivative}
The continuous map $U: \mathcal{P}(\mathbb{R}^d) \rightarrow \mathbb{R}$ is said to have a linear functional derivative if there exists a real-valued bounded measurable function
$$
 \mathcal P(\mathbb{R}^d) \times \mathbb{R}^d  \ni (m, x) \mapsto \frac {\delta U}{\delta m}(m)(x) \in \mathbb{R},
$$
\noindent such that for all $x$ in $\mathbb{R}^d$, the map $\mathcal{P}(\mathbb{R}^d) \ni m \mapsto [\delta U/\delta m](m)(x)$ is continuous and such that for all $m$ and $m'$ in $\mathcal{P}(\mathbb{R}^d)$, it holds
\begin{equation}
\label{limit:equation:linear:functional:derivative}
\lim_{\varepsilon \downarrow 0} \frac{U((1-\varepsilon) m + \varepsilon m') - U(m)}{\varepsilon} = \int_{\mathbb{R}^d} \frac{\delta U}{\delta m}(m)(y) \, d(m'-m)(y).
\end{equation}
The map $y\mapsto [\delta U/\delta m](m)(y)$ being defined up to an additive constant, we will follow the usual normalization convention $\int_{\mathbb{R}^d} [\delta U/\delta m](m)(y) \, dm(y) = 0$. Observe from the above definition that for all $m$ and $m'$ in $\mathcal{P}(\mathbb{R}^d)$
\begin{equation}
\label{mean:value:theorem:flat:derivative}
 U(m')-U(m) = \int_0^1 \int_{\rr^d} \frac{\delta U}{\delta m}((1-\lambda) m + \lambda m')(y) \, d(m'-m)(y)\, d\lambda.
\end{equation}
Note that the boundedness assumption of the map $x\mapsto [\delta U/\delta m](m)(x)$, uniformly in $m$ guarantees the well-posedness of the integral appearing in the right-hand side of \eqref{mean:value:theorem:flat:derivative}.
\end{definition}

\begin{remark}
\label{rem:linderivandLip}
If a map $U$ admits a flat derivative in the above sense then one directly deduces that 
 for all $m$ and $m'$ in $\mathcal{P}(\mathbb{R}^d)$
\begin{equation}\label{eq:linFuncDerivtoLipdtv}
 |U(m) - U(m')| \leq \sup_{m'' \in \mathcal{P}(\mathbb{R}^d)} |\frac{\delta U}{\delta m}(m'')(.)|_\infty\, d_{{\rm TV}}(m, m').
\end{equation}

Therefore, if the map $U$ admits a linear functional derivative in the sense of Definition \ref{continuous:L:derivative} then it is Lipschitz continuous with respect to the total variation metric.

\end{remark}

%It is then readily seen that if $(m, y) \mapsto \frac{\delta U}{\delta m}(m)(y)$ is bounded, then one has
%$$
%\forall m, m' \in \pp, \quad |U(m) - U(m')| \leq \sup_{m'' \in \pp} \|\frac{\delta U}{\delta m}(m'')(.)\|_\infty\, d_{TV}(m, m')
%$$
%
%\noindent where $d_{TV}$ is the total variation metric so that $U$ is Lipschitz continuous with respect to this distance. If, $y\mapsto \frac{\delta U}{\delta m}(m)(y)$ is Lipschitz continuous, with a Lipschitz modulus bounded uniformly with respect to the variable $m$, then from the Monge-Kantorovich duality principle
%$$
%\forall m, m' \in \pp, \quad |U(m) - U(m')| \leq \sup_{m''} \| \partial_y [\frac{\delta U}{\delta m}(m'')(.)]\|_{\infty} W_1(m, m').
%$$
%

We will also work with higher order derivatives. This is naturally defined by induction as follows.

\begin{definition}\label{continuous:L:derivative:higher:order}Let $p\geq1$. The continuous map $U: \mathcal{P}(\mathbb{R}^d) \rightarrow \mathbb{R}$ is said to have a continuous linear functional derivative at order $p$ if there exists a real valued bounded measurable map $[\delta^{p} U/\delta m^{p}]: \mathcal{P}(\mathbb{R}^d) \times (\mathbb{R}^d)^{p-1}\times \mathbb{R}^d \rightarrow \mathbb{R}$ such that for all $(\bold{y}_{p-1}, y_p) \in (\mathbb{R}^d)^{p-1}\times \mathbb{R}^d$, the map $\mathcal{P}(\mathbb{R}^d) \ni m \mapsto [\delta^{p} U/\delta m^{p}](m)(\bold{y}_{p-1}, y_p)$ is continuous and such that for any $m, m' \in \mathcal{P}(\mathbb{R}^d)$ and for any $\bold{y}_{p-1}\in (\mathbb{R}^d)^{p-1}$
$$
\lim_{\varepsilon \downarrow 0} \varepsilon^{-1} \Big(\frac{\delta^{p-1}}{\delta m^{p-1}}U((1-\varepsilon) m + \varepsilon m')(\bold{y}_{p-1}) - \frac{\delta^{p-1}}{\delta m^{p-1}}U(m)(\bold{y}_{p-1})\Big) = \int_{\rr^d} \frac{\delta^{p} U}{\delta m^{p}}(m)(\bold{y}_{p}) \, d(m'-m)(y_p)
$$
\noindent provided the $(p-1)$th order derivative is well-defined, with the notation $\bold{y}_p := (\bold{y}_{p-1}, y_p)$ and the convention $[\delta^{0}/\delta m^0]U(m) \equiv U(m)$. We again follow the usual normalization convention which ensures uniqueness 
$$
\int_{\mathbb{R}^d} \frac{\delta^{p} U}{\delta m^p}(m)(\bold{y}_{p-1}, y_p) \, dm(y_p) = 0
$$
\noindent for any $\bold{y}_{p-1} \in (\mathbb{R}^d)^{p-1}$.
\end{definition}

%With the above definition in mind, one may again investigate the smoothness of $m \mapsto \frac{\delta U}{\delta m}(m)(y)$ for a fixed $y \in \rr^d$. We will say that $U$ has two continuous linear functional derivative and denote $\frac{\delta^2 U}{\delta m^2}(m)(y)$ its second derivative taken at $(m, y)$ if $m \mapsto \frac{\delta U}{\delta m}(m)(y)$ has a continuous linear functional derivative in the sense of Definition \ref{continuous:L:derivative}. As a consequence,  
%$$
%\forall m, m' \in \pp, \quad \frac{\delta U}{\delta m}(m')(y) -  \frac{\delta U}{\delta m}(m)(y) = \int_0^1 \int_{\rr^d} \frac{\delta^2 U}{\delta m^2}( (1-\lambda) m +\lambda m')(y, y') \, d(m' -m)(y')\ d\lambda
%$$
%
%\noindent and if $\pp\times (\rr^d)^2 \ni (m, y, y') \mapsto \frac{\delta^2 U }{\delta m^2}(m)(y, y')$ is continuous then $\frac{\delta^2 U }{\delta m^2}(m)(y , y') = \frac{\delta^2 U }{\delta m^2}(m)(y' , y)$ for all $(m, y, y') \in \pp \times (\rr^d)^2$. 

Again, for more details on the above notion of derivative, we refer to Chapter 5 of \cite{carmona2018probabilistic}.

 \medskip
 
We now briefly present the second notion of derivatives as originally introduced by Lions \cite{lecture:lions:college}. The basic strategy consists in considering the canonical lift of the real-valued function $U : \pp \ni \mu \mapsto U(\mu)$ into a function $\mathcal{U} : \mathbb{L}_2  \ni Z \mapsto \mathcal{U}(Z) = U([Z]) \in \rr$, $(\Omega, \mathcal{F}, \P)$ standing for an atomless probability space, with $\Omega$ a Polish space, $\mathcal{F}$ its Borel $\sigma$-algebra, $\mathbb{L}_2:=\mathbb{L}_2(\Omega,\mathcal{F},\mathbb{P}, \rr^d)$ standing for the space of $\rr^d$-valued random variables defined on $\Omega$ with finite second moment and $Z$ being a random variable with law $\mu$. The function $U$ is then said to be differentiable at $\mu\in \pp$ if its canonical lift $\mathcal{U}$ is Fr\'{e}chet differentiable at some point $Z$ such that $[Z]=\mu$. We denote by $D\mathcal{U}$ its gradient. The Riesz representation theorem then allows to identify $D\mathcal{U}$ as an element of $\mathbb{L}^2$. It turns out that $D\mathcal{U}$ is a random variable which is $\sigma(Z)$-measurable and given by a function $DU(\mu)(.)$ from $\rr^d$ to $\rr^d$, which depends on the law $\mu$ of $Z$ and satisfying $DU(\mu)(.) \in \mathbb{L}^2(\rr^d, \mathcal{B}(\rr^d), \mu; \rr^d)$. As in \cite{chaudruraynal:frikha}, we adopt the notation $\partial_\mu U(\mu)(.)$ in order to emphasize that we are taking the derivative of the map $U$ with respect to its measure argument. The $L$-derivative of $U$ at $\mu$ is the map $\partial_\mu U(\mu)(.): \rr^d \ni v \mapsto \partial_\mu U(\mu)(v) \in \rr^d$, satisfying $D \mathcal{U} = \partial_\mu U(\mu)(Z)$. 

It is important to note that this representation holds irrespectively of the choice of the original probability space $(\Omega, \mathcal{F}, \P)$. We will restrict our considerations to functions which are $\mathcal{C}^1$, that is, functions for which the associated canonical lift is $\mathcal{C}^1$ on $\mathbb{L}^2$ and for which there exists a continuous version of the mapping $\pp \times \rr^d \ni (\mu, v) \mapsto \partial_\mu U(\mu)(v) \in \rr^d$. It then appears that this version is unique. We straightforwardly extend the above discussion to $\rr^d$-valued or $\rr^{d \times d}$-valued maps $U$ defined on $\pp$, component by component.

%Motivated by the chain rule formula for functions of the form $U(\mu_t)_{t\geq0}$, $U: \pp^2 \rightarrow \rr$, where $(\mu_t)_{t\geq0}$ is the flow of probability measures of $\pp$ generated by the marginal distributions of an $\rr^d$-valued It\^o process $X=(X_t)_{t\geq0}$, we introduce the notion of \emph{partial $\mathcal{C}^2(\pp)$-regularity} as exposed in  Chapter 5 of \cite{carmona2018probabilistic}.
% 

In order to establish the existence and uniqueness of a fundamental solution of the Kolmogorov PDE on the Wasserstein space as well as our quantitative estimates for the mean-field approximation by systems of particles, we will employ at several places a chain rule formula for $(U(t, Y_t, [X_t]))_{t\geq0}$, where $(X_t)_{t\geq0}$ and $(Y_t)_{t\geq 0}$ are two It\^o processes defined for sake of simplicity on the same probability space $(\Omega, \mathcal{F}, \mathbb{F}, \mathbb{P})$ assumed to be equipped with a right-continuous and complete filtration $\mathbb{F}=(\mathcal{F}_t)_{t\geq0}$. Their dynamics are given by
\begin{align}
X_t &= X_0 + \int_0^t b_s \, ds + \int_0^t \sigma_s \, dW_s, \, X_0 \in \mathbb{L}_2, \label{ito:process:X} \\
Y_t & = Y_0 + \int_0^t \eta_s \, ds + \int_0^t \gamma_s \, dW_s \label{ito:process:Y}
\end{align}

\noindent where $W=(W_t)_{t\geq0}$ is an $\mathbb{F}$-adapted $d$-dimensional Brownian, $(b_t)_{t\geq 0}$, $(\eta_t)_{t\geq0}$, $(\sigma_t)_{t\geq 0}$ and $(\gamma_t)_{t\geq 0}$ are $\mathbb{F}$-progressively measurable processes, with values in $\rr^d$, $\rr^d$, $\rr^{d \times d}$ and $\rr^{d\times q}$ respectively, satisfying the following conditions 
\begin{equation}
\label{cond:integrab:ito:process}
\forall T>0, \quad \E\Big[\int_0^T (|b_t|^2 + |\sigma_t|^4) \, dt\Big] < \infty \quad \mbox{ and } \quad \P\left(\int_0^T (|\eta_t| + |\gamma_t|^2) \, dt < + \infty \right) = 1.
\end{equation}

We now introduce two spaces of smooth functions we will work with throughout the paper.

\begin{definition}(The spaces $\mathcal{C}^{p, 2, 2}([0,T] \times \rr^d \times \pp)$ and $\mathcal{C}^{p, 2, 2}_f([0,T] \times \rr^d \times \pp)$, for $p=0, \,1$)\label{def:space:c122}
Let $T>0$ and $p \in \left\{0, 1\right\}$.

The continuous function $U : [0,T] \times \rr^d \times \pp$ is in $\mathcal{C}^{p, 2, 2}([0,T] \times \rr^d \times \pp)$ if the following conditions hold:
\begin{itemize}
\item[(i)] For any $\mu \in \mathcal{P}_2(\rr^d)$, the mapping $[0,T] \times \rr^d \ni (t,x) \mapsto U(t, x,\mu)$ is in $\mathcal{C}^{p,2}([0,T] \times \rr^d)$ and the functions $[0,T] \times \rr^d \times \pp \ni (t, x, \mu) \mapsto \partial^{p}_t U(t, x, \mu),\, \partial_x U(t, x, \mu),\, \partial_{x}^2 U(t, x, \mu)$ are continuous.

\item[(ii)] For any $(t, x)\in [0,T] \times \rr^d$, the mapping $\mathcal{P}_2(\rr^d) \ni \mu \mapsto U(t ,x, \mu)$ is continuously L-differentiable and for any $\mu \in \mathcal{P}_2(\rr^d)$, we can find a version of the mapping $\rr^d \ni v \mapsto \partial_\mu U(t, x,\mu)(v)$ such that the mapping $[0,T]\times \rr^d \times \mathcal{P}_2(\rr^d) \times \rr^d \ni (t ,x, \mu, v) \mapsto \partial_\mu U(t, x, \mu)(v)$ is locally bounded and is continuous at any $(t, x, \mu, v)$ such that $v \in \supp(\mu)$.

\item[(iii)] For the version of $\partial_\mu U$ mentioned above and for any $(t, x,\mu)$ in $[0,T] \times \rr^d \times \mathcal{P}_2(\rr^d)$, the mapping $\rr^d \ni v \mapsto \partial_\mu U(t, x,\mu)(v)$ is continuously differentiable and its derivative $\partial_v [\partial_\mu U(t, x, \mu)](v) \in \rr^{d\times d}$ is jointly continuous in $(t, x, \mu, v)$ at any point $(t, x, \mu, v)$ such that $v \in \supp(\mu)$.
\end{itemize}

The continuous function $U : [0,T] \times \rr^d \times \pp$ is in $\mathcal{C}^{p, 2, 2}_f([0,T] \times \rr^d \times \pp)$ if $U \in \mathcal{C}^{p, 2, 2}([0,T] \times \mathbb{R}^d \times \pp)$ in the above sense and the following additional condition holds:
\begin{itemize}
\item[(iv)] For each $v\in \rr^d$, the version $\pp \ni \mu \mapsto \partial_\mu U(t, x, \mu)(v)$ discussed in (ii) is L-differentiable (component by component) with a derivative given by $(\mu, v, v')\mapsto \partial^2_\mu U(t, x, \mu)(v)(v') \in \rr^{d\times d}$ such that for any $\mu \in \pp$ and $X \in \mathbb{L}_2$ with $[X]=\mu$, the $\mathbb{R}^{d\times d}$-valued random variable $\partial^2_\mu U(t, x, \mu)(v)(X)$ gives the Fr\'echet derivative of the map $\mathbb{L}_2\ni X'\mapsto \partial^2_\mu U(t, x, [X'])(v)$ for every $v\in \rr^d$. Denoting $\partial^2_\mu U(t, x, \mu)(v)(v')$ by $\partial^2_\mu U(t, x, \mu)(v, v')$, the map $[0,T] \times \rr^d \times \pp \times (\rr^d)^2 \ni (t, x, \mu, v, v')\mapsto \partial^2_\mu U(t, x, \mu)(v, v')$ is also assumed to be continuous for the product topology.
\end{itemize}
\end{definition}

\begin{remark}\label{space:restriction} We will also consider the spaces $\mathcal{C}^{1,p}([0,T] \times \pp)$ for $p=1, \, 2$ and $\mathcal{C}^{1,2}_f([0,T] \times \pp)$, where we adequately remove the space variable in the Definition \ref{def:space:c122}. We will say that $U \in \mathcal{C}^{1, 1}([0,T] \times \pp)$ if $U$ is continuous, $t\mapsto U(t,\mu) \in \mathcal{C}^{1}([0,T])$ for any $\mu \in \pp$, $(t,\mu) \mapsto \partial_t U(t, \mu)$ being continuous and if for any $t\in [0,T]$, $\mu \mapsto U(t,\mu)$ is continuously L-differentiable such that we can find a version of $v\mapsto \partial_\mu U(t, \mu)(v)$ satisfying: $(t,\mu, v) \mapsto \partial_\mu U(t, \mu)(v)$ is locally bounded and continuous at any $(t, \mu, v)$ satisfying $v\in \supp(\mu)$. 

We will say that $U\in \mathcal{C}^{1,2}_f([0,T] \times \pp)$ if $U \in \mathcal{C}^{1,2}([0,T] \times \pp)$ and for the version of $\partial_\mu U$ previously considered, for any $(t, v) \in [0,T] \times \rr^d$, the mapping $\pp \ni \mu \mapsto \partial_\mu U(t, \mu)(v)$ is L-differentiable with a derivative given by $(t ,\mu, v, v') \mapsto \partial_\mu U(t, \mu)(v, v') \in \rr^{d\times d}$ such that for any $\mu \in \pp$ and $X \in \mathbb{L}_2$ with $[X]=\mu$, $\partial_\mu U(t, \mu)(v, X)$ gives the Fr\'echet derivative of the map $\mathbb{L}_2\ni X'\mapsto \partial_\mu U(t, [X'])(v)$ for every $(t, v)\in [0,T] \times \rr^d$. Moreover, the map $[0,T]  \times \pp \times (\rr^d)^2 \ni (t, \mu, v, v')\mapsto \partial^2_\mu U(t, \mu)(v, v')$ is assumed to be continuous for the product topology. 

\end{remark}

\noindent \textbf{Notations:} We will use the following notations throughout the paper. For a smooth map $U:\mathcal{P}_2(\rr^d)\rightarrow \rr$ and for $\mu \in \pp$, $v, v' \in \rr^d$ 
\begin{align*}
\partial_v[\partial_\mu U(\mu)(v)] & = (\partial_{v_j} [\partial_\mu U(\mu)]_i(v))_{1\leq i, j \leq d}\\
\partial^2_\mu U(\mu)(v, v') & = ([\partial_\mu [\partial_\mu U(\mu)]_i(v)]_j(v'))_{1\leq i, j \leq d}.
\end{align*}

With the above definitions and notations, we can now provide the chain rule formula on the Wasserstein space that will be play a central role in our analysis.

\begin{prop}[\cite{carmona2018probabilistic}, Proposition 5.102]\label{prop:chain:rule:joint:space:measure}
Let $X$ and $Y$ be two It\^o processes, with respective dynamics \eqref{ito:process:X} and \eqref{ito:process:Y}, satisfying \eqref{cond:integrab:ito:process}. Assume that $U\in \mathcal{C}^{1, 2, 2}([0,T]\times \rr^d \times \pp)$ in the sense of Definition \ref{def:space:c122} such that for any compact set $\mathcal{K}\subset \rr^d \times \pp$, 
\begin{equation}
\label{cond:integrab:ito:process:second:version}
\sup_{(t, x, \mu) \in [0,T] \times \mathcal{K}}\left\{\int_{\rr^d} | \partial_\mu U(t, x, \mu)(v)|^2 \, \mu(dv) +  \int_{\rr^d} |\partial_v [\partial_\mu U(t, x, \mu)](v)|^2 \, \mu(dv) \right\} < \infty.
\end{equation}

%If $(Y_t)_{t\in [0,T]}$ is a $d$-dimensional It\^o process defined on the same filtered probability space $(\Omega, \mathcal{F},\mathbb{F}, \P)$ with dynamics $Y_t = Y_0 + \int_0^t \eta_s \, ds + \int_0^t \gamma_s \, dW_s$, for two $\mathbb{F}$-progressively measurable processes $(\eta_t)_{t\in[0,T]}$ and $(\gamma_t)_{t\in [0,T]}$ with values in $\rr^d$ and $\rr^{d\times q}$ respectively, such that
%$$
%\P\left(\int_0^T (|\eta_t| + |\gamma_t|^2) \, dt < + \infty \right) = 1
%$$

 Then, $\P$-a.s., $\forall t\in [0,T]$, one has
\begin{align}
  U(t, Y_t, [X_t]) & = U(0, Y_0, [X_0]) + \int_0^t \partial_x U(s, Y_s, [X_s]) \, . \gamma_s \, dW_s  \nonumber \\
&   + \int_0^t \left\{\partial_s U(s, Y_s, [X_s]) +  \partial_x U(s, Y_s, [X_s]) . \eta_s + \frac12 \tr(\partial^2_{x} U(s, Y_s, [X_s]) \gamma_s \gamma^{*}_s) \right\} ds \label{chain:rule:mes}\\
&   + \int_0^t \left\{ \widetilde{\E}\big[\partial_\mu U(s , Y_s, [X_s])(\widetilde{X}_s) . \widetilde{b}_s\big] + \frac12 \widetilde{\E}\big[\tr(\partial_v [\partial_\mu U(s, Y_s, [X_s])](\widetilde{X}_s)  \widetilde{\sigma}_s \widetilde{\sigma}_s^{*})\big] \right\} ds \nonumber
\end{align}

\noindent where the It\^o process $(\widetilde{X}_t, \widetilde{b}_t, \widetilde{\sigma}_t)_{0\leq t \leq T}$ is a copy of the original process $(X_t, b_t, \sigma_t)_{0\leq t \leq T}$ defined on a copy $(\widetilde{\Omega}, \widetilde{\mathcal{F}}, \widetilde{\P})$ of the original probability space $(\Omega, \mathcal{F}, \P)$.
\end{prop}

We conclude this subsection by enlightening the connection between the L-derivative of a map $U: \pp \rightarrow \rr$ and the standard partial derivatives of its empirical projection $U^{N}: (\rr^d)^N \rightarrow \rr$, $N$ being a positive integer, defined by 
\begin{equation}
U^{N}: (\rr^d)^{N} \ni (x_1, \cdots, x_N) \mapsto U\Big(\frac{1}{N}\sum_{i=1}^{N} \delta_{x_i}\Big).
\end{equation}

We refer to Propositions 5.35 and 5.91 of \cite{carmona2018probabilistic} for a proof of the following result.

\begin{prop}[Connection between L-derivatives and empirical projection]\label{empirical:projection:prop}
If $U$ is a real-valued function that belongs to $\mathcal{C}^{2}_f(\pp)$ (fully $\mathcal{C}^2$) then its empirical projection $U^{N}$ is two times differentiable on $(\mathbb{R}^d)^{N}$ and, for all $x_1, \cdots, x_N \in (\mathbb{R}^d)^N$, for all $(i, j)\in \left\{1,\cdots, N\right\}^{2}$
$$
\partial_{x_i} U^{N}(x_1, \cdots, x_N) = \frac{1}{N}\partial_\mu U(\frac{1}{N}\sum_{\ell=1}^{N} \delta_{x_\ell})(x_i)
$$
\noindent and
$$
\partial_{x_i, x_j} U^{N}(x_1, \cdots, x_N) = \frac{1}{N} \partial_v\partial_\mu U(\frac{1}{N}\sum_{\ell=1}^{N} \delta_{x_\ell})(x_i) \delta_{i, j} + \frac{1}{N^{2}}\partial^2_\mu U(\frac{1}{N}\sum_{\ell=1}^{N} \delta_{x_\ell})(x_i, x_j)
$$

\noindent with the notation $\delta_{i, j} = 1 $ if $i=j$ and $\delta_{i, j} = 0$ otherwise. 

\end{prop}

\subsection{Regularization properties by smooth flow of probability measures}\label{subsec:structural} 

In order to establish quantitative estimates for propagation of chaos for the mean field approximation of the dynamics \eqref{SDE:MCKEAN} by its system of particles \eqref{SDE:particle:system}, an additional study of the regularity properties of the transition density of the McKean-Vlasov SDE is required. We build on our previous work \cite{chaudruraynal:frikha} which highlights the key feature to investigate the smoothing properties of the transition density in the uniformly elliptic framework. Namely, our analysis is mainly based on how a continuous map defined on $\mathcal{P}(\mathbb{R}^d)$ admitting only flat derivatives in the sense of Definition \ref{continuous:L:derivative} can be regularized in \emph{the intrinsic sense} by a smooth flow of probability measures. Assuming that the coefficients $b$ and $a$ are uniformly H\"older continuous in the space variable and admit bounded, uniformly H\"older continuous linear functional derivatives at order $2$, it turns out that the density taken at time $t>0$ of the unique weak solution of the McKean-Vlasov SDE with dynamics \eqref{SDE:MCKEAN} achieves better regularity with respect to its measure argument and is \emph{partially $\mathcal{C}^2$}. Clearly, this phenomenon has to be understood as a smoothing property of McKean-Vlasov SDEs in a uniformly elliptic setting. We refer to Section 2.2 in \cite{chaudruraynal:frikha} for a detailed introduction and discussion of this regularization property. 

We here want to go one step further by analyzing the \emph{full $\mathcal{C}^2$ regularity} of the density. The following result will play central role in our analysis.

\begin{prop}\label{structural:class} Let $h: \mathcal{P}(\mathbb{R}^d) \rightarrow \mathbb{R}$ be a continuous map that admits two bounded linear functional derivatives. For some prescribed $T>0$ and $z\in \mathbb{R}^d$, consider a map $ (t, x, \mu) \mapsto p(\mu, t, T, x, z) \in \mathcal{C}^{1,2, 2}_f([0,T)\times \mathbb{R}^d \times \pp)$, $z\mapsto p(\mu, t, T, x, z)$ being a density function, such that the probability measure given by $(p(\mu, t, T, ., z) \sharp\mu) dz$ belongs to $\pp$, locally uniformly with respect to $(t, \mu) \in [0,T) \times \pp$, i.e. uniformly in $(t, \mu)$ in bounded subsets of $[0,T) \times \pp$. Assume that for all $(\mu, t, x, z, v)\in \pp \times [0,T) \times (\mathbb{R}^d)^3$ the maps $\mathbb{R}^d \ni x\mapsto \partial_\mu p(\mu, t, T, x, z)(v)$ and $\pp \ni \mu\mapsto  \partial_x p(\mu, t, T, x, z)$ are continuously differentiable, with derivatives $\partial_x \partial_\mu p(\mu, t, T, x, z)(v)$, $ \partial_\mu \partial_x p(\mu, t, T, x, z)(v)$ being continuous in $\mu, x, v$, and of at most linear growth in $v$, uniformly in $(\mu, x)$ in bounded subsets of $\pp \times \mathbb{R}^d$\footnote{In this case, by Clairaut's theorem it holds $\partial_x \partial_\mu p(\mu, t, T, x, z)(v) =  (\partial_\mu \partial_x p(\mu, t, T, x, z)(v))^{*}$ for all $(\mu, x, v) \in \pp \times (\mathbb{R}^d)^2$. } and that the mappings $\mathbb{R}^d \ni x\mapsto  \int_{\mathbb{R}^d} |\partial^2_\mu p(\mu, t, T, x, z)(v, v')| \, dz$, $\int_{\mathbb{R}^d} |\partial^{n}_v[\partial_{\mu} p(\mu, t, T, x, z)](v)| \, dz$, $\int_{\mathbb{R}^d} |\partial_t p(\mu, t, T, x, z)| \, dz$, $\int_{\mathbb{R}^d} |\partial_x \partial_\mu p(\mu, t, T, x, z)(v)| \, dz$, $n \in \left\{0, 1\right\}$, are at most of quadratic growth, uniformly in $(t, \mu, v, v')$ in bounded subsets of $[0,T) \times \pp \times (\mathbb{R}^d)^2$ and such that for any bounded subset $\mathcal{K}' \subset [0,T) \times \pp \times (\mathbb{R}^d)^3$, for any $n \in \left\{ 0, 1\right\}$
\begin{eqnarray}
\label{integrability:condition}
&&\int_{\mathbb{R}^d} \sup_{(t, \mu, x, v, v') \in \mathcal{K}'}  \Big\{|\partial^{n}_t p(\mu, t, T, x, z)| + | \partial^{1+n}_x p(\mu, t, T, x, z)|  \notag\\
&& \qquad + |\partial^{n}_{v}[\partial_\mu p(\mu, t, T, x, z)](v)| + |\partial_x \partial_\mu p(\mu, t, T, x, z)(v)| + |\partial^2_\mu p(\mu, t, T, x, z)(v, v')|  \Big\} \, dz < \infty. 
\end{eqnarray}

\noindent Consider the map $\Theta: [0,T) \times \pp \rightarrow \mathcal{P}_2(\mathbb{R}^d)$ defined by 
$$
\Theta(t, \mu)(dz) = (p(\mu, t, T, ., z) \sharp\mu)(dz) = \int_{\mathbb{R}^d} p(\mu, t, T, x, z) \mu(dx) \, dz.
$$
 Then, the following statements hold:
\begin{itemize}
\item the map $ h( \Theta(., .))$ belongs to $\mathcal{C}^{1,2}_f([0,T) \times \pp)$,
\item the Lions and time derivatives satisfy for $n \in \left\{0,1\right\}$: 
\begin{align}
\partial^{n}_v[\partial_{\mu} [h(\Theta(t, \mu))]](v) & = \partial^{n}_v\Big[\partial_{\nu} \Big[\int_{(\mathbb{R}^d)^2} \frac{\delta h}{\delta m} (\Theta(t,\mu))(z) \, p(\nu, t, T, x, z) \, dz \, \nu(dx) \Big]_{|\nu =\mu}\Big](v) \nonumber \\
& =  \int_{\mathbb{R}^d} \Big[\frac{\delta h}{\delta m}(\Theta(t,\mu))(z) - \frac{\delta h}{\delta m} (\Theta(t,\mu))(v)\Big] \partial^{1+n}_x p(\mu, t, T, v, z) \, dz \label{condition1:structural:class:measure:derivative} \\
& \quad  + \int_{(\mathbb{R}^d)^2}  \Big[ \frac{\delta h}{\delta m}(\Theta(t,\mu))(z)  - \frac{\delta h}{\delta m}(\Theta(t,\mu))(x) \Big] \partial^{n}_v[ \partial_\mu p(\mu, t, T, x, z)](v) \, dz \,\mu(dx),  \nonumber 
\end{align}

 \begin{align}
 \partial_{t} [h(\Theta(t, \mu))] & = \partial_{s} \Big[\int_{(\mathbb{R}^d)^2} \frac{\delta h}{\delta m}(\Theta(t, \mu))(z) \, p(\mu, s, T, x, z) \, dz \, \mu(dx) \Big]_{|s = t} \nonumber \\
 & = \int_{(\mathbb{R}^d)^2} \Big[ \frac{\delta h}{\delta m} (\Theta(t,\mu))(z) -  \frac{\delta h}{\delta m}(\Theta(t,\mu))(x)\Big]\ \partial_t p(\mu, t, T, x, z) \, dz \, \mu(dx)\label{condition1:structural:class:time:derivative}
\end{align}

\noindent and

\begin{align}
\partial^2_{\mu} [h(\Theta(t, \mu))](v, v') & = \partial_{\mu} \Big[\int_{\rr^d} \frac{\delta h}{\delta m} (\Theta(t,\mu))(z) \, \partial_x p(\mu, t, T, v, z) \, dz  \Big](v') \nonumber \\
& \quad + \partial_{\mu} \Big[\int_{(\rr^d)^2} \frac{\delta h}{\delta m} (\Theta(t,\mu))(z) \, \partial_\mu p(\mu, t, T, x, z)(v) \, dz \, \mu(dx) \Big](v') \nonumber \\
& = \int_{\mathbb{R}^d} \partial_x p(\mu, t, T, v, z)\otimes  \partial_\mu \Big[\frac{\delta h}{\delta m} (\Theta(t,\mu))(z)\Big](v') \, dz \nonumber \\
& + \int_{\mathbb{R}^d}  \big[\frac{\delta h}{\delta m} (\Theta(t,\mu))(z) -\frac{\delta h}{\delta m} (\Theta(t,\mu))(v) \big]\, \partial_\mu \partial_x  p(\mu, t, T, v, z)(v') \, dz \nonumber \\
& +  \int_{(\mathbb{R}^d)^2}  \partial_\mu p(\mu, t, T, x, z)(v) \otimes  \partial_\mu \Big[\frac{\delta h}{\delta m} (\Theta(t,\mu))(z) \Big](v')\, dz \, \mu(dx) \label{condition2:structural:class:measure:derivative} \\
&  + \int_{(\mathbb{R}^d)^2} \big[ \frac{\delta h}{\delta m} (\Theta(t,\mu))(z) -  \frac{\delta h}{\delta m} (\Theta(t,\mu))(x) \big]  \partial^2_\mu p(\mu, t, T, x, z)(v, v') \, dz\,  \mu(dx) \nonumber \\ 
&  + \int_{\mathbb{R}^d} \big[ \frac{\delta h}{\delta m} (\Theta(t,\mu))(z) - \frac{\delta h}{\delta m} (\Theta(t,\mu))(v')  \big] \partial_{x} \partial_\mu p(\mu, t, T, v', z)(v) \, dz \nonumber
 \end{align}
 
\end{itemize}
\noindent with the notations $\partial_x \partial_\mu  p(\mu, t, T, v', z)(v) = (\partial_{x_j}[\partial_\mu p(\mu, t, T, v, z)]_i(v'))_{1\leq i, j \leq d}$ and $ \partial_\mu  \partial_x p(\mu, t, T, v, z)(v') = (\partial_{\mu}[\partial_{x_i} p(\mu, t, T, v, z)]_j(v'))_{1\leq i, j \leq d}$.

\end{prop}

\begin{proof} Under the current assumption, from \cite{chaudruraynal:frikha}, we already know that $(t, \mu) \mapsto h(\Theta(t, \mu)) \in \mathcal{C}^{1, 2}([0,T) \times\pp)$ and that \eqref{condition1:structural:class:measure:derivative} as well as \eqref{condition1:structural:class:time:derivative} are satisfied. It thus remains to prove that for any $(t, v) \in [0,T) \times \mathbb{R}^d$, the map $\mu \mapsto \partial_\mu [h(\Theta(t, \mu))](v)$ is $\mathcal{C}^{1} (\pp)$ and that for any $\mu \in \pp$, we can find a version of $v'\mapsto \partial^2_\mu [h(\Theta(t, \mu))](v, v')$ satisfying \eqref{condition2:structural:class:measure:derivative} and such that the mapping $(t, \mu, v, v')\mapsto  \partial^2_\mu h(\Theta(t, \mu))(v, v')$ is continuous for the product topology.

From \eqref{condition1:structural:class:measure:derivative} with $n=0$ 
\begin{align}
\partial_{\mu} [h(\Theta(t, \mu))](v) & = \int_{\mathbb{R}^d} \Big[\frac{\delta h}{\delta m}(\Theta(t,\mu))(z) - \frac{\delta h}{\delta m} (\Theta(t,\mu))(v)\Big] \partial_x p(\mu, t, T, v, z) \, dz \label{first:order:L:deriv:test:function:h} \\
& \quad  + \int_{(\mathbb{R}^d)^2}  \Big[ \frac{\delta h}{\delta m}(\Theta(t,\mu))(z)  - \frac{\delta h}{\delta m}(\Theta(t,\mu))(x) \Big]  \partial_\mu p(\mu, t, T, x, z)(v) \, dz \,\mu(dx).\notag
\end{align}

Observe now that for any $(t, x, z, v)\in [0,T) \times (\mathbb{R}^d)^3$, the maps $[\delta h / \delta m](\Theta(t, .))(v), \,  \partial_x p(., t, T, v, z)$ and $\partial_\mu p(., t, T, x, z)(v)$ are continuously L-differentiable. Moreover, from the above identity applied to the map $m\mapsto [\delta h/\delta m](m)(v)$ instead of $h$, for any fixed $z\in \mathbb{R}^d$, we deduce that $(t, \mu, v')\mapsto \partial_\mu [[\delta h/\delta m](\Theta(t,\mu))(z)](v')$ is continuous and thus locally bounded. Hence, the integrability condition \eqref{integrability:condition} allows to differentiate under the integral sign. We thus deduce the L-differentiability of $\mu \mapsto \partial_{\mu}[h(\Theta(t, \mu))](v)$ and \eqref{condition2:structural:class:measure:derivative} follows by differentiating under the integral sign in \eqref{first:order:L:deriv:test:function:h}. Finally, we remark that each integrand appearing in the five integrals of the right-hand side of \eqref{condition2:structural:class:measure:derivative} are continuous with respect to the variables $t, \mu, v, v'$. The integrability condition \eqref{integrability:condition} then allows to deduce the global continuity of each term with respect to the variables $(t, \mu, v, v') \in [0,T) \times \pp \times (\mathbb{R}^d)^2$.

\end{proof}

 An explicit expression of $\partial^2_\mu [h(\Theta(t, \mu))](v, v')$ can be derived by plugging the identity \eqref{first:order:L:deriv:test:function:h} applied to the map $\delta h/\delta m$ into \eqref{condition2:structural:class:measure:derivative}. We thus obtain
\begin{align}
 \partial^2_{\mu} & [h(\Theta(t, \mu))](v, v') \nonumber  \\
 & = \int_{(\mathbb{R}^d)^2} \Big\{\frac{\delta^2 h}{\delta m^2}(\Theta(t,\mu))(z, z') - \frac{\delta^2 h}{\delta m^2}(\Theta(t,\mu))(z, v')\Big\} \partial_x p(\mu, t, T, v, z)\otimes \partial_{x} p(\mu, t, T, v',z')\, dz\, dz' \nonumber \\
&+\int_{(\mathbb{R}^d)^3}  \Big\{\frac{\delta^2 h}{\delta m^2}(\Theta(t,\mu))(z, z')-\frac{\delta^2 h}{\delta m^2}(\Theta(t,\mu))(z, x)\Big\}\,\partial_x  p(\mu, t, T, v, z) \otimes \partial_\mu p(\mu, t, T, x, z')(v')\, dz\, dz'\, \mu(dx) \nonumber \\
&+\int_{\mathbb{R}^d}  \Big\{\frac{\delta h}{\delta m}(\Theta(t,\mu))(z)-\frac{\delta h}{\delta m}(\Theta(t,\mu))(v)\Big\}\, \partial_\mu\partial_x p(\mu, t, T, v ,z)(v')\, dz \nonumber \\
&+\int_{\mathbb{R}^d}  \Big\{\frac{\delta h}{\delta m} (\Theta(t, \mu))(z)-\frac{\delta h}{\delta m} (\Theta(t, \mu))(v') \Big\} \partial_{x}\partial_\mu p(\mu, t, T, v', z)(v) \, dz \label{full:expression:second:deriv}\\
&+\int_{(\mathbb{R}^d)^3} \Big\{\frac{\delta^2 h}{\delta m^2}(\Theta(t,\mu))(z, z') - \frac{\delta^2 h}{\delta m^2}(\Theta(t,\mu))(z, v')\Big\}\, \partial_\mu p(\mu, t, T, x, z)(v)\otimes  \partial_{x}p(\mu, t, T,v',z') \, dz\, dz'\, \mu(dx) \nonumber \\
&+ \int_{(\mathbb{R}^d)^4}\Big\{\frac{\delta^2 h}{\delta m^2}(\Theta(t,\mu))(z, z')-\frac{\delta^2 h}{\delta m^2}(\Theta(t,\mu))(z, x')\Big\} \partial_\mu p(\mu, t, T, x, z)(v)\otimes  \partial_\mu p(\mu, t, T, x', z')(v')\,dz\, dz'\, \mu(dx')\, \mu(dx) \nonumber\\
&+\int_{(\mathbb{R}^d)^2}  \Big\{\frac{\delta h}{\delta m} (\Theta(t, \mu))(z)- \frac{\delta h}{\delta m} (\Theta(t, \mu))(x)\Big\}\, \partial_\mu^2 p(\mu, t, T, x, z)(v, v')\,  dz\, \mu(dx). \nonumber
\end{align}

The three relations \eqref{condition1:structural:class:measure:derivative}, \eqref{condition1:structural:class:time:derivative} and \eqref{full:expression:second:deriv} play a central role for the analysis of the regularity properties of the transition density related to the dynamics \eqref{SDE:MCKEAN}. Indeed, under the additional assumption that the maps $v \mapsto [\delta h/\delta m](m)(v), \,  v'\mapsto [\delta^2h/\delta m^2](m)(v, v')$ are uniformly H\"older continuous and if $(t, \mu, x) \mapsto p(\mu, t, T, x, z)$ as well as its derivatives satisfy some suitable Gaussian-type bounds, they allow thanks to the space-time inequality \eqref{space:time:inequality} to match the diagonal regime of the underlying heat kernel and to benefit from the so-called smoothing property of Gaussian kernels. These key observation will be used repeatedly in the proofs of Lemma \ref{additional:regularity:lemma} and Proposition \ref{fully:c2:regularity:property:decoupling:density}.

\section{Overview, assumptions and main results}\label{assumptions:results}

\subsection{Some remainders from \cite{chaudruraynal:frikha}: well-posedness of \eqref{SDE:MCKEAN}, existence and regularity of its transition density}\label{sec:some:remainders}

Let us give a few practical reminders of our previous work \cite{chaudruraynal:frikha} concerning the well-posedness of \eqref{SDE:MCKEAN}, the existence and regularity properties of its transition density. We first provide some assumptions on the coefficients made therein. 
 \begin{itemize}

\item[\HRp]
\begin{itemize}
\item[(i)] The drift coefficient $b: \mathbb{R}_+ \times \mathbb{R}^d \times \mathcal{P}(\mathbb{R}^d) \rightarrow \mathbb{R}^d$ and the diffusion coefficient $a :  \mathbb{R}_+ \times \mathbb{R}^d \times \mathcal{P}(\mathbb{R}^d) \rightarrow \mathbb{R}^d \otimes \mathbb{R}^d$, where $a(t, x, m) = (\sigma \sigma^{*})(t, x, m)$, are bounded and continuous functions. The maps $ \mathbb{R}^d \ni x \mapsto b(t, x, m), \, a(t, x, m)$ are uniformly $\eta$-H\"older continuous for some $\eta \in (0,1]$,
$$
\sup_{t \geq 0, \, x \neq y, \, m \in \mathcal{P}(\mathbb{R}^d)} \left\{  \frac{| b(t, x,m) - b(t, y,m) |}{|x-y|^{\eta}}  + \frac{| a(t, x,m) - a(t, y,m) |}{|x-y|^{\eta}} \right\} < \infty.
$$

\item[(ii)] For any $(i, j) \in \left\{1, \cdots, d\right\}^2$ and any $(t, x, y) \in \mathbb{R}_+\times (\mathbb{R}^d)^2$, the map $m\mapsto a_{i, j}(t, x, m)$ has a bounded and continuous linear functional derivative, such that $(x, y)\mapsto [\delta a_{i,j}/\delta m](t, x, m)(y)$ is a bounded and $\eta$-H\"older continuous function, for some $\eta\in (0,1]$, uniformly with respect to the other variables. The map $m \mapsto [\delta a_{i, j}/\delta m](t, x, m)(y)$ has a bounded and continuous linear functional derivative, such that $(x, y')\mapsto [\delta^2a_{i, j}/\delta m^2](t, x, m)(y, y')$ is $\eta$-H\"older continuous uniformly with respect to the other variables.

\item[(iii)] For any $i \in \left\{1, \cdots, d\right\}$ and any $(t, x) \in \mathbb{R}_+\times\mathbb{R}^d$, the map $m \mapsto b_{i}(t, x, m)$ has a bounded and continuous linear functional derivative, such that $y\mapsto [\delta b_i/\delta m](t, x, m)(y)$ is $\eta$-H\"older continuous uniformly with respect to the other variables. Moreover, for any $(t, x, y) \in \rr_+\times (\mathbb{R}^d)^2$, the map $m \mapsto [\delta b_i/\delta m] (t, x, m)(y)$ has a bounded and continuous linear functional derivative, such that $y'\mapsto [\delta^2 b_i/\delta m^2](t, x, m)(y, y')$ is $\eta$-H\"older continuous uniformly with respect to the other variables.
%There exists a function $B: \rr_+\times (\rr^d)^2 \times (\mathcal{P}_2(\rr^d))^2 \rightarrow \rr^N$ such that for every $(t, x, \nu, \mu)\in \rr_+ \times \rr^d \times (\pp)^2$ the function $z\mapsto B(t, x, z, \nu, \mu)$ is $\eta$-H\"older continuous, with modulus denoted by $[B]_H$, for some $\eta\in (0,1]$, uniformly with respect to the other variables and satisfies the following estimate: for all $t\geq0$, for all $(x,  \nu, \mu) \in \rr^d \times (\pp)^2$
%$$
%| b(t, x, \nu) - b(t, x, \mu)| \leq  \Big|\int B(t, x, z, \nu,\mu) (\nu-\mu)(dz)\Big| .
%$$
\end{itemize} 
 \medskip
 
\item[\HE] The diffusion coefficient is uniformly elliptic, that is, there exists $\lambda \geq 1$ such that for every $(t, m) \in [0,\infty) \times \mathcal{P}(\mathbb{R}^d)$ and $(x,\xi) \in (\mathbb{R}^d)^2 $, $\lambda^{-1} |\xi|^2 \leq \langle a(t, x, m) \xi,\xi \rangle \leq \lambda |\xi|^2$ where $a(t, x, m) = (\sigma \sigma^{*})(t, x, m)$.
\bigskip

\end{itemize}

Throughout the paper, we will frequently use the following notation. We will denote by $K^{+}:=K(T, \HRp, \HE)$ some generic constant which depends only upon $T$ and the parameters appearing in \A{HR}$_+$ and \HE. With a slight abuse of notation, we will proceed similarly and denote by $K:=K(T, \HR, \HE)$ some generic constant which depends upon $T$, $a$, $b$, $\delta a/\delta m$, $\delta b/ \delta m$, $\lambda$ and $\eta$. In particular, their values may vary from line to line. We will emphasize the dependence of the constants $K$ or $K_+$ with respect to a prescribed parameter $\beta$ by writing $K_\beta$ or $K^+_\beta$. 

Under \HRp\, and \HE, the martingale problem associated with \eqref{SDE:MCKEAN} is well-posed for any initial distribution $\mu \in \mathcal{P}(\mathbb{R}^d)$. Note that in \cite{chaudruraynal:frikha}, the well-posedness is actually tackled under weaker regularity assumptions on the coefficients, especially with respect to the measure argument, see Theorem 3.4 therein. In particular, weak existence and uniqueness in law holds for the SDE \eqref{SDE:MCKEAN}. 
%We refer to \cite{chaudruraynal:frikha} for a proof using a fixed point argument in a suitable complete metric space under weaker assumptions, notably concerning the regularity in the measure argument. 
The law of the process $(X^{s, \xi}_t)_{t\geq s}$ given by the unique solution to the SDE \eqref{SDE:MCKEAN} starting from the initial distribution $\mu = [\xi]$ at time $s$ thus only depends upon $\xi$ through its law $\mu$. Given $\mu \in \mathcal{P}(\mathbb{R}^d)$, it thus makes sense to consider $([X^{s, \xi}_t])_{t\geq s}$ as a function of the initial distribution $\mu$ (and of the time variable $s$) without specifying the choice of the lifted random variable $\xi$ that has $\mu$ as distribution. We then introduce, for any $x\in \rr^d$, the following \emph{decoupled stochastic flow} associated to the SDE \eqref{SDE:MCKEAN}
\begin{equation}
\label{SDE:MCKEAN:decoupling}
X^{s, x, \mu}_t = x + \int_s^t b(r, X^{s, x, \mu}_r, [X^{s, \xi}_r]) \, dr + \int_s^t \sigma(r, X^{s, x , \mu}_r, [X^{s, \xi}_r]) \, dW_r.
\end{equation}

We note that the previous equation is not a McKean-Vlasov SDE since the law appearing in the coefficients is not $[X^{s, x, \mu}_r]$ but rather $[X^{s, \xi}_r]$, that is, the law of the solution to the SDE \eqref{SDE:MCKEAN} (starting at time $s$ from the initial distribution $\mu$) at time $r$. Under \HRp(i) and \HE, the time-inhomogeneous martingale problem associated to the SDE \eqref{SDE:MCKEAN:decoupling} is well-posed, see e.g. Stroock and Varadhan \cite{stroock:varadhan}. In particular, weak existence and uniqueness in law holds for the SDE \eqref{SDE:MCKEAN:decoupling}.

\smallskip
Once weak well-posedness holds for both the McKean-Vlasov SDE and its decoupled flow, one may consider the densities of the random variables $X^{s,\xi}_t$ and $X^{s, x, \mu}_t$. Indeed, under \HRp\, and \HE, it turns out that both random variables admit a density for any $t>0$. Moreover, denoting by $z \mapsto p(\mu,s,t,z)$ the density of $X^{s,\xi}_t$ and by $z \mapsto p(\mu,s,t,x,z)$ the one of $X^{s, x, \mu}_t$, the following key relation holds
\begin{equation}
\label{relation:density:mckean:decoupling:field}
p(\mu, s, t , z) = \int_{\mathbb{R}^d} p(\mu, s, t ,x ,z) \, \mu(dx).
\end{equation}
In other words, $z \mapsto p(\mu, s, t, z)$ is the density of the image measure of the map $x\mapsto p(\mu, s, t, x, z)$ by the measure $\mu$.\\

It now follows from Friedman \cite{friedman:64} (see also McKean and Singer \cite{mcke:sing:67}) that $ p(\mu, s, t ,x ,z)$ admits the following representation in infinite series also known as \emph{parametrix expansion} 
\begin{equation}
p(\mu, s, t, x, y) = \sum_{k\geq0} (\widehat{p} \otimes \mH^{(k)})(\mu, s, t, x, y) \label{parametrix:series:expansion}
\end{equation}

\noindent where for any $(\mu, r, x, y) \in \pp \times [s, t) \times (\mathbb{R}^d)^2$ the parametrix kernel $\mH(\mu, s, r, t, x, y)$ and the Gaussian type kernel $\widehat{p}^{y}(\mu, s, r, t, x, z)$ are given by 
\begin{align}
\widehat{p}(\mu, s, r, t, x, z) & := \widehat{p}^{z}(\mu, s, r, t, x, z), \notag\\
\widehat{p}^{y}(\mu, s, r, t, x, z) & := g\left(\int_{r}^{t} a(r, y, [X^{s, \xi}_{r'}]) \, dr', z-x\right), \label{kernel:p:hat:definition}\\
\mH(\mu, s, r, t, x, y) & := \sum_{i=1}^d b_i(r, x, [X^{s, \xi}_r]) \partial_{x_i} \widehat{p}(\mu, s,  r, t, x, y) \notag \\
& \quad + \frac12 \sum_{i, j=1}^d (a_{i, j}(r, x, [X^{s, \xi}_r]) - a_{i, j}(r, y, [X^{s, \xi}_r]))\partial^2_{x_i, x_j} \widehat{p}(\mu, s, r, t, x, y) \notag
\end{align}
\noindent and the space-time convolution operator
$$
(f \otimes h)(\mu, s, r, t, x, y) := \int_r^t \int_{\mathbb{R}^d} f(\mu, s, r, r', x, z) h(\mu, s, r', t, z, y) \, dz \, dr'
$$
\noindent together with its iterate $f \otimes \mH^{(k)} = (f\otimes \mH^{(k-1)}) \otimes \mH$ for $k\geq1$ with the convention that $f \otimes \mH^{(0)} \equiv f$. Note that to simplify the notation we will write $(f \otimes g)(\mu, s, t, x, y) := (f \otimes g)(\mu, s, s, t, x, y)$, $\mH(\mu, s, t, x, z) = \mH(\mu, s, s, t, x, z)$ and proceed similarly for other maps. 

We will also need the following estimates: for any $\beta \in [0,1]$, there exist positive constants $K:=K(T, b, a , \eta, \lambda)$, $K_\beta:=K(T, b, a , \eta, \lambda, \beta)$ and $c:= c(\lambda)$ such that for any positive integer $k$, any $(\mu, x, y, z) \in \pp \times (\mathbb{R}^d)^3$, any $0\leq s < t \leq T$, any $r\in [s, t)$ and any integer $n$
\begin{equation}
\Big |\partial_x^{n} \widehat{p}^{y}(\mu, s, r, t, x, z) - \partial_x^{n}\widehat{p}^{y}(\mu, s, r, t, x, z)\Big| \leq K_\beta \frac{|x-x'|^\beta}{(t-r)^{\frac{n+\beta}{2}}} \, \left\{ g(c(t-r), z-x) + g(c(t-r), z-x') \right\}, \label{holder:reg:deriv:p:hat}
\end{equation}

\begin{equation}
\label{iter:parametrix:kernel:true:kernel}
| \mH^{(k)}(\mu, s, r, t, x, z)| \leq \frac{K^{k}}{(t-r)^{1-k \frac{\eta}{2}}} \prod_{\ell=1}^{k-1} B\left(\ell \frac{\eta}{2}, \frac{\eta}{2}\right) \, g(c(t-r), z-x),
\end{equation}
\noindent and
\begin{equation}
\label{iter:param:convol:kernel}
| \widehat{p} \otimes \mH^{(k)}(\mu, s, t, x, z)| \leq K^{k} (t-s)^{k \frac{\eta}{2}} \prod_{\ell=1}^{k} B\left(1+ \frac{(\ell-1)\eta}{2}, \frac{\eta}{2}\right) g(c(t-s), z-x)
\end{equation}

\noindent where $B(k, \ell) = \int_0^1 (1-v)^{-1+k} v^{-1+\ell} dv$ stands for the Beta function. From the asymptotics of the Beta function, the series \eqref{parametrix:series:expansion} converges absolutely and uniformly for $(\mu, x, z) \in \pp \times (\mathbb{R}^d)^2$. Moreover, it satisfies the following Gaussian upper-bounds: for any $0\leq s < t \leq T$ and any $(\mu, x, y) \in \pp \times (\mathbb{R}^d)^2$ 
\begin{equation}
\label{bound:density:parametrix}
|\partial^{n}_x p(\mu, s, t, x, y)| \leq  \frac{C}{(t-s)^{\frac{n}{2}}} \,  g(c (t-s), y-x), \, n=0,1, 2
\end{equation}

\noindent and for any $n\in \left\{0, 1, 2\right\}$, for any $\beta \in [0,1]$ if $n \in \left\{0, 1\right\}$ or any $\beta \in [0,\eta)$ if $n=2$
\begin{equation}
\label{second:derivatives:holder:estimates:parametrix:series}
|\partial^{n}_x p(\mu, s, t, x_1, y) - \partial^n_x p(\mu, s, t, x_2, y) | \leq  C_\beta \frac{|x_1-x_2|^{\beta}}{(t-s)^{ \frac{n+\beta}{2}}} \, \left\{ g(c (t-s), y-x_1) + g(c (t-s), y-x_2)\right\}
\end{equation}

\noindent where $C:=C(T, b, a, \lambda, \eta)$, $C_\beta:=C(T, b, a, \lambda, \eta, \beta)$ and $c:=c(\lambda)$ are positive constants. 

We now introduce the solution $\Phi(\mu, s, r, t, y, z)$ to the following Volterra integral equation
\begin{equation}
\Phi(\mu, s, r, t, y , z) = \mH(\mu, s,  r, t, y, z) + (\mH \otimes \Phi)(\mu, s, r, t, y, z).  \label{volterra:equation:density:decoupling:field}
\end{equation}  

Observe from \eqref{iter:parametrix:kernel:true:kernel} for $k=1$ that the singular kernel $\mH(\mu, s, r, t, y, z)$ induces an integrable singularity in time in the above space-time convolution so that the solution to the above equation exists and is given by the (uniformly) convergent series
\begin{equation}
\label{infinite:series:Phi:step}
\Phi(\mu, s, r, t, y, z) = \sum_{k\geq1} \mH^{(k)}(\mu, s, r, t, y, z)
\end{equation}

\noindent which by \eqref{iter:parametrix:kernel:true:kernel} and the asymptotics of the Beta function satisfies
\begin{equation}
\label{Gaussian:estimate:Phi}
| \Phi(\mu, s, r, t, y, z) | \leq \frac{K}{(t-r)^{1-\frac{\eta}{2}}} \, g(c(t-r), z-y)
\end{equation}
\noindent for some positive constants $K:=K(T, b, a, \eta, \lambda), \, c:=c(\lambda)$.

In view of the relation \eqref{relation:density:mckean:decoupling:field} and the above discussion, the regularity properties of the map $(s, \mu) \mapsto p(\mu, s , t, z)$ stem from those satisfied by $(s, x, \mu) \mapsto p(\mu, s, t, x, z)$. In this regard, we recall the following result established in \cite{chaudruraynal:frikha}.

\begin{theorem}\label{derivative:density:sol:mckean:and:decoupling} Assume that \HE\, and \HRp\, hold. Let $T>0$ and $(t, z) \in (0,T] \times \mathbb{R}^d$. Then, the mapping $[0,t) \times \rr^d \times \pp \ni (s, x, \mu) \mapsto p(\mu, s, t, x, z)$ is in $\mathcal{C}^{1, 2, 2}([0,t) \times \mathbb{R}^d \times \pp)$. 
%that is, $\mathcal{P}_2(\rr^d) \times \rr^d \ni(\mu,v)\mapsto \partial_\mu p(\mu, s, t, x, z)(v)$ exists and is continuous at any point $(\mu,v)$ such that $v\in \supp(\mu)$ and, for any $\mu \in \pp$, $\rr^d \ni v \mapsto \partial_\mu p(\mu, s, t, x, z)(v)$ is continuously differentiable, its derivative denoted by $\partial_v [ \partial_\mu p(\mu, s, t, x, z)(v)]$ being jointly continuous with respect to $\mu$ and $v$ for any $(\mu, v)$ such that $v\in \supp(\mu)$. first:second:lions:derivative:mckean:decoupling
Moreover, for any $(\mu, \mu', s, x, x', v, v') \in (\mathcal{P}_2(\mathbb{R}^d))^2 \times [0,t) \times (\mathbb{R}^d)^4$ and any $(s_1, s_2) \in [0,t)$, 
\begin{align}
%|  \partial_\mu p(\mu, s, t, x, z)(v) | &Ê\leq  \frac{C}{(t-s)^{\frac{1-\eta}{2}}} g(c (t-s), z-x), \label{first:second:lions:derivative:mckean:decoupling}\\
|\partial^{n}_v [\partial_\mu p(\mu, s, t, x, z)](v)| & \leq  \frac{C}{(t-s)^{\frac{1+n-\eta}{2}}} g(c (t-s), z-x), \, n=0,1, \label{first:second:lions:derivative:mckean:decoupling} \\
|\partial_s p(\mu, s, t, x, z)| & \leq  \frac{C}{t-s} g(c (t-s), z-x), \label{first:time:derivative:mckean:decoupling}
\end{align}

%\begin{align}
%\forall \beta \in [0, \eta), \quad | \partial_s p(\mu, s, t, x, z) (v) & - \partial_s p(\mu, s, t , x', z) (v) | \nonumber \\
%&  \leq C \frac{|x-x'|^{\beta}}{(t-s)^{1+\frac{\beta}{2}}} \left\{ g(c(t-s), z-x) + g(c(t-s), z-x') \right\} , \label{equicontinuity:space:estimate:time:deriv:decoupling:mckean:final} 
%\end{align}

%\begin{align}
%\forall \beta \in [0, 1], \quad | \partial_\mu p_{m}(\mu, s, t, x, z) (v) & - \partial_\mu p_{m}(\mu, s, t , x', z) (v) | \nonumber \\
%&  \leq C \frac{|x-x'|^{\beta}}{(t-s)^{\frac{1+\beta-\eta}{2}}} \left\{ g(c(t-s), z-x) + g(c(t-s), z-x') \right\} , \label{equicontinuity:second:estimate:decoupling:mckean} 
%\end{align}

\noindent for any $\beta \in [0,1]$ if $n=0$ and any $\beta \in [0,\eta)$ if $n=1$,

\begin{align}
 \quad | \partial^{n}_v [\partial_\mu p(\mu, s, t, x, z)](v) & - \partial^{n}_v [\partial_\mu p(\mu, s, t , x', z)] (v) |\nonumber \\
&  \leq C_\beta \frac{|x-x'|^{\beta}}{(t-s)^{\frac{1+n+\beta-\eta}{2}}} \left\{ g(c(t-s), z-x) + g(c(t-s), z-x') \right\},  \label{equicontinuity:second:third:estimate:decoupling:mckean:final} 
\end{align}

\noindent for any $\beta \in [0,\eta)$, 
\begin{align}
 |\partial_v [\partial_\mu p(\mu, s, t, x, z)](v) - \partial_v [\partial_\mu p(\mu, s, t, x, z)](v')| & \leq C_\beta \frac{|v-v'|^{\beta}}{(t-s)^{1+ \frac{\beta-\eta}{2}}} g(c(t-s), z-x), \label{equicontinuity:first:estimate:mckean:decoupling}
%\forall \beta \in [0, 1], \quad | \partial_\mu p(\mu, s, t, x, z) (v) - \partial_\mu p(\mu', s, t , x , z) (v) | & \leq C \frac{W^{\beta}_2(\mu, \mu')}{(t-s)^{\frac{1+\beta-\eta}{2}}} g(c(t-s), z- x), \label{regularity:measure:estimate:v1:mckean:decoupling} \\
\end{align}

%\noindent There exist positive constants $C := C(\A{HE}, \A{HR_+}, T)$,  $c:=c(\lambda)$, such that for any $(\mu, \mu', s, x, z, v) \in (\mathcal{P}_2(\rr^d))^2 \times [0,t) \times (\rr^d)^3$,

\noindent for any $\beta \in [0,1]$ if $n \in \left\{0, \, 1\right\}$ and any $\beta \in [0,\eta)$ if $n=2$,
\begin{align}
 |\partial^{n}_x p(\mu, s, t, x, z) & - \partial^{n}_x p(\mu', s, t, x, z)](v)| \leq C_\beta \frac{W_2(\mu, \mu')^{\beta}}{(t-s)^{ \frac{n+\beta- \eta}{2}}} \, g(c(t-s), z-x), \label{regularity:measure:estimate:v1:v2:v3:mckean:decoupling} 
\end{align}

\noindent for any $\beta\in [0,1]$ if $n=0$ and any $\beta\in [0,\eta)$ if $n=1$,
\begin{align}
 |\partial^{n}_v [\partial_\mu p(\mu, s, t, x, z)](v) - \partial^{n}_v [\partial_\mu p(\mu', s, t, x, z)](v)| & \leq C_\beta^{+} \frac{W_2(\mu, \mu')^{\beta}}{(t-s)^{ \frac{1+n+\beta- \eta}{2}}} g(c(t-s), z-x) \label{regularity:measure:estimate:v1:v2:mckean:decoupling} 
\end{align}

\noindent for any $\beta \in [0,1]$ if $n=0$, any $\beta \in [0,\frac{1+\eta}{2})$ if $n=1$ and any $\beta \in [0, \frac{\eta}{2})$ if $n=2$, 
\begin{align}
 | & \partial^{n}_x p (\mu, s_1, t, x, z) - \partial^{n}_x p(\mu, s_2, t, x, z)  | \nonumber \\
 & \leq C_\beta \left\{ \frac{|s_1-s_2|^{\beta}}{(t-s_1)^{\frac{n}{2} + \beta }} \, g(c(t-s_1), z-x) + \frac{|s_1-s_2|^{\beta}}{(t- s_2)^{ \frac{n}{2} +\beta }} \, g(c(t-s_2), z-x) \right\}, \label{regularity:time:estimate:v1:v2:v3:mckean:decoupling} 
\end{align}

%\noindent where $\beta \in [0,1]$ for $n=0$, $\beta \in [0,(1+\eta)/2)$ for $n=1$ and $\beta \in [0, \eta/2)$ for $n=2$ and

%\begin{align}
%\forall \beta \in \Big[0, \frac{1+\eta}{2}\Big), \quad |&  \partial_\mu p(\mu, s_1, t, x, z) (v)  - \partial_\mu p(\mu, s_2, t , x , z) (v) | \nonumber \\ 
%& \leq C \left\{ \frac{|s_1-s_2|^{\beta}}{(t-s_1)^{\frac12 +\beta - \frac{\eta}{2}}} \, g(c(t-s_1), z-x) + \frac{|s_1-s_2|^{\beta}}{(t- s_2)^{\frac12 +\beta - \frac{\eta}{2}}} \, g(c(t-s_2), z-x) \right\}, \label{regularity:time:estimate:v1:mckean:decoupling}
%\end{align}
\noindent for any $\beta \in [0, \frac{1+\eta}{2})$ for $n=0$ and any $\beta \in [0, \frac{\eta}{2})$ for $n=1$
\begin{align}
 | & \partial^{n}_v [\partial_\mu p(\mu, s_1, t, x, z)](v) - \partial^{n}_v [\partial_\mu p(\mu, s_2, t, x, z)](v)| \nonumber \\
 & \leq C_\beta^{+} \left\{ \frac{|s_1-s_2|^{\beta}}{(t-s_1)^{ \frac{1+n-\eta}{2}+\beta}} \, g(c(t-s_1), z-x) + \frac{|s_1-s_2|^{\beta}}{(t- s_2)^{\frac{1+n-\eta}{2} + \beta}} \, g(c(t-s_2), z-x) \right\},  \label{regularity:time:estimate:v1:v2:mckean:decoupling} 
\end{align}
 \noindent where $C := C(T, \HR, \HE)$, $C_\beta := C(T, \HR, \HE, \beta)$, $C^{+}_\beta := C^{+}_\beta(T, \HRp,\, \HE, \beta)$ and $c:=c(\lambda)$ are positive constants.
 
\end{theorem}

\smallskip

\subsection{Additional regularity of the transition density}\label{sec:additional:regularity}

Our approach requires to investigate additional (with respect to the aforementioned results) regularity properties of the map $(s, \mu) \mapsto p(\mu, s , t, z)$ and $(s, x, \mu) \mapsto p(\mu, s, t, x, z)$. To be more specific, our aim is to establish that $(s, \mu) \mapsto p(\mu, s, t, z)$ belongs to $\mathcal{C}^{1,2}_f([0,t) \times \pp)$ and that the transition density of the decoupled flow \eqref{SDE:MCKEAN:decoupling} satisfies the assumptions of Proposition \ref{structural:class}. This second claim is the purpose of the following lemma whose proof follows from similar arguments as those employed to obtain Theorem \ref{derivative:density:sol:mckean:and:decoupling} in \cite{chaudruraynal:frikha} and is thus postponed to \ref{proof:lemma:additional:regularity}. 

\begin{lem}\label{additional:regularity:lemma}
Assume that \HE\, and \HRp\, hold. Let $T>0$ and $(t, z) \in (0,T] \times \mathbb{R}^d$. Then, for all $(\mu, s , x, v)\in \pp \times [0,t) \times (\mathbb{R}^d)^2$, the maps $\mathbb{R}^d \ni x\mapsto \partial_\mu p(\mu, s, t, x, z)(v)$ and $\pp \ni \mu \mapsto \partial_x p(\mu, s, t, x, z)(v)$ are continuously differentiable, with derivatives $\partial_x \partial_\mu p(\mu, s, t, x, z)(v)$ and $\partial_\mu \partial_x p(\mu, s, t, x, z)(v)$ being continuous in $\mu, x, v$ and bounded with respect to the same variables. 

Moreover, for any $\beta \in [0,\eta)$ there exist $C:= C(T, \HR, \HE)$, $C_\beta^+ := C(T, \HRp, \HE,\beta)$, $c:=c(\lambda)>0$ such that for all $(\mu, \mu', s , x, x', v, v')\in (\pp)^2 \times [0,t) \times (\mathbb{R}^d)^4$, one has
\begin{equation}\label{deriv:cross:space:mes:pm}
|\partial_x [\partial_\mu p(\mu, s, t, x, z)](v) | \leq  \frac{C}{(t-s)^{1-\frac{\eta}{2}}}\, g(c(t-s), z-x)%\label{bound:deriv:space:mes:density:Lemm}
\end{equation}
and
\begin{align}
|& \partial_x [\partial_\mu p(\mu, s, t, x, z)](v)  - \partial_x[\partial_\mu p(\mu', s, t, x', z)](v')| \notag\\
& \leq \frac{C_\beta^+}{(t-s)^{1+ \frac{\beta-\eta}{2}}} [W_2(\mu, \mu')^{\beta} + |x-x'|^\beta + |v-v'|^\beta ] \,\left\{ g(c(t-s), z-x) + g(c(t-s), z-x') \right\}. \label{sensitivity:mes:deriv:cross:space:mes:p}
\end{align}
\end{lem}

In view of the results recalled from \cite{chaudruraynal:frikha} and the relation \eqref{relation:density:mckean:decoupling:field}, a sufficient condition to obtain the $\mathcal{C}^{1,2}_f([0,t) \times \pp)$ regularity of the map $(s, \mu) \mapsto p(\mu, s, t, z)$ consists in establishing that the second derivative $\partial^2_\mu p(\mu, s, t, x, z)(\gv)$ exists and is continuous in its arguments. This is the purpose of the following proposition which provides as well the H\"older regularity of $[0,t) \times \mathbb{R}^d \times \pp \times (\mathbb{R}^d)^2 \ni (s, x, \mu, \gv)\mapsto \partial^2_\mu p(\mu, s, t, x, z)(\gv)$ and some sharp Gaussian type estimates provided the coefficients satisfy the following additional regularity assumptions. Roughly speaking, we need the existence of an additional linear functional derivative which is H\"older continuous with respect to its space arguments.\\

% $x \mapsto \partial_\mu  p(\mu, s, x, t, z)(v)$ is continuously differentiable. In view of the above results from \cite{chaudruraynal:frikha} and the relation \eqref{relation:density:mckean:decoupling:field}, it thus suffices to prove the existence of the second derivative $\partial^2_\mu p(\mu, s, t, x, z)(v,v')$ together with the H\"older regularity of $\pp \times \rr^d \times (\rr^d)^2 \ni (\mu, x, v,v')\mapsto \partial^2_\mu p(\mu, s, t, x, z)(v,v')$.
%
%
%We will also need the existence of the second derivative $\partial^2_\mu p(\mu, s, t, x, z)(v)$ together with the H\"older regularity of $\pp \times \rr^d \times (\rr^d)^2 \ni (\mu, x, v)\mapsto \partial^2_\mu p(\mu, s, t, x, z)(v)$ and some sharp Gaussian type estimates. In this regard, we introduce the following additional assumption on the coefficients:
 \begin{itemize}
\item[\A{HR$_{++}$}] The coefficients $b$ and $\sigma$ satisfy \A{HR$_+$}. Moreover, for any $(i, j) \in \left\{1, \cdots, d\right\}^2$ and any $(t, x, v, v') \in \rr_+\times (\rr^d)^3$, the maps $\mathcal{P}(\mathbb{R}^d) \ni m \mapsto [\delta^2 a_{i, j}/\delta m^2](t, x, m)(v, v'),\, [\delta^2 b_i/\delta m^2](t, x, m)(v, v')$ admit a bounded and continuous linear functional derivative, such that $(x, v'')\mapsto [\delta^3 a_{i, j}/\delta m^3] (t, x, m)(v, v', v'')$ and $v''\mapsto [\delta^3b_i/\delta m^3](t, x, m)(v, v', v'')$ are $\eta$-H\"older continuous uniformly with respect to the other variables.

%\item[(iii)] For any $i \in \left\{1, \cdots, d\right\}$, for any $(t, x) \in \rr_+\times\rr^d$, the map $m \mapsto b_{i}(t, x, m)$ has a bounded continuous linear functional derivative, such that $y\mapsto \frac{\delta}{\delta m} b_{i}(t, x, m)(y)$ is $\eta$-H\"older continuous uniformly with respect to the other variables. Moreover, for any $i \in \left\{1, \cdots, d\right\}$, for any $(t, x, y) \in \rr_+\times (\rr^d)^2$, the map $m \mapsto \frac{\delta}{\delta m} b_{i}(t, x, m)(y)$ has a bounded continuous linear functional derivative, such that $y'\mapsto \frac{\delta^2}{\delta m^2} b_{i}(t, x, m)(y, y')$ is $\eta$-H\"older continuous uniformly with respect to the other variables. 
\end{itemize}

As previously done, we will denote by $K^{++}:=K(T, \HRpp, \HE)$ some generic constant that depends only on $T$ and the parameters in \A{HR}$_{++}$, \A{HE}. We emphasize its dependence with respect to a prescribed parameter $\beta$ by writing $K^{++}_\beta$.

We are now in position to state the $\mathcal{C}^{1, 2, 2}_f([0,t) \times \mathbb{R}^d \times \pp)$ regularity of the map $(s, x, \mu)\mapsto p(\mu, s, t, x, z)$ which is a key step toward our quantitative estimates for propagation of chaos. Its proof being rather long and technical is postponed to Section \ref{full:c2:regularity:section}.

\begin{prop}\label{fully:c2:regularity:property:decoupling:density} Assume that \A{HE} and \A{HR$_{++}$} hold. Let $T>0$ and $(t, z) \in (0,T] \times \mathbb{R}^d$. Then, the map $(s, x, \mu)\mapsto p(\mu, s, t, x, z) \in \mathcal{C}^{1,2, 2}_f([0,t) \times \rr^d \times \pp)$. In particular, for any fixed $(s, x) \in [0,t) \times \mathbb{R}^d$, the map $\mu \mapsto p(\mu, s, t, x, z)$ is fully $\mathcal{C}^2$. Moreover, for any $\beta \in [0,\eta)$, there exist positive constants $C^+ := C(T, \HRp, \HE)$, $C_\beta^{++}=C(T, \HRpp, \HE, \beta)$ and $c:=c(\lambda)$ such that for any $(\mu, s, x, x', z, \gv) \in \pp \times [0,t) \times (\mathbb{R}^d)^3 \times (\mathbb{R}^d)^2$
\begin{align}
| \partial^{2}_\mu p(\mu, s, t, x, z)(\gv)| & \leq  \frac{C^+}{(t-s)^{1-\frac{\eta}{2}}} g(c (t-s), z-x)  \label{second:lions:derivative:mckean:decoupling} 
\end{align}

\noindent for any $\gv_1 = (v_1,v_1')$, $\gv_2 = (v_2,v_2')$ in $\mathbb{R}^d \times \mathbb{R}^d$,
\begin{align}
|\partial^2_\mu & p(\mu, s, t, x, z)(\gv_1)  - \partial^2_\mu p(\mu', s, t, x' , z)(\gv_2)| \nonumber \\
& \leq \frac{C_\beta^{++}}{(t-s)^{1+\frac{\beta-\eta}{2}}} \big[W_2(\mu, \mu')^{\beta} +|x-x'|^{\beta} + |\gv_1-\gv_2|^{\beta} \big]\, \Big\{ g(c (t-s), z-x) + g(c(t-s), z-x')\Big\}.\label{second:deriv:mes:reg:mes:space:estimate:induction:mckean:decoupling}
\end{align}

\noindent and for any $s_1, s_2 \in [0,t)$
\begin{align}
|\partial^2_\mu & p(\mu, s_1, t, x, z)(\gv)  - \partial^2_\mu p(\mu', s, t, x' , z)(\gv)| \nonumber \\
& \leq C_\beta^{++}  \left\{ \frac{|s_1-s_2|^{\frac{\beta}{2}}}{(t-s_1)^{1+\frac{\beta-\eta}{2}  }} \, g(c(t-s_1), z-x) + \frac{|s_1-s_2|^{\frac{\beta}{2}}}{(t- s_2)^{ 1+ \frac{\beta-\eta}{2} }} \, g(c(t-s_2), z-x) \right\}.\label{second:deriv:mes:reg:time:estimate:induction:mckean:decoupling}
\end{align}
\end{prop}

\subsection{Fundamental solution of the backward Kolmogorov PDE on the Wasserstein space.} 

A key feature of our analysis of convergence rate in the propagation of chaos phenomenom is to bring to light a connection between the transition density functions of the system of particles \eqref{SDE:particle:system} and of its mean-field limit \eqref{SDE:MCKEAN} by means of the notion of fundamental solution of the parabolic backward Kolmogorov PDE defined on the strip $[0,T] \times \pp$ that we now present. 

Let us consider the following linear differential operator
\begin{equation}\label{Gen:flow:weak:mckean}
\mathscr{L}_{t} U(\mu) = \int_{\rr^d}\left\{ \sum_{i= 1}^{d} b_i(t, v, \mu) [\partial_\mu U(\mu)]_i(v) + \frac12 \sum_{i, j=1}^d a_{i ,j}(t, v, \mu) \partial_{v_j}[\partial_\mu U(\mu)]_i(v)  \right\} \mu(dv), \, t\in [0,T]
\end{equation}

\noindent acting on a smooth real-valued function $U$ defined on $\pp$. The parabolic backward Kolmogorov PDE defined on the strip $[0,T] \times \pp$ is given by
\begin{align}
\begin{cases}
(\partial_t + \mathscr{L}_t) U(t, \mu)  = 0,& \quad  (t,  \mu) \in [0,T) \times \pp,  \\
U(T, \mu)   = h(\mu), & \quad \mu \in  \pp.
\end{cases}
 \label{kolmogorov:pde:wasserstein:space}
\end{align}

Let us underline that under mild assumptions on the functions $h$, $b$ and $a$, the above PDE admits a unique classical solution given by $U(t, \mu) = h([X^{t, \xi}_T])$, $(X^{t, \xi}_s)_{s\in [t, T]}$ being the unique weak solution to the SDE \eqref{SDE:MCKEAN} starting from the initial distribution $[\xi]=\mu$ at time $t$. We refer to \cite{chaudruraynal:frikha} for irregular terminal condition $h$ and coefficients $b$ and $a$, in the uniformly elliptic setting. We also refer to \cite{Crisan2017} when the terminal condition $h$ is irregular by means of Malliavin's calculus still in the uniformly elliptic setting. We finally mention the recent work \cite{buckdahn2017} for the case of smooth functions $h$, $b$ and $a$ without any non-degeneracy assumption. Let us now introduce the notion of fundamental solution related to \eqref{kolmogorov:pde:wasserstein:space}. 

\begin{definition}\label{def:fundamental:solution}
A fundamental solution of $\partial_s + \mathscr{L}_s = 0$ in $[0, T] \times \pp$ is a map $[0, t) \times \pp \ni (s, \mu) \mapsto p(\mu, s, t, z)$ defined for all $(t, z) \in (0,T] \times \mathbb{R}^d$ satisfying the following two conditions:
\begin{itemize}
\item[(i)] For every fixed $(t, z) \in (0,T] \times \rr^d$, the map $[0, t) \times \pp \ni (s, \mu) \mapsto p(\mu, s, t, z)$ belongs to $\mathcal{C}^{1, 2}([0,t) \times \pp)$ and satisfies the equation
\begin{equation}
\label{backward:kolmogorov:pde}
(\partial_s+\mathscr{L}_s) p(\mu, s, t, z)  = 0 \quad \mbox{ on } [0,t)\times \pp.
\end{equation}

\item[(ii)] For every real-valued continuous function $f$ defined on $\rr^d$ with at most quadratic growth, for any $\mu \in \pp$
\begin{equation}
\lim_{s\uparrow t} \int_{\mathbb{R}^d} f(z) \, p(\mu, s, t, z) \, dz = \int_{\mathbb{R}^d} f(z) \, \mu(dz). \label{terminal:condition}
\end{equation}

When there is no possible confusion, we will write $\lim_{s\uparrow t} p(\mu, s, t, z) = \delta_z(.)\star \mu$ (``$\star$'' denoting the usual convolution operator), instead of \eqref{terminal:condition}.

\end{itemize}
\end{definition}

\begin{theorem}\label{fundamental:solution:wasserstein:space}
Assume that \A{HE} and \A{HR$_+$} hold. Let $(t, z) \in (0,T] \times \rr^d$. The map $[0,t) \times \pp \ni (s, \mu) \mapsto p(\mu, s, t, z)$ defined by \eqref{relation:density:mckean:decoupling:field}-\eqref{parametrix:series:expansion} is a fundamental solution of $\partial_s + \mathscr{L}_s = 0$. 

Moreover, it is the unique solution among the class of fundamental solutions $(s, \mu) \mapsto q(\mu, s, t, z)$ defined for all $(t, z) \in (0,T] \times \mathbb{R}^d$, being continuous with respect to $z$, satisfying \eqref{cond:integrab:ito:process:second:version} for any fixed $(t, z) \in (0,T] \times \mathbb{R}^d$, $T$ being replaced by any $t' \in [0,t)$, and satisfying the terminal condition \eqref{terminal:condition} locally uniformly in $\mu \in \pp$, that is, uniformly in $\mu \in \mathcal{K}$, $\mathcal{K}$ being any compact set of $\pp$.
\end{theorem}

\subsection{Three types of propagation of chaos estimates for the system of particles \eqref{SDE:particle:system}}
Our primary objective is to study the propagation of chaos for the system of particles \eqref{SDE:particle:system} by quantifying in an appropriate sense its distance from its mean field limit \eqref{SDE:MCKEAN}.\\

Let us first emphasize that under \HE\, and \HRp, the system of particles with dynamics \eqref{SDE:particle:system} is well posed in the weak sense. Indeed, for any $(t, x) \in \rr_+ \times \mathbb{R}^d$ and $\bold{x}, \bold{y} \in (\mathbb{R}^d)^{N}$: 
\begin{align*}
b_i( t, x, m^{N}_{\bold{x}}) - b_i(t, x, m^{N}_{\bold{y}}) & =  \int_0^{1} \int_{\rr^d} \frac{\delta}{\delta m} b_i(t, x,  m^{\lambda, N})(z) (m^{N}_{\bold{x}} -m^{N}_{\bold{y}})(dz) d\lambda
\end{align*}

\noindent where we used the notations $m^{N}_{\bold{x}} := N^{-1} \sum_{i=1}^{N} \delta_{x_i}$ and $m^{\lambda, N} := \lambda m^{N}_{\bold{x}} +(1-\lambda)m^{N}_{\bold{y}}$. From the uniform $\eta$-H\"older regularity of $z \mapsto [\delta b_i /\delta m] (t, x, m)(z)$, it is thus readily seen
$$
|b_i( t, x, m^{N}_{\bold{x}}) - b_i(t, x, m^{N}_{\bold{y}})| \leq \sup_{t \in \mathbb{R}_+, x \in \mathbb{R}^d, m \in \mathcal{P}(\mathbb{R}^d)}\Big[ \frac{\delta}{\delta m}b_i( t, x, m)(.) \Big]_{H} |\bold{x} - \bold{y}|^{\eta}
$$ 

\noindent where $\Big[ \frac{\delta}{\delta m}b_i( t, x, m)(.) \Big]_{H}$ stands for the H\"older norm of the map $[\delta b_i / \delta m]( t, x, m)(.)$. The same inequality also holds with the map $a_{i, j}$ instead of $b_i$. As a consequence, the measurable maps $\mathbb{R}^d \times (\mathbb{R}^d)^{N} \ni (x, \bold{x}) \mapsto b(t, x, m^{N}_{\bold{x}}), \, a(t, x, m^{N}_{\bold{x}})$ are bounded and $\eta$-H\"older continuous uniformly in time so that the martingale problem related to \eqref{SDE:particle:system} is well posed, see e.g. \cite{stroock:varadhan}. In particular, weak existence and uniqueness holds for the SDE \eqref{SDE:particle:system}.\\

Also, from \cite{friedman:64}, the $N\times d$ dimensional random variable $\mathbf{X}_t = (X^1_t, \cdots, X^N_t)$ given by the unique weak solution to \eqref{SDE:particle:system} taken at time $t>0$ starting from the $N$-fold product measure $\mu^{N}$ admits a density function $(\rr^d)^{N} \ni \bold{z} \mapsto \mathbf{p}^{N}(\mu, 0, t, \bold{z})$, $\bold{z}= (z_1, \cdots, z_{N})$, with respect to the Lebesgue measure on $(\rr^d)^{N}$. For any fixed $i$ in $\left\{1,\ldots, N\right\}$, we denote by $p^{i,N}$ the density of the $i^{{\rm th}}$ particle obtained by integrating the joint density of the particles $\bold{z} \mapsto \mathbf{p}^{N}(\mu, 0, t, \bold{z})$ over $z_{j}$ for $j\neq i$. By weak uniqueness of the SDE \eqref{SDE:particle:system} and exchangeability in law of the i.i.d. initial conditions $(\xi^{i})_{1\leq i \leq N}$, the one-dimensional marginal distributions of the random variable $\mathbf{X}^{N}_t$ are equal. In particular, one has $p^{i, N} \equiv p^{1,N}$ for any $i \in \left\{1, \cdots, N\right\}$. Moreover, for any fixed time $T>0$, there exist two constants $\bold{C}:=\bold{C}(T, a, b, N)>1, \, \bold{c}:= \bold{c}(\lambda, N)>0$ such that for any $(t, \mu, \bold{z}) \in (0,T] \times \mathcal{P}(\mathbb{R}^d) \times (\mathbb{R}^d)^N$ the following two sided Gaussian estimate holds
\begin{equation}
\bold{C}^{-1} \int_{(\rr^d)^{N}}g(\bold{c}^{-1}t, \bold{z} -\bold{x}) \, \mu^{N}(\bold{dx}) \leq \mathbf{p}^{N}(\mu, 0, t, \bold{z}) \leq \bold{C} \int_{(\rr^d)^{N}}g(\bold{c}t, \bold{z} -\bold{x}) \, \mu^{N}(\bold{dx}). \label{two:sided:gaussian:estimate:pN}
\end{equation}

\begin{remark}
Hence, it is readily seen that a similar two sided Gaussian estimate hold for $p^{1, N}$ instead of $\mathbf{p}^{N}$ but with constant $\bold{C}, \, \bold{c}$ that depend on $N$. As a by product of our result, we will establish below a Gaussian upper-bound with two constants $C$, $c$ that do not depend on $N$. To the best of our knowledge, this result is new.
\end{remark}

The first propagation of chaos estimate is an error bound of order $N^{-1}$ for the difference $(p^{1,N} - p)(\mu, 0, t, z)$ under \A{HR$_{++}$} and \HE. We then establish a first order expansion for this difference with an explicit control of the remainder term under the additional assumption that $\mathbb{R}^d \times \pp (x, \mu) \mapsto \sigma(t, x, \mu)$ is uniformly Lipschitz continuous and that $M_q(\mu)= (\int_{\mathbb{R}^d} |x|^q \, \mu(dx))^{1/q} <\infty$ for some $q>4$. The proof of the following result is postponed to Section \ref{propagation:of:chaos}.

\begin{theorem}\label{propagation:chaos:fundamental:sol:theorem}Assume that \HE\, and \HRpp\, hold. Then, there exist positive constants $ K^+ :=K(T, \HRp, \HE)$, $c:=c(\lambda)$, $T\mapsto K(T, \HRp, \HE)$ being non-decreasing, such that for any $(t, \mu, z) \in (0,T]\times \pp \times \mathbb{R}^d$
\begin{align}
p^{1, N}(\mu, 0, t, z) \leq K^+ \int_{\mathbb{R}^d} g(c t, z-x) \mu(dx) \label{Gaussian:upper:bound:density:p1N}
\end{align}
\noindent and 
\begin{align}
|(p^{1, N}-p)(\mu, 0, t, z)| \leq \frac{K^+}{N} \left\{ \frac{1}{t^{\frac{1-\eta}{2}}} \int_{\mathbb{R}^d} g(c t, z-x) |x| \mu(dx) + \frac{1}{t^{1-\frac{\eta}{2}}} \int_{\mathbb{R}^d} g(c t, z-x)  \mu(dx) \right\}.\label{error:bound:prop:chaos}
\end{align}

Assume additionally that $\mathbb{R}^d \times \pp \ni (x, \mu) \mapsto \sigma(t, x, \mu)$ is uniformly Lipschitz continuous, with modulus $[\sigma]_L$, and that $M_q(\mu) < \infty$ for some $q>4$. Then, for all $(t, \mu, z) \in (0,T]\times \pp \times \mathbb{R}^d$, the following first order expansion holds
\begin{align}
 (p^{1,N} - p)(\mu, 0, t, z) & =  \frac{1}{N}   \E\Big[\frac{\delta}{\delta m}p(\mu, 0, t, \xi^1, z)(\xi^1)- \frac{\delta}{\delta m}p(\mu, 0, t, \xi^1, z)(\widetilde{\xi})\Big]\nonumber \\
& \quad + \frac{1}{2N} \E\Big[\frac{\delta^2}{\delta m^2}p(\mu, 0, t, \xi^1, z)(\widetilde{\xi}, \widetilde{\xi}) - \frac{\delta^2}{\delta m^2}p(\mu, 0, t, \xi^1, z)(\widetilde{\xi}, \xi^2) \Big]  \nonumber\\
&  \quad +\frac{1}{N} \int_0^t \E[\mathcal{A}_sp(\mu_s, s, t, z)]\, ds + \frac{1}{N}\mathcal{R}_N(\mu, 0, t, z)\label{first:order:exp:prop:chaos}
\end{align}

\noindent where $\widetilde{\xi}$ is an $\mathbb{R}^d$-valued random variable independent of $(\xi^{i})_{1\leq i\leq N}$ with law $\mu$ and $\mathcal{A}_s$ is the differential operator on $\pp$ acting on smooth function $\phi : \pp \to \mathbb{R}$
$$
 \mathcal{A}_s \phi(\mu) = \frac{1}{2} \int_{\mathbb{R}^d} \tr\big(\partial^2_\mu \phi(\mu)(v, v) a(s, v, \mu) \big) \, \mu(dv)
$$ 

\noindent with the following estimate on the remainder term: for any $\phi: \mathbb{R}^d \rightarrow \mathbb{R}$ with at most quadratic growth and any $\beta \in [0,\eta)$ 
$$
 \int_{\mathbb{R}^d} |\phi(z)|  |\mathcal{R}_N(\mu, 0, t, z)| \, dz \leq K^{++} \left\{ \frac{\varepsilon^{1/2}_N}{t^{1-\frac{\eta}{2}}} + \frac{\varepsilon^{\beta/2}_N}{t^{1+ \frac{\beta-\eta}{2}}} \right\}
$$

\noindent where $T\mapsto K^{++}:=K(T, \HRpp, \HE, [\sigma]_L, q, M_q(\mu))$ is a positive non-decreasing function, $[\sigma]_L$ standing for the uniform Lipschitz modulus of the map $\sigma(t, ., .)$, and where $\varepsilon_N$ is defined by
\begin{equation}\label{convergence:rate}
\varepsilon_N :=  \left\lbrace\begin{array}{ll} \displaystyle N^{-1/2} \text{ if } d < 4,\\ \displaystyle  N^{-1/2}\log(1+N)  \text{ if } d=4,\\ \displaystyle N^{-2/d}  \text{ if } d > 4. \end{array}\right.
\end{equation}

\end{theorem}

 Inspired by the previous result as well as Remark 5.110 in \cite{carmona2018probabilistic}, we now provide a kind of weak propagation of chaos estimate as well as an error estimate for the difference between the semigroup generated by the system of particles \eqref{SDE:particle:system} and the semigroup associated to its mean-field limit both living on $\pp$. Below, for all $t \ge 0$, we denote by $\mu_t$ the law of the solution $X_t$ of \eqref{SDE:MCKEAN}. 

\begin{theorem}\label{theo:weakChaos} Assume that \A{HE} and \A{HR$_{++}$} hold. For $\alpha\in (0,1]$, let $\mathscr{C}^{2, \alpha}(\pp)$ be the class of continuous functions $\phi : \pp \rightarrow \mathbb{R}$ that admit two continuous linear functional derivatives on $\pp$ (see Definition 5.43 in section 5.4.1 of \cite{carmona2018probabilistic}) and satisfying the following regularity and growth assumptions: there exist $C\geq0$ such that for any $m\in \pp$, any $x, x' \in \mathbb{R}^d$ and any bounded set $D\subset \mathbb{R}^d$
\begin{equation}
\label{local:holder:reg:first:flat:deriv:terminal:condition}
\sup_{x\neq x', x,  x' \in D} |x-x'|^{-\alpha}\Big| \frac{\delta \phi}{\delta m}(m)(x) - \frac{\delta \phi}{\delta m}(m)(x') \Big|  \leq C (1 + M_2(m)),
\end{equation}

\begin{equation}
\label{local:holder:reg:second:flat:deriv:terminal:condition}
\sup_{x''\neq x', x', x'' \in D} |x'-x''|^{-\alpha}\Big| \frac{\delta^2 \phi}{\delta m^2}(m)(x, x') - \frac{\delta^2 \phi}{\delta m^2}(m)(x, x'') \Big|  \leq C (1 + |x| + M_2(m)),
\end{equation}

\noindent and
\begin{equation}
\label{growth:condtion:first:second:flat:deriv}
\Big|\frac{\delta \phi}{\delta m}(m)(x)\Big| + \Big|\frac{\delta^2 \phi}{\delta m^2}(m)(x, x')\Big| \leq C (1+ |x| + |x'| + M_2(m))
\end{equation}
\noindent where we recall that $M_2(m)= (\int_{\mathbb{R}^d}|x|^2 m(dx))^{1/2}$.

 Then, there exists a positive constant $K^{+}:=K(T, \HRp, \HE, \alpha, M_2(\mu) )$, $T\mapsto K(T,\HRp, \HE, \alpha, M_2(\mu))$ being non-decreasing such that for all $\phi \in \mathscr C^{2, \alpha}(\pp)$ it holds
\begin{eqnarray}
|\E[\phi(\mu_T^N)] -\phi(\mu_T)|& \leq& \frac{K^+}{T^{1-\frac \alpha 2}} \frac{1}{N}, \label{esti:sem:chaos:part1}\\
\E\left[|\phi(\mu_T^N) -\phi(\mu_T)|\right] &\leq & K^+ \left\{\frac{1}{T^{\frac{1-\alpha}{2}}}  \mathbb{E}[W_2(\mu_0^{N}, \mu)^2]^{1/2} + \frac{1}{N^{\frac12}}\right\}.\label{esti:sem:chaos:part2}
%\left\lbrace\begin{array}{ll} \displaystyle N^{-1/2} \text{ if } d < 2,\\ \displaystyle  N^{-1/2}\log(1+N)  \text{ if } d=2,\\ \displaystyle N^{-1/d}  \text{ if } d > 2. \end{array}\right.\label{weakconv}
\end{eqnarray}
\end{theorem}
\begin{remark}
\begin{itemize}

\item The linear growth assumption with respect to the space variable and second order moment of the probability measure variable $m$ appearing in the definition of the space $\mathscr{C}^{2, \alpha}(\pp)$ is tailor-made to ensure the linear growth of the solution (as well as of its first and second order derivatives) of the corresponding backward Kolmogorov PDE stated in the strip $[0,T] \times \pp$ with terminal condition $\phi$. This together with the fact that the initial distribution satisfies $M_2(\mu)<\infty$ play a central role in the proof of \eqref{esti:sem:chaos:part2}. Larger spaces of test function could be considered under stronger integrability assumptions on the initial distribution $\mu$. 

\item Note that when $\alpha=1$, it holds 
$$
\big\{\phi : \pp\to \mathbb{R},\, \phi(\mu) = \int_{\mathbb{R}^d} \varphi(x) \mu(dx),\, \varphi \, {\rm is\ 1\mbox{-}Lipschitz}\big\} \subset \mathscr C^{2,1}(\pp)
$$ 

\noindent so that, in particular, \eqref{esti:sem:chaos:part1} implies the convergence of the probability measure $\mathbb{E}[\mu_T^{N}]$ toward $\mu_T$ with respect to the first order Wasserstein distance by the Kantorovitch-Rubinstein duality theorem. 
\item From the proof of \eqref{esti:sem:chaos:part1}, it will be apparent that one could also obtain a first order expansion at the level of the semigroup, that is, an ad-hoc version of \eqref{first:order:exp:prop:chaos}. However, we refrain from going further in this direction here.

\item Let us finally observe that since $M_2(\mu) < \infty$, one has $\lim_{N\rightarrow \infty} \mathbb{E}[W_2(\mu_0^{N}, \mu)^2] =0$ and that a non-asymptotic estimate which quantifies the rate of convergence in this limit is available under the assumption $M_q(\mu) < \infty$ for some $q>4$, see e.g. Theorem 1 of \cite{Fournier2015} and Theorem 5.8 of \cite{carmona2018probabilistic}. Indeed, if this stronger integrability condition on $\mu$ is satisfied, then there exists a positive constant $C:=C(d, q, M_q(\mu))$ such that for all $N\geq2$, $\mathbb{E}[W_2(\mu_0^{N}, \mu)^2]\leq C \varepsilon_N$, where $\varepsilon_N$ is defined by \eqref{convergence:rate}. 
\end{itemize}
\end{remark}

Our last objective is to prove that the system of particles \eqref{SDE:particle:system} converges in the strong sense to the solution of the McKean-Vlasov SDE \eqref{SDE:MCKEAN} by extending the classical result of propagation of chaos on the trajectories of the particles to our framework. As in the standard case, we shall quantify the convergence rate of propagation of chaos through a coupling argument with an auxiliary system of particles as in \cite{Sznitman}.

Under the additional assumption that $\mathbb{R}^d \times \pp \ni (x, \mu) \mapsto \sigma(t, x, \mu)$ is Lipschitz continuous uniformly in time, from \cite{veretennikov_strong_1980}, strong uniqueness holds for the system of particles \eqref{SDE:particle:system} and from Corollary 3.5 in \cite{chaudruraynal:frikha} the same conclusion holds for its mean-field limit \eqref{SDE:MCKEAN}. Hence, strong well-posedness for both SDEs follows from the Yamada-Watanabe theorem. 

In the above framework, we thus choose a probability space $(\Omega, \mathcal{F}, \P )$ as well as $N$ independent $q$-Brownian motion $(W^{i})_{1\leq i \leq N}$ on it. We also assume that the probability space carries the i.i.d. sequence of $\mathbb{R}^d$-valued and $\mathcal{F}_0$-measurable random variables $(\xi^{i})_{1\leq i\leq N}$ with common law $\mu$ satisfying $M_2(\mu)< \infty$.

For any $i\in \left\{1, \cdots, N\right\}$, we then introduce the process $\bar{X}^{i}=(\bar{X}^i_t)_{0\leq t\leq T}$ given by the unique strong solution to the McKean-Vlasov SDE \eqref{SDE:MCKEAN} but with the input $(\xi^{i}, W^{i})_{1\leq i \leq N}$ instead of $(\xi, W)$
\begin{equation}
\label{SDE:particle:system:coupling}
\bar{X}^{i}_t = \xi^{i} + \int_0^t b(s, \bar{X}^{i}_s, [\bar{X}^{i}_s]) ds + \int_0^t \sigma(s, \bar{X}^{i}_s, [\bar{X}^{i}_s]) dW^{i}_s, \quad i=1, \cdots, N.
\end{equation}

By weak uniqueness for the SDE \eqref{SDE:MCKEAN}, the two processes $\bar{X}^i$ and $X$ have the same law, in particular $[\bar X^i_t] = [X_t] =\mu_t$, for any $t\in [0,T]$ and for any $i \in \left\{1, \cdots, N\right\}$. Our last result quantifies the propagation of chaos at level of the trajectories. Its proof is postponed to Section \ref{propagation:of:chaos}.

\begin{theorem}\label{propagation:prop:path} Assume that \A{HE} and \A{HR$_{++}$} hold and that $M_{q}(\mu) < +\infty$, for some $q>4$. Assume that for any $t\in [0,T]$, the map $\mathbb{R}^d \times \pp \ni (x, \mu) \mapsto \sigma(t, x, \mu)$ is Lipschitz continuous, uniformly in time. Then, there exists a positive constant $K^+:=K(T, \HRp, \HE,  [\sigma]_L, M_q(\mu))$ such that
\begin{equation}\label{Esti:propchaos:path}
\sup_{0\leq t\leq T}\E[W_2(\mu_t,\mu_t^N)^2] + \max_{i=1,\ldots,N}\sup_{0\leq t \leq T}\E\Big[|X^i_t-\bar X^i_t|^2 \Big] \leq K^+ \varepsilon_N 
\end{equation}
\noindent and
\begin{equation}\label{Esti:propchaos:path:2}
\E[\sup_{0\leq t\leq T}W_2(\mu_t,\mu_t^N)^2] + \max_{i=1,\ldots,N}\E\Big[\sup_{0\leq t \leq T}|X^i_t-\bar X^i_t|^2 \Big] \leq K^+ \sqrt{\varepsilon_N} 
\end{equation}

\noindent where we recall that $\varepsilon_N$ is defined by \eqref{convergence:rate}. 

\end{theorem}

\begin{remark} The Zvonkin's transform applied in our framework shows that the rate of convergence provided in \eqref{Esti:propchaos:path:2} is actually ruled by the quantity $\E[\sup_{0\leq t\leq T} W^2_2(\mu_t, \bar{\mu}^N_t)]$ where $\bar \mu^{N}_t := \frac1N \sum_{i=1}^{N} \delta_{\bar X^{i}_t}$ which is in turn known to be of order $\sqrt{\varepsilon_N}$, see e.g. Briand et al. \cite{BCCdRH:19}. This last estimate could be improved under stronger integrability assumption on the initial distribution $\mu$. We also mention the fact that one could achieve a convergence rate of order $\varepsilon_N$ under the additional assumption that the map $\mu \mapsto b(t, x, \mu)$ is Lipschitz continuous uniformly with respect to the variables $t$ and $x$ but we do not engage into further reflections in this direction.  
\end{remark}

\section{The backward Kolmogorov equation}\label{fully:c2:regularity:section}
This section is dedicated to the proofs of Theorem \ref{fundamental:solution:wasserstein:space} and Proposition \ref{fully:c2:regularity:property:decoupling:density}. Hence, we assume that \HE\, and \HRp\, are in force in subsection \ref{section:kolmogorov:backward:pde} and that \HE\, and \HRpp \, are in force in subsection \ref{full:c2:regularity:section}.

\subsection{Proof of Theorem \ref{fundamental:solution:wasserstein:space}}\label{section:kolmogorov:backward:pde}

The proof is divided into two steps.\\

\noindent \emph{Step 1: Existence of a fundamental solution.}\\

We fix $T>0$ and $(t,z)\in (0,T] \times \mathbb{R}^d$. From the identity \eqref{relation:density:mckean:decoupling:field} and Theorem \ref{derivative:density:sol:mckean:and:decoupling}, we already know that the map $(s, \mu) \mapsto p(\mu, s, t, z)$ is in $\mathcal{C}^{1, 2}([0,t) \times \pp)$ with derivatives $\partial_s p(\mu, s, t, z) = \int_{\rr^d} \partial_s p(\mu, s, t, x,  z) \mu(dx)$ and 
\begin{equation}
\partial^{n}_v[\partial_\mu p(\mu, s, t, z)](v) = \partial^{1+n}_x p(\mu, s, t, v, z) + \int_{\rr^d} \partial^{n}_v[\partial_\mu p(\mu, s, t, x,  z)](v) \, \mu(dx), \quad n\in\left\{0, 1\right\}. \label{identity:deriv:mes:dens:mckean}
\end{equation} 

We now prove that it satisfies \eqref{backward:kolmogorov:pde}. 

From the Markov property satisfied by the SDE \eqref{SDE:MCKEAN}, stemming from the well-posedness of the related martingale problem, see Theorem 3.4 in \cite{chaudruraynal:frikha}, the following relation is satisfied for all $0<h<s$
$$
p(\mu, s-h, t,  z) = p([X^{s-h, \xi}_s], s, t, z).
$$

From the relation \eqref{identity:deriv:mes:dens:mckean} and the estimates \eqref{first:second:lions:derivative:mckean:decoupling}, we deduce that the condition \eqref{cond:integrab:ito:process:second:version} of the chain rule formula of Proposition \ref{prop:chain:rule:joint:space:measure} (with respect to the measure variable only) is satisfied so that
$$
 p([X^{s-h, \xi}_s], s, t, z) = p(\mu, s, t,  z) + \int_{s-h}^{s} \mathscr{L}_r p([X^{s-h, \xi}_r], s, t, z) \, dr 
$$

\noindent which in turn yields
$$
\frac{1}{h}(p(\mu, s-h, t, x, z)  - p(\mu, s, t, x, z)) = \frac{1}{h}\int_{s-h}^s \mathscr{L}_r p([X^{s-h, \xi}_r], s, t, z) \, dr.
$$

\noindent Letting $h\downarrow 0$, from the boundedness and the continuity of the coefficients as well as the continuity of the maps $(\mu, v) \mapsto \partial_\mu p(\mu, s, t, z)(v), \partial_v[\partial_\mu p(\mu, s, t, z)](v)$ and the differentiability of $[0, t) \ni s\mapsto p(\mu, s, t, x, z)$, we get that $(s, \mu) \mapsto p(\mu, s, t, z)$ satisfies \eqref{backward:kolmogorov:pde}.

 We now prove that \eqref{terminal:condition} is satisfied locally uniformly, that is, uniformly on compact sets $\mathcal{K} \subset \pp$. From \eqref{parametrix:series:expansion}, one gets
\begin{equation}
p(\mu, s, t, x, z) = \widehat{p}(\mu, s, t, x, z) + \mathcal{R}(\mu, s, t, x, z), \quad \mathcal{R}(\mu, s, t, x, z):= \sum_{k\geq1} (\widehat{p} \otimes \mathcal{H}^{(k)})(\mu, s, t, x, z) \label{decomposition:short:time:parametrix}
\end{equation}

\noindent and \eqref{iter:param:convol:kernel} implies that the infinite series defining $\mathcal{R}(\mu, s, t, x, z)$ converges and that the following estimate is satisfied
\begin{equation}
|\mathcal{R}(\mu, s, t, x, z)| \leq K (t-s)^{\frac{\eta}{2}} \, g(c(t-s), z-x) \label{rest:parametrix:series:time:estimate}
\end{equation}

\noindent for some positive constant $K:=K(T, b, a, \lambda, \eta)$. From the mean-value theorem, the uniform $\eta$-H\"older continuity of $x\mapsto a(t, x, m)$ and the space-time inequality \eqref{space:time:inequality}, one has 
\begin{equation}
|(\widehat{p}^{z}- \widehat{p}^{x})(\mu, s, t, x, z)| \leq K |z-x|^\eta g(c(t-s), z-x) \leq K (t-s)^{\frac{\eta}{2}} g(c(t-s), z-x). \label{diff:p:hat:frozen:coeff}
\end{equation}

Let $f$ be a real-valued continuous function defined on $\mathbb{R}^d$ with at most quadratic growth. The key relation \eqref{relation:density:mckean:decoupling:field} together with the fact that $\int_{\mathbb{R}^d}  \widehat{p}^{x}(\mu, s, t, x, z) \, dz=1$ yield
\begin{align*}
\int_{\mathbb{R}^d} f(z) \, p(\mu, s, t, z)\, dz - \int_{\mathbb{R}^d} f(x) \, \mu(dx) & = \int_{(\mathbb{R}^d)^2} [f(z)-f(x)] \, \widehat{p}^{x}(\mu, s, t, x, z) \, dz \mu(dx) \\
& +  \int_{(\mathbb{R}^d)^2} f(z) \,  [\widehat{p}^{z}(\mu, s, t, x, z) - \widehat{p}^{x}(\mu, s, t, x, z)] \, dz \mu(dx) \\
& \quad + \int_{(\mathbb{R}^d)^2} f(z) \mathcal{R}(\mu, s, t, x, z) \, dz \mu(dx).
\end{align*}

Thanks to \eqref{rest:parametrix:series:time:estimate}, \eqref{diff:p:hat:frozen:coeff} and using the fact that $f$ has at most quadratic growth, for any compact set $\mathcal{K} \subset \pp$, it holds
\begin{align*}
\sup_{\mu \in \mathcal{K}} & | \int_{(\rr^d)^2} f(z) [\widehat{p}^{z}(\mu, s, t, x, z) - \widehat{p}^{x}(\mu, s, t, x, z)] \, dz \mu(dx)| \\
& \quad \quad + \sup_{\mu \in \mathcal{K}} | \int_{(\mathbb{R}^d)^2} f(z) \mathcal{R}(\mu, s, t, x, z) \, dz \mu(dx)| \leq K (t-s)^{\frac{\eta}{2}} \big(1+ \sup_{\mu \in \mathcal{K}} M_2(\mu)\big).
\end{align*}

The uniform continuity of the map $[0,t] \times \mathcal{K} \ni (s, \mu) \mapsto \int_{(\mathbb{R}^d)^2} [f(z)-f(x)] \, \widehat{p}^{x}(\mu, s, t, x, z) \, dz \mu(dx) = \int_{(\mathbb{R}^d)^2} [f(x+ \Sigma^{1/2}_{s, t} z) -f(x)] e^{-\frac{|z|^2}{2}} (2\pi)^{-\frac{d}{2}} \, dz \mu(dx)$, where $\Sigma^{1/2}_{s, t}$ is the unique principal square root of the positive-semidefinite matrix $\int_{s}^t a(r, x, [X^{s, \xi}_r]) \, dr$, implies that $\left\{[0,t] \ni s\mapsto \int_{(\mathbb{R}^d)^2} [f(z)-f(x)] \, \widehat{p}^{x}(\mu, s, t, x, z) \, dz \mu(dx), \mu \in \mathcal{K}\right\}$ is equicontinuous and the quadratic growth of $f$ implies its boundedness. We thus deduce
$$
\lim_{s\uparrow t} \sup_{\mu \in \mathcal{K}} \Big|\int_{(\mathbb{R}^d)^2} [f(z)-f(x)] \, \widehat{p}^{x}(\mu, s, t, x, z) \, dz \mu(dx) \Big| = 0.
$$

Combining the previous results, we deduce that \eqref{terminal:condition} is satisfied locally uniformly on $\pp$. We thus conclude that $[0,t) \times \pp \ni (s, \mu) \mapsto p(\mu, s, t, z)$ is a fundamental solution of $\partial_s + \mathcal{L}_s=0$ in $[0,T] \times \pp$. \\

\noindent \emph{Step 2: Uniqueness}\\

In order to get the uniqueness result, let us consider any solution $(s, \mu) \mapsto q(\mu, s, t, z)$ to the backward Kolmogorov equation \eqref{backward:kolmogorov:pde} satisfying \eqref{cond:integrab:ito:process:second:version} on any interval $[0,t']$, with $t'<t$, and \eqref{terminal:condition} uniformly in $\mu \in \mathcal{K}$, $\mathcal{K}$ being a compact set of $\pp$. We apply the chain rule formula of Proposition \ref{prop:chain:rule:joint:space:measure} to $\left\{ q([X^{s, \xi}_r], r, t, z), \,  s \leq r  < t  \right\}$ and use the fact that $(\partial_s + \mathcal{L}_s)q(\mu, s, t, z)  = 0$, for any $(s, \mu) \in [0,t) \times \pp$ to get that for any $r\in [s, t)$
\begin{align}
 q([X^{s, \xi}_r], r, t, z) &  = q(\mu, s, t, z). \label{first:relation:toward:uniqueness}
\end{align}

We now aim to pass to the limit as $r \uparrow t$ in the previous relation. To do this, we first remark that from \eqref{parametrix:series:expansion}, \eqref{relation:density:mckean:decoupling:field}, the Gaussian upper-bound \eqref{bound:density:parametrix} and the continuity of $(s, t]\ni r\mapsto p(\mu, s, r, z)$, one has $\lim_{r \uparrow t} \int_{\mathbb{R}^d} f(z) \, p(\mu, s, r, z) \, dz = \int_{\mathbb{R}^d} f(z) \, p(\mu, s, t, z) \, dz$ for any real-valued measurable function $f$ with at most quadratic growth so that $\lim_{r\uparrow t} W_2([X^{s, \xi}_r], [X^{s, \xi}_t])=0$. Hence using the local uniform convergence in $\mu$ of $r\mapsto \int_{\mathbb{R}^d} f(z) \, q(\mu, r, t, z) \, dz$ towards $\int_{\mathbb{R}^d} f(x) \, \mu(dx)$ as $r\uparrow t$ and \eqref{first:relation:toward:uniqueness}, we obtain
$$
\int_{\mathbb{R}^d} f(z) \, q(\mu, s, t, z) \, dz =  \lim_{r \uparrow t} \int_{\mathbb{R}^d} f(z) \, q([X^{s, \xi}_r], r, t, z) \, dz =  \int_{\mathbb{R}^d} f(z) \, [X^{s, \xi}_t](dz) = \int_{\mathbb{R}^d} f(z)\,  p(\mu, s, t, z) \, dz.
$$

\noindent for any continuous function $f$ with at most quadratic growth. From the continuity of the maps $q(\mu, s, t, .)$ and $p(\mu, s, t, .)$, we deduce that $q(\mu, s, t, z) = p(\mu, s, t, z)$ which completes the proof of Theorem \ref{fundamental:solution:wasserstein:space}.

\subsection{Proof of Proposition \ref{fully:c2:regularity:property:decoupling:density}}\label{full:c2:regularity:section}

The proof of Proposition \ref{fully:c2:regularity:property:decoupling:density} relies on similar arguments as those employed to prove Theorem 3.6 in \cite{chaudruraynal:frikha}. To be more specific, our strategy is based on an approximation argument of the transition density $p(\mu, s, t, x, z)$ by a Picard iteration scheme and sharp uniform estimates on its derivatives from which we can extract a uniformly convergent subsequence by using Arzela-Ascoli's theorem.
 
\medskip

\noindent \emph{Step 1: Construction of an approximation sequence and related estimates}

\medskip

For a given initial condition $(s,\mu) \in \mathbb{R}_+ \times \pp$ and a probability measure $\nu \in \pp$, $\nu \neq \mu$, we let $\P^{(0)} = (\P^{(0)}(t))_{t\geq s}$ be the probability measure on $\mathcal{C}([s,\infty), \rr^d)$, endowed with its canonical filtration, satisfying $\P^{(0)}(t)=\nu$, $t\geq s$. Let us consider the following recursive sequence of probability measures $\left\{ \P^{(m)}; m\geq 0 \right\}$, with time marginals $(\P^{(m)}(t))_{t\geq s}$, where, $\P^{(m)}$ being given, $\P^{(m+1)}$ is the unique solution to the following martingale problem
\begin{itemize}
\item[(i)] $\P^{(m+1)}(y(r) \in \Gamma; 0\leq r \leq s)  =  \mu(\Gamma)$, for all $\Gamma \in \mathcal{B}(\rr^d)$.
\item[(ii)] For all $f\in \mathcal{C}^2_b(\rr^d)$, 
$$
f(y_t)  - f(y_s) - \int_s^t \left\{ \sum_{i=1}^d b_i(r, y_r, \P^{(m)}(r)) \partial_i f(y_r) + \sum_{i, j =1}^d \frac12 a_{i, j}(r, y_r, \P^{(m)}(r)) \partial^2_{i, j} f(y_r) \right\} \, dr
$$

\noindent is a continuous square-integrable martingale under $\mathbb{P}^{(m+1)}$.
\end{itemize}

Note that, under the considered assumptions, the well-posedness of the above standard martingale problem follows from classical results, see e.g. \cite{stroock:varadhan}. In particular, there exists a unique weak solution to the SDE with dynamics
\begin{align}
X_t^{s,\xi, (m+1)} & = \xi + \int_s^t b(r,X_r^{s,\xi, (m+1)},[X_r^{s,\xi, (m)}]) dr + \int_s^t \sigma(r,X_r^{s,\xi, (m+1)},[X_r^{s,\xi, (m)}]) d W_r. \label{iter:mckean} 
\end{align}

 We will also work with the decoupled stochastic flow or characterics given by the unique weak solution to the SDE with dynamics
\begin{align}
X_t^{s, x, \mu, (m+1)} & = x + \int_s^t b(r,X_r^{s, x,\mu, (m+1)},[X_r^{s,\xi, (m)}]) dr + \int_s^t \sigma(r,X_r^{s, x, \mu, (m+1)},[X_r^{s,\xi, (m)}]) d W_r.  \label{iter:mckean:decoupl}
\end{align}

We point out that the notation $X^{s, x, \mu, (m+1)}_t$ makes sense since by weak uniqueness of solution to the SDE \eqref{iter:mckean}, the law $[X^{s, \xi, (m)}_t]$ only depends on the initial condition $\xi$ through its law $\mu$. 

From \cite{friedman:64}, for any positive integer $m$, the two random variables $X^{s,\xi,(m)}_t$ and $X^{s, x, \mu,(m)}_t$ admit a density respectively denoted by $p_m(\mu, s, t, z)$ and $p_m(\mu, s, t, x, z)$. Moreover, the following relation is satisfied for any $z\in \mathbb{R}^d$
\begin{equation}
 p_m(\mu, s, t, z) = \int p_m(\mu, s, t, x, z) \,  \mu(dx) \label{conv:relation:step:m}
\end{equation}

\noindent where
\begin{align}
p_m(\mu, s, t, x, z) & = \sum_{k \geq 0} (\widehat{p}_{m} \otimes \mH^{(k)}_m)(\mu, s, t, x, z), \label{series:approx:mckean} 
\end{align}

\noindent with
\begin{align}
\widehat{p}_m(\mu, s, r,  t , x, z) & = \widehat{p}^{z}_m(\mu, s, r,  t , x, z), \nonumber \\
\widehat{p}^{y}_m(\mu, s, r,  t , x, z) & = g\left(\int_{r}^{t} a(r', y, [X^{s,\xi, (m-1)}_{r'}]) dr', z-x\right), \label{eq:def:de:phat:m}\\
\mH_m(\mu, s, r ,t ,x ,z) & = \left\{- \sum_{i=1}^d b_i(r, x, [X^{s, \xi, (m-1)}_r]) H^{i}_1\left(\int_r^{t} a(r', z, [X^{s, \xi, (m-1)}_{r'}]) dr', z-x\right) \right. \nonumber \\
&  \left. + \frac12 \Big(a_{i, j}(r, x, [X^{s, \xi, (m-1)}_{r}]) - a_{i, j}(r, z, [X^{s, \xi, (m-1)}_r])\Big) \right. \label{eq:def:de:mH} \\
& \left. \quad \quad \times H^{i, j}_2\left(\int_r^{t} a(r', z, [X^{s, \xi, (m-1)}_{r'}]) dr', z-x\right)  \right\} \widehat{p}_m(\mu, s, r, t, x, z)  \nonumber
\end{align}

\noindent and $\mH^{(k+1)}_m(\mu, s, t, x, z) = (\mH^{(k)}_m \otimes \mH_m)(\mu, s, t, x, z)$, $\mH^{(0)}_m = I_d$, with the convention that $[X^{s, \xi, (0)}_t] = \P^{(0)}(t)=\nu$, $t\geq0$. In what follows, we will often make use of the following estimates which are reminiscent of \eqref{iter:param:convol:kernel}, \eqref{bound:density:parametrix} and \eqref{second:derivatives:holder:estimates:parametrix:series}: there exist positive constant $c:= c(\lambda)$, $C:=C(T, b, a, \lambda, \eta)$, such that for any positive integer $k$, any $(\mu, x, z) \in \pp \times (\mathbb{R}^d)^2$, any $0\leq s < t \leq T$ and any $r\in [s, t)$, it holds
\begin{equation}
\label{iter:parametrix:kernel}
| \mH^{(k)}_m(\mu, s, r, t, x, z)| \leq \frac{C^{k}}{(t-r)^{1-k \frac{\eta}{2}}} \prod_{\ell=1}^{k-1} B\left(\ell \frac{\eta}{2}, \frac{\eta}{2}\right) \, g(c(t-r), z-x)
\end{equation}
\noindent and
\begin{equation}
\label{iter:param:classic}  
| \widehat{p}_{m} \otimes \mH^{(k)}_m(\mu, s, t, x, z)| \leq C^{k} (t-s)^{k \frac{\eta}{2}} \prod_{\ell=1}^{k} B\left(1+ \frac{(\ell-1)\eta}{2}, \frac{\eta}{2}\right) g(c(t-s), z-x)
\end{equation}

\noindent where we recall that $B(., .)$ stands for the Beta function. Consequently, the series \eqref{series:approx:mckean} converge absolutely and uniformly for $(\mu, x, z) \in \pp \times (\mathbb{R}^d)^2$ and satisfies: for any positive integer $m$, any $0\leq s < t \leq T$, any $(\mu, x, z) \in \pp \times (\mathbb{R}^d)^2$ and any $n\in \left\{0,1, 2\right\}$
\begin{equation}
\label{bound:derivative:heat:kernel}
| \partial^{n}_x p_m(\mu, s, t, x, z)| \leq C (t-s)^{-\frac{n}{2}} \, g(c (t-s), z-x), 
\end{equation}

\noindent and for all $(x, x')\in (\mathbb{R}^d)^2$, all $\beta \in [0,1]$ if $n =0,\, 1$ and all $\beta \in [0,\eta)$ if $n=2$
\begin{align}
\label{reg:heat:kernel:deriv}
 |\partial^{n}_x p_{m}(\mu, s, t, x, z) & - \partial^{n}_x p_{m}(\mu, s, t , x', z)| \nonumber \\
 & \leq C_\beta \frac{|x-x'|^{\beta}}{(t-s)^{\frac{n+\beta}{2}}} \left\{ g(c (t-s), z-x) + g(c(t-s), z-x') \right\},
\end{align}

\noindent for some positive constants $C:=C(T, b, a, \lambda, \eta)$, $C_\beta:=C(T, b, a, \lambda, \eta, \beta)$. We refer again to \cite{friedman:64} for a proof of the above estimate. 

Similarly to \eqref{volterra:equation:density:decoupling:field}, we denote by $\Phi_m(\mu, s, r, t, x_1, x_2)$ the unique solution to the following Volterra integral equation
\begin{equation}
\Phi_m(\mu, s, r, t, x_1, x_2) = \mH_m(\mu, s,  r, t, x_1, x_2) + (\mH_m \otimes \Phi_m)(\mu, s, r, t, x_1, x_2)
\end{equation}  

\noindent which is given by the (uniform) convergent series
\begin{equation}
\label{infinite:series:Phi:step:m}
\Phi_m(\mu, s, r, t, x_1, x_2) = \sum_{k\geq1} \mH^{(k)}_m(\mu, s, r, t, x_1, x_2)
\end{equation}

\noindent and \eqref{series:approx:mckean} now writes
\begin{equation}
p_m(\mu, s, t, x, z) = \widehat{p}_m(\mu, s,  t , x, z)  + \int_s^t \int_{\mathbb{R}^d} \widehat{p}_m(\mu, s,  r , x, y) \,\Phi_m(\mu, s, r, t, y, z) \, dy \, dr. \label{other:representation:parametrix:series}
\end{equation}

Finally, from Theorem 7, Chapter 1 in \cite{friedman:64}, for any positive integer $m$, the map $\Phi_m(\mu, s, r, t, x, z)$ satisfies the following estimates: for any $\beta \in[0, \eta)$, there exist positive constants $C_\beta:=C(T, a, b, \eta, \lambda, \beta), C:=C(T, a, b, \eta, \lambda), \, c:=c(\lambda)$, which do not depend on $m$, such that for any $(\mu, x, y, z)\in \pp \times (\mathbb{R}^d)^{3}$ and any $0\leq s \leq r < t \leq T$
\begin{equation}
\label{Gaussian:estimate:Phim}
| \Phi_m(\mu, s, r, t, x, z) | \leq \frac{C}{(t-r)^{1-\frac{\eta}{2}}} \, g(c(t-r), z-x)
\end{equation}
\noindent and
\begin{align}
|\Phi_m(\mu, s, r, t, x, z) & - \Phi_m(\mu, s, r, t, y, z) |\nonumber \\
& \leq C_\beta \frac{|x-y|^{\beta}}{(t-r)^{1+ \frac{\beta-\eta}{2}}} \left\{ g(c(t-r), z -x) + g(c(t-r), z-y)\right\}. \label{holder:reg:voltera:kernel}
\end{align}

We now recall from \cite{chaudruraynal:frikha} some important notations, properties and estimates. For some positive integer $m$, $n\in \left\{0,1\right\}$, $\beta \in [0, 1+\eta)$ if $n=0$ or $\beta \in [0,\eta)$ if $n=1$, $C \geq0$ and $t\in [0,T]$, we define
\begin{equation}\label{definition:series:Cmnbeta}
\mathscr{C}^{n ,\beta}_m(C, t):=  \sum_{k=1}^{m} C^{k} t^{(k-1) \frac{\eta}{2}} \prod_{i=1}^{k-1} B\left(\frac{\eta}{2}, \frac{1-n+\eta-\beta}{2} +  (i-1) \frac{\eta}{2} \right).
\end{equation}

Let $T>0$. For any fixed $(t, z) \in (0,T] \times \mathbb{R}^d$ and any positive integer $m$, it holds:
\begin{itemize}
\item The mapping $[0,t) \times \rr^d \times \pp \ni (s, x, \mu) \mapsto p_m(\mu, s, t, x, z)$ is in $\mathcal{C}^{1, 2, 2}([0,t) \times \rr^d \times \pp)$.

\item  There exist positive constants $C := C(T, \HR, \HE)$, $C_\beta := C(T, \HR, \HE, \beta)$, $c:=c(\lambda)$, which do not depend on $m$, such that for any $(\mu, s, x, x', z, v, v') \in \mathcal{P}_2(\mathbb{R}^d) \times [0,t) \times (\mathbb{R}^d)^5$ and any $n\in \left\{0, 1\right\}$
\begin{align}
|\partial^{n}_v [\partial_\mu p_m(\mu, s, t, x, z)](v)| & \leq  \frac{\mathscr{C}^{n, 0}_m(C, t-s)}{(t-s)^{\frac{1+n-\eta}{2}}} \, g(c (t-s), z-x), \label{first:second:estimate:induction:decoupling:mckean} \\
| \partial_s p_m(\mu, s, t, x, z) | & \leq  \frac{\mathscr{C}^{1, 0}_m(C, t-s)}{t-s} \, g(c (t-s), z-x), \label{time:derivative:induction:decoupling:mckean}
\end{align}

\begin{align}
 \quad | \partial^{n}_v [\partial_\mu p_{m}(\mu, s, t, x, z)](v) & - \partial^{n}_v [\partial_\mu p_{m}(\mu, s, t , x', z)] (v) |\nonumber \\
&  \leq \mathscr{C}^{n, \beta}_{m}(C_\beta, t-s) \frac{|x-x'|^{\beta}}{(t-s)^{\frac{1+n+\beta-\eta}{2}}} \left\{ g(c(t-s), z-x) + g(c(t-s), z-x') \right\},  \label{equicontinuity:second:third:estimate:decoupling:mckean} 
\end{align}

\noindent where $\beta \in [0,1]$ for $n=0$ and $\beta \in [0, \eta)$ for $n=1$,

\begin{align}
 |\partial_v [\partial_\mu p_m(\mu, s, t, x, z)](v) & - \partial_v [\partial_\mu p_m(\mu, s, t, x, z)](v')| \nonumber \\
& \leq \mathscr{C}^{1, \beta}_{m}(C_\beta, t-s) \frac{|v-v'|^{\beta}}{(t-s)^{1+\frac{\beta-\eta}{2}}} \, g(c(t-s), z-x) \label{equicontinuity:first:estimate:decoupling:mckean}
\end{align}
\noindent where $\beta\in [0,\eta)$.
\item There exist three positive constants $C^{+}_\beta := C(T, \HRp, \HE, \beta)$, $C_\beta:= C(T, \HR, \HE, \beta)$, $c:=c(\lambda)$, which do not depend on $m$, such that for any $(\mu, \mu', s, x, z, v) \in (\mathcal{P}_2(\mathbb{R}^d))^2 \times [0,t) \times (\mathbb{R}^d)^3$ and any $(s_1, s_2) \in [0,t)^2$

\begin{align}
 |\partial^{n}_x p_{m}(\mu, s, t, x, z) & - \partial^{n}_x p_{m}(\mu', s, t, x, z)]| \leq C_\beta \frac{W_2(\mu, \mu')^{\beta}}{(t-s)^{ \frac{n+\beta}{2}}} \, g(c(t-s), z-x), \label{regularity:measure:estimate:v1:v2:v3:decoupling:mckean} 
\end{align}

\noindent where $\beta \in [0,1]$ for $n \in \left\{0, \, 1\right\}$ and $\beta \in [0,\eta)$ for $n=2$,
\begin{align}
 |\partial^{n}_v [\partial_\mu p_{m}(\mu, s, t, x, z)](v) & - \partial^{n}_v [\partial_\mu p_{m}(\mu', s, t, x, z)](v)| \nonumber \\
 & \leq\mathscr{C}^{n, \beta}_{m}(C^+_{\beta}, t-s) \frac{W_2(\mu, \mu')^{\beta}}{(t-s)^{ \frac{1+n+\beta- \eta}{2}}} \, g(c(t-s), z-x), \label{regularity:measure:estimate:v1:v2:decoupling:mckean} 
\end{align}

\noindent where $\beta \in [0,1]$ for $n=0$ and $\beta \in [0,\eta)$ for $n=1$,
\begin{align}
 | & \partial^{n}_x p_{m}(\mu, s_1, t, x, z) - \partial^{n}_x p_{m}(\mu, s_2, t, x, z)  | \nonumber \\
 & \leq C_\beta \left\{ \frac{|s_1-s_2|^{\beta}}{(t-s_1)^{\frac{n}{2} + \beta }} \, g(c(t-s_1), z-x) + \frac{|s_1-s_2|^{\beta}}{(t- s_2)^{ \frac{n}{2} +\beta }} \, g(c(t-s_2), z-x) \right\}, \label{regularity:time:estimate:v1:v2:v3:decoupling:mckean} 
\end{align}

\noindent where $\beta \in [0,1]$ for $n=0$, $\beta \in [0,(1+\eta)/2)$ for $n=1$ and $\beta \in [0, \eta/2)$ for $n=2$ and
\begin{align}
 | & \partial^{n}_v [\partial_\mu p_{m}(\mu, s_1, t, x, z)](v) - \partial^{n}_v [\partial_\mu p_{m}(\mu, s_2, t, x, z)](v)| \nonumber \\
 & \leq  \mathscr{C}^{n, 2 \beta}_{m}(C^+_{\beta}, t-s_1 \vee s_2) \left\{ \frac{|s_1-s_2|^{\beta}}{(t-s_1)^{\frac{1+n-\eta}{2} + \beta }} \, g(c(t-s_1), z-x) + \frac{|s_1-s_2|^{\beta}}{(t- s_2)^{ \frac{1+n-\eta}{2} +\beta }} \, g(c(t-s_2), z-x) \right\}, \label{regularity:time:estimate:v1:v2:decoupling:mckean} 
\end{align}

\noindent where $\beta \in [0,(1+\eta)/2)$ if $n=0$ and $\beta \in [0,\eta/2)$ if $n=1$.
 
 \end{itemize}
 
With the above notations and properties at hand, we can now state the following key proposition whose proof is postponed to \ref{proof:main:prop}.

\begin{prop}\label{proposition:reg:density:recursive:scheme:mckean} Let $T>0$. Assume that \HE\, and \HRp\, hold. Then, for any fixed $(t, z) \in (0,T] \times \mathbb{R}^d$ and any positive integer $m$, the following properties hold:
\begin{itemize}
\item The mapping $[0,t) \times \mathbb{R}^d \times \pp \ni (s, x, \mu) \mapsto p_m(\mu, s, t, x, z)$ is in $\mathcal{C}^{1, 2, 2}_f([0,t) \times \mathbb{R}^d \times \pp)$.

\item For any $(\mu, s, x, v)\in \pp \times [0,t) \times (\mathbb{R}^d)^2$, the maps $\mathbb{R}^d \ni x\mapsto \partial_\mu p_m(\mu, s, t, x, z)(v)$ and $\pp \ni \mu\mapsto  \partial_x p_m(\mu, s, t, x, z)$ are continuously differentiable, with derivatives $\partial_x \partial_\mu p_m(\mu, s, t, x, z)(v)$, $ \partial_\mu \partial_x p_m(\mu, s, t, x, z)(v)$ being continuous in $\mu, s, x, v$ and bounded with respect to $\mu, x, v$. 

Moreover, for any $\beta \in [0,\eta)$, there exist positive constants $C := C(T, \HR, \HE)$, $C^{+}_\beta:= C(T, \HRp, \HE, \beta)$ and $c:=c(\lambda)$ such that for any $\mu, \mu' \in \pp$, any $s, s_1, s_2 \in [0,t)$ and any $x, x', v, v' \in \mathbb{R}^d$
\begin{align}
| \partial_{x} \partial_\mu p_m(\mu, s, t, x, z)(v)| \leq \frac{C}{(t-s)^{1-\frac{\eta}{2}}} \, g(c(t-s), z-x), \label{cross:deriv:mes:space:induction:decoupling:mckean}
\end{align}
\begin{align}
|& \partial_x [\partial_\mu p_m(\mu, s, t, x, z)(v)]  - \partial_x[\partial_\mu p_m(\mu', s, t, x', z)(v')]| \label{sensitivity:mes:deriv:cross:space:mes:pm} \\
& \leq \frac{C_\beta^+}{(t-s)^{1+ \frac{\beta-\eta}{2}}} [W_2(\mu, \mu')^{\beta} + |x-x'|^\beta + |v-v'|^\beta ] \,\left\{ g(c(t-s), z-x) + g(c(t-s), z-x') \right\} \notag
\end{align}
\noindent and
\begin{align}
|& \partial_x [\partial_\mu p_m(\mu, s_1, t, x, z)(v)]  - \partial_x[\partial_\mu p_m(\mu, s_2, t, x, z)(v)]| \label{sensitivity:time:deriv:cross:space:mes:pm} \\
& \leq C_\beta^+\left\{ \frac{|s_1-s_2|^\frac{\beta}{2}}{(t-s_1)^{1+\frac{\beta-\eta}{2}}} g(c(t-s_1), z-x) + \frac{|s_1-s_2|^\frac{\beta}{2}}{(t-s_2)^{1+\frac{\beta-\eta}{2}}} g(c(t-s_2), z-x) \right\}. \notag 
\end{align}

\item For any $\beta \in [0,\eta)$, there exist positive constants $C^+ := C(T, \HRp, \HE)$, $C_\beta^+ := C(T, \HRp, \HE, \beta)$, $c:=c(\lambda)$, such that for any $(\mu, s, x, x', z) \in \mathcal{P}_2(\rr^d) \times [0,t) \times (\mathbb{R}^d)^3$ and any $\gv=(v, v')$, $\gv_1 = (v_1,v_1')$, $\gv_2 = (v_2,v_2')$ in $\mathbb{R}^{d} \times \mathbb{R}^d$
\begin{align}
|\partial^2_\mu p_m(\mu, s, t, x, z)(\gv)| & \leq  \frac{\mathscr C_m^{1,0}(C^+,t-s)}{(t-s)^{1-\frac{\eta}{2}}} \, g(c (t-s), z-x), \label{second:deriv:mes:induction:decoupling:mckean}
\end{align}
\begin{align}
|\partial^2_\mu p_m(\mu, s, t, x, z)(\gv_1) & - \partial^2_\mu p_m(\mu, s, t, x, z)(\gv_2)| \label{second:deriv:mes:reg:space:deriv:arg:estimate:induction:decoupling:mckean} \\
& \leq \mathscr C_m^{1,\beta}(C_\beta^+,t-s)  \frac{|\gv_1-\gv_2|^{\beta}}{(t-s)^{1+\frac{\beta-\eta}{2}}} \, g(c (t-s), z-x), \notag
\end{align}
\begin{align}
|\partial^2_\mu p_m(\mu, s, t, x, z)(\gv) & - \partial^2_\mu p_m(\mu, s, t, x', z)( \gv)| \label{second:deriv:mes:reg:space:deriv:arg:estimate:induction:decoupling:mckean2} \\
& \leq  \mathscr C_{m}^{1,\beta}(C_\beta^+,t-s) \frac{|x-x'|^{\beta}}{(t-s)^{1+\frac{\beta-\eta}{2}}} \, \left\{ g(c (t-s), z-x) + g(c(t-s), z-x')\right\}. \notag
\end{align}

\item  If additionally \HRpp\, is satisfied, then for any $\beta \in [0,\eta)$, there exist positive constants $C^{++}_\beta := C(T, \HRpp, \HE, \beta), \, c:=c(\lambda)>0$ such that for any $(\mu, \mu', s, x, z) \in (\mathcal{P}_2(\mathbb{R}^d))^2 \times [0,t) \times (\mathbb{R}^d)^2$, any $\gv \in \mathbb{R}^d \times \mathbb{R}^d$ and any $s_1,s_2$ in $[0,t)^2$
\begin{align}
|\partial^2_\mu p_m(\mu, s, t, x, z)(\gv) - \partial^2_\mu p_m(\mu', s, t, x, z)(\gv)| & \leq  \mathscr C_m^{1,\beta}(C_\beta^{++},t-s) \frac{W_2(\mu, \mu')^{\beta}}{(t-s)^{1+\frac{\beta-\eta}{2}}} \, g(c (t-s), z-x), \label{second:deriv:mes:reg:mes:estimate:induction:decoupling:mckean}
\end{align}

\noindent and
\begin{align}
 | & \partial_\mu^2 p_{m}(\mu, s_1, t, x, z)(\gv) - \partial_\mu^2 p_{m}(\mu, s_2, t, x, z)(\gv)| \nonumber \\
 & \leq \mathscr{C}^{1, \beta}_{m}(C^{++}_{\beta}, t-s_1 \vee s_2) \left\{ \frac{|s_1-s_2|^{\frac{\beta}{2}}}{(t-s_1)^{1+\frac{\beta-\eta}{2} }} \, g(c(t-s_1), z-x) + \frac{|s_1-s_2|^{\frac{\beta}{2}}}{(t- s_2)^{1+ \frac{\beta-\eta}{2} }} \, g(c(t-s_2), z-x) \right\}.\label{regularity:time:estimate:secon:deriv:decoupling:mckean} 
\end{align}

\end{itemize}

\end{prop}

\medskip

\noindent \emph{Step 2: Extraction of a convergent subsequence}

\medskip

Our next step now is to extract from the sequences $\left\{\mathbb{L}^{2} \ni \xi \mapsto \widetilde{p}_m(\xi, s, t, x, z) , m \geq 0 \right\}$ (the lifting of $\mu \mapsto p_m(\mu, s, t, x, z)$), $\left\{ [0,t) \times \mathbb{R}^d \times \pp \times \mathbb{R}^d \ni (s, x, \mu, v)  \mapsto \partial_\mu p_m(\mu, s, t, x, z)(v) , m \geq 0 \right\}$,\\ $\left\{ [0,t) \times \mathbb{R}^d \times \pp \times (\mathbb{R}^d)^2 \ni (s, x, \mu, \gv) \mapsto \partial^2_\mu p_m(\mu, s, t, x, z)(\gv) , m \geq 0 \right\}$ the corresponding subsequences which converge locally uniformly using the Arzel\`a-Ascoli theorem.

Since the coefficients $b_i, \, a_{i, j}$ are bounded and the initial condition $\mu \in \pp$, the sequence $(\P^{(m)})_{m\geq0}$ constructed in \emph{Step 1} is tight. Relabelling the indices if necessary, we may assert that $(\P^{(m)})_{m \geq 0}$ converges weakly to a probability measure $\P^{\infty}$. From standard arguments that we omit (passing to the limit in the characterisation of the martingale problem solved by $\P^{(m)}$), we deduce that $\P^\infty$ is the probability measure $\P$ induced by the unique weak solution to the McKean-Vlasov SDE \eqref{SDE:MCKEAN}. As a consequence, every convergent subsequence converges to the same limit $\P$ and so does the original sequence $(\P^{(m)})_{m\geq1}$. 

By Lebesgue's dominated convergence theorem, for any fixed $t >0$ and $z \in \rr^d$, using \eqref{iter:param:classic}, one may pass to the limit as $m\uparrow \infty$ in the parametrix infinite series \eqref{series:approx:mckean} and thus deduce that the sequence of functions $\left\{\mathcal{K} \ni (s, x, \mu) \mapsto p_m(\mu, s, t, x, z), \, m\geq1\right\}$, $\mathcal{K}$ being a compact set of $[0,t) \times \mathbb{R}^d \times \pp$, converges to $(s, x, \mu) \mapsto p(\mu, s, t, x, z)$ given by the infinite series \eqref{parametrix:series:expansion} for any fixed $(s, x, \mu)$. Moreover, it is clearly uniformly bounded and from \eqref{first:second:estimate:induction:decoupling:mckean}, \eqref{time:derivative:induction:decoupling:mckean} and \eqref{bound:derivative:heat:kernel}, it is equicontinuous. Relabelling the indices if necessary, from the Arzel\`a-Ascoli theorem, we may assert that it converges uniformly. Hence, $[0,t) \times \rr^d \times \pp \ni (s, x, \mu) \mapsto p(\mu, s, t, x, z)$ is continuous.

For any $\mu \in \pp$ and any positive integer $m$, the mapping $ (s, x) \mapsto p_m(\mu, s, t ,x ,z)$ is in $\mathcal{C}^{0, 2 }([0,t) \times \rr^d)$. Moreover, from the estimates \eqref{regularity:measure:estimate:v1:v2:v3:decoupling:mckean}, \eqref{regularity:time:estimate:v1:v2:v3:decoupling:mckean} and \eqref{reg:heat:kernel:deriv} (for $n=1,2$), the sequence of functions $\mathcal{K} \ni (s, x, \mu) \mapsto \partial_x p_m(\mu, s, t, x, z)$, $\partial^2_x p_m(\mu, s, t, x, z)$,  $\mathcal{K}$ being again a compact set of $[0,t) \times \mathbb{R}^d \times \pp$, are uniformly bounded and equicontinuous. Hence, from Arzel\`a-Ascoli's theorem, we may assert that $(s, x) \mapsto p(\mu, s, t, x, z)\in \mathcal{C}^{0, 2 }([0,t) \times \mathbb{R}^d)$ and that the mappings $[0,t) \times \mathbb{R}^d \times \pp \ni (s, x, \mu) \mapsto \partial_x p(\mu, s, t, x, z)$, $\partial^2_x p(\mu, s, t, x, z)$ are continuous.

Considering now the sequence $\left\{ \mathcal{K} \ni \xi \mapsto D\widetilde{p}_{m}(\xi, s, t, x, z) = \partial_\mu p _m([\xi], s, t, x, z)(\xi),\, m\geq 1 \right\}$, $\mathcal{K}$ being any compact set of $\mathbb{L}^{2}$, from \eqref{first:second:estimate:induction:decoupling:mckean} (with $n=0$), we deduce that it is uniformly bounded. From \eqref{first:second:estimate:induction:decoupling:mckean} (with $n=1$) and \eqref{regularity:measure:estimate:v1:v2:decoupling:mckean} (with $n=0$), it is equicontinuous. Relabelling the indices if necessary, from the Arzel\`a-Ascoli theorem, we may assert that it converges uniformly. We thus deduce that the map $\mathcal{K} \ni \xi \mapsto \widetilde{p}(\xi, s, t, x, z)$ is continuously differentiable. As a consequence, $\pp \ni \mu \mapsto p(\mu, s, t, x, z)$ is continuously L-differentiable. 

From \eqref{equicontinuity:second:third:estimate:decoupling:mckean}, \eqref{regularity:measure:estimate:v1:v2:decoupling:mckean} and \eqref{regularity:time:estimate:v1:v2:decoupling:mckean} (with $n=0$) and \eqref{first:second:estimate:induction:decoupling:mckean} (with $n=0$ and $n=1$), the sequence $ \{\mathcal{K} \ni (s, x, \mu, v) \mapsto \partial_\mu p_{m}(\mu, s, t, x, z)(v)$, $m\geq 1\}$, $\mathcal{K}$ being a compact set of $[0,t) \times \rr^d \times \pp \times \rr^d$, is uniformly bounded and equicontinuous so that the map $[0,t) \times \mathbb{R}^d \times \pp \times \mathbb{R}^d \ni (s, x, \mu, v) \mapsto \partial_\mu p (\mu, s, t, x , z)(v)$ is continuous. 

From \eqref{first:second:estimate:induction:decoupling:mckean} (with $n=1$) and \eqref{equicontinuity:first:estimate:decoupling:mckean}, the sequence $\big\{\mathbb{R}^d \supset B(0,R) \ni v \mapsto \partial_v [\partial_\mu p_{m}(\mu, s, t, x, z)](v), m\geq1\big\}$, is bounded and equicontinuous so that we deduce that the map $\mathbb{R}^d \ni v \mapsto \partial_\mu p(\mu, s, t, x, z)(v)$ is continuously differentiable. Also, the continuity of the map $[0,t) \times \mathbb{R}^d \times \pp \times \mathbb{R}^d \ni(s, x, \mu, v) \mapsto \partial_v [\partial_\mu p(\mu, s, t, x, z)](v)$ can be deduced from the uniform convergence of the sequence of continuous mappings $\big\{ \mathcal{K}  \ni (s, x, \mu, v) \mapsto \partial_v [\partial_\mu p_{m}(\mu, s, t, x, z)](v), m\geq1\big\}$, $\mathcal{K}$ being a compact set of $[0,t) \times \rr^d \times \pp \times\rr^d$, along a subsequence, derived by combining the estimates \eqref{first:second:estimate:induction:decoupling:mckean}, \eqref{equicontinuity:second:third:estimate:decoupling:mckean}, \eqref{equicontinuity:first:estimate:decoupling:mckean}, \eqref{regularity:measure:estimate:v1:v2:decoupling:mckean} and \eqref{regularity:time:estimate:v1:v2:decoupling:mckean} for $n=1$ with the Arzel\`a-Ascoli theorem.

For each fixed $v\in \mathbb{R}^d$, we now consider the following sequence of Fr\'echet derivatives of the map $L^{2}(\Omega, \mathcal{A}, \P) \ni \xi' \mapsto \partial_\mu p_m([\xi'], s, t, x, z)(v)$ given by $\left\{ \mathcal{K} \ni \xi \mapsto  \partial^2_\mu p _m([\xi], s, t, x, z)(v)(\xi),\, m\geq 1 \right\}$, $\mathcal{K}$ being a compact set of $L^{2}(\Omega, \mathcal{A}, \P)$. From \eqref{second:deriv:mes:induction:decoupling:mckean}, \eqref{second:deriv:mes:reg:space:deriv:arg:estimate:induction:decoupling:mckean} and \eqref{second:deriv:mes:reg:mes:estimate:induction:decoupling:mckean}, this sequence is uniformly bounded and equicontinuous. Relabelling the indices if necessary, from the Arzel\`a-Ascoli theorem, we may assert that it converges uniformly. Hence, for each fixed $v\in \mathbb{R}^d$, $ \pp \ni \mu \mapsto \partial_\mu p(\mu, s, t, x, z)(v)$ is continuously L-differentiable and we denote its derivative $\partial^2_\mu p(\mu, s, t, x, z)(v)(v')$ by $\partial^2_\mu p(\mu, s, t, x, z)(v, v')$. 

From the estimates \eqref{equicontinuity:second:third:estimate:decoupling:mckean}, \eqref{regularity:measure:estimate:v1:v2:decoupling:mckean} and \eqref{regularity:time:estimate:v1:v2:decoupling:mckean} (the three for $n=0$) and \eqref{first:second:estimate:induction:decoupling:mckean} on the one hand and \eqref{second:deriv:mes:induction:decoupling:mckean}, \eqref{second:deriv:mes:reg:space:deriv:arg:estimate:induction:decoupling:mckean2}, \eqref{second:deriv:mes:reg:mes:estimate:induction:decoupling:mckean}, \eqref{second:deriv:mes:reg:space:deriv:arg:estimate:induction:decoupling:mckean} and \eqref{regularity:time:estimate:secon:deriv:decoupling:mckean} on the other hand, both sequences $ \mathcal{K} \ni (s, x, \mu, v) \mapsto \partial_\mu p_{m}(\mu, s, t, x, z)(v)$ and $ \mathcal{K}' \ni (s, x, \mu, \gv) \mapsto \partial^2_\mu p_{m}(\mu, s, t, x, z)(\gv)$, $m\geq 1$, $\mathcal{K}$ and $\mathcal{K}'$ being compact sets of $[0,t) \times \mathbb{R}^d \times \pp \times \mathbb{R}^d$ and $[0,t) \times \mathbb{R}^d \times \pp \times (\mathbb{R}^d)^2$, are uniformly bounded and equicontinuous so that, from the Arzel\`a-Ascoli theorem, the map $[0,t) \times \mathbb{R}^d \times \pp \times \mathbb{R}^d \ni (s, x, \mu, v) \mapsto \partial_\mu p (\mu, s, t, x , z)(v)$, $[0,t) \times \mathbb{R}^d \times \pp \times (\mathbb{R}^d)^2 \ni (s, x, \mu, \gv) \mapsto \partial^2_\mu p (\mu, s, t, x , z)(\gv)$ are continuous. 

%Passing to the limit (along the considered subsequence) in \eqref{second:derivative:nextstep:decoupling}, we get
%\begin{align}
%& \partial_v [\partial_\mu \bar{p}(\mu, s, t , x, z)](v) \nonumber \\
%&  =  Df_{z-x}(\int_s^t a(r, [\bar{X}^{s, \xi}_r]) dr) .\int_s^t \left\{\int (\widetilde{a}(r, y',  [\bar{X}^{s,\xi}_r]) - \widetilde{a}(r, v, [\bar{X}^{s,\xi}_r]))  \partial^2_x \bar{p}(\mu, s, r, v, y') \, dy' \right. \nonumber\\
% &  \left. +  \int  \int ( \widetilde{a}(r, y', [\bar{X}^{s,\xi}_r]) - \widetilde{a}(r, x', [\bar{X}^{s,\xi}_r]) )  \partial_v [\partial_\mu \bar{p}(\mu, s, r, x', y')](v) dy' \, \mu(dx')\right\} dr. \label{second:derivative:final:decoupling}
% \end{align}

%\noindent and the estimates \eqref{first:second:lions:derivative:decoupling} for $n=1$ and \eqref{equicontinuity:first:estimate} hold. 

The estimates \eqref{second:lions:derivative:mckean:decoupling}, \eqref{second:deriv:mes:reg:mes:space:estimate:induction:mckean:decoupling} and \eqref{second:deriv:mes:reg:time:estimate:induction:mckean:decoupling} then follow by passing to the limit in the corresponding upper-bounds proved in the first step.

 \medskip

\noindent \emph{Step 3: $\mathcal{C}^{1, 2, 2}_f([0,t) \times \rr^d \times \pp)$ regularity and related estimates.}

\medskip

Let us now prove that $(s, x, \mu) \mapsto p(\mu, s, t, x, z)$ is in $\mathcal{C}^{1, 2, 2}_f([0,t) \times \rr^d \times \pp)$. We here follow the same lines of reasonings as those employed in \cite{chaudruraynal:frikha}. From the Markov property satisfied by the SDE \eqref{SDE:MCKEAN}, stemming from the well-posedness of the related martingale problem, the following relation is satisfied for all $h>0$
$$
p(\mu, s-h, t, x, z) = \E[p([X^{s-h, \xi}_s], s, t, X^{s-h, x, \mu}_s, z)].
$$

Combining estimates \eqref{first:second:lions:derivative:mckean:decoupling} and \eqref{bound:density:parametrix} (for $n=1$) with the chain rule formula of Proposition \ref{prop:chain:rule:joint:space:measure} (with respect to the space and measure variables only) we obtain
$$
  \E[p([X^{s-h, \xi}_s], s, t, X^{s-h, x, \mu}_s, z)] = p(\mu, s, t, x, z) + \E\left[\int_{s-h}^{s} \mathcal{L}_r p([X^{s-h, \xi}_r], s, t, X^{s-h, x, \mu}_r, z) \, dr \right]
$$

\noindent \noindent where the operator $\mathcal{L}_r$ is given by \eqref{inf:generator:mckean:vlasov}.

 Hence, one has
$$
\frac{1}{h}(p(\mu, s-h, t, x, z)  - p(\mu, s, t, x, z)) = \frac{1}{h} \E\left[ \int_{s-h}^s \mathcal{L}_r p([X^{s-h, \xi}_r], s, t, X^{s-h, x, \mu}_r, z) \, dr \right]
$$

%\noindent with
%\begin{align*}
%& \bar{\mathcal{L}}_r \bar{p}([\bar{X}^{s-h, \xi}_r], s, t, x, z)  =  \frac12 \sum_{i, j=1}^d a_{i, j}(r, [\bar{X}^{s-h, \xi}_r]) \partial_{x_i, x_j}\bar{p}([\bar{X}^{s-h, \mu}_r], s, t, x, z)  \\
%& \quad \quad + \frac12 \int\int \sum_{i, j=1}^d a_{i, j}(r, [\bar{X}^{s-h, \xi}_r])  \partial_{v_i} [\partial_{\mu} \bar{p}([\bar{X}^{s-h, \mu}_r], s, t, x, z)(v)]_j  \, \bar{p}(\mu, s-h, r, x', v) \, dv \, \mu(dx').
%\end{align*}

\noindent so that, letting $h\downarrow 0$, from the differentiability of $[0, t) \ni s\mapsto p(\mu, s, t, x, z)$, the boundedness and continuity of the coefficients as well as the continuity of the maps $(\mu, x, v) \mapsto p(\mu, s, t, x, z), \, \partial^{1+n}_x p(\mu, s, t, x, z), \, \partial^{n}_v[\partial_\mu p(\mu, s, t, x, z)](v)$, for $n=0,1$, we deduce  
 \begin{eqnarray*}
\partial_s p(\mu, s, t, x, z) = - \mathcal{L}_s p(\mu, s, t, x, z) \quad \mbox{ on } [0,t) \times \mathbb{R}^d \times \pp
\end{eqnarray*} 

\noindent so that $[0,t) \times \mathbb{R}^d \times \pp \ni (s, x, \mu) \mapsto \partial_s p(\mu, s, t, x, z)$ is continuous.

\section{Propagation of chaos}\label{propagation:of:chaos}

This section is devoted to the proof of Theorems \ref{propagation:chaos:fundamental:sol:theorem}, \ref{theo:weakChaos} and \ref{propagation:prop:path}. As already mentioned, our propagation of chaos results crucially rely on the regularity properties provided by Theorems \ref{derivative:density:sol:mckean:and:decoupling}, \ref{fundamental:solution:wasserstein:space} and Proposition \ref{fully:c2:regularity:property:decoupling:density}. 

%\subsection{Construction of an approximation sequence}\label{subsection:construction:approximation:sec}
%We first present here how we construct an approximation of the solution of the Kolmogorov backward PDE on the Wasserstein space. For a positive integer $n$ and any $i \in \left\{1, \cdots, d\right\}$, denoting by $X$ and $N$ two independent random variables with law $m$ and $\mathcal{N}(0, I_d)$ respectively, we set
%$$
%b^n_i(t, x, m) = \int_{\mathbb{R}^d} b_i(t, x-y, [X+ n^{-1} N]) \, g(n^{-1}, y) \, dy
%$$   
%\noindent and define $a^n_{i, j}(t, x, m)$ accordingly for $(i, j) \in \left\{1, \cdots, d\right\}^2$. By doing so, it is then clear that if the coefficients $b$ and $a$ satisfy \A{HR}, \A{HR}$_+$ or \A{HR}$_{++}$ then so do the mollified version $b^n$ and $a^n$, uniformly in $n$. The key point is that under \A{HR}$_+$, it holds
%$$
%(b_i^n, a_{i, j}^{n}) \rightarrow (b, a),\quad ([\delta b_i/\delta m], [\delta a_{i, j}/\delta m]) \rightarrow  \, \mbox{ as } n\rightarrow \infty
%$$
%
%\noindent uniformly

\subsection{Proof of Theorem \ref{propagation:chaos:fundamental:sol:theorem}.}

The strategy  consists in testing the fundamental solution $p(\mu, s, t, z)$ to the backward Kolmogorov PDE \eqref{backward:kolmogorov:pde} stated on the Wasserstein space as an approximate solution to the one-dimensional marginal density of the $N$-dimensional particle systems. For any fixed $(t, z) \in (0,T] \times \mathbb{R}^d$, the natural candidate for being an approximate solution is 
$$
p(\mu^{N}_s, s, t, z), \quad \mbox{ with } \quad \mu^{N}_s := \frac{1}{N} \sum_{i=1}^{N} \delta_{X^{i}_s}
$$

\noindent where $\left\{ (X^{i}_t)_{t\in [0,T]}, 1\leq i \leq N\right\}$ are given by the unique weak solution to the system of particles with dynamics given by \eqref{SDE:particle:system}. We start with the following lemma concerning the control of the initial error induced by the difference of the fundamental solution taken along the initial empirical measure $\mu_0^N$ and $\mu$.

\begin{lem}\label{lem:first:order:expansion:initial:error} Under \HE\, and \HRp, for any $(t, \mu, z) \in (0,T] \times \pp \times \mathbb{R}^d$ and any positive integer $N$, the following error bound is satisfied
\begin{align}
|\E[ p(\mu^{N}_0, 0, t, z) - p(\mu, 0, t, z)]| \leq \frac{K^+}{N} \left\{ \frac{1}{t^{\frac{1-\eta}{2}}} \int_{\mathbb{R}^d} g(ct, z-x) |x| \mu(dx) + \frac{1}{t^{1-\frac{\eta}{2}}} \int_{\mathbb{R}^d} g(c t, z-x)  \mu(dx) \right\} \label{error:bound:initial:density}
\end{align}
\noindent for some positive constants $K^+:=K(T, \HRp, \HE)$, $c:=c(\lambda)$, $T\mapsto K(T, \HRp, \HE)$ being non-decreasing.

Assume additionally that \HRpp\, holds. Then, recalling that $\widetilde{\xi}$ stands for an $\mathbb{R}^d$-valued random variable independent of $(\xi^{i})_{1\leq i\leq N}$ with law $\mu$, the following first order expansion holds
\begin{align}
\E[ p(\mu^{N}_0, 0, t, z) - p(\mu, 0, t, z)] & = \frac{1}{N}   \E\Big[\frac{\delta}{\delta m}p(\mu, 0, t, \xi^1, z)(\xi^1)- \frac{\delta}{\delta m}p(\mu, 0, t, \xi^1, z)(\widetilde{\xi})\Big] \notag \\
& \quad + \frac{1}{2N} \E\Big[\frac{\delta^2}{\delta m^2}p(\mu, 0, t, \xi^1, z)(\widetilde{\xi}, \widetilde{\xi}) - \frac{\delta^2}{\delta m^2}p(\mu, 0, t, \xi^1, z)(\widetilde{\xi}, \xi^2) \Big] \label{first:order:expansion:initial:error}\\
& \quad + \frac{1}{N} \mathcal{R}^{N}_2(\mu, 0, t, z) \notag
\end{align}
\noindent where for any $\beta \in [0,\eta)$
\begin{align}
 |& \mathcal{R}^{N}_2(\mu, 0, t, z) |  \leq K_\beta^{++} \Big(\frac{1}{t^{1-\frac{\eta}{2}}} \E\Big[g(ct, z-\xi^1) \left(W_2(\mu^{N}_0, \mu)  (1+ |\xi^1|) + \frac{1}{N}\right) \Big] \label{bound:remainder:term:2:initial:error} \\
 & \quad + \frac{1}{t^{1+\frac{\beta-\eta}{2}}} \E\Big[ g(c t, z-\xi^1) W_2(\mu^{N}_0, \mu)^{\beta}(1+|\xi^2|)\Big] \Big) \notag
\end{align}
\noindent for some positive constants $K^{++}_\beta:=K(T, \HRpp, \HE, \beta)$, $c:=c(\lambda)$, $T\mapsto K(T, \HRpp, \HE, \beta)$ being non-decreasing.

\end{lem}

\begin{proof}

\emph{Step 1: proof of the error bound \eqref{error:bound:initial:density}.} \\
We consider the sequence $\left\{[0,t) \times \pp \ni (s, \mu) \mapsto p_m(\mu, s , t, z), \, m\geq1 \right\}$ of $\mathcal{C}^{1, 2}_f([0,t)\times\pp)$ maps constructed in Section \ref{full:c2:regularity:section} and recall that $p_m(\mu, 0, t, z)$ converges to $p(\mu, 0, t, z)$ for any fixed $\mu,\, t, \, z$. Hence, using the relation \eqref{conv:relation:step:m}, the estimate \eqref{bound:derivative:heat:kernel} together with the dominated convergence theorem, we get $\lim_{m \uparrow \infty} \E[ p_m(\mu^{N}_0, 0, t, z) - p_m(\mu, 0, t, z)] = \E[ p(\mu^{N}_0, 0, t, z) - p(\mu, 0, t, z)] $. It thus suffices to prove the error bound \eqref{error:bound:initial:density} for the difference $\E[ p_m(\mu^{N}_0, 0, t, z) - p_m(\mu, 0, t, z)] $. \\

 By exchangeability in law of the random variables $(\xi^{i})_{1\leq i\leq N}$ and the mean-value theorem
\begin{align}
\E[p_m(\mu^{N}_0, 0, t, z) & - p_m(\mu, 0, t, z)]  \notag\\
 & = \E[p_m(\mu^{N}_0, 0, t, \xi^1, z) - p_m(\mu, 0, t, \xi^1, z)] \notag \\
& = \int_{0}^1\int_{\mathbb{R}^d} \E\Big[\frac{\delta}{\delta m}p_m(\mu^{\lambda_1, N}_0, 0, t, \xi^1, z)(y) \,  (\mu^{N}_0 - \mu)(dy)\Big] \, d\lambda_1\notag \\
& = \frac{1}{N} \int_{0}^1  \E\Big[\frac{\delta}{\delta m}p_m(\mu^{\lambda_1, N}_0, 0, t, \xi^1, z)(\xi^1)- \frac{\delta}{\delta m}p_m(\mu^{\lambda_1, N}_0, 0, t, \xi^1, z)(\widetilde{\xi})\Big] \, d\lambda_1\label{eq:cont:p0} \\
& \quad + \frac{N-1}{N} \int_0^1  \E\Big[\frac{\delta}{\delta m}p_m(\mu^{\lambda_1, N}_0, 0, t, \xi^1, z)(\xi^2) - \frac{\delta}{\delta m}p_m(\mu^{\lambda_1, N}_0, 0, t, \xi^1, z)(\widetilde{\xi})\Big] d\lambda_1\notag
\end{align}

\noindent where we introduced the notation $\mu^{\lambda_1, N}_0 := \lambda_1 \mu^{N}_0 + (1-\lambda_1) \mu$ and recall that $\widetilde{\xi}$ is a random variable independent of the sequence $(\xi^{i})_{1\leq i \leq N}$ with law $\mu$. We now introduce the measure $\widetilde{\mu}^{\lambda_1, N}_0 := \lambda_1 \widetilde{\mu}^{N}_0 + (1-\lambda_1) \mu$ with $\widetilde{\mu}^{N}_0 := \mu^{N}_0  + \frac{1}{N}(\delta_{\widetilde{\xi}} - \delta_{\xi^2})$ and notice that
$$
\E\Big[\frac{\delta}{\delta m} p_m(\mu^{\lambda_1, N}_0, 0, t,\xi^1, z)(\xi^2)\Big] = \E\Big[\frac{\delta}{\delta m}p_m(\widetilde{\mu}^{\lambda_1, N}_0, 0, t,\xi^1, z)(\widetilde{\xi})\Big].
$$
%$$
%\E\Big[\partial_\mu p_m(\mu^{\lambda_1, N}_0, 0, t,\xi^1, z)(\lambda_2 \xi^2)\cdot \xi^2\Big] = \E\Big[\partial_\mu p_m(\widetilde{\mu}^{\lambda_1, N}_0, 0, t,\xi^1, z)(\lambda_2 \widetilde{\xi})\cdot \widetilde{\xi}\Big].
%$$

\noindent so that, again by the mean-value theorem
\begin{align*}
\E\Big[\frac{\delta}{\delta m} &p_m(\mu^{\lambda_1, N}_0, 0, t, \xi^1, z)(\xi^2)  - \frac{\delta}{\delta m} p_m(\mu^{\lambda_1, N}_0, 0, t,\xi^1, z)(\widetilde{\xi})\Big] \\
& =  \frac{\lambda_1}{N}\int_0^1 \E\Big[\frac{\delta^2}{\delta m^2} p_m(\widetilde{\mu}^{\lambda_1, \lambda_2, N}_0, 0, t, \xi^{1}, z)(\widetilde{\xi},\widetilde{\xi}) -\frac{\delta^2}{\delta m^2} p_m(\widetilde{\mu}^{\lambda_1, \lambda_2, N}_0, 0, t, \xi^1, z)(\widetilde{\xi}, \xi^2) \Big]  \, d\lambda_2
%& =  \frac{\lambda_1}{N}\int_{[0,1]^2} \E\Big[ \widetilde{\xi}\cdot \partial^2_\mu p_m(\widetilde{\mu}^{\lambda_1, \lambda_3, N}_0, 0, t, \xi^{1}, z)(\lambda_2 \widetilde{\xi},\lambda_4 \widetilde{\xi})\cdot\widetilde{\xi} - \xi^2. \partial^2_\mu p_m(\widetilde{\mu}^{\lambda_1, \lambda_3, N}_0, 0, t, \xi^1, z)(\lambda_2 \widetilde{\xi}, \lambda_4 \xi^2)\cdot\widetilde{\xi} \Big]  \, d\lambda_3 d\lambda_4
\end{align*}

%Hence, again by the mean-value theorem
%\begin{align*}
%\E\Big[\partial_\mu &p_m(\mu^{\lambda_1, N}_0, 0, t, \xi^1, z)(\lambda_2 \xi^2). \xi^2  - \partial_\mu p_m(\mu^{\lambda_1, N}_0, 0, t,\xi^1, z)(\lambda_2 \widetilde{\xi}). \widetilde{\xi}\Big] \\
%& =  \frac{\lambda_1}{N}\int_0^1 \E\Big[\frac{\delta}{\delta m} \partial_\mu p_m(\widetilde{\mu}^{\lambda_1, \lambda_3, N}_0, 0, t, \xi^{1}, z)(\lambda_2 \widetilde{\xi},\widetilde{\xi})\cdot\widetilde{\xi} -\frac{\delta}{\delta m}\partial_\mu p_m(\widetilde{\mu}^{\lambda_1, \lambda_3, N}_0, 0, t, \xi^1, z)(\lambda_2 \widetilde{\xi}, \xi^2).\widetilde{\xi} \Big]  \, d\lambda_3\\
%& =  \frac{\lambda_1}{N}\int_{[0,1]^2} \E\Big[ \widetilde{\xi}\cdot \partial^2_\mu p_m(\widetilde{\mu}^{\lambda_1, \lambda_3, N}_0, 0, t, \xi^{1}, z)(\lambda_2 \widetilde{\xi},\lambda_4 \widetilde{\xi})\cdot\widetilde{\xi} - \xi^2. \partial^2_\mu p_m(\widetilde{\mu}^{\lambda_1, \lambda_3, N}_0, 0, t, \xi^1, z)(\lambda_2 \widetilde{\xi}, \lambda_4 \xi^2)\cdot\widetilde{\xi} \Big]  \, d\lambda_3 d\lambda_4
%\end{align*}

\noindent with the notation $\widetilde{\mu}^{\lambda_1, \lambda_2, N}_0 := \lambda_2 \widetilde{\mu}^{\lambda_1, N}_0 + (1-\lambda_2) \mu^{\lambda_1, N}_0$. Plugging the previous identity into \eqref{eq:cont:p0}, we derive
\begin{align}
\E[& p_m(\mu^{N}_0, 0, t, z)  - p_m(\mu, 0, t, z)]  \notag \\
& = \frac{1}{N} \int_{0}^1  \E\Big[\frac{\delta}{\delta m}p_m(\mu^{\lambda_1, N}_0, 0, t, \xi^1, z)(\xi^1)- \frac{\delta}{\delta m}p_m(\mu^{\lambda_1, N}_0, 0, t, \xi^1, z)(\widetilde{\xi})\Big] \, d\lambda_1\notag \\
& \quad + \frac{N-1}{N^2} \int_{[0,1]^2} \lambda_1 \E\Big[\frac{\delta^2}{\delta m^2} p_m(\widetilde{\mu}^{\lambda_1, \lambda_2, N}_0, 0, t, \xi^{1}, z)(\widetilde{\xi},\widetilde{\xi}) -\frac{\delta^2}{\delta m^2} p_m(\widetilde{\mu}^{\lambda_1, \lambda_2, N}_0, 0, t, \xi^1, z)(\widetilde{\xi}, \xi^2) \Big]  \,  d\lambda_1d\lambda_2 \label{first:order:expansion:pm} \\
& = \frac{1}{N} \int_{[0,1]^2}  \E\Big[\partial_\mu p_m(\mu^{\lambda_1, N}_0, 0, t, \xi^1, z)(\lambda_2 \xi^1 + (1-\lambda_2) \widetilde{\xi})\cdot (\xi^1-\widetilde{\xi})\Big] \, d\lambda_1 d\lambda_2\notag \\
& +\frac{(N-1)}{N^2} \int_{[0, 1]^4} \lambda_1 \E\Big[ \widetilde{\xi}\cdot \partial^2_\mu p_m(\widetilde{\mu}^{\lambda_1, \lambda_2, N}_0, 0, t, \xi^{1}, z)(\lambda_3 \widetilde{\xi},\lambda_4 \widetilde{\xi})\cdot\widetilde{\xi} \notag \\
&\qquad  - \xi^2\cdot \partial^2_\mu p_m(\widetilde{\mu}^{\lambda_1, \lambda_2, N}_0, 0, t, \xi^1, z)(\lambda_3 \widetilde{\xi}, \lambda_4 \xi^2)\cdot\widetilde{\xi} \Big] \, d\lambda_1 d\lambda_2 d\lambda_3 d\lambda_4 \notag
\end{align}

%Therefore,
%\begin{align*}
%& \E[ p_m(\mu^{N}_0, 0, t, z) - p_m(\mu, 0, t, z)] \\
%& = \frac{1}{N} \int_{[0,1]^2}  \E\Big[\partial_\mu p_m(\mu^{\lambda_1, N}_0, 0, t, \xi^1, z)(\lambda_2 \xi^1 + (1-\lambda_2) \widetilde{\xi})\cdot (\xi^1-\widetilde{\xi})\Big] \, d\lambda_1 d\lambda_2 \\
%& +\frac{(N-1)}{N^2} \int_{[0, 1]^4} \lambda_1 \E\Big[ \widetilde{\xi}. \partial^2_\mu p_m(\widetilde{\mu}^{\lambda_1, \lambda_3, N}_0, 0, t, \xi^{1}, z)(\lambda_2 \widetilde{\xi},\lambda_4 \widetilde{\xi})\cdot\widetilde{\xi}\\
%&\qquad  - \xi^2\cdot \partial^2_\mu p_m(\widetilde{\mu}^{\lambda_1, \lambda_3, N}_0, 0, t, \xi^1, z)(\lambda_2 \widetilde{\xi}, \lambda_4 \xi^2)\cdot\widetilde{\xi} \Big] \, d\lambda_1 d\lambda_2 d\lambda_3 d\lambda_4
%%= \frac{1}{N}\int_{[0,1]^2} \E\Big[\frac{\delta^2}{\delta m^2} p_m(\widetilde{\mu}^{\lambda, \lambda', N}_0, 0, t, z)(\widetilde{\xi},\widetilde{\xi}) -\frac{\delta^2}{\delta m^2} p_m(\widetilde{\mu}^{\lambda, \lambda', N}_0, 0, t, z)(\widetilde{\xi}, \xi^1) \Big] \, d\lambda d\lambda'.
%\end{align*}

\noindent which in turn by using \eqref{first:second:estimate:induction:decoupling:mckean} and \eqref{second:deriv:mes:induction:decoupling:mckean} eventually yield
\begin{align*}
|\E[ p_m(\mu^{N}_0, 0, t, z) - p_m(\mu, 0, t, z)]| \leq \frac{K^+}{N} \left\{ \frac{1}{t^{\frac{1-\eta}{2}}} \int_{\rr^d} g(ct, z-x) |x| \mu(dx) + \frac{1}{t^{1-\frac{\eta}{2}}} \int_{\rr^d} g(c t, z-x)  \mu(dx) \right\}.
\end{align*}

The proof of \eqref{error:bound:initial:density} is now complete. \\

\noindent \emph{Step 2: proof of the first order expansion \eqref{first:order:expansion:initial:error}.} \\

We here assume that \HRpp\, holds. In a completely analogous manner, one obtains the identity \eqref{first:order:expansion:pm} for $\E[p(\mu^{N}_0, 0, t, z)  - p(\mu, 0, t, z)]$. We thus write
\begin{align*}
\E[& p(\mu^{N}_0, 0, t, z)  - p(\mu, 0, t, z)]  \notag \\
& = \frac{1}{N} \int_{0}^1  \E\Big[\frac{\delta}{\delta m}p(\mu^{\lambda_1, N}_0, 0, t, \xi^1, z)(\xi^1)- \frac{\delta}{\delta m}p(\mu^{\lambda_1, N}_0, 0, t, \xi^1, z)(\widetilde{\xi})\Big] \, d\lambda_1\notag \\
& \quad + \frac{N-1}{N^2} \int_{[0,1]^2} \lambda_1 \E\Big[\frac{\delta^2}{\delta m^2} p(\widetilde{\mu}^{\lambda_1, \lambda_2, N}_0, 0, t, \xi^{1}, z)(\widetilde{\xi},\widetilde{\xi}) -\frac{\delta^2}{\delta m^2} p(\widetilde{\mu}^{\lambda_1, \lambda_2, N}_0, 0, t, \xi^1, z)(\widetilde{\xi}, \xi^2) \Big]  \,  d\lambda_1d\lambda_2\\
& = \frac{1}{N}   \E\Big[\frac{\delta}{\delta m}p(\mu, 0, t, \xi^1, z)(\xi^1)- \frac{\delta}{\delta m}p(\mu, 0, t, \xi^1, z)(\widetilde{\xi})\Big] \\
& \quad + \frac{1}{2N} \E\Big[\frac{\delta^2}{\delta m^2}p(\mu, 0, t, \xi^1, z)(\widetilde{\xi}, \widetilde{\xi}) - \frac{\delta^2}{\delta m^2}p(\mu, 0, t, \xi^1, z)(\widetilde{\xi}, \xi^2) \Big] \\
& \quad + \frac{1}{N} \mathcal{R}^{N}_2(\mu, 0, t, z)
\end{align*}

\noindent with
\begin{align*}
\mathcal{R}^{N}_2(\mu, 0, t, z) & :=  \int_{0}^1 \Big( \E\Big[\frac{\delta}{\delta m}p(\mu^{\lambda_1, N}_0, 0, t, \xi^1, z)(\xi^1)- \frac{\delta}{\delta m}p(\mu, 0, t, \xi^1, z)(\xi^1)\Big] \\
& \quad - \E\Big[\frac{\delta}{\delta m}p(\mu^{\lambda_1, N}_0, 0, t, \xi^1, z)(\widetilde{\xi})- \frac{\delta}{\delta m}p(\mu, 0, t, \xi^1, z)(\widetilde{\xi})\Big] \Big) \, d\lambda_1 \\
& \quad +  \int_{[0,1]^2} \Big(\E\Big[\frac{\delta^2}{\delta m^2}p(\widetilde{\mu}^{\lambda_1, \lambda_2, N}_0, 0, t, \xi^1, z)(\widetilde{\xi}, \widetilde{\xi}) -\frac{\delta^2}{\delta m^2}p(\mu, 0, t, \xi^1, z)(\widetilde{\xi}, \widetilde{\xi}) \Big]  \\
& \quad - \E\Big[\frac{\delta^2}{\delta m^2}p(\widetilde{\mu}^{\lambda_1, \lambda_2, N}_0, 0, t, \xi^1, z)(\widetilde{\xi}, \xi^2) -\frac{\delta^2}{\delta m^2}p(\mu, 0, t, \xi^1, z)(\widetilde{\xi}, \xi^2) \Big] \Big) \lambda_1 \, d\lambda_1 d\lambda_2 \\
& \quad -\frac{1}{N} \int_{[0,1]^2} \E\Big[\frac{\delta^2}{\delta m^2}p(\widetilde{\mu}^{\lambda_1, \lambda_2, N}_0, 0, t, \xi^1, z)(\widetilde{\xi}, \widetilde{\xi}) - \frac{\delta^2}{\delta m^2}p(\widetilde{\mu}^{\lambda_1, \lambda_2, N}_0, 0, t, \xi^1, z)(\widetilde{\xi}, \xi^2) \Big] \lambda_1 \, d\lambda_1 d\lambda_2\\
& =: \mathcal{R}^{N, 1}_2(\mu, 0, t, z)  + \mathcal{R}^{N, 2}_2(\mu, 0, t, z) + \mathcal{R}^{N, 3}_2(\mu, 0, t, z). 
\end{align*}

It thus remains to provide an estimate for the three terms of the remainder $\mathcal{R}^{N}_2(\mu, 0, t, z) $. In order to deal with $\mathcal{R}^{N, 1}_2(\mu, 0, t, z)$, we first write
\begin{align*}
\frac{\delta}{\delta m}p(\mu^{\lambda_1, N}_0, 0, t, \xi^1, z)(\xi^1)& - \frac{\delta}{\delta m}p(\mu, 0, t, \xi^1, z)(\xi^1) \\
& = \int_0^1 \Big[\partial_\mu p(\mu^{\lambda_1, N}_0, 0, t, \xi^1, z)(\lambda_2 \xi^1) - \partial_\mu p(\mu, 0, t, \xi^1, z)(\lambda_2 \xi^1) \Big] \cdot \xi^1 \, d\lambda_2
\end{align*}

\noindent so that, from \eqref{regularity:measure:estimate:v1:v2:mckean:decoupling} with $n=0, \beta=1$ and noting that $W^2_2(\mu^{\lambda_1, N}_0, \mu) = \lambda_1 W^2_2(\mu^{N}_0, \mu)$, we get
$$
|\frac{\delta}{\delta m}p(\mu^{\lambda_1, N}_0, 0, t, \xi^1, z)(\xi^1)- \frac{\delta}{\delta m}p(\mu, 0, t, \xi^1, z)(\xi^1) | \leq K^+ \frac{W_2(\mu^{N}_0, \mu)}{t^{1-\frac{\eta}{2}}} g(ct, z-\xi^1)|\xi^1|.
$$

Similarly, we obtain
$$
|\frac{\delta}{\delta m}p(\mu^{\lambda_1, N}_0, 0, t, \xi^1, z)(\widetilde{\xi})- \frac{\delta}{\delta m}p(\mu, 0, t, \xi^1, z)(\widetilde{\xi})| \leq K^+ \frac{W_2(\mu^{N}_0, \mu) }{t^{1-\frac{\eta}{2}}} g(ct, z-\xi^1)|\widetilde{\xi}|.
$$

Gathering the two previous estimates and using the fact that $\widetilde{\xi}$ is independent of $\mu_0^{N}$, we conclude
$$
|\mathcal{R}^{N, 1}_2(\mu, 0, t, z)| \leq \frac{K^+}{t^{1-\frac{\eta}{2}}} \E\Big[ g(ct, z-\xi^1)W_2(\mu^{N}_0, \mu)(1+ |\xi^1|) \Big].
$$

%\noindent and by Cauchy-Schwarz's inequality
%$$
%\int_{\rr^d} |\mathcal{R}^{N, 1}_2(\mu, 0, t, z)| |z|^2 \, dz \leq \frac{K}{t^{1-\frac{\eta}{2}}} \E\Big[W_2(\mu^{N}_0, \mu) (1+ |\xi_1|^3) \Big] \leq  \frac{K}{t^{1-\frac{\eta}{2}}} \E[W^4_2(\mu^{N}_0, \mu)]^{\frac14} = \frac{K}{t^{1-\frac{\eta}{2}}} \varepsilon^{\frac12}_N
%%\varepsilon^{\frac12}_N
%$$
%
%\noindent where we used the bound $\E[W^{4}_2(\mu^{N}_0, \mu)]^{\frac14} \leq K \varepsilon^{\frac12}_N$  stemming, after some standard computations, from the concentration inequality established in Theorem 2 by Fournier and Guillin \cite{Fournier2015}.

From similar arguments, using \eqref{second:deriv:mes:reg:mes:space:estimate:induction:mckean:decoupling} and the fact that $W_2(\widetilde{\mu}^{\lambda_1, \lambda_2, N}_0, \mu) \leq W_2(\mu^{N}_0, \mu)$, for any $\beta \in [0,\eta)$ it holds
$$
\Big| \E\Big[\frac{\delta^2}{\delta m^2}p(\widetilde{\mu}^{\lambda_1, \lambda_2, N}_0, 0, t, \xi^1, z)(\widetilde{\xi}, \widetilde{\xi}) -\frac{\delta^2}{\delta m^2}p(\mu, 0, t, \xi^1, z)(\widetilde{\xi}, \widetilde{\xi}) \Big] \Big| \leq \frac{K^{++}_\beta}{t^{1+ \frac{\beta-\eta}{2}}} \E\Big[ g(ct, z-\xi^1) W_2(\mu^{N}_0, \mu)^{\beta}\Big] 
$$
\noindent and
$$
\Big| \E\Big[\frac{\delta^2}{\delta m^2}p(\widetilde{\mu}^{\lambda_1, \lambda_2, N}_0, 0, t, \xi^1, z)(\widetilde{\xi}, \xi^2) -\frac{\delta^2}{\delta m^2}p(\mu, 0, t, \xi^1, z)(\widetilde{\xi}, \xi^2) \Big] \Big| \leq \frac{K^{++}_\beta}{t^{1+\frac{\beta-\eta}{2}}} \E[ g(ct, z-\xi^1)W_2(\mu^{N}_0, \mu)^{\beta} |\xi^2|]
$$

\noindent so that
$$
|\mathcal{R}^{N, 2}_2(\mu, 0, t, z)| \leq \frac{K^{++}_\beta}{t^{1+\frac{\beta-\eta}{2}}} \E[ g(ct, z-\xi_1) W_2(\mu^{N}_0, \mu)^{\beta} (1+|\xi^2|)].
$$
%\noindent and
%$$
%\int_{\rr^d} |\mathcal{R}^{N, 2}_2(\mu, 0, t, z)| |z|^2 \, dz \leq \frac{K}{t^{1+\frac{\beta-\eta}{2}}} \varepsilon^{\frac{\beta}{2}}_N
%%\varepsilon^{\frac{\beta}{2}}_N
%$$
%\noindent for any $\beta \in [0,\eta)$. 

For the last term, from \eqref{second:lions:derivative:mckean:decoupling}, we directly get
$$
|\mathcal{R}^{N, 3}_2(\mu, 0, t, z)| \leq \frac{K^+}{t^{1-\frac{\eta}{2}}} \frac{1}{N}\E[g(c t, z-\xi_1)]
$$
%\noindent so that
%$$
%\int_{\rr^d} |\mathcal{R}^{N, 3}_2(\mu, 0, t, z)| |z|^2 \, dz \leq \frac{K}{t^{1-\frac{\eta}{2}}} N^{-1}.
%$$

Gathering the previous estimates on $\mathcal{R}^{N, 1}_2$ $\mathcal{R}^{N, 2}_2$ and $\mathcal{R}^{N, 3}_2$ concludes the proof of \eqref{first:order:expansion:initial:error}.
%$$
%\forall \beta \in [0,\eta), \quad \int_{\rr^d} |\mathcal{R}^{N}_2(\mu, 0, t, z)| |z|^2 \, dz \leq \frac{K}{t^{1+\frac{\beta-\eta}{2}}} \varepsilon^{\beta}_N.
%%\varepsilon^{\frac{\beta}{2}}_N.
%$$

\end{proof}

We now move to the proof of Theorem \ref{propagation:chaos:fundamental:sol:theorem}.\\

\noindent \emph{Step 1: proof of the Gaussian upper-bound \eqref{Gaussian:upper:bound:density:p1N}.}\\ 

Under \HE\, and \HRpp, the map $(s, \mu) \mapsto p(\mu, s , t, z)$ belongs to $\mathcal{C}^{1, 2}_f([0,t)\times\pp)$ so that, from Proposition \ref{empirical:projection:prop}, we deduce that the empirical projection function defined by 
$$
[0,t) \times (\mathbb{R}^d)^{N} \ni (s, (x_1,\cdots, x_N)) \mapsto p(\frac{1}{N}\sum_{i=1}^{N} \delta_{x_i}, s, t, z)
$$

\noindent belongs to the space $\mathcal{C}^{1,2}([0,t)\times (\mathbb{R}^d)^{N})$. Hence, from standard It\^{o}'s lemma
\begin{align}
p(\mu^{N}_s, s, t, z) & = p(\mu^{N}_0, 0, t, z) + \int_0^s  (\partial_r + \mathscr{L}_r) p(\mu^{N}_r, r, t, z) \, dr  \nonumber  \\
& \quad + \int_0^s \frac{1}{N} \sum_{i=1}^{N} \partial_\mu p(\mu^{N}_r, r, t, z)(X^{i}_r). \big(\sigma(r, X^{i}_r, \mu^{N}_r) dW^{i}_r \big) \nonumber \\
& \quad + \int_0^s \frac{1}{2 N^2}\sum_{i=1}^{N} \tr\Big( a(r, X^{i}_r, \mu^{N}_r) \partial^2_\mu p(\mu^{N}_r, r, t, z)(X^{i}_r, X^{i}_r)\Big) \, dr \nonumber \\
& = p(\mu^{N}_0, 0, t, z)  +  \int_0^s \frac{1}{N} \sum_{i=1}^{N} \partial_\mu p(\mu^{N}_r, r, t, z)(X^{i}_r). \big(\sigma(r, X^{i}_r, \mu^{N}_r) dW^{i}_r \big)  \label{pm:ito:rule} \\
& \quad +  \int_0^s \frac{1}{2 N^2}\sum_{i=1}^{N} \tr\Big( a(r, X^{i}_r, \mu^{N}_r) \partial^2_\mu p(\mu^{N}_r, r, t, z)(X^{i}_r, X^{i}_r)\Big) \, dr\nonumber
\end{align}

\noindent where we used the fact that $(\partial_s + \mathscr{L}_s) p(\mu, s, t, z) = 0$ for any $(\mu, s) \in \pp \times [0,t)$.

From the relation 
\begin{align}
\partial_\mu p(\mu, r, t, z)(v) = \partial_x p(\mu, r, t, v,  z)+ \int_{\mathbb{R}^d} \partial_\mu p(\mu, r, t, x, z)(v) \, \mu(dx) \label{first:mes:deriv:relation}
\end{align}

\noindent and the estimates \eqref{bound:density:parametrix} and \eqref{first:second:lions:derivative:mckean:decoupling}, we get $| \partial_\mu p(\mu^{N}_r, r, t, z)(v)| \leq K:=K(t-s, \HR, \HE)$, for any $r\in [0,s]$, so that the local martingale appearing in the right-hand side of \eqref{pm:ito:rule} is a true martingale. Taking expectation in both sides of \eqref{pm:ito:rule}, we thus obtain
\begin{align}
\E\Big[&p(\mu^{N}_s, s, t, z)   \Big] \nonumber \\
%& =  \E[p_m(\mu^{N}_0, 0, t, z) - p_m(\mu, 0, t, z)] + \int_0^s \frac{1}{2 N^2} \sum_{i=1}^{N} \E\left[Tr\Big( a(r, X^{i}_r, \mu^{N}_r) \partial^2_\mu p_m(\mu^{N}_r, r, t, z)(X^{i}_r, X^{i}_r)\Big)\right] \, dr\\
& = \E[p(\mu^{N}_0, 0, t, z) ]  + \int_0^s \frac{1}{2N} \E\left[\tr\Big( a(r, X^{1}_r, \mu^{N}_r) \partial^2_\mu p(\mu^{N}_r, r, t, z)(X^{1}_r, X^{1}_r)\Big)\right] \,dr. \label{first:step:approximation:density}
\end{align}

Now, in order to handle the second term appearing in the right-hand side of the above identity, we first use the relation 
\begin{align}
\partial^2_\mu p(\mu, r, t, z)(v, v') & = \partial_\mu[\partial_x p(\mu, r, t, v, z)](v') + \partial_{x}[\partial_\mu p(\mu, r, t, v', z)(v)] \label{second:mes:deriv:relation} \\
& \quad + \int_{\mathbb{R}^d} \partial^2_\mu p(\mu, r, t, x, z)(v, v') \, \mu(dx) \nonumber
\end{align}

\noindent and the estimates \eqref{deriv:cross:space:mes:pm} and \eqref{second:lions:derivative:mckean:decoupling}, so that we get the following upper-bound
\begin{align}
|\partial^2_\mu&  p(\mu, r, t, z)(v, v')| \notag \\
& \leq \frac{K^+}{{(t-r)^{1-\frac{\eta}{2}}}} \left\{ g(c(t-r), z-v) + g(c(t-r), z-v') +  \int_{\mathbb{R}^d} g(c(t-r), z-x) \, \mu(dx) \right\} \label{gaussian:estimate:second:order:derivative:mes}
\end{align}

\noindent for some positive constants $K^+:=K(T, \HRp, \HE)$ and $c:=c(\lambda)$. Hence, using the boundedness of $a$ as well as the previous estimate, we derive the following estimate for the integrand of the second term appearing in the right-hand side of \eqref{first:step:approximation:density}
\begin{align}
& \left|\E\left[\tr\Big( a(r, X^{1}_r, \mu^{N}_r)  \partial^2_\mu p(\mu^{N}_r, r, t, z)(X^{1}_r, X^{1}_r)\Big)\right] \right| \notag \\
& \quad \quad   \leq  \frac{K^+}{(t-r)^{1-\frac{\eta}{2}}} \int_{\mathbb{R}^d} g(c(t-r), z-y) \, p^{1, N}(\mu, 0, r, y) \, dy  \label{bound:second:term:exp}
\end{align}

\noindent which, plugged into \eqref{first:step:approximation:density}, in turn yields
\begin{align}
\E\Big[p(\mu^{N}_s, s, t, z) \Big] & \leq K^+\left\{ \E[p(\mu^{N}_0, 0, t, z) ]  + \frac{1}{N} \int_0^s \frac{1}{(t-r)^{1-\frac{\eta}{2}}} \int_{\mathbb{R}^d} g(c(t-r), z-y) p^{1, N}(\mu, 0, r, y) \, dy \, dr \right\}. \label{first:bound:density:before:time:limit}
\end{align}

In order to conclude the proof of the Gaussian upper-estimate \eqref{Gaussian:upper:bound:density:p1N}, it remains to pass to the limit as $s\uparrow t$ in the previous inequality. We first note that by interchangeability in law  
$$
\E[p(\mu^{N}_s, s, t, z)] = \E\Big[\int_{\mathbb{R}^d} p(\mu^N_s, s, t, x, z) \mu^N_s(dx)\Big] = \E\Big[ p(\mu^N_s, s, t, X^{1}_s, z)\Big] 
$$ 
\noindent so that
\begin{equation}
\label{identity:density:expectation}
\E[p(\mu^{N}_s, s, t, z)]  = \int_{(\mathbb{R}^d)^{N}} p(m^{N}_{\bold{x}}, s, t, x_1, z) \, \bold{p}^N(\mu, 0, s, \bold{x})\, d\bold{x}
\end{equation}

\noindent where we recall that we use the notations $\bold{x} = (x_1, \cdots, x_N)$, $d\bold{x} = dx_1 \cdots dx_N$, $m^N_{\bold{x}} = \frac{1}{N} \sum_{i=1}^{N} \delta_{x_i}$ and denoted by $\bold{x}\mapsto \bold{p}^N(\mu, 0, s, \bold{x})$ the density function of the $N$-tuple $\bold{X}_s = (X^{1}_s, \cdots, X^{N}_s)$ given by the unique weak solution to the particle system at time $s$ starting at time $0$ from the $N$-fold product measure $\mu^{N}$. We then make use again of the decomposition \eqref{decomposition:short:time:parametrix} and the computations that appear shortly after, namely, we write
$$
p(m^{N}_{\bold{x}}, s, t, x_1, z) = \widehat{p}^{z}(m^{N}_{\bold{x}}, s, t, x_1, z) +  \mathcal{R}(m^{N}_{\bold{x}}, s, t, x_1, z)
$$

\noindent with
$$
 | \mathcal{R}(m^{N}_{\bold{x}}, s, t, x_1, z)| \leq C (t-s)^{\frac{\eta}{2}} g(c(t-s), z-x_1).
$$

Denoting $\widetilde{\bold{x}} = (z, x_2, \cdots, x_N)$, $\xi, \xi'$ two random variables with $[\xi] = m_{\bold{x}}^{N}$, $[\xi']=m_{\widetilde{\bold{x}}}^N$ and using the estimate (A.45) of Lemma A.2 (with $\beta=\eta$) in \cite{chaudruraynal:frikha}, we get 
\begin{align*}
& | a_{i, j}(r, z, [X^{s, \xi }_r]) - a_{i, j}(r, z, [X^{s, \xi'}_r]) |\\
&= \lim_{m \uparrow \infty} | a_{i, j}(r, z, [X^{s, \xi , (m)}_r]) - a_{i, j}(r, z, [X^{s, \xi', (m)}_r]) |  \leq C W_2(m_{\bold{x}}^{N}, m_{\widetilde{\bold{x}}}^N)^{\eta} \leq C |z-x_1|^{\eta}
\end{align*}

%from the mean-value theorem and the $\eta$-H\"older regularity of $y\mapsto \frac{\delta}{\delta m} a_{i, j}(t, x, m)(y)$, one gets
%\begin{align*}
%\Big| a_{i, j}(t, x, m^{N}_{\bold{x}}) - a_{i, j}(t, x, m^{N}_{\widetilde{\bold{x}}}) \Big| & = \Big|\int_0^1\int_{\rr^d} \frac{\delta}{\delta m}a_{i, j}(t, x, m^{N}_{\lambda, \bold{x}})(y) \, (m^{N}_{\bold{x}} - m^{N}_{\widetilde{\bold{x}}})(dy) \, d\lambda\Big|  \leq C |z-x_1|^{\eta}
%\end{align*}

\noindent which in turn by recalling \eqref{kernel:p:hat:definition} and using the mean-value theorem and the space-time inequality \eqref{space:time:inequality} yield
$$
|\widehat{p}^{z}(m^{N}_{\bold{x}}, s, t, x_1, z) - \widehat{p}^{z}(m^{N}_{\widetilde{\bold{x}}}, s, t, x_1, z)| \leq C (t-s)^{\frac{\eta}{2}} g(c(t-s), z-x_1).
$$

Hence, plugging the previous estimates into \eqref{identity:density:expectation} we deduce
\begin{align*}
\E[p(\mu^{N}_s, s, t, z)]  & =  \int_{(\mathbb{R}^d)^{N}} p(m^{N}_{\bold{x}}, s, t, x_1, z) \bold{p}^N(\mu, 0, s, \bold{x})\, d\bold{x} \\
& =  \int_{(\rr^d)^{N}} \widehat{p}^{z}(m^{N}_{\widetilde{\bold{x}}}, s, t, x_1, z) \bold{p}^N(\mu, 0, s, \bold{x})\, d\bold{x} + \mathcal{O}((t-s)^{\frac{\eta}{2}}) \int g(c t, z - x) \mu(dx).
\end{align*}

In order to pass to the limit as $s \uparrow t$ in the previous identity, we finally perform the change of variable $x_1 = \Sigma^{1/2}_{s, t} y_1 + z $, where $\Sigma^{1/2}_{s, t}$ is the unique principal square root of the positive definite matrix $\int_s^t a(r, z, [X^{s, \xi'}_r]) \, dr$, recalling that $[\xi']= m_{\widetilde{\bold{x}}}^{N}$, in the integral appearing in the right-hand side of the previous equality and then let $s\uparrow t$, by dominated convergence
$$
\lim_{s\uparrow t} \int_{(\mathbb{R}^d)^{N}} p(m^{N}_{\bold{x}}, s, t, x_1, z) \bold{p}^N(\mu, 0, s, \bold{x})\, d\bold{x} = \lim_{s \uparrow t} \int_{(\rr^d)^{N}} \widehat{p}^{z}(m^{N}_{\widetilde{\bold{x}}}, s, t, x_1, z) \bold{p}^N(\mu, 0, s, \bold{x})\, d\bold{x} = p^{1,N}(\mu, 0, t, z).
$$

Coming back to  \eqref{first:bound:density:before:time:limit}, passing to the limit as $s\uparrow t$ in \eqref{identity:density:expectation} and using the previous identity, we thus obtain
%\begin{align*}
%\Big|\E\Big[p^{1, N}(\mu, 0, t, z) - p(\mu, 0, t, z)\Big] \Big|& = \Big|\E\Big[p^{1, N}(\mu, 0, t, z) - p(\mu^{N}_0, 0, t, z)\Big] \Big|\\
%& = \lim_{s\uparrow t}\Big|\E\Big[p(\mu^{N}_s, s, t, z) - p(\mu^{N}_0, 0, t, z)\Big] \Big| \\
%& \leq \frac{K}{N} \int_0^t \frac{1}{(t-r)^{1-\frac{\eta}{2}}} \int_{\rr^d} g(c(t-r), z-y) p^{1, N}(\mu, 0, r, y) \, dy \, dr
%\end{align*}
%
%\noindent so that
\begin{align}
p^{1, N}(\mu, 0, t, z) & = \lim_{s \uparrow t}\E\Big[p(\mu^{N}_s, s, t, z) \Big] \notag \\
&\leq K^+ \bigg\{ \E[p(\mu^{N}_0, 0, t, z) ]     \label{esti:lim:approx:part}\\
&  + \frac{1}{N} \int_0^t \frac{1}{(t-r)^{1-\frac{\eta}{2}}} \int_{\mathbb{R}^d} g(c(t-r), z-y) p^{1, N}(\mu, 0, r, y) \, dy \, dr \bigg\}.\notag\\
&\leq K^+ \left\{ \int_{\mathbb{R}^d} g(c t, z-x) \mu(dx) \right.  \notag \\
& \left. + \frac{1}{N} \int_0^t \frac{1}{(t-r)^{1-\frac{\eta}{2}}} \int_{\mathbb{R}^d} g(c(t-r), z-y) p^{1, N}(\mu, 0, r, y) \, dy \, dr \right\}.\notag
\end{align}

\noindent where we also used the Gaussian upper-bound \eqref{bound:density:parametrix} with $n=0$ for the last inequality.

Observe now that the space-time convolution kernel $\mathbb{R}^d \times [0,t) \ni (y, r) \mapsto (t-r)^{-1+\frac{\eta}{2}} g(c(t-r), z-y)$ leads to an integrable time singularity so that the previous inequality can be iterated and by an induction argument that we omit, we conclude
\begin{align}
p^{1, N}(\mu, 0, t, z) \leq K^+ \int_{\mathbb{R}^d} g(c t, z-x) \mu(dx) \label{bound:density:p1N}.
\end{align}

The proof of the Gaussian upper-estimate \eqref{Gaussian:upper:bound:density:p1N} is thus complete.\\

\noindent \emph{Step 2: proof of the error bound \eqref{error:bound:prop:chaos}.}\\

We now come back to the identity \eqref{first:step:approximation:density}, substract $p(\mu, 0, t, z)$ from its both sides, then use \eqref{bound:second:term:exp} together with \eqref{bound:density:p1N} so that  
\begin{align}
\Big|\E\Big[p(\mu^{N}_s, s, t, z) & - p(\mu, 0, t, z)  \Big] \Big| \leq \Big|\E\Big[p(\mu^{N}_0, 0, t, z) - p(\mu, 0, t, z)\Big]\Big|  + \frac{K^+}{N} \int_{\mathbb{R}^d} g(ct, z-x) \mu(dx) \notag
\end{align}

\noindent which in turn combined with \eqref{error:bound:initial:density} yields
\begin{align*}
\Big|\E\Big[p(\mu^{N}_s, s, t, z)  - p(\mu, 0, t, z) \Big] \Big|  & \leq \frac{K^+}{N} \left\{  \frac{1}{t^{\frac{1-\eta}{2}}} \int_{\rr^d} g(ct, z-x) |x| \mu(dx) + \frac{1}{t^{1-\frac{\eta}{2}}} \int_{\rr^d} g(c t, z-x)  \mu(dx) \right\}.
\end{align*}

We eventually conclude the proof of \eqref{error:bound:prop:chaos} by letting $s \uparrow t$ in the previous inequality following similar arguments as those used in the previous step
\begin{align*}
\Big| p^{1, N}(\mu, 0, t, z)  - p(\mu, 0, t, z) \Big|&  = \lim_{s \uparrow t}\Big|\E\Big[p(\mu^{N}_s, s, t, z)  - p(\mu, 0, t, z) \Big] \Big| \\
&  \leq \frac{K^+}{N} \left\{  \frac{1}{t^{\frac{1-\eta}{2}}} \int_{\mathbb{R}^d} g(ct, z-x) |x| \mu(dx) + \frac{1}{t^{1-\frac{\eta}{2}}} \int_{\mathbb{R}^d} g(c t, z-x)  \mu(dx) \right\}.\\
\end{align*}
% From \eqref{conv:relation:step:m}, we get
%$$
%\frac{\delta}{\delta m}p_m(\mu, 0, t, z)(y) = \int_{\rr^d} \frac{\delta}{\delta m}p_m(\mu, 0, t, x', z)(y) \, \mu(dx') + p_m(\mu, 0, t, y, z) - p_m(\mu, 0, t, 0, z)
%$$

%From the previous estimate, \eqref{esti:lim:approx:part_temp}, \eqref{esti:lim:approx:part}, \eqref{bound:second:term:exp} and \eqref{first:step:approximation:density}, we thus deduce
%\begin{align*}
%|(p^{1, N}-p)(\mu, 0, t, z)| & \leq \frac{K}{N} \left\{  \frac{1}{t^{\frac{1-\eta}{2}}} \int_{\rr^d} g(ct, z-x) |x| \mu(dx) + \frac{1}{t^{1-\frac{\eta}{2}}} \int_{\rr^d} g(c t, z-x)  \mu(dx) \right.\\
%& \quad \left. + \frac{1}{N} \int_0^t \frac{1}{(t-r)^{1-\frac{\eta}{2}}} \int_{\rr^d} g(c(t-r), z-y) p^{1, N}(\mu, 0, r, y) \, dy \, dr \right\}
%\end{align*}
%
%\noindent which in turn by \eqref{bound:density:p1N} yields
%$$
%|(p^{1, N}-p)(\mu, 0, t, z)| \leq \frac{K}{N} \left\{ \frac{1}{t^{\frac{1-\eta}{2}}} \int_{\rr^d} g(ct, z-x) |x| \mu(dx) + \frac{1}{t^{1-\frac{\eta}{2}}} \int_{\rr^d} g(c t, z-x)  \mu(dx) \right\}.
%$$

\noindent \emph{Step 3: proof of the first order expansion \eqref{first:order:exp:prop:chaos}.}\\

We here establish the first order expansion \eqref{first:order:exp:prop:chaos} under the additional assumption that $(x, \mu) \mapsto \sigma(t, x, \mu)$ is uniformly Lipschitz continuous and that $M_q(\mu) < \infty$ for some $q>4$. Coming back to \eqref{first:step:approximation:density} and substracting $p(\mu, 0, t, z)$ from its both sides, we get
\begin{align}
\E\Big[p(\mu^{N}_s, s, t, z) & - p(\mu, 0, t, z)  \Big] \nonumber \\
%& =  \E[p_m(\mu^{N}_0, 0, t, z) - p_m(\mu, 0, t, z)] + \int_0^s \frac{1}{2 N^2} \sum_{i=1}^{N} \E\left[Tr\Big( a(r, X^{i}_r, \mu^{N}_r) \partial^2_\mu p_m(\mu^{N}_r, r, t, z)(X^{i}_r, X^{i}_r)\Big)\right] \, dr\\
& = \E[p(\mu^{N}_0, 0, t, z) - p(\mu, 0, t, z)]  + \int_0^s \frac{1}{2N} \E\left[\tr\Big( a(r, X^{1}_r, \mu^{N}_r) \partial^2_\mu p(\mu^{N}_r, r, t, z)(X^{1}_r, X^{1}_r)\Big)\right] \,dr. \label{first:step:exp:prop:chaos:density}
\end{align}

We then pass to the limit as $s\uparrow t$ in the previous identity using similar arguments as those previously employed and apply the first order expansion \eqref{first:order:expansion:initial:error} of Lemma \ref{lem:first:order:expansion:initial:error}. We thus obtain
\begin{align*}
(p^{1,N} - p)(\mu, 0, t, z) & =  \E[p(\mu^{N}_0, 0, t, z) - p(\mu, 0, t, z)]  + \int_0^t \frac{1}{2N} \E\left[\tr\Big( a(r, X^{1}_r, \mu^{N}_r) \partial^2_\mu p(\mu^{N}_r, r, t, z)(X^{1}_r, X^{1}_r)\Big)\right] \,dr \\
& = \frac{1}{N}   \E\Big[\frac{\delta}{\delta m}p(\mu, 0, t, \xi^1, z)(\xi^1)- \frac{\delta}{\delta m}p(\mu, 0, t, \xi^1, z)(\widetilde{\xi})\Big] \\
& \quad + \frac{1}{2N} \E\Big[\frac{\delta^2}{\delta m^2}p(\mu, 0, t, \xi^1, z)(\widetilde{\xi}, \widetilde{\xi}) - \frac{\delta^2}{\delta m^2}p(\mu, 0, t, \xi^1, z)(\widetilde{\xi}, \xi^2) \Big] \\
& \quad \quad + \int_0^t \frac{1}{2N} \E[\mathcal{A}_sp(\mu_s, s, t, z)]\, ds +  \frac{1}{N} \left(\mathcal{R}^{N}_1(\mu, 0, t, z) + \mathcal{R}^{N}_2(\mu, 0, t, z)\right)
\end{align*}

\noindent where $\mathcal{R}^{N}_2(\mu, 0, t, z)$ is defined in \eqref{first:order:expansion:initial:error} and
$$
\mathcal{R}^{N}_1(\mu, 0, t, z) :=\frac{1}{2}   \int_0^t \E\left[\tr\Big( a(r, X^{1}_r, \mu^{N}_r) \partial^2_\mu p(\mu^{N}_r, r, t, z)(X^{1}_r, X^{1}_r) - a(r, X_r, \mu_r) \partial^2_\mu p(\mu_r, r, t, z)(X_r, X_r)\Big)\right] \,dr.
$$

\smallskip

Observe that from \eqref{bound:remainder:term:2:initial:error}, for any $\phi$ with at most quadratic growth and any $\beta \in [0,\eta)$, it holds
\begin{align*}
& \int_{\mathbb{R}^d} |\phi(z)| |\mathcal{R}^{N}_2(\mu, 0, t, z)| \, dz \\
& \quad \leq K^{++} \left\{ \frac{1}{t^{1-\frac{\eta}{2}}} \mathbb{E}\Big[(1+|\xi^1|^3) W_2(\mu_0^{N}, \mu) + \frac{(1+|\xi^1|^2)}{N}\Big] + \frac{1}{t^{1+ \frac{\beta-\eta}{2}}} \mathbb{E}\Big[ (1+|\xi^1|^2) (1+|\xi^2|) W_2(\mu_0^N, \mu)^\beta \Big] \right\}  \\
& \quad \leq K^{++} \left\{ \frac{1}{t^{1-\frac{\eta}{2}}} \left(\mathbb{E}[W_2(\mu_0^{N}, \mu)^4]^{1/4} + \frac{1}{N}\right) + \frac{1}{t^{1+ \frac{\beta-\eta}{2}}} \mathbb{E}[W_2(\mu_0^N, \mu)^2]^{\beta/2}\right\} 
\end{align*}
\noindent where we used H\"older's inequality in the last inequality. Now, using the fact that $M_q(\mu) < \infty$ for some $q>4$, one deduces from the concentration inequality established in Theorem 2 by Fournier and Guillin \cite{Fournier2015} that $\mathbb{E}[W_2(\mu^{N}_0, \mu)^4]^{1/4} \leq K \varepsilon^{1/2}_N$. We thus conclude
\begin{equation}
\label{integral:test:function:R2N}
\int_{\mathbb{R}^d} |\phi(z)| |\mathcal{R}^{N}_2(\mu, 0, t, z)| \, dz \leq K^{++} \left\{ \frac{\varepsilon^{1/2}_N}{t^{1-\frac{\eta}{2}}} + \frac{\varepsilon^{\beta/2}_N}{t^{1+ \frac{\beta-\eta}{2}}} \right\}.
\end{equation}
It thus remains to establish an appropriate estimate for $\mathcal{R}_{1}^{N}(\mu, 0, t, z)$. Introducing the coupling dynamics $\bar{X}^{1}$, we write
\begin{align*}
\mathcal{R}^{N}_1(\mu, 0, t, z) & =\frac{1}{2}   \int_0^t \E\left[\tr\Big( a(r, X^{1}_r, \mu^{N}_r) \partial^2_\mu p(\mu^{N}_r, r, t, z)(X^{1}_r, X^{1}_r)\Big) - \tr\Big( a(r, \bar{X}^{1}_r, \mu_r) \partial^2_\mu p(\mu_r, r, t, z)(\bar{X}^{1}_r, \bar{X}^{1}_r) \Big)\right] \,dr
\end{align*}

\noindent and decompose the integrand appearing in right-hand side as the sum of the three following terms $\mathcal{R}^{N,1}_1(\mu, r, t, z)$, $\mathcal{R}^{N, 2}_1(\mu, r, t, z)$ and $\mathcal{R}^{N,3}_1(\mu, r, t, z)$ defined by
\begin{align*}
\mathcal{R}^{N,1}_1(\mu, r, t, z) & := \E\left[\tr\Big( [a(r, X^{1}_r, \mu^{N}_r) - a(r, \bar{X}^{1}_r, \mu_r)] \partial^2_\mu p(\mu^{N}_r, r, t, z)(X^{1}_r, X^{1}_r)\Big)\right],  \\
\mathcal{R}^{N,2}_1(\mu, r, t, z) & := \E\left[\tr\Big( a(r, \bar{X}^{1}_r, \mu_r) [\partial^2_\mu p(\mu^{N}_r, r, t, z)(X^{1}_r, X^{1}_r) - \partial^2_\mu p(\mu_r, r, t, z)(X^{1}_r, X^{1}_r)]\Big)\right], \\
\mathcal{R}^{N,3}_1(\mu, r, t, z) & := \E\left[\tr\Big( a(r, \bar{X}^{1}_r, \mu_r) [\partial^2_\mu p(\mu_r, r, t, z)(X^{1}_r, X^{1}_r) - \partial^2_\mu p(\mu_r, r, t, z)(\bar{X}^{1}_r, \bar{X}^{1}_r)]\Big)\right].
\end{align*}

From \eqref{gaussian:estimate:second:order:derivative:mes}, we first obtain
\begin{align}
|  \partial^2_\mu p(\mu^{N}_r, r, t, z)(X^{1}_r, X^{1}_r) | & \leq \frac{K^+}{(t-r)^{1-\frac{\eta}{2}}} \left\{ g(c(t-r), z-X^{1}_r) + \int g(c(t-r), z-x) \mu^{N}_r(dx) \right\} \nonumber \\
& = \frac{K^+}{(t-r)^{1-\frac{\eta}{2}}} \left\{ g(c(t-r), z-X^{1}_r) + \frac{1}{N}\sum_{i=1}^{N} g(c(t-r), z-X^{i}_r) \right\}. \label{bound:deriv:second:mes:part:system}
\end{align}

This estimate will be used in the sequel. The uniform Lipschitz regularity of the map $(x, \mu) \mapsto a(t, x, \mu)$ then gives
\begin{equation}
|a(r, X^{1}_r, \mu^{N}_r) - a(r, \bar{X}^{1}_r, \mu_r)| \leq K \big[|X^{1}_r - \bar{X}^{1}_r| + W_2(\mu^{N}_r, \mu_r) \big]. \label{lip:reg:a:particle:system}
\end{equation}

Combining the two previous estimates with the Fubini theorem, the Cauchy-Schwarz inequality, the fact that $\sup_{0\leq t \leq T} \max_{1\leq i \leq N} \mathbb{E}[|X^i_t|^4]^{1/4}< C M_4(\mu)$ and eventually using the estimate \eqref{Esti:propchaos:path} of Theorem \ref{propagation:prop:path} yield
\begin{align*}
\int_{\mathbb{R}^d}  |\phi(z)| |\mathcal{R}^{N,1}_1(\mu, 0, r, z)| \, dz & \leq \frac{K^+}{(t-r)^{1-\frac{\eta}{2}}} \left\{ \mathbb{E}[|X^1_r - \bar{X}^1_r|^2]^{1/2} + \E[W_2(\mu^{N}_r, \mu_r)^2]^{1/2}\right\} \E[|X^1_r|^4]^{1/2}\\
&  \leq \frac{K^+}{(t-r)^{1-\frac{\eta}{2}}} \varepsilon_N^{1/2}.
\end{align*}

In order to deal with $\mathcal{R}^{N,2}_1(\mu, r, t, z)$, we first establish an estimate for the difference $ \partial^2_\mu p(\mu, r, t, z)(v, v) - \partial^2_\mu p(\nu, r, t, z)(v, v) $, for $\mu, \, \nu \in \pp$. From \eqref{second:mes:deriv:relation}, the estimates \eqref{sensitivity:mes:deriv:cross:space:mes:p}, \eqref{second:deriv:mes:reg:mes:space:estimate:induction:mckean:decoupling}, for any $\mu, \nu \in \pp$ and for any coupling $\pi$ between $\mu$ and $\nu$, we get
\begin{align*}
| \partial^2_\mu p(\mu, r, t, z)(v, v) & - \partial^2_\mu p(\nu, r, t, z)(v, v) | \\
& \leq  | \partial_x\partial_\mu p(\mu, r, t, v, z)(v) - \partial_x \partial_\mu p(\nu, r, t, v, z)(v)| \\
& \quad +  | \int_{(\mathbb{R}^d)^2} [\partial^2_\mu p(\mu, r, t, x, z)(v, v) - \partial^2_\mu p(\mu, r, t, y, z)(v, v)] \, \pi(dx, dy)| \\
& \quad +  \int_{\mathbb{R}^d} |\partial^2_\mu p(\mu, r, t, x, z)(v, v) - \partial^2_\mu p(\nu, r, t, x, z)(v,v)| \nu(dx) \\
& \leq K^{++}_\beta \frac{W_2(\mu, \nu)^\beta}{(t-r)^{1+\frac{\beta-\eta}{2}}} \left\{ g(c(t-r), z-v) + \int_{\mathbb{R}^d} g(c(t-r), z-x) \nu(dx) \right\} \\
& + \frac{K^{++}_\beta}{(t-r)^{1+\frac{\beta-\eta}{2}}} \int_{(\rr^d)^2}  | x- y|^\beta \left\{g(c(t-r), z-x) + g(c(t-r), z-y) \right\} \, \pi(dx, dy)
\end{align*}

\noindent which directly yields
\begin{align*}
\int_{\mathbb{R}^d} |\phi(z)| &| \partial^2_\mu p(\mu, r, t, z)(v, v)  - \partial^2_\mu p(\nu, r, t, z)(v, v) |\, dz \\
&  \leq  \frac{K^{++}_\beta}{(t-r)^{1+\frac{\beta-\eta}{2}}}\left\{ W_2(\mu, \nu)^\beta (1+|v|^2 + M_2(\nu)^2) + \int_{(\mathbb{R}^d)^2} |x-y|^\beta (1+ |x|^2+ |y|^2) \, \pi(dx, dy) \right\} \\
& \leq K^{++}_\beta \frac{W_2(\mu, \nu)^\beta}{(t-r)^{1+\frac{\beta-\eta}{2}}} (1+|v|^2 + M_4(\mu)^2 + M_4(\nu)^2)
\end{align*}

\noindent where we used Cauchy-Schwarz's inequality and then took the infimum over $\pi$ for the last inequality.

Now, having in mind the preceding estimate and using again the Fubini theorem, the Cauchy-Schwarz inequality and the fact that $\sup_{0\leq t \leq T}\E[M_4(\mu^{N}_t)^4] \leq K M_4(\mu)^4$ and $\sup_{0\leq t \leq T} M_4(\mu_t) \leq K M_4(\mu)$, we get
$$
\int_{\mathbb{R}^d} |\phi(z)| \, \E[| \partial^2_\mu p(\mu^{N}_r, r, t, z)(X^{1}_r, X^{1}_r) - \partial^2_\mu p(\mu_r, r, t, z)(X^{1}_r, X^{1}_r) |] \, dz \leq K^{++}_\beta  \frac{\E[W_2(\mu^{N}_r, \mu_r)^2]^{\beta /2}}{(t-r)^{1+\frac{\beta-\eta}{2}}}  
$$

\noindent so that
$$
\int_{\mathbb{R}^d} |\phi(z)|\,  |\mathcal{R}^{N,2}_1(\mu, r, t, z)|  \, dz \leq  \frac{K^{++}_\beta}{(t-r)^{1+\frac{\beta-\eta}{2}}} \varepsilon^{\beta/2}_N.
%\varepsilon^{\frac{\beta}{2}}_N.
$$

We finally deal with $\mathcal{R}^{N, 3}_1(\mu, r, t, z)$. From \eqref{second:mes:deriv:relation}, \eqref{sensitivity:mes:deriv:cross:space:mes:p} and \eqref{second:deriv:mes:reg:mes:space:estimate:induction:mckean:decoupling}, for any $(v_1, v_2) \in \mathbb{R}^d$, we obtain
\begin{align*}
& | \partial^2_\mu p(\mu, r, t, z)(v_1, v_1) - \partial^2_\mu p(\mu, r, t, z)(v_2, v_2) | \\
& \leq K^{++}_\beta \frac{ |v_1-v_2|^{\beta}}{(t-r)^{1+ \frac{\beta-\eta}{2}}} \left\{ g(c(t-r), z-v_1) + g(c(t-r), z-v_2) +\int_{\mathbb{R}^d} g(c(t-r), z-x) \, \mu(dx) \right\} .   
\end{align*}

Hence, from the preceding estimate, Fubini's theorem, Cauchy-Schwarz's inequality and \eqref{Esti:propchaos:path}, we obtain 
\begin{align*}
\int_{\mathbb{R}^d} |\phi(z)| |\mathcal{R}^{N,3}_1(\mu, r, t, z)|  \, dz & \leq \frac{K^{++}_\beta}{(t-r)^{1+ \frac{\beta-\eta}{2}}} \E[|X^1_r - \bar{X}^1_r|^\beta (1+ |X^1_r|^2 + |\bar{X}^1_r|^2 + M_2(\mu)^2)] \\
&  \leq  \frac{K^{++}_\beta}{(t-r)^{1+ \frac{\beta-\eta}{2}}}  \varepsilon^{\beta / 2}_N.
%\varepsilon^{\frac{\beta}{2}}_N
\end{align*}

Gathering the previous estimates, we eventually conclude that for any $\beta \in [0,\eta)$
$$
\int_{\mathbb{R}^d} |\phi(z)| \, |\mathcal{R}^{N}_1(\mu, 0, t, z)| \, dz \leq K^{++}_\beta \left\{ t^{\frac{\eta}{2}} \varepsilon^{1/2}_N + t^{\frac{\eta-\beta}{2}}  \varepsilon^{\beta / 2}_N \right\}
$$
\noindent for some positive constant $K^{++}_\beta:=K(T, \HRpp, \HE, [\sigma]_L, q, M_q(\mu), \beta)$. The previous estimate together with \eqref{integral:test:function:R2N} allows to concluide the proof.

\subsection{Proof of Theorem \ref{theo:weakChaos}.} 
For a fixed function $\phi$ in $\mathscr C^{2, \alpha}(\pp)$, we consider the following PDE on the Wasserstein space
\begin{eqnarray}\label{PDEChaosSemGroup:weakchaos}
(\partial_t + \mathscr L_t)U = 0,\qquad U(T,\cdot) = \phi(\cdot),
\end{eqnarray}
where the operator $\mathscr L_t$ is given by \eqref{Gen:flow:weak:mckean}. From Theorem 3.8 in \cite{chaudruraynal:frikha}, under \HE\, and \HRp, there exists a unique solution $U \in \mathcal{C}^{1, 2}([0,T) \times \pp)$ to the above PDE \eqref{PDEChaosSemGroup:weakchaos} given by
\begin{equation*}
\forall (t,\mu)\in [0,T]\times \pp,\quad U(t,\mu) = \phi([X_T^{t,\mu}]).
\end{equation*}

%From proposition \ref{structural:class} together with \eqref{derivative:density:sol:mckean:and:decoupling} and \eqref{relation:density:mckean:decoupling:field}  
%\begin{equation}\label{estiU}
%|\partial_v^n \partial_\mu U(t,\mu)(v) | \leq (T-t)^{-\frac{1+n-\alpha}{2}},\quad n=0,1.
%\end{equation}

%Again, we will work with an approximation sequence $(U^{(m)})_{m\geq0}$ of $U$ defined by
%\begin{equation}
%\forall (t,\mu) \in [0,T)\times \pp,\quad U^{(m)}(t,\mu) = \phi([X_T^{t,\mu,(m)}]),
%\end{equation}
%where $X^{t,\mu,(m)}$ is the unique weak solution to the SDE \eqref{iter:mckean} starting from the initial distribution $\mu$ at time $t$. 

Using Proposition \ref{fully:c2:regularity:property:decoupling:density} together with the estimates \eqref{bound:density:parametrix}, \eqref{first:second:lions:derivative:mckean:decoupling}, \eqref{first:time:derivative:mckean:decoupling}, \eqref{deriv:cross:space:mes:pm} and \eqref{second:lions:derivative:mckean:decoupling}, one may apply Proposition \ref{structural:class} to deduce that $U \in \mathcal C_f^{1,2}([0,T)\times \pp)$. Note carefully that in Proposition \ref{structural:class}, the linear functional derivatives of $h$ are assumed to be bounded for sake of simplicity while here the linear function derivatives of $\phi$ is of linear growth, see \eqref{growth:condtion:first:second:flat:deriv}. However, using the pointwise Gaussian estimates \eqref{bound:density:parametrix}, \eqref{first:second:lions:derivative:mckean:decoupling}, \eqref{deriv:cross:space:mes:pm} and \eqref{second:lions:derivative:mckean:decoupling}, one can extend the analysis performed in the proof of Proposition \ref{structural:class} to the current setting.

 Moreover, the first and second order L-derivatives satisfy the identities \eqref{condition1:structural:class:measure:derivative} and \eqref{full:expression:second:deriv}. Now, proceeding as in the proof of Proposition 6.1 in \cite{chaudruraynal:frikha}, namely, using \eqref{local:holder:reg:first:flat:deriv:terminal:condition} and \eqref{local:holder:reg:second:flat:deriv:terminal:condition} as well as the estimates provided by Proposition \ref{fully:c2:regularity:property:decoupling:density}, one may prove the following estimates: there exists a positive constant $K^{+}:=K(T, \HRp,\HE)$, $T\mapsto K(T, \HRp,\HE)$ being non-decreasing, such that for all $(t,\mu)\in [0,T)\times \pp$ and $\gv=(v,v') \in (\mathbb{R}^d)^2$:
\begin{align}
\partial_v^{n}\partial_\mu U(t,\mu)(v) & \leq  K^{+} (T-t)^{-\frac{n+1-\alpha}{2}} (1+ |v| + M_2(\mu)),\, n=0,1, \, \label{estimate:approxim:kolmo:backward:part1}\\
 \partial_\mu^2 U(t,\mu)(\gv) & \leq  K^{+}(T-t)^{-1+\frac \alpha 2} (1+ |\gv| + M_2(\mu)). \label{estimate:approxim:kolmo:backward:part2}
\end{align}

Since the arguments are completely analogous to those employed in the proof of Proposition 6.1 in \cite{chaudruraynal:frikha}, we omit the proof of the above estimates.

Note carefully that the time singularities appearing in the previous bounds on the first and second L-derivatives of $U$ are integrable over $[0,T)$. 
%Moreover, from the estimates \eqref{first:second:estimate:induction:decoupling:mckean} to \eqref{regularity:time:estimate:v1:v2:decoupling:mckean} and the expressions of $\partial_v^{n}\partial_\mu U^{(m)}(t,\mu)(v)$ stemming from Proposition \ref{structural:class}, relabelling the indices if necessary, one may assert that the sequence $(U^{(m)}(t, \mu), \mathscr L_tU^{(m)}(t, \mu))_{m\geq0}$ converges towards $(U(t, \mu),\mathscr L_tU(t,\mu))$ for every $(t,\mu)$ in $[0,T) \times \pp$.

On the one hand, from standard It\^o's formula and \eqref{PDEChaosSemGroup:weakchaos}, we have
\begin{eqnarray*}
U(t,\mu_t^N) &=& U(0,\mu_0^N) + \frac 1N \sum_{i=1}^N \int_0^t  \partial_\mu U(s,\mu_s^N) (X_s^i) . (\sigma(s,X_s^i,\mu^{N}_s) d W_s^i) \\
&&+ \frac 1{2N^2} \sum_{i=1}^{N} \int_0^t  \tr\left(a(s, X^{i}_s,\mu_s^N)\partial_{\mu}^2U(s,\mu_s^N)( X^{i}_s, X^{i}_s) \right)  ds.
\end{eqnarray*}
On the other hand, from the Markov property stemming from the well-posedness of the martingale problem related to \eqref{SDE:MCKEAN}
$$
U(t,[X^{s, \xi}_t]) = \phi([X^{t,X^{s, \xi}_t}_T]) = \phi([X_T^{s,\xi}]) = U(s,\mu),
$$
for all $s$ in $[0,t]$ and especially for $s=0$. Hence

\begin{eqnarray}\label{weakChaos:inter}
U(t,\mu_t^N) - U(t,\mu_t) &=& \big(U(0,\mu_0^N)-U(0,\mu)\big) + \frac 1N \sum_{i=1}^N \int_0^t  \partial_\mu U(s,\mu_s^N) (X_s^i) . (\sigma(s,X_s^i,\mu^{N}_s)d W_s^i) \notag\\
&&+ \frac 1{2N^2} \sum_{i=1}^N \int_0^t \tr\left(a(s,X_s^i,\mu_s^N)\partial_{\mu}^2U(s,\mu_s^N)(X_s^i,X_s^i) \right) ds.
\end{eqnarray}
Using the Burkholder-Davis-Gundy inequality, the estimates \eqref{estimate:approxim:kolmo:backward:part1}, \eqref{estimate:approxim:kolmo:backward:part2} and the fact that there exists a positive constant $C$ such that for any $1\leq i \leq N$ and any $s\in [0,T]$, $\mathbb{E}[|X^{i}_s|^2] + \mathbb{E}[M_2(\mu^{N}_s)^2] \leq C(1+M_2(\mu)^2)$, which directly stems from the dynamics \eqref{SDE:particle:system} together with the boundedness of the coefficients, we get
$$
\E[|U(t,\mu_t^N) - U(t,\mu_t) |] \leq   \E\left[|U(0,\mu)-U(0,\mu_0^N)|\right]  + \frac{K^{+}}{N^{\frac12}} 
$$
\noindent where $T\mapsto K^{+}:=K(T, \HRp, \HE, M_2(\mu))$ is a non-decreasing positive function. Finally, letting $t \uparrow T$
$$
\E\left[|U(T,\mu_T) - U(T,\mu_T^N)| \right] \leq   \E\left[|U(0,\mu)-U(0,\mu_0^N)|\right] + \frac {K^{+}}{N^{\frac12}} \leq \frac{K^{+}}{T^{\frac{1-\alpha}{2}}}  \E[W_2(\mu,\mu_0^N)^2]^{1/2} + \frac{K^{+}}{N^{\frac12}}
$$
\noindent where we used the Lipschitz continuity of the map $\pp \ni \mu \mapsto U(0, \mu)$ thanks to \eqref{estimate:approxim:kolmo:backward:part1} for the last inequality together with the Cauchy-Schwarz inequality and the fact that $M_2(\mu) < \infty$. This completes the proof of \eqref{esti:sem:chaos:part2}.

%The conclusion of point (ii) then follows from Theorem 1 of \cite{Fournier2015} using the fact that the initial distribution $\mu \in \pp$. 

In order to prove \eqref{esti:sem:chaos:part1}, we first take the expectation in \eqref{weakChaos:inter}. Doing so we get ride of the martingale part therein. Then, using \eqref{estimate:approxim:kolmo:backward:part2}, we hence obtain
\begin{equation}
\left|U(t,\mu_t) - \E\left[U(t,\mu_t^N)\right]\right| \leq \left|U(0,\mu)-\E\left[U(0,\mu_0^N)\right]\right| + \frac {K^{+}}{N} . \label{strong:error:regularized:semigroup}
\end{equation}
It remains to establish an error bound for the quantity $\E[U(0, \mu^{N}_0)]  - U(0, \mu)$. We follow similar lines of reasonings as those employed in \eqref{eq:cont:p0}. One may also refer to \cite{chassagneux:szpruch:tse} for a similar argument. We briefly repeat the proof here for sake of completeness.  From the mean-value theorem and the exchangeability in law of $(\xi^{i})_{1\leq i \leq N}$
\begin{align*}
\E[U(0, \mu^{N}_0) & - U(0, \mu)]  \notag\\
& = \int_{0}^1\int_{\mathbb{R}^d} \E\Big[\frac{\delta}{\delta m}U(0, \mu^{\lambda_1, N}_0)(y) \,  (\mu^{N}_0 - \mu)(dy)\Big] \, d\lambda_1\notag \\
& =  \int_{0}^1  \E\Big[\frac{\delta}{\delta m} U(0, \mu^{\lambda_1, N}_0)(\xi^1)- \frac{\delta}{\delta m}U(0, \mu^{\lambda_1, N}_0)(\widetilde{\xi})\Big] \, d\lambda_1\notag\\
& = \int_{0}^1  \E\Big[\frac{\delta}{\delta m} U(0, \widetilde{\mu}^{\lambda_1, N}_0)(\widetilde{\xi})- \frac{\delta}{\delta m}U(0, \mu^{\lambda_1, N}_0)(\widetilde{\xi})\Big] \, d\lambda_1\\
& = \frac{1}{N} \int_{[0,1]^2} \lambda_1 \E\Big[\frac{\delta^2}{\delta m^2} U(0, \widetilde{\mu}^{\lambda_1, \lambda_2, N}_0)(\widetilde{\xi}, \widetilde{\xi}) - \frac{\delta^2}{\delta m^2} U(0, \widetilde{\mu}^{\lambda_1, \lambda_2, N}_0)(\widetilde{\xi}, \xi^{1})\Big] \, d\lambda_1 d\lambda_2
\end{align*}

\noindent where we used the notations: $\mu^{\lambda_1, N}_0 := \lambda_1 \mu^{N}_0 + (1-\lambda_1) \mu$, $\widetilde{\mu}^{\lambda_1, N}_0 := \lambda_1 \widetilde{\mu}^{N}_0 + (1-\lambda_1) \mu$, $\widetilde{\mu}^{N}_0 := \mu^{N}_0  + \frac{1}{N}(\delta_{\widetilde{\xi}} - \delta_{\xi^1})$, $\widetilde{\mu}^{\lambda_1, \lambda_2, N}_0 := \lambda_2 \widetilde{\mu}^{\lambda_1, N}_0 + (1-\lambda_2) \mu^{\lambda_1, N}_0$,
 $\widetilde{\xi}$ being a random variable independent of $(\xi^{i})_{1\leq i \leq N}$ with law $\mu$. The previous identity together with \eqref{estimate:approxim:kolmo:backward:part2} finally yield
\begin{equation*}
\left|U(0,\mu_0)-\E\left[U(0,\mu_0^N)\right]\right| \leq T^{-1+\frac{\alpha}{2}}\frac{K^{+}}{N}
\end{equation*}

\noindent for some positive constant $K^{+}:=K(T, \HRp, \HE, M_2(\mu))$, $T\mapsto K(T, \HRp, \HE, M_2(\mu))$ being non decreasing. Plugging the previous bound into \eqref{strong:error:regularized:semigroup} and finally letting $t\uparrow T$ allows to conclude the proof of \eqref{esti:sem:chaos:part1}.

\subsection{Proof of Theorem \ref{propagation:prop:path}} As already mentioned in the introduction, the strategy relies on Zvonkin's transform. To do so, we introduce the following PDE
\begin{equation}\label{PDE_zvon}
(\partial_t+ \mathcal L_t)U(t, x,\mu) = b(t, x,\mu),\quad U(T,\cdot,\cdot) = 0_d,
\end{equation}
where the operator $(\mathcal L_t)_{t\geq 0}$ is given by \eqref{inf:generator:mckean:vlasov}. Note that under \A{HR}$_{+}$(iii) the drift coefficient $b$ is continuous on $[0,T]\times \mathbb{R}^d \times \pp$ and satisfies the assumption of Theorem 3.8 in \cite{chaudruraynal:frikha} so that there exists a unique solution $U \in \mathcal{C}^{1, 2, 2}([0,T)\times \mathbb{R}^d \times \pp)$ to the above PDE \eqref{PDE_zvon} satisfying for any $(t, x,\mu)\in [0,T)\times \mathbb{R}^d \times \pp$,
\begin{align*} 
U(t, x,\mu) & = \E\left[ \int_t^T b(s,X_s^{t, x,\mu},[X_s^{t,\xi}]) \, ds\right]  = \int_t^T \int_{\mathbb{R}^d} b(s, y,[X_s^{t,\xi}]) \, p(\mu, t, s, x, y) \, dy \, ds.
\end{align*}

Following the lines of proof of Proposition 6.1 in \cite{chaudruraynal:frikha}, we readily obtain that $U$ satisfies the following estimates: there exists a positive constant $K:=K(T, \HR, \HE)$, $T\mapsto K(T, \HR, \HE)$ being non-decreasing, such that for all $(t, x, \mu) \in [0,T) \times \mathbb{R}^d \times \pp$
\begin{eqnarray}\label{esi:U:1}
|\partial_v^n\partial_\mu U(t, x, \mu)(v)| +  |\partial_x^{n+1} U(t, x, \mu)| \leq  C (T-t)^{\frac{1-n+\eta}{2}},\quad n=0,1.
\end{eqnarray} 
%From Lemma \ref{additional:regularity:lemma} and the expression of the derivative $\partial_\mu U(t, x, \mu)$ provided in the proof of Proposition 6.1 of \cite{chaudruraynal:frikha}, the map $x\mapsto \partial_\mu U(t, x, \mu)(v)$ is continuously differentiable for all $(t,\mu, v)$ in $[0,T)\times \pp \times \mathbb{R}^d$. Moreover, from \eqref{deriv:cross:space:mes:pm} and \eqref{bound:density:parametrix}, its derivative satisfies 
%\begin{equation}\label{esi:U:2}
%\partial_x\partial_\mu U(t, x,\mu)(v) \leq C(T-t)^{\frac{\eta}{2}}.
%\end{equation}
%up to a modification on $\delta$. Indeed we have
%\begin{eqnarray*}
%\partial_\mu U(t, x, \mu)(v)  & =& \int_t^T  \int_{\rr^d} \partial_\mu b(s, y, [X^{t, \mu}_s])(v) p(\mu, t, s, x, y) dy ds \\
%&&+ \int_t^T \int_{\rr^d} b(s, y, [X^{t,\mu}_s]) \partial_\mu p(\mu, t, s, x, y)(v) dy ds. \nonumber
%\end{eqnarray*}
%From computations done in the proof of Lemma \ref{additional:regularity:lemma} (see section \ref{proof:lemma:additional:regularity} in the Appendix) we have that 
%\begin{equation}
%|\partial_x [\partial_\mu p(\mu, s, t, x, z)](v) | \leq  \frac{K}{(t-s)^{1-\frac{\eta}{2}}}\, g(c(t-s), z-x).
%\end{equation}
%Using the fact that $b$ is bounded and that $|\partial_\mu b(s, y, [X^{t, \mu}_s])(v)| \leq C (t-s)^{-\frac {1+\eta}{2}}$ the claim easily follows from Lebesgue differentiation theorem and usual gaussian estimates on space derivatives of $p$.\\
%
From Theorem \ref{derivative:density:sol:mckean:and:decoupling}, Lemma \ref{additional:regularity:lemma} and Proposition \ref{fully:c2:regularity:property:decoupling:density}, in particular the estimates \eqref{bound:density:parametrix}, \eqref{first:second:lions:derivative:mckean:decoupling}, \eqref{first:time:derivative:mckean:decoupling}, \eqref{deriv:cross:space:mes:pm}, \eqref{second:lions:derivative:mckean:decoupling}, one may apply Proposition \ref{structural:class} to the density function $(t, x, \mu) \mapsto p(\mu, t, s, x, y)$ and the map $m \mapsto b(s, y, m)$ to deduce that $(t, \mu) \mapsto b(s, y, [X_s^{t, \xi}]) \in \mathcal{C}^{1,2}_f([0,s)\times \pp)$.

Moreover, we deduce from the identities \eqref{condition1:structural:class:measure:derivative}, \eqref{condition1:structural:class:time:derivative}, \eqref{full:expression:second:deriv} combined with the estimates \eqref{bound:density:parametrix}, \eqref{first:second:lions:derivative:mckean:decoupling}, \eqref{first:time:derivative:mckean:decoupling}, \eqref{deriv:cross:space:mes:pm}, \eqref{second:lions:derivative:mckean:decoupling} as well as the uniform $\eta$-H\"older regularity of the linear functional derivative $[\delta b_i/\delta m](t, x, m)(.)$ and the space time inequality \eqref{space:time:inequality} that the following estimates hold:
\begin{eqnarray*}
| \partial_v^n[\partial_\mu [b(s, y ,[X_s^{t,\xi,(m)}])]](v)| \leq K^+ (s-t)^{\frac{-1-n+\eta}{2}},\, n=0,1,\quad |\partial_\mu^2 [b(s, y,[X_s^{t,\xi,(m)}])](v,v')| \leq K^+ (s-t)^{-1+\frac{\eta}{2}}
\end{eqnarray*}
\noindent and
$$
|\partial_t [b(s, y ,[X_s^{t,\xi,(m)}])]| \leq K^+ (s-t)^{-1+ \frac{\eta}{2}},
$$
\noindent for some positive constant $K^+ := K(T, \HRp, \HE)$.

Then, combining the above estimates together with the estimates \eqref{bound:density:parametrix}, \eqref{first:second:lions:derivative:mckean:decoupling}, \eqref{first:time:derivative:mckean:decoupling}, \eqref{deriv:cross:space:mes:pm}, \eqref{second:lions:derivative:mckean:decoupling} as well as the dominated convergence theorem, we eventually deduce that $U$ is in $\mathcal C_f^{1,2,2}([0,T)\times \mathbb{R}^d \times \pp)$ with a cross derivative $\partial_x \partial_\mu U(t, x, \mu)(v) = \partial_\mu[\partial_x U(t, x, \mu)](v)^{*}$ and an L-derivative of second order $\partial^2_\mu U(t, x, m)(v, v')$ given by
\begin{align*}
\partial_x \partial_\mu U(t, x, \mu)(v) & = \int_t^T \int_{\mathbb{R}^d} \Big[\partial_\mu [b(s, y,[X_s^{t,\xi}])](v) \otimes \partial_x p(\mu, t, s, x, y) + b(s, y,[X_s^{t,\xi}])(v)  \partial_x \partial_\mu p(\mu, t, s, x, y)(v) \Big] \, dy \, ds,
\end{align*}

\begin{align*}
\partial_\mu^2 U(t, x, \mu)(v, v') & = \int_t^T \int_{\mathbb{R}^d} \Big[\partial_\mu^2 [b(s, y,[X_s^{t,\xi}])](v, v') \, p(\mu, t, s, x, y) + \partial_\mu [b(s, y,[X_s^{t,\xi}])](v) \otimes \partial_\mu p(\mu, t, s, x, y)(v')  \\
& \quad +  \partial_\mu p(\mu, t, s, x, y)(v) \otimes  \partial_\mu [b(s, y,[X_s^{t,\xi}])](v') +  b(s, y,[X_s^{t,\xi}]) \partial^2_\mu p(\mu, t, s, x, y)(v, v')  \Big] \, dy \, ds 
\end{align*}
\noindent and satisfying
\begin{align}
|\partial_x \partial_\mu U(t, x, \mu)(v)| + |\partial^2_\mu U(t, x, \mu)(v, v')| \leq  K^{+}(T-t)^{\frac{\eta}{2}}. \label{esi:U:2}
\end{align}

We are now ready to complete the proof of Theorem \ref{propagation:prop:path}. The chain rule formula of Proposition \ref{prop:chain:rule:joint:space:measure} yields

\begin{eqnarray}\label{repZvonIndep}
&&\qquad \bar X^i_t - U(t,\bar X^i_t,\mu_t)  \\
&=&  \xi^i- U(0,\xi^i,\mu_0) - \int_{0}^t \left(\left[\partial_{x} U- \mathbf{1}\right]\sigma \right) (s,\bar X^i_s,\mu_s ) dW_s^i - \int_0^t[(\partial_s+\mathcal L_s)U-b](s,\bar X_s^i, \mu_s) ds\notag
\end{eqnarray}

\noindent and from classical Itô's formula
\begin{eqnarray}\label{repZvonPart}
&&\qquad X_t^i - U(t,X_t^i,\mu_t^N)  \\
&=&  \xi^i- U(0,\xi^i,\mu_0^N) - \int_{0}^t \left(\left[\partial_{x}U - \mathbf{1}\right]\sigma\right) (s,X_s^i, \mu_s^N ) dW_s^i - \int_0^t [(\partial_s+\mathcal L_s)U-b](s,X_s^i,\mu_s^N) ds\notag\\
&& + \frac 1{2N^2} \sum_{j=1}^N \int_0^t \tr\left(a(s,X_s^i,\mu_s^N)\partial_{\mu}^2U(s,X_s^i,\mu_s^N)( X_s^j, X_s^j) \right) ds\notag\\
&& \qquad + \frac1N \sum_{j=1}^N \int_0^t \partial_\mu U(s,X_s^i,\mu_s^N)(X_s^j) . (\sigma(s, X_s^{j}, \mu_s^N) dW^j_s).\notag
\end{eqnarray}

Taking the difference between \eqref{repZvonIndep} and \eqref{repZvonPart} and using the fact that $U$ solves the PDE \eqref{PDE_zvon} yield
\begin{eqnarray}\label{eq:diffZvonkp}
\bar X^i_t - X_t^i  & =  &  U(0,\xi^i,\mu_0^N) - U(0,\xi^i,\mu_0) +  U(t,\bar X_t^i,\mu_t) - U(t,X_t^i,\mu_t^N) \\
&&- \frac 1{2N^2} \sum_{j=1}^N \int_0^t \tr\left(a(s,X_s^i,\mu_s^N)\partial_{\mu}^2U (s,X_s^i,\mu_s^N)( X_s^j, X_s^j) \right) ds\notag\\
&&-\int_{0}^t \left(\left[\partial_{x}U- \mathbf{1}\right]\sigma\right) (s,\bar X_s^i, \mu_s )-\left(\left[\partial_{x}U - \mathbf{1}\right]\sigma \right)(s,X_s^i, \mu_s^N ) dW_s^i\notag \\
&&- \frac1N \sum_{j=1}^N \int_0^t \partial_\mu U (s,X_s^i,\mu_s^N)(X_s^j). (\sigma(s, X_s^{j}, \mu_s^N) dW^j_s). \notag
%&&- \int_0^t[(\partial_s+\mathcal L_s)U - b](s,X_s^i,\mu_s^N) ds +  \int_0^t [(\partial_s+\mathcal L_s)U -b](s,\bar X_s^i, \mu_s) ds.\notag
\end{eqnarray}

Now, it follows from the estimates \eqref{esi:U:1}, \eqref{esi:U:2} and the boundedness of $\sigma$ that the terms in the third and fourth lines appearing in the right-hand side of the above identity are true square integrable martingales and that the term in the second line of the above equality is of order $N^{-1}$.

% On the other hand, we have from the point $(iv)$ together with the estimates $(ii)$ and the dominated convergence theorem that, along a subsequence, for any $0\leq t<T$
%\begin{eqnarray}\label{eq:lim:mollif}
%&&\lim_{m \uparrow \infty} \Bigg\{\max_{1\leq i \leq N}\E\left[ \Big(\int_0^t |[(\partial_s+\mathcal L_s)U^{(m)}-b](s,X_s^i,\mu_s^N)| \, ds \Big)^2\right]\\
%&&\qquad +\max_{1\leq i \leq N}\E\left[ \Big(\int_0^t | [(\partial_s+\mathcal L_s)U^{(m)}-b](s,\bar X_s^i, \mu_s)| \, ds\Big)^2\right] \Bigg\}= 0\notag.
%\end{eqnarray}
 Therefore, taking the square of the norm in both sides of the identity \eqref{eq:diffZvonkp}, then summing over $i$ and eventually using Burkholder-Davis-Gundy's inequality give
\begin{align*}
\E\left[ \frac{1}{N}\sum_{i=1}^N |\bar X_t^i - X_t^i|^2\right] &\leq C\Bigg\{ \E\left[|U(0,\xi^1,\mu_0^N) - U(0,\xi^1,\mu_0)|^2\right] +  \frac 1N \sum_{i=1}^N\E\left[|U(t,X_t^i,\mu_t) - U(t,\bar X_t^i,\mu_t^N)|^2\right]\\
& + \frac 1N \sum_{i=1}^N\int_{0}^t \E\left[\left|\left(\left[\partial_{x}U- \mathbf{1}\right]\sigma\right) (s,\bar X_s^i, \mu_s )-\left(\left[\partial_{x}U - \mathbf{1}\right]\sigma \right)(s,X_s^i, \mu_s^N )\right|^2\right]ds\Bigg\}+\frac{K^{+}}{N}.
\end{align*}
Using \eqref{esi:U:1}, \eqref{esi:U:2} and the uniform Lipschitz regularity of $(x, \mu) \mapsto \sigma(t, x, \mu)$, we deduce that there exists $K^{+}_T:=K(T, \HRp, \HE)>0$ satisfying $K^{+}_T \downarrow 0$ when $T\downarrow 0$ such that
\begin{align*}
\E\left[ \frac{1}{N}\sum_{i=1}^N |\bar X_t^i - X_t^i|^2\right] &\leq  K^{+}_{T} \Bigg\{\E\left[W_2(\mu_0,\mu_0^N)^2\right] + \frac{1}{N} \sum_{i=1}^N\E[|\bar X_t^i-X_t^i|^2] + \E[W_2(\mu_t,\mu_t^N)^2] \Bigg\}\\
 &+\frac{K^{+}}{N} \sum_{i=1}^N \int_0^t (\E[|\bar X_s^i-X_s^i|^2] + \E\left[W_2(\mu_s, \mu_s^N)^2\right]) ds +\frac{K^{+}}{N} .
\end{align*}
We now introduce $\bar \mu_t^N := N^{-1} \sum_{i=1}^N \delta_{\bar X_t^i}$, $t\in [0,T]$, the empirical measure associated with the i.i.d. random variable $(\bar X_t^i)_{1\leq i \leq N}$. Noticing that for all $t\in [0,T]$, 
\begin{equation}
W_2(\mu_t,\mu_t^N)^2 \leq  2W_2(\mu_t,\bar \mu_t^N)^2 +  2W_2(\bar \mu_t^N, \mu_t^N)^2 \leq 2W_2(\mu_t,\bar \mu_t^N)^2 +  \frac{2}{N}\sum_{i=1}^N |X_t^i -\bar X_t^i|^2 \label{triangular:inequality}
\end{equation}
 
 \noindent and choosing $T$ small enough\footnote{There exists $\mathcal{T}=\mathcal T(\HRp, \HE)>0$ such that for all $T \leq \mathcal T$ we have $K^{+}_{T}\leq 1/4$.} so that $K^{+}_{T}\leq 1/4$ and using Gr\"onwall's lemma lead to 
\begin{eqnarray*}
 \frac{1}{N} \sum_{i=1}^{N}\E\left[ |\bar X_t^i - X_t^i|^2\right] \leq  K^{+} \left\{\E\left[W_2(\mu_0,\bar \mu_0^N)^2\right] + \sup_{0\leq t\leq T} \E[W_2(\mu_t,\bar \mu_t^N)^2] +\int_0^T \E\left[W_2(\mu_s,\bar \mu_s^N)^2\right] ds + \frac{1}{N}\right\}.
\end{eqnarray*}

Finally, the strong well-posedness of the SDEs \eqref{SDE:particle:system:coupling} and \eqref{SDE:particle:system} together with the exchangeability of $(\xi^{i}, W^{i})_{1\leq i \leq N}$ imply that the random variables $(\bar{X}^i, X^i)_{1\leq i\leq N}$ are identically distributed so that $N^{-1}  \sum_{i=1}^{N}\E\left[ |\bar X_t^i - X_t^i|^2\right] = \E\left[ |\bar X_t^1 - X_t^1|^2\right] $. Hence,
$$
\sup_{0\leq t \leq T} \E\left[ |\bar X_t^i - X_t^i|^2\right] \leq  K^{+} \left\{\E\left[W_2(\mu_0,\bar \mu_0^N)^2\right] + \sup_{0\leq t\leq T} \E[W_2(\mu_t,\bar \mu_t^N)^2] + \frac{1}{N} \right\}\leq K^{+} \varepsilon_N
$$

\noindent with $K^{+}:=K(T, \HRp, \HE, q, M_q(\mu))>0$ and where we used Theorem 1 in Fournier and Guillin \cite{Fournier2015} or Theorem 5.8 of \cite{carmona2018probabilistic} for the last inequality. One may then extend the above estimate to an arbitrary finite time horizon $T$ by considering a partition of the time interval $[0,T]$ with a sufficiently small time mesh and repeating the above argument, observing that the estimates \eqref{esi:U:1} and \eqref{esi:U:2} are uniform in $x$ and $\mu$. Taking expectation in \eqref{triangular:inequality}, one then concludes that a similar estimate holds for the quantity
$$
\sup_{0\leq t \leq T} \E[W_2(\mu_t,\mu_t^N)^2]. 
$$

Finally, coming back to \eqref{eq:diffZvonkp}, one can apply similar lines of reasonings but taking first the square of the norm, then the supremum in time and obtain, thanks to the above estimate, for $T$ small enough
\begin{equation*}
\max_{1\leq i \leq N}\E\left[ \sup_{0\leq t \leq T} |X_t^i- \bar X^i_t|^2\right] \leq K^{+} \left\{\E\left[W_2(\mu_0,\bar \mu_0^N)^2\right] +  \E[\sup_{0\leq t\leq T}W_2(\mu_t, \bar \mu_t^N)^2] +\int_0^T \E\left[W_2(\mu_s,\bar \mu_s^N)^2\right] ds + \frac{1}{N}\right\}.
\end{equation*}
The first and third terms appearing in the right-hand side of the above inequality are handled using Theorem 1 in \cite{Fournier2015}. The second term provides the rate of convergence and requires the following lemma borrowed from \cite{BCCdRH:19}.
\begin{lem}\label{ConvPart}
Let $\{Y^i_\cdot\}_{1 \leq i \leq N}$ be an i.i.d. sequence of copies of a process $Y$ satisfying $\sup_{t \in [0,1]} \E |Y_t|^q < +\infty,\ \text{for some } q>4$ and for some $p>2$:
\begin{eqnarray}
\E [ |Y_s - Y_r|^p |Y_s-Y_t|^p] &\leq &C |t-r|^2,\quad {\rm for }\ 0 \leq r < s < t \leq 1;\notag\\
\E [|Y_t-Y_s|^p] &\leq &C |t-s|,\quad {\rm for }\ 0 \leq  s \leq  t \leq 1; \label{hypRacRu}\\
\E [|Y_t - Y_s |^2]& \leq& C |t-s|,\quad {\rm for }\ 0 \leq  s \leq  t \leq 1. \notag
\end{eqnarray}
Then, introducing the notations $\nu_s:=[Y_s]$ and $\bar \nu^N_s := N^{-1} \sum_{i=1}^N \delta_{Y_s^i}$, there exists $C>0$ such that
\begin{equation}
\E \left[\sup_{0\leq s \leq T} W_2(\bar \nu_{s}^N, \nu_{s})^2 \right]\leq C\sqrt{\varepsilon_N}.
\end{equation}
\end{lem}
%
%Recall that the solution $U$ of \eqref{PDE_zvon}
% belongs to  $\mathcal{C}^{1, 2, 2}([0,T)\times \rr^d \times \pp)$. Applying It\^{o}'s formula from proposition \ref{prop:chain:rule:joint:space:measure} to $X_t^i - U(t,X_t^i,\mu_t)$ we obtain that for any $p\geq 2$, for any $0\leq s \leq t \leq T$:
%\begin{eqnarray}\label{diff:repZvonIndep}
%\E[ |X_t - X_s|^p] &\leq& C_p\left\{  \E[|U(s,X_s,\mu_s) - U(t,X_t,\mu_t)|^p] + \E\left[\left|\int_{s}^t \left(\left[\partial_{x} U- \mathbf{1}\right]\sigma \right) (s, X_s,\mu_s ) dW_s\right|^p\right]\right\}\notag\\
%&\leq& C_{T,p}\left\{ \E[|X_t-X_s|^p]  + |t-s|^{p/2}\right\}\notag,
%\end{eqnarray}
%thanks to BDG's inequality and \eqref{esi:U:1}. The estimates in the above lemma are hence fulfilled provided that $T$ is small enough. 

We thus derive
$$
\max_{1\leq i \leq N}\E\left[ \sup_{0\leq t \leq T} |X_t^i- \bar X^i_t|^2\right]  \leq K^{+} \sqrt{\varepsilon_{N}}.
$$
\noindent for some positive constant $K^{+}:=K(T, \HRp, \HE, [\sigma]_L, q, M_q(\mu))$. Taking first the supremum in time and then expectation in \eqref{triangular:inequality}, one then concludes that a similar estimate holds for the quantity
$$
\E[\sup_{0\leq t \leq T} W_2(\mu_t,\mu_t^N)^2]. 
$$

\appendix
\section{proof of Lemma \ref{additional:regularity:lemma}} \label{proof:lemma:additional:regularity}

We here freely use the notations and the results established in \cite{chaudruraynal:frikha}. Since the arguments and the computations are quite similar to those employed in \cite{chaudruraynal:frikha}, we will deliberately be short on some technical details. We start by recalling some important estimates established in \cite{chaudruraynal:frikha}. Let us emphasize that these estimates are established for the corresponding approximation sequences, namely $(p_m(\mu, s, t, x, z))_{m\geq1}$, $(\widehat{p}^{y}_m(\mu, s, r, t, x, z))_{m\geq1}$, $(\mH_m(\mu, s, r, t, x, z))_{m\geq1}$, ..., constructed in Section \ref{full:c2:regularity:section} but are still valid for the corresponding limiting object by copying verbatim the corresponding proof except that one directly uses the estimates provided by Theorem 3.6 therein.

\begin{lem}\label{technical:lem:from:chaudru:frikha}
 Let $n\in \left\{0,1\right\}$. For any $\beta \in [0,1]$ if $n=0$ or any $\beta \in [0,\eta)$ if $n=1$, there exist some positive constants $K:=K(T, \HR, \HE)$, $K_\beta^{+}:=K(T, \HRp, \HE,\beta)$ and $c:=c(\lambda)$ such that for any $\mu, \mu' \in \pp$ and any $x, x', v, v', y \in \mathbb{R}^d$, it holds
\begin{align}
\Big| \partial^{n}_v [\partial_\mu & \widehat{p}^{y}(\mu, s, t, x, z)](v)  - \partial^{n}_v [\partial_\mu \widehat{p}^{y}(\mu', s, t, x', z)](v') \Big| \notag \\
& \leq K^+_\beta \frac{[W_2(\mu, \mu')^\beta + |x-x'|^\beta + |v-v'|^\beta] }{(t-s)^{\frac{1+n+\beta-\eta}{2}}} \, \left\{ g(c(t-s), z-x) + g(c(t-s), z-x')\right\}, \label{diff:space:mes:v:partial:mes:phat}
\end{align}

\begin{align}
 \Big| \partial^{n}_v &[\partial_\mu [a_{i, j}(t, x, [X^{s,\xi}_t])]](v) -   \partial^{n}_v [\partial_\mu [a_{i, j}(t, x, [X^{s, \xi}_t])]_{\mu=\mu'}](v') \Big|  \notag \\
& \leq K_\beta^{+} \frac{[W_2(\mu, \mu')^{\beta}+ |v-v'|^\beta]}{(t-s)^{\frac{1+n+\beta-\eta}{2}}}, \label{diff:mes:v:partial:deriv:diff:coeff}
\end{align}

\begin{align}
\Big| \partial^{n}_v &  [\partial_\mu \mH(\mu, s, r, t, x, z)](v) - \partial^{n}_v [\partial_\mu \mH(\mu, s, r, t, x, z)]_{\mu= \mu'}(v') \Big| \notag \\
& \leq K_\beta^{+} [W_2(\mu, \mu')^{\beta} + |v-v'|^\beta ] \left\{ \frac{1}{(t-r)(r-s)^{\frac{1+n+\beta-\eta}{2}}} \wedge \frac{1}{(t-r)^{1- \frac{\eta}{2}}(r-s)^{\frac{1+n+\beta}{2}}}  \right\} \label{diff:v:and:mes:partial:mes:parametrix:kernel} \\
& \quad \quad \times \, g(c(t-r), z-x).\notag
\end{align}

\noindent Similarly, for any $\beta \in [0,1]$ if $n\in \left\{0,1\right\}$ or any $\beta \in [0,\eta)$ if $n=2$
\begin{align}
 |\partial^{n}_x p(\mu, s, t, x, z) & - \partial^{n}_x p(\mu', s, t, x, z)](v)| \leq K_\beta \frac{W_2(\mu, \mu')^{\beta}}{(t-s)^{ \frac{n+\beta}{2}}} \, g(c(t-s), z-x), \label{regularity:measure:estimate:v1:v2:v3:mckean:decoupling} 
\end{align}

\begin{align}
 | \partial^{n}_x \widehat{p}^{y}(\mu, s, t, x, z) - \partial^{n}_x \widehat{p}^{y}(\mu', s, t, x, z)  | \leq K \frac{W_2(\mu, \mu')^{\beta}}{(t-s)^{\frac{n+\beta}{2}}} \, g(c(t-s), z-x), \label{gaussian:bound:diff:deriv:hat:pm:same:time}
\end{align}

\noindent and for any $\beta \in [0,1]$
\begin{align}
 |  \mH(\mu, s, r ,t , x ,z) - \mH(\mu', s, r ,t , x ,z) |    \leq K \frac{W_2(\mu, \mu')^{\beta}}{(t-r)^{1-\frac{\eta}{2}}(r-s)^{\frac{\beta}{2}}}  \, g(c(t-r), z - x),\label{bound:diff:H:mu}
\end{align}
\begin{align}
|a_{i, j}(t, x, [X^{s, \xi}_t])  - & a_{i, j}(t, x, [X^{s, \xi'}_t])|  \leq K \frac{W_2(\mu, \mu')^{\beta}}{(t-s)^{\frac{\beta}{2}}}, \label{diff:mes:drift:diff:coefficients:true:mes}
\end{align}
\begin{align}
 |  \Phi(\mu, s, r ,t , x ,z) - \Phi(\mu', s, r ,t , x ,z) |    \leq K \frac{W_2(\mu, \mu')^{\beta}}{(t-r)^{1-\frac{\eta}{2}}(r-s)^{\frac{\beta}{2}}}  \, g(c(t-r), z - x). \label{bound:diff:Phi:mu}
\end{align}

\end{lem}

We now move to the proof of Lemma \ref{additional:regularity:lemma}. \\

\noindent \emph{Step 1: smoothness of the maps $x\mapsto \partial_\mu p(\mu, s, t, x, z)(v)$, $\mu \mapsto \partial_x p(\mu, s, t, x, z)$ and proof of the estimate \eqref{deriv:cross:space:mes:pm}.} \\

First, combining Proposition 2.2 (applied to the maps $m\mapsto b_i(t, x, m), \, a_{i, j}(t, x, m)$) with Theorem 3.6 in \cite{chaudruraynal:frikha}, and following the lines of proof of (A.9) and (A.15) in Corollaries A.1 and A.2 therein, we deduce that for any $\beta \in [0,1]$ there exists some positive constants $K_\beta:=K(T, \HR, \HE, \beta)$, $K:=K(T, \HR, \HE)$ and $c:=c(\lambda)$ such that 
\begin{align}
| \partial_\mu [b_{i}(t, x, [X^{s,\xi}_t])](v) + \partial_\mu [a_{i, j}(t, x, [X^{s,\xi}_t])](v) | \leq \frac{K}{(t-s)^{\frac{1-\eta}{2}}}, \label{deriv:mes:a}
\end{align}
\begin{align}
|\partial^{n}_v [\partial_\mu \mH(\mu, s, r ,t ,x ,z)](v)| \leq \frac{K_\beta}{(t-r)^{1-\frac{\beta \eta}{2}}(r-s)^{\frac{1+n-(1-\beta)\eta}{2}}} \, g(c(t-r), z-x), \label{bound:deriv:mes:kernel}
\end{align}

\noindent and
\begin{align}
| \partial^{n}_v  [\partial_\mu \widehat{p}^{y}(\mu, s, r, t, x, z)](v)|  \leq \frac{K}{t-r}  \int_r^t \frac{1}{(r'-s)^{\frac{1+n - \eta}{2}}} \,dr'  g(c(t-r), z-x). \label{deriv:v:deriv:mes:p:hat}
\end{align}

Similarly to (A.23) of Proposition A.1 in \cite{chaudruraynal:frikha}, the following representation formulae holds:
\begin{align}
\partial_\mu p(\mu, s, t, x, z)(v) & = \sum_{k\geq0} \Big(\big[\partial_\mu \widehat{p} + p \otimes \partial_\mu \mathcal{H} \big] \otimes \mathcal{H}^{(k)}\Big)(\mu, s, t, x, z)(v) \notag \\
&  = \partial_\mu \widehat{p}(\mu, s, t, x, z)(v) + (p\otimes \partial_\mu \mathcal{H})(\mu, s, t, x, z)(v) + ((\partial_\mu \widehat{p} + (p \otimes \partial_\mu \mathcal{H})) \otimes \Phi)(\mu, s, t, x, z)(v) \label{infinite:series:deriv:mes:density}
\end{align}
\noindent and 
\begin{align}
\partial_x  p(\mu, s, t, x, z) = \partial_x \widehat{p}(\mu, s, t, x, z) +  (\partial_x \widehat{p} \otimes \Phi)(\mu, s, t, x, z) \label{representation:formula:first:derivative:p}
\end{align}
\noindent where for $n\in \left\{0,1\right\}$
\begin{align}
& \partial^n_v[\partial_\mu \widehat{p}^{y}(\mu, s, r, t, x, z)](v)  \notag\\
& = -\frac12\left\{ \tr\left(\left(\int_r^t a(r', y, [X^{s, \xi}_{r'}])\, dr'\right)^{-1} \int_r^t \partial^n_v[\partial_\mu [a(r', y, [X^{s, \xi}_{r'}])]](v) \, dr' \right) \right. \notag\\
& \quad \left. - (z-x)^{t} \left(\int_r^t a(r', y, [X^{s, \xi}_{r'}])\, dr'\right)^{-1} \int_r^t \partial^n_v[\partial_\mu [a(r', y, [X^{s, \xi}_{r'}])]](v) \, dr'  \right. \label{representation:formula:deriv:mes:p:hat:after:limit}\\ 
& \quad \left. \quad \times \left(\int_r^t a(r', y, [X^{s, \xi}_{r'}])\, dr'\right)^{-1} (z-x) \right\} \widehat{p}^{y}(\mu, s, r, t, x, z) \notag
\end{align}
\noindent and
\begin{equation}
\label{expression:first:space:derivative:phat}
\partial_x \widehat{p}^{y}(\mu, s, r, t, x, z) = H_1\Big(\int_r^t a(r', y, [X^{s, \xi}_{r'}])\, dr', z-x\Big) \widehat{p}^{y}(\mu, s, r, t, x, z).
\end{equation}

%The identities \eqref{infinite:series:deriv:mes:density} and \eqref{representation:formula:deriv:mes:p:hat:after:limit} stem from Proposition A.1 together with (A.15) of Corollary A.2 as well as (B.12) in the proof of Corollary (A.1) together with (B.3) in the proof of Lemma A.1 in \cite{chaudruraynal:frikha} by passing to the limit as $m\uparrow\infty$.

%\begin{alig

Now, it is readily seen from the identity \eqref{representation:formula:deriv:mes:p:hat:after:limit} that $x\mapsto \partial^n_v[\partial_\mu \widehat{p}^{y}(\mu, s, r, t, x, z)](v) $ is continuously differentiable with a derivative being continuous in $x, \mu, v$. Moreover, using the space time inequality \eqref{space:time:inequality} and \eqref{deriv:mes:a}, we obtain
\begin{equation}
\label{estimate:partialx:partialmes:phat}
\Big| \partial_x [\partial_\mu \widehat{p}^{y}(\mu, s,  t, x, z)](v) \Big| \leq \frac{K}{(t-s)^{1-\frac{\eta}{2}}} \, g(c(t-s), z-x).
\end{equation}

Using the Gaussian estimate \eqref{bound:density:parametrix} with $n=1$, splitting the time integral over $[s, t]$ of the space time convolution operator $\otimes$ into the two disjoint intervals $[s, \frac{t+s}{2}]$ and $(\frac{t+s}{2}, t]$ and using \eqref{bound:deriv:mes:kernel} (with $\beta=0$ on $[s, \frac{t+s}{2}]$ and $\beta=1$ otherwise) to balance the time singularity in the integral, by the dominated convergence theorem, we deduce that $x\mapsto (p\otimes \partial_\mu \mathcal{H}(.)(v))(\mu, s, t, x, z)$ is continuously differentiable with a derivative being continuous in $x, \mu, v$ and satisfying
\begin{equation}
\label{partialx:p:convol:partialmes:parametrix:kernel}
| (\partial_x p\otimes \partial_\mu \mathcal{H}(.)(v))(\mu, s, t, x, z)| \leq \frac{K}{(t-s)^{1-\frac{\eta}{2}}} \, g(c(t-s), z-x).
\end{equation}

Again, using the two previous estimates, \eqref{Gaussian:estimate:Phi} and the dominated convergence theorem, we deduce that the map $x \mapsto  ((\partial_\mu \widehat{p}(.)(v) + (p \otimes \partial_\mu \mathcal{H}(.)(v))) \otimes \Phi)(\mu, s, t, x, z)$ is continuously differentiable with a derivative being continuous in $x, \mu, v$, satisfying $\partial_x((\partial_\mu \widehat{p}(.)(v) + (p \otimes \partial_\mu \mathcal{H}(.)(v))) \otimes \Phi)(\mu, s, t, x, z) =((\partial_x \partial_\mu \widehat{p}(.)(v) + (\partial_x p \otimes \partial_\mu \mathcal{H}(.)(v))) \otimes \Phi)(\mu, s, t, x, z)$ and
\begin{equation}
\label{partialx:last:term:deriv:mes:p}
|\partial_x((\partial_\mu \widehat{p}(.)(v) + (p \otimes \partial_\mu \mathcal{H}(.)(v))) \otimes \Phi)(\mu, s, t, x, z)| \leq \frac{K}{(t-s)^{1- \eta}} \, g(c(t-s), z-x).
\end{equation}

We thus conclude from the preceding discussion and the identity \eqref{infinite:series:deriv:mes:density} that $x\mapsto \partial_\mu p(\mu, s, t, x, z)(v)$ is continuously differentiable with a derivative being continuous in $x, \mu, v$ and satisfying 
\begin{align*}
\partial_x & \partial_\mu p(\mu, s, t, x, z)(v) \\
&  = \partial_x\partial_\mu \widehat{p}(\mu, s, t, x, z)(v) + (\partial_x p\otimes \partial_\mu \mathcal{H}(.)(v))(\mu, s, t, x, z) + ((\partial_x \partial_\mu \widehat{p}(.)(v) + (\partial_x p \otimes \partial_\mu \mathcal{H}(.)(v))) \otimes \Phi)(\mu, s, t, x, z). 
\end{align*}
Moreover, from the estimates \eqref{estimate:partialx:partialmes:phat}, \eqref{partialx:p:convol:partialmes:parametrix:kernel} and \eqref{partialx:last:term:deriv:mes:p}, the following pointwise Gaussian estimate is satisfied
$$
|\partial_x  \partial_\mu p(\mu, s, t, x, z)(v)| \leq \frac{K}{(t-s)^{1-\frac{\eta}{2}}} \, g(c(t-s), z-x).
$$

We now investigate the smoothness of the map $\pp \ni \mu \mapsto \partial_x  p(\mu, s, t, x, z)$ using similar arguments. Starting from the identity \eqref{representation:formula:first:derivative:p}, our aim is to prove that one is allowed to differentiate each term with respect to $\mu$ and that the derivatives of each term is a continuous function with respect to $x, \mu, v$. From \eqref{representation:formula:deriv:mes:p:hat:after:limit}, \eqref{deriv:mes:a} and the dominated convergence theorem, the map $\mu\mapsto \partial_x \widehat{p}^{y}(\mu, s, r, t, x, z) $ given by \eqref{expression:first:space:derivative:phat} is continuously L-differentiable with a derivative being continuous in $x, \mu, v$ and satisfying
\begin{align}
\partial_\mu [\partial_x \widehat{p}^{y}(\mu, s, r, t, x, z)](v) & = \partial_\mu \Big[H_1\Big(\int_r^t a(r', y, [X^{s, \xi}_{r'}])\, dr', z-x\Big)\Big](v) \, \widehat{p}^{y}(\mu, s, r, t, x, z) \notag\\
& \quad +  H_1\Big(\int_r^t a(r', y, [X^{s, \xi}_{r'}])\, dr', z-x\Big) \partial_\mu \widehat{p}^{y}(\mu, s, r, t, x, z)(v). \label{decomposition:partial:mes:partial:space:phat}
\end{align}

\noindent Moreover, from the above expression, \eqref{representation:formula:deriv:mes:p:hat:after:limit}, \eqref{deriv:mes:a}, the space time inequality \eqref{space:time:inequality} and standard computations, we obtain
\begin{align}
|\partial_\mu [\partial_x \widehat{p}^{y}(\mu, s, t, x, z)](v)| & \leq \frac{K}{(t-s)^{1-\frac{\eta}{2}}} \, g(c(t-s), z-x). \label{estimate:partial:mes:partial:space:phat}
\end{align}

Following similar lines of reasonings as those used in the proof of Corollary A.3 in \cite{chaudruraynal:frikha}, namely using the relation
\begin{equation}
\label{iterated:parametrix:kernel:next:step}
\mH^{(k+1)}(\mu, s, r, t, x, z) = \int_r^t \int_{ \mathbb{R}^d }  \mH(\mu, s, r , r' , x , y)  \mH^{(k)}(\mu, s, r' ,t , y ,z) \, dy \, dr'
\end{equation}
\noindent as well as the estimates \eqref{iter:parametrix:kernel:true:kernel} and \eqref{bound:deriv:mes:kernel}, by induction on $k$, we derive that for any positive integer $k$ the map $\mu \mapsto \mH^{(k)}(\mu, s, r, t, x, z)$ is continuously L-differentiable with a derivative $\partial_\mu [\mH^{(k)}(\mu, s, r, t, x, z)](v)$ being continuous in $\mu, v$ and continuously differentiable with respect to the variable $v$ with a continuous derivative in $\mu, v$ and satisfying for any $\beta \in (0,1]$
\begin{align}
| \partial^{n}_v [& \partial_\mu[\mH^{(k)}(\mu, s, r, t, x, z)]](v) | \notag \\
& \leq \frac{k K_\beta^k}{(r-s)^{\frac{1+n-(1-\beta)\eta}{2}}(t-r)^{1-\beta \frac{\eta}{2} - (k-1)\frac{\eta}{2}}} \prod_{\ell=1}^{k-1}B(\frac{\eta}{2}, \frac{\eta}{2}\beta+ (\ell-1) \frac{\eta}{2}) \, g(c(t-r), z-x) \label{estimate:partial:deriv:mes:parametrix:kernel:iterated}
\end{align}  
\noindent for some positive constants $K_\beta:=K(T, \HR, \HE, \beta)$, $c:=c(\lambda)$. It follows from the previous estimate, the asymptotics of the Beta function, the identity \eqref{infinite:series:Phi:step} and the dominated convergence theorem that $\mu \mapsto \Phi(\mu, s, r, t, x, z)$ is continuously L-differentiable with a derivative $\partial_\mu \Phi(\mu, s, r, t, x, z)(v)$ being continuous in $\mu, v$ and continuously differentiable with respect to the variable $v$ with a continuous derivative in $\mu, v$ and satisfying for any $\beta \in (0,1]$
\begin{align}
| \partial^n_v [\partial_\mu [\Phi(\mu, s, r, t, x, z)]](v)| \leq \frac{K_\beta}{(r-s)^{\frac{1+n-(1-\beta)\eta}{2}}(t-r)^{1-\beta \frac{\eta}{2}}} \, g(c(t-r), z-x). \label{estimate:deriv:mes:Phi}
\end{align}

Now, it follows from \eqref{estimate:partial:mes:partial:space:phat}, \eqref{estimate:deriv:mes:Phi} and the dominated convergence theorem that $\mu \mapsto  (\partial_x \widehat{p} \otimes \Phi)(\mu, s, t, x, z) $ is continuously L-differentiable with a derivative given by 
$$
\partial_\mu (\partial_x \widehat{p} \otimes \Phi)(\mu, s, t, x, z) = (\partial_\mu [\partial_x \widehat{p}](.)(v) \otimes \Phi)(\mu, s, t, x, z) + (\partial_x \widehat{p} \otimes \partial_\mu \Phi(.)(v))(\mu, s, t, x, z), 
$$
\noindent being continuous in $\mu, x, v$ and satisfying
$$
|\partial_\mu [(\partial_x \widehat{p} \otimes \Phi)(\mu, s, t, x, z)](v)| \leq \frac{K}{(t-s)^{1-\frac{\eta}{2}}} \, g(c(t-s), z-x).
$$

Hence, coming back to the identity \eqref{representation:formula:first:derivative:p}, we conclude that $\mu \mapsto  \partial_x  p(\mu, s, t, x, z)$ is continuously L-differentiable with a derivative being continuous in $\mu, x, v$. Moreover, it follows from the previous estimate and \eqref{estimate:partial:mes:partial:space:phat} that
$$
|  \partial_\mu [\partial_x p(\mu, s, t, x, z)](v)| \leq \frac{K}{(t-s)^{1-\frac{\eta}{2}}} \, g(c(t-s), z-x).
$$

Let us finally mention that it follows from \eqref{bound:density:parametrix}, \eqref{iter:parametrix:kernel:true:kernel} with $k=1$, \eqref{estimate:partial:deriv:mes:parametrix:kernel:iterated} and the previous estimate as well as the relation $p = \widehat{p} + p \otimes \mH$ which directly stems from \eqref{parametrix:series:expansion} that 
$$
\partial_\mu \partial_x p(\mu, s, t, x, z)(v) = \partial_\mu [\partial_x \widehat{p}(\mu, s, t, x, z)](v) + (\partial_x p \otimes \partial_\mu \mH(.)(v))(\mu, s, t, x, z) + (\partial_\mu [\partial_x p](.)(v) \otimes \mH)(\mu, s, t, x, z). 
$$

Moreover, the estimates \eqref{estimate:partial:mes:partial:space:phat} and \eqref{partialx:p:convol:partialmes:parametrix:kernel} show that the kernel $ \partial_\mu [\partial_x \widehat{p}(\mu, s, t, x, z)] + (\partial_x p \otimes \partial_\mu \mH(.)(v))(\mu, s, t, x, z)$ is non singular so that one may iterate the previous relation. Hence, it holds
\begin{align}
\partial_\mu \partial_x p(\mu, s, t, x, z)(v) &  =  \partial_\mu [\partial_x \widehat{p}(\mu, s, t, x, z)](v) + (\partial_x p \otimes \partial_\mu \mH(.)(v))(\mu, s, t, x, z)  \notag \\
& + ((\partial_\mu [\partial_x \widehat{p}](.)(v) + \partial_x p \otimes \partial_\mu \mH(.)(v))\otimes \Phi)(\mu, s, t, x, z). \label{identity:partial:mes:partial:space:p}
\end{align}

\noindent \emph{Step 2: proof of the estimate \eqref{sensitivity:mes:deriv:cross:space:mes:p}.}\\

 Starting from the decomposition \eqref{identity:partial:mes:partial:space:p}, we see that it suffices to investigate the H\"older regularity of each term with respect to the variables $\mu, x, v$. We first derive an estimate for the difference $\partial_\mu \partial_x \widehat{p}(\mu, s, t, x, z)(v) -  \partial_\mu \partial_x \widehat{p}(\mu', s, t, x', z)(v')$. 

We split the computations into the two disjoint cases $W_2(\mu, \mu') + |v-v'| + |x-x'| \leq (t-s)^{1/2}$ and $W_2(\mu, \mu') + |v-v'| + |x-x'| > (t-s)^{1/2}$. In the first case, from \eqref{diff:mes:v:partial:deriv:diff:coeff}, \eqref{diff:mes:drift:diff:coefficients:true:mes}, \eqref{deriv:mes:a}, for any $\beta \in [0,1]$, we get
\begin{align*}
\Big | \partial_\mu & \Big[H_1\Big(\int_s^t a(r, y, [X^{s, \xi}_{r}])\, dr, z-x\Big)\Big](v) - \partial_\mu \Big[H_1\Big(\int_s^t a(r, y, [X^{s, \xi}_{r}])\, dr, z-x'\Big)\Big]_{|\mu=\mu'}(v') \Big| \\
& \leq K^+_\beta \left\{ \frac{|z-x|}{(t-s)^{1+ \frac{1+\beta-\eta}{2}}} [W_2(\mu, \mu')^{\beta} + |v-v'|^\beta] + \frac{|x-x'|^\beta}{(t-s)^{1+\frac{\beta-\eta}{2}}} \right\}
\end{align*}
\noindent and
\begin{align*}
\Big |  & H_1\Big(\int_s^t a(r, y, [X^{s, \xi}_{r}])\, dr, z-x\Big) -  H_1\Big(\int_s^t a(r, y, [X^{s, \xi'}_{r}])\, dr, z-x'\Big) \Big| \\
& \leq K \left\{ \frac{|z-x|}{(t-s)^{1+ \frac{\beta}{2}}} W_2(\mu, \mu')^{\beta}  + \frac{|x-x'|^\beta}{(t-s)^{\frac{1+\beta}{2}}} \right\}
\end{align*}

\noindent so that using the space time inequality \eqref{space:time:inequality}
\begin{align*}
\Big | \partial_\mu & \Big[H_1\Big(\int_s^t a(r, y, [X^{s, \xi}_{r}])\, dr, z-x\Big)\Big](v) - \partial_\mu \Big[H_1\Big(\int_s^t a(r, y, [X^{s, \xi}_{r}])\, dr, z-x'\Big)\Big]_{\mu=\mu'}(v') \Big|  \widehat{p}^{y}(\mu, s, t, x, z)\\
& \leq K^{+}_{\beta} \frac{[W_2(\mu, \mu')^{\beta} + |v-v'|^\beta + |x-x'|^\beta]}{(t-s)^{1+\frac{\beta-\eta}{2}}} \, g(c(t-s), z-x).
\end{align*}
\noindent and
\begin{align*}
\Big |  \Big[ & H_1\Big(\int_s^t a(r, y, [X^{s, \xi}_{r}])\, dr, z-x\Big) -  H_1\Big(\int_s^t a(r, y, [X^{s, \xi'}_{r}])\, dr, z-x'\Big)\Big]  \partial_\mu \widehat{p}^{y}(\mu, s, t, x, z)(v) \Big|  \\
& \leq K \left\{ \frac{|z-x|}{(t-s)^{1+ \frac{\beta}{2}}} W_2(\mu, \mu')^{\beta}  + \frac{|x-x'|^\beta}{(t-s)^{\frac{1+\beta}{2}}} \right\} \frac{1}{(t-s)^{\frac{1-\eta}{2}}} \, g(c(t-s), z-x) \\
& \leq K \frac{[W_2(\mu, \mu')^{\beta} + |x-x'|^\beta]}{(t-s)^{1+\frac{\beta-\eta}{2}}} \, g(c(t-s), z-x)
\end{align*}
\noindent where we used \eqref{deriv:v:deriv:mes:p:hat} (with $n=0$ and $r=s$) and the space time inequality \eqref{space:time:inequality} for the last but one inequality.

Also, it directly follows from \eqref{deriv:mes:a}, \eqref{holder:reg:deriv:p:hat}, \eqref{gaussian:bound:diff:deriv:hat:pm:same:time} (both with $n=0$), the inequality $|z-x| \leq |z-x'| + |x-x'|\leq |z-x'| + (t-s)^{1/2}$ and the space time inequality \eqref{space:time:inequality} that
\begin{align*}
\Big|  \partial_\mu & \Big[H_1\Big(\int_s^t a(r, y, [X^{s, \xi}_{r}])\, dr, z-x'\Big)\Big]_{\mu=\mu'}(v') [\widehat{p}^{y}(\mu, s,  t, x, z) - \widehat{p}^{y}(\mu', s,  t, x', z)] \Big| \\
& \leq K_\beta |z-x|  \frac{[W_2(\mu, \mu')^{\beta} + |x-x'|^\beta]}{(t-s)^{1+\frac{1+\beta-\eta}{2}}} \, \left\{ g(c(t-s), z-x) + g(c(t-s), z-x') \right\} \\
& \leq K_\beta \frac{[W_2(\mu, \mu')^{\beta} + |x-x'|^\beta]}{(t-s)^{1+\frac{\beta-\eta}{2}}} \, \left\{ g(c(t-s), z-x) + g(c(t-s), z-x') \right\}
\end{align*}

\noindent and, from the space time inequality \eqref{space:time:inequality}, \eqref{diff:space:mes:v:partial:mes:phat} and the inequality $|z-x'| \leq |z-x| + |x-x'|\leq |z-x|+ (t-s)^{1/2}$, one also obtains
\begin{align*}
\Big| & H_1\Big(\int_s^t a(r, y, [X^{s, \xi'}_{r}])\, dr, z-x' \Big)  [\partial_\mu \widehat{p}^{y}(\mu, s, t, x, z)(v) - \partial_\mu \widehat{p}^{y}(\mu', s, t, x', z)(v')] \Big| \\
& \leq K_{\beta}^{+}\frac{|z-x'|}{(t-s)}\frac{[W_2(\mu, \mu')^\beta + |x-x'|^\beta + |v-v'|^\beta] }{(t-s)^{\frac{1+\beta-\eta}{2}}} \, \left\{ g(c(t-s), z-x) + g(c(t-s), z-x')\right\} \\
& \leq K_{\beta}^{+}\frac{[W_2(\mu, \mu')^\beta + |x-x'|^\beta + |v-v'|^\beta] }{(t-s)^{1+\frac{\beta-\eta}{2}}} \, \left\{ g(c(t-s), z-x) + g(c(t-s), z-x')\right\}.
\end{align*}

Coming back to the decomposition \eqref{decomposition:partial:mes:partial:space:phat} and using the previous estimates yield
\begin{align*}
\Big| \partial_\mu & [\partial_x \widehat{p}^{y}(\mu, s, r, t, x, z)](v) - \partial_\mu [\partial_x \widehat{p}^{y}(\mu, s, r, t, x', z)]_{\mu = \mu'}(v') \Big| \\
& \leq K_\beta^{+} \frac{[W_2(\mu, \mu')^\beta + |x-x'|^\beta + |v-v'|^\beta] }{(t-s)^{1+\frac{\beta-\eta}{2}}} \, \left\{ g(c(t-s), z-x) + g(c(t-s), z-x')\right\}
\end{align*}

\noindent for any $\beta \in [0,\eta)$ in the diagonal regime $W_2(\mu, \mu') + |v-v'| + |x-x'| \leq (t-s)^{1/2}$. 

In the off-diagonal regime $W_2(\mu, \mu') + |v-v'| + |x-x'| > (t-s)^{1/2}$, we directly use the estimate \eqref{estimate:partial:mes:partial:space:phat} 
\begin{align*}
\Big| \partial_\mu & [\partial_x \widehat{p}^{y}(\mu, s, r, t, x, z)](v) - \partial_\mu [\partial_x \widehat{p}^{y}(\mu, s, r, t, x', z)]_{\mu = \mu'}(v') \Big| \\
& \leq\Big| \partial_\mu  [\partial_x \widehat{p}^{y}(\mu, s, r, t, x, z)](v) \Big| + \Big| \partial_\mu [\partial_x \widehat{p}^{y}(\mu, s, r, t, x', z)]_{\mu = \mu'}(v') \Big| \\
& \leq K \frac{[W_2(\mu, \mu')^\beta + |x-x'|^\beta + |v-v'|^\beta] }{(t-s)^{1+\frac{\beta-\eta}{2}}} \, \left\{ g(c(t-s), z-x) + g(c(t-s), z-x')\right\}.
\end{align*}
\noindent for any $\beta \in [0,1]$.

Hence, for any $\beta \in [0,\eta)$, there exist some positive constants $K_\beta^{+}:=K(T, \HRp, \HE)$ and $c:=c(\lambda)$ such that for any $(\mu, \mu', s , x, x', v, v', y)\in (\pp)^2 \times [0,t) \times (\mathbb{R}^d)^5$
\begin{align}
\Big| \partial_\mu & [\partial_x \widehat{p}^{y}(\mu, s, r, t, x, z)](v) - \partial_\mu [\partial_x \widehat{p}^{y}(\mu, s, r, t, x', z)]_{\mu = \mu'}(v') \Big| \notag \\
& \leq K_\beta^{+} \frac{[W_2(\mu, \mu')^\beta + |x-x'|^\beta + |v-v'|^\beta] }{(t-s)^{1+\frac{\beta-\eta}{2}}} \, \left\{ g(c(t-s), z-x) + g(c(t-s), z-x')\right\}. \label{partial:mu:partial:mes:p:hat:reg:space:mes:v}
\end{align}

From \eqref{second:derivatives:holder:estimates:parametrix:series}, \eqref{estimate:partial:deriv:mes:parametrix:kernel:iterated}, separating the time integral of the space time convolution into the two disjoint intervals $[s, (t+s)/2)$ and $[(t+s)/2, t]$ in order to balance the time singularity induced by the two estimates, after some standard computations that we omit, we deduce that for any $\beta\in [0,\eta)$
\begin{align}
|(\partial_x p \otimes \partial_\mu \mH (.) (v))(\mu, s, t, x, z) & - (\partial_x p \otimes \partial_\mu \mH(.)(v))(\mu, s, t, x', z)| \notag\\
&  \leq K_\beta \frac{|x-x'|^\beta}{(t-s)^{1+\frac{\beta-\eta}{2}}} \, \left\{ g(c(t-s), z-x) + g(c(t-s), z-x') \right\}. \label{diff:space:partial:x:density:convol:partial:mes:parametrix:kernel}
\end{align}

Similarly, it follows from \eqref{bound:density:parametrix}, \eqref{diff:v:and:mes:partial:mes:parametrix:kernel} and similar computations that for any $\beta \in [0,\eta)$
\begin{align}
|(\partial_x p \otimes \partial_\mu \mH(.)(v))(\mu, s, t, x, z) & - (\partial_x p \otimes \partial_\mu \mH(.)(v'))(\mu, s, t, x, z)| \notag\\
&  \leq K_\beta^+ \frac{|v-v'|^\beta}{(t-s)^{1+\frac{\beta-\eta}{2}}} \, g(c(t-s), z-x). \label{diff:v:partial:x:density:convol:partial:mes:parametrix:kernel}
\end{align}
 
 Finally, from \eqref{regularity:measure:estimate:v1:v2:v3:mckean:decoupling}, \eqref{diff:v:and:mes:partial:mes:parametrix:kernel}, \eqref{bound:density:parametrix} and \eqref{estimate:partial:deriv:mes:parametrix:kernel:iterated} (with $k=1$ and $\beta$ sufficiently small), separating the time integral of the space time convolution into the two disjoint intervals $[s, (t+s)/2)$ and $[(t+s)/2, t]$ in order to balance the time singularity induced by \eqref{diff:v:and:mes:partial:mes:parametrix:kernel}, after some standard computations that we omit, we deduce that for any $\beta \in [0,\eta)$
 \begin{align}
|(\partial_x p \otimes \partial_\mu \mH(.)(v))(\mu, s, t, x, z) & - (\partial_x p \otimes \partial_\mu \mH(.)(v))(\mu', s, t, x, z)| \notag\\
&  \leq K_\beta^+ \frac{W_2(\mu, \mu')^\beta}{(t-s)^{1+\frac{\beta-\eta}{2}}} \, g(c(t-s), z-x). \label{diff:mes:partial:x:density:convol:partial:mes:parametrix:kernel}
\end{align}

Combining \eqref{diff:space:partial:x:density:convol:partial:mes:parametrix:kernel}, \eqref{diff:v:partial:x:density:convol:partial:mes:parametrix:kernel} and \eqref{diff:mes:partial:x:density:convol:partial:mes:parametrix:kernel}, we thus obtain
 \begin{align}
|(\partial_x p \otimes \partial_\mu & \mH(.)(v))(\mu, s, t, x, z)  - (\partial_x p \otimes \partial_\mu \mH(.)(v'))(\mu', s, t, x', z)| \notag\\
&  \leq K_\beta^+ \frac{[W_2(\mu, \mu')^\beta + |v-v'|^\beta + |x-x'|^\beta]}{(t-s)^{1+\frac{\beta-\eta}{2}}} \, \left\{ g(c(t-s), z-x) + g(c(t-s), z-x') \right\} \label{diff:mes:space:v:partial:x:density:convol:partial:mes:parametrix:kernel}
\end{align}
\noindent which combined with \eqref{partial:mu:partial:mes:p:hat:reg:space:mes:v}, \eqref{Gaussian:estimate:Phi} and then \eqref{estimate:partial:mes:partial:space:phat}, \eqref{partialx:p:convol:partialmes:parametrix:kernel} and \eqref{bound:diff:Phi:mu} eventually yield
\begin{align}
\Big| ((\partial_\mu [\partial_x \widehat{p}](.)(v) & + \partial_x p \otimes \partial_\mu \mH(.)(v))\otimes \Phi)(\mu, s, t, x, z) - ((\partial_\mu [\partial_x \widehat{p}](.)(v') + \partial_x p \otimes \partial_\mu \mH(.)(v'))\otimes \Phi)(\mu', s, t, x', z) \Big| \notag \\
& \leq K_\beta^+ \frac{[W_2(\mu, \mu')^\beta + |v-v'|^\beta + |x-x'|^\beta]}{(t-s)^{1+\frac{\beta}{2}-\eta}} \, \left\{ g(c(t-s), z-x) + g(c(t-s), z-x') \right\} \label{partial:phat:convol:partial:parametrix:kernel:convol:Phi}
\end{align}

\noindent for any $\beta \in [0,\eta)$.
The identity \eqref{identity:partial:mes:partial:space:p} together with \eqref{partial:mu:partial:mes:p:hat:reg:space:mes:v}, \eqref{diff:mes:space:v:partial:x:density:convol:partial:mes:parametrix:kernel} and \eqref{partial:phat:convol:partial:parametrix:kernel:convol:Phi} allow to conclude the proof of \eqref{sensitivity:mes:deriv:cross:space:mes:p}, recalling that $\partial_\mu [\partial_x p(\mu, s, t, x, z)](v) = [\partial_x [\partial_\mu p(\mu, s, t, x, z)(v)]]^{*}$.

\section{Proof of Proposition \ref{proposition:reg:density:recursive:scheme:mckean}}\label{proof:main:prop} 

This section is dedicated to the proof of Proposition \ref{proposition:reg:density:recursive:scheme:mckean}. The strategy of proof is the same as the one developed for Proposition 5.1 in \cite{chaudruraynal:frikha}. In order to foster the understanding on the main steps of the proof, we will collect intermediate technical results into several auxiliary lemmas and associated corollaries which are based on standard but cumbersome Gaussian like computations and postpone their proof to \ref{sec:proof:technical:estimates}. The reader could skip some of these derivations in a first reading. \\

This section is organized as follows: in \ref{subsubsection:base:case}, we deal with the base case in a completely analogous manner to the base case $m=1$ of Proposition 5.1 in \cite{chaudruraynal:frikha}. We provide it for sake of completeness. The regularity of the maps $x\mapsto \partial_\mu p_m(\mu, s, t, x, z)$ and $\mu \mapsto \partial_x p_m(\mu, s, t, x, z)$ and the related estimates \eqref{cross:deriv:mes:space:induction:decoupling:mckean}, \eqref{sensitivity:mes:deriv:cross:space:mes:pm} and \eqref{sensitivity:time:deriv:cross:space:mes:pm} are obtained as a consequence of the results established in our previous work \cite{chaudruraynal:frikha} and are thus tackled in \ref{subsubsection:regularity:derivatives:heat:kernel:approx}. In particular, in Lemma \ref{lem:collect:technical:estimates}, we will recall some important technical estimates established in \cite{chaudruraynal:frikha} that will be used in our analysis. In \ref{sec:technical:results:first:part}, we provide some technical results which are necessary to address the proof of the first part of the induction step. Then, the first part of the induction step, namely, the $\mathcal{C}^{1, 2, 2}_f([0,t) \times \mathbb{R}^d \times \pp)$ regularity of the map $(s, x, \mu) \mapsto p_{m+1}(\mu, s, t, x, z)$ and the proof of the estimates \eqref{second:deriv:mes:induction:decoupling:mckean}, \eqref{second:deriv:mes:reg:space:deriv:arg:estimate:induction:decoupling:mckean} and \eqref{second:deriv:mes:reg:space:deriv:arg:estimate:induction:decoupling:mckean2} at step $m+1$, is treated in \ref{sec:first:part:induction:step}. We eventually address the second part of the induction step, namely the estimates \eqref{second:deriv:mes:reg:mes:estimate:induction:decoupling:mckean} and \eqref{regularity:time:estimate:secon:deriv:decoupling:mckean} at step $m+1$ in \ref{sec:proof:second:part:induction:step}. \\

Apart from \ref{subsubsection:base:case} and \ref{subsubsection:regularity:derivatives:heat:kernel:approx}, we will work under the following assumption. For a fixed positive time horizon $T>0$ and positive integer $m$, we assume that for any fixed $(t, z) \in (0,T] \times \mathbb{R}^d$, the map $(s, x, \mu) \mapsto p_m(\mu, s, t, x, z)$ defined by \eqref{series:approx:mckean} belongs to $\mathcal{C}^{1,2, 2}_f([0,t) \times \mathbb{R}^d \times \pp)$ and denote by $[X^{s, \xi, (m)}_t]$ the probability measure on $\mathbb{R}^d$ with density function $z\mapsto (p_m(\mu, t, T, ., z) \sharp\mu)$.\\

\noindent\textbf{Notations.} We recall some notations that will be used in this section. As already mentioned before, we denote by $K := K(T, \HR, \HE)$ a positive constant depending on $T$, $a$, $b$, $\delta a/\delta m$, $\delta b/\delta m$ and $\lambda$. We also denote by $K^+ := K(T, \HRp, \HE)$ and $K^{++} := K(T, \HRpp, \HE)$ two positive constants depending only upon $T$ and the parameters appearing in \A{HR}$_{+}$, \A{HE} and \A{HR}$_{++}$ and \A{HE} respectively. If $\beta$ is any (other) parameter, we denote by $K_\beta$,  $K_{\beta}^+$ or $K^{++}_\beta$ a positive constant depending on $\beta$ and the corresponding aforementioned parameters. We also denote by $c:=c(\lambda)>0$ a constant depending only on the parameter $\lambda$ in \A{HE}. As usual, all these constants may vary from line to line.\\

\subsection{Base case $m=1$.\\ }\label{subsubsection:base:case}

As far as the base case $m=1$ is concerned, as already underlined in \cite{chaudruraynal:frikha}, since $P^{(0)}(t) = \nu$ for any $t\geq s$, it is readily seen from \eqref{series:approx:mckean}, \eqref{eq:def:de:phat:m}, \eqref{eq:def:de:mH} for $m=1$ that the law argument in the coefficients depends neither on the initial measure $\mu$ nor on the initial time $s$ but only on $\nu$.  It thus follows from \cite{friedman:64} that the map $[0,t) \times \mathbb{R}^d \ni (s, x) \mapsto p_1(\mu, s, t, x, z) =  \sum_{k \geq 0} (\widehat{p}_{1} \otimes \mH^{(k)}_1)(\mu, s, t, x, z)$ belongs to $\mathcal{C}^{1, 2}([0,t) \times \mathbb{R}^d)$ with derivatives that do not depend on $\mu$. Obviously, the map $\pp \ni \mu \mapsto p_1(\mu, s, t, x, z)$ is two times continuously L-differentiable and satisfies $\partial_\mu p_1(\mu, s, t, x, z)(v) = \partial_v [\partial_\mu p_1(\mu, s, t, x, z)](v) = \partial^2_\mu p_1(\mu, s, t, x, z)(v,v') = 0 $ for any $(s, x ,\mu, v, v') \in [0,t) \times \mathbb{R}^d \times \pp \times (\mathbb{R}^d)^2$. We thus conclude that the map $[0,t) \times \mathbb{R}^d \times \pp \ni (s, x, \mu) \mapsto p_1(\mu, s, t, x, z)$ is in $\mathcal{C}_f^{1, 2, 2}([0,t)\times \mathbb{R}^d \times \pp)$. The estimates \eqref{cross:deriv:mes:space:induction:decoupling:mckean} up to \eqref{regularity:time:estimate:secon:deriv:decoupling:mckean} are straightforward since $\partial_x [\partial_\mu p_1(\mu, s, t, x, z)(v)] = \partial^2_\mu p_1(\mu, s, t, x, z)(v,v') = 0$. \\

\subsection{On the regularity of the maps $x\mapsto \partial_\mu p_m(\mu, s, t, x, z)$ and $\mu \mapsto \partial_x p_m(\mu, s, t, x, z)$ and the related estimates \eqref{cross:deriv:mes:space:induction:decoupling:mckean}, \eqref{sensitivity:mes:deriv:cross:space:mes:pm} and \eqref{sensitivity:time:deriv:cross:space:mes:pm}.\\} \label{subsubsection:regularity:derivatives:heat:kernel:approx}
The proof of the continuous differentiability of the maps $x\mapsto \partial_\mu p_m(\mu, s, t, x, z)$ and $\mu \mapsto \partial_x p_m(\mu, s, t, x, z)$ and the continuity of their respective derivatives with respect to the variables $x, \mu, v$ as well as the estimates \eqref{cross:deriv:mes:space:induction:decoupling:mckean} and \eqref{sensitivity:mes:deriv:cross:space:mes:pm} is similar to the proof of Lemma \ref{additional:regularity:lemma}. One just copies verbatim the proof except that one has to replace $p$, $\widehat{p}$, $\mH$, $\Phi$, etc... by their respective approximation sequences $(p_m)_{m\geq1}$, $(\widehat{p}_m)_{m\geq1}$, $(\mH_m)_{m\geq1}$, $(\Phi_m)_{m\geq1}$, etc ... and use the corresponding estimates with constants being uniform in $m$. In particular, the identity \eqref{identity:partial:mes:partial:space:p} here writes 
\begin{align}
\partial_\mu \partial_x p_m(\mu, s, t, x, z)(v) &  =  \partial_\mu [ \partial_x \widehat{p}_m(\mu, s, t, x, z)](v) + (\partial_x p_m \otimes \partial_\mu \mH_m(.)(v))(\mu, s, t, x, z)  \notag \\
& + ((\partial_\mu [\partial_x \widehat{p}_m](.)(v) + \partial_x p_m \otimes \partial_\mu \mH_m(.)(v))\otimes \Phi_m)(\mu, s, t, x, z) \label{identity:partial:mes:partial:space:pm}
\end{align}

\noindent and the estimates \eqref{estimate:partial:mes:partial:space:phat}, \eqref{deriv:v:deriv:mes:p:hat}, \eqref{estimate:partial:deriv:mes:parametrix:kernel:iterated}, \eqref{estimate:deriv:mes:Phi} and \eqref{deriv:mes:a} become 
\begin{align}
|\partial_\mu [\partial_x \widehat{p}^{y}_m(\mu, s, t, x, z)](v)| & \leq \frac{K}{(t-s)^{1-\frac{\eta}{2}}} \, g(c(t-s), z-x), \label{estimate:partial:mes:partial:space:pmhat}
\end{align}
\begin{align}
| \partial^{n}_v  [\partial_\mu \widehat{p}^{y}_m(\mu, s, r, t, x, z)](v)|  \leq \frac{K}{t-r}  \int_r^t \frac{1}{(r'-s)^{\frac{1+n - \eta}{2}}} \,dr'  g(c(t-r), z-x), \label{deriv:v:deriv:mes:pm:hat}
\end{align}
\noindent for any $\beta \in (0,1]$ and any positive integer $k$
\begin{align}
| \partial^{n}_v [& \partial_\mu[\mH^{(k)}_m(\mu, s, r, t, x, z)]](v) | \notag \\
& \leq \frac{k K^k_\beta}{(r-s)^{\frac{1+n-(1-\beta)\eta}{2}}(t-r)^{1-\beta \frac{\eta}{2} - (k-1)\frac{\eta}{2}}} \prod_{\ell=1}^{k-1}B(\frac{\eta}{4}, \frac{\eta}{2}\beta+ (\ell-1) \frac{\eta}{2}) \, g(c(t-r), z-x), \label{estimate:partial:deriv:mes:parametrix:kernel:iterated:m}
\end{align}  
\begin{align}
| \partial^n_v [\partial_\mu [\Phi_m(\mu, s, r, t, x, z)]](v)| \leq \frac{K_\beta}{(r-s)^{\frac{1+n-(1-\beta)\eta}{2}}(t-r)^{1-\beta \frac{\eta}{2}}} \, g(c(t-r), z-x), \label{estimate:deriv:mes:Phim}
\end{align}
\noindent and
\begin{align}
| \partial_\mu [b_{i}(t, x, [X^{s,\xi, (m)}_t])](v)| + | \partial_\mu [a_{i, j}(t, x, [X^{s,\xi, (m)}_t])](v)| \leq \frac{K}{(t-s)^{\frac{1-\eta}{2}}}. \label{deriv:mes:a:m}
\end{align}

In particular, it follows from \eqref{bound:derivative:heat:kernel} (with $n=1$) and \eqref{estimate:partial:deriv:mes:parametrix:kernel:iterated:m} (with $n=0$, $k=1$ and $\beta \in (0,1)$) that
\begin{equation}
\label{}
|(\partial_x p_m \otimes \partial_\mu \mH_m(.)(v))(\mu, s, t, x, z)| \leq  \frac{K}{(t-s)^{1-\frac{\eta}{2}}} \, g(c(t-s), z-x). \label{convolution:partial:space:pm:partial:mes:parametrix:kernel:m}
\end{equation}

The above estimate will be used in the sequel.

We thus only prove the estimate \eqref{sensitivity:time:deriv:cross:space:mes:pm}. Before proceeding with the proof, we recall some important technical estimates established in \cite{chaudruraynal:frikha} that will be used in the sequel.

\begin{lem}\label{lem:collect:technical:estimates} For any $\beta \in [0,(1+\eta)/2)$ and any positive integer $m$, there exist positive constants $K$, $K_\beta^{+}$, $c$ such that for any $(t, z) \in (0,T] \times \mathbb{R}^d$, any $\mu, \mu' \in \pp$ (denoting by $\xi$ and $\xi'$ any random variables with respective law $\mu$ and $\mu'$), any $s, s_1, s_2 \in [0,t)$, any $r'\in (s_1 \vee s_2, t)$, any $x, y, v, v_1, v_2 \in \mathbb{R}^d$ and any $(i, j) \in \left\{1, \cdots, d\right\}^2$
\begin{align}
\Big| \partial_\mu [\partial_x&  \widehat{p}^{y}_m(\mu, s_1, t, x, z)](v)  -  \partial_\mu [\partial_x \widehat{p}^{y}_m(\mu, s_2, t, x, z)](v) \Big| \notag \\
& \leq K_\beta^{+} \left\{ \frac{|s_1-s_2|^{\beta}}{(t-s_1)^{1-\frac{\eta}{2}+\beta}} g(c(t-s_1), z-x) +  \frac{|s_1-s_2|^{\beta}}{(t-s_2)^{1-\frac{\eta}{2}+\beta}} g(c(t-s_2), z-x) \right\}, \label{diff:time:deriv:mes:deriv:space:p:hat:m}
\end{align}

\begin{align}
 | & \partial_x p_{m}(\mu, s_1, t, x, z) - \partial_x p_{m}(\mu, s_2, t, x, z)  | \nonumber \\
 & \leq K_\beta \left\{ \frac{|s_1-s_2|^{\beta}}{(t-s_1)^{\frac{1}{2} + \beta }} \, g(c(t-s_1), z-x) + \frac{|s_1-s_2|^{\beta}}{(t- s_2)^{ \frac{1}{2} +\beta }} \, g(c(t-s_2), z-x) \right\}, \label{regularity:time:estimate:v1:v2:v3:decoupling:mckean:prop:statement} 
\end{align}

 \begin{align}
  \Big|&  \partial_\mu \widehat{p}^y_{m}(\mu, s_1, t, x, z)(v) -  \partial_\mu \widehat{p}^y_{m}(\mu, s_2, t, x, z)(v) \Big| \notag \\
  & \leq K^{+}_\beta \left\{ \frac{|s_1-s_2|^{\beta}}{(t-s_1)^{\frac{1-\eta}{2}+\beta}} g(c(t-s_1), z-x) +  \frac{|s_1-s_2|^{\beta}}{(t-s_2)^{\frac{1-\eta}{2}+\beta}} g(c(t-s_2), z-x) \right\}, \label{diff:time:L:deriv:p:hat:same:time} 
 \end{align}
 
  \begin{align}
  \Big|&  \partial_\mu \widehat{p}^y_{m}(\mu, s_1, r', t, x, z)(v) -  \partial_\mu \widehat{p}^y_{m}(\mu, s_2, r',  t, x, z)(v) \Big| \notag \\
  & \leq K^{+}_\beta \frac{|s_1-s_2|^{\beta}}{(r'-s_1\vee s_2)^{\frac{1-\eta}{2}+\beta}} g(c(t-r'), z-x), \label{diff:time:L:deriv:p:hat:different:time} 
 \end{align}
 
\begin{align}
 & \Big|  \partial_\mu \mH_{m}(\mu, s_1, r', t, x, z)(v) - \partial_\mu \mH_{m}(\mu, s_2, r', t, x, z)(v) \Big| \nonumber \\
 & \leq K^{+}_\beta  |s_1-s_2|^\beta \left\{ \frac{1}{(t-r')(r'-s_1\vee s_2)^{\frac{1-\eta}{2}+\beta}} \wedge \frac{1}{(t-r')^{1- \frac{\eta}{2}}(r'-s_1 \vee s_2)^{\frac{1}{2} + \beta}} \right\} \label{diff:time:L:deriv:parametrix:kernel:pmp1} \\
 & \quad \quad  \times g(c(t-r'), z-x), \nonumber
 \end{align}
 
\noindent for any $\beta \in [0,1]$

\begin{align}
|a_{i, j}(t, x, [X^{s_1 , \xi, (m)}_t])  - & a_{i, j}(t, x, [X^{s_2, \xi, (m)}_t])| + |b_{i}(t, x, [X^{s_1, \xi, (m)}_t])  - b_{i}(t, x, [X^{s_2, \xi, (m)}_t])| \notag \\
& \leq K  \left\{ \frac{|s_1-s_2|^\beta}{(t-s_1)^{\beta-\frac{\eta}{2}}} +\frac{|s_1-s_2|^\beta}{(t-s_2)^{\beta-\frac{\eta}{2}}}  \right\}, \label{diff:time:drift:diff:coefficients}
\end{align}

\begin{align}
 \Big| \partial^{n}_x \widehat{p}_{m}(\mu, s_1, r', t, x, z) & - \partial^{n}_x \widehat{p}_{m}(\mu, s_2, r', t, x, z)\Big| \notag \\
  & \leq K \frac{|s_1-s_2|^\beta}{(t-r')^{\frac{n}{2}}(r'-s_1\vee s_2)^{\beta}} \, g(c(t-r'), z-x), \quad n\in \left\{0,1, 2, 3\right\}, \label{gaussian:bound:diff:time:hat:pm:different:time}
\end{align}

\begin{align}
 |  \Phi_{m}(\mu, s_1, r' ,t , x ,z) -  \Phi_{m}(\mu, s_2, r' ,t , x ,z) |  \leq K  \frac{|s_1-s_2|^{\beta}}{(t-r')^{1-\frac{\eta}{2}}(r'-s_1\vee s_2)^{\beta}}  \, g(c(t-r'), z-x), \label{gaussian:bound:diff:time:Phim}
\end{align}

\begin{align}
 \Big| &  \partial^{n}_x \widehat{p}^y_{m}(\mu, s_1, t, x, z) - \partial^{n}_x \widehat{p}^y_{m}(\mu, s_2, t, x, z) \Big|  \notag \\
 & \leq K \left\{ \frac{|s_1-s_2|^\beta}{(t-s_1)^{\frac{n}{2}+\beta}} \, g(c(t-s_1), z-x) + \frac{|s_1-s_2|^\beta}{(t-s_2)^{\frac{n}{2}+\beta}} \, g(c(t-s_2), z-x) \right\}, \quad n\in \left\{0,1, 2\right\}, \label{gaussian:bound:diff:time:hat:pm:same:time}
\end{align}

\begin{align}
|  & \partial^{n}_v [\partial_\mu [a_{i, j}(t, x, [X^{s, \xi, (m)}_{t}]) - a_{i, j}(t, z, [X^{s, \xi, (m)}_t])]](v) |  \leq  K_{\beta} \frac{|z-x|^{\beta \eta}}{(t-s)^{\frac{1+n-(1-\beta)\eta}{2}}}, \quad n\in \left\{0,1\right\},  \label{recursive:bound:deriv:mes:holder:reg:a}
\end{align}

\noindent for any $\beta \in [\eta,1]$
\begin{align}
|a_{i, j}(t, x, [X^{s , \xi, (m)}_t])  - & a_{i, j}(t, x, [X^{s, \xi', (m)}_t])| + |b_{i}(t, x, [X^{s, \xi, (m)}_t])  - b_{i}(t, x, [X^{s, \xi', (m)}_t])| \notag \\
& \leq K \frac{W_2(\mu, \mu')^\beta}{(t-s)^{\frac{\beta-\eta}{2}}}, \label{diff:mes:drift:diff:coefficients}
\end{align}

\noindent for any $\beta \in [0,1]$ and any $r\in [s, t)$

\begin{align}
  |& \partial_\mu \widehat{p}^{y}_{m}(\mu, s, r, t, x, z)(v) - \partial_\mu \widehat{p}^{y}_{m}(\mu', s, r, t, x, z)(v)  | \notag \\
  & \leq  K^{+}_\beta   W_2(\mu, \mu')^{\beta}\left( \frac{1}{(t-s)^{ \frac{1+\beta-\eta}{2}}} \I_\seq{r=s} + \frac{1}{(r-s)^{ \frac{1+\beta-\eta}{2}}} \I_\seq{r>s} \right)\, g(c(t-r), z-x),  \label{diff:L:deriv:pm:mu:mup}
\end{align}

\noindent for any $\alpha \in [0,\eta]$ and any $\beta \in [\alpha,1]$
\begin{align}
|a_{i, j}& (t, x, [X^{s, \xi, (m)}_t])  -  a_{i, j}(t, x, [X^{s, \xi', (m)}_t]) - (a_{i, j}(t, z, [X^{s, \xi, (m)}_t])  - a_{i, j}(t, z, [X^{s, \xi', (m)}_t]))| \notag \\
 & + |b_{i}(t, x, [X^{s, \xi, (m)}_t])  - b_{i}(t, x, [X^{s, \xi', (m)}_t]) - (b_{i}(t, z, [X^{s, \xi, (m)}_t])  - b_{i}(t, z, [X^{s, \xi', (m)}_t]))| \notag \\
& \quad \quad \leq  K ( |x-z|^{\eta-\alpha} \wedge 1)\frac{W_2(\mu, \mu')^{\beta}}{(t-s)^{\frac{\beta-\alpha}{2}}},  \label{diff:mes:with:holder:reg:space:drift:diff:coefficients}
\end{align}

\noindent for any $\alpha \in [0,\eta]$ and any $\beta \in [0,1]$
\begin{align}
|a_{i, j}& (t, x, [X^{s_1 , \xi, (m)}_t])  -  a_{i, j}(t, x, [X^{ s_2, \xi, (m)}_t]) - (a_{i, j}(t, z, [X^{s_1 , \xi, (m)}_t])  - a_{i, j}(t, z, [X^{ s_2, \xi, (m)}_t]))| \notag \\
\quad + |b_{i}& (t, x, [X^{s_1 , \xi, (m)}_t])  -  b_{i}(t, x, [X^{ s_2, \xi, (m)}_t]) - (b_{i}(t, z, [X^{s_1 , \xi, (m)}_t])  - b_{i}(t, z, [X^{ s_2, \xi, (m)}_t]))| \notag \\
& \quad \quad \leq   K_{\alpha} (|z-x|^\alpha \wedge 1) |s_1-s_2|^\beta \left\{\frac{1}{(t-s_1)^{\beta+\frac{\alpha-\eta}{2}}} + \frac{1}{(t-s_2)^{\beta+\frac{\alpha-\eta}{2}}}\right\},  \label{diff:time:with:holder:reg:space:drift:diff:coefficients}
\end{align} 

\noindent for any $\beta \in [0,1]$ and any $r\in (s, t)$
\begin{align}
|& \partial_\mu [a_{i, j} (t, x, [X^{s, \xi, (m)}_t])  -  a_{i, j}(t, z, [X^{s, \xi, (m)}_t])](v) \notag\\
& \quad - \partial_\mu [a_{i, j} (t, x, [X^{s, \xi, (m)}_t])  -  a_{i, j}(t, z, [X^{s, \xi, (m)}_t])]_{|\mu=\mu'}(v) | \label{diff:mes:L:deriv:diff:diff:coeff:holder:reg} \\
& \quad \leq  K^{+}_\beta W_2(\mu, \mu')^{\beta} \left\{ \frac{(|z-x|^\eta\wedge 1)}{(t-s)^{\frac{1+\beta}{2}}} \wedge \frac{1}{(t-s)^{\frac{1+\beta-\eta}{2}}} \right\}, \notag
\end{align}

\begin{align}
 \Big| &  \partial^{n}_x \widehat{p}^y_{m}(\mu, s, r,  t, x, z) - \partial^{n}_x \widehat{p}^y_{m}(\mu', s, r, t, x, z) \Big|  \notag \\
 & \leq K \frac{W_2(\mu, \mu')^\beta}{(t-r)^{\frac{n}{2}}(r-s)^{\frac{\beta}{2}}} \, g(c(t-r), z-x) , \quad n\in \left\{0,1, 2\right\}, \label{gaussian:bound:diff:mes:hat:pm:same:time}
\end{align}

\begin{align}
 & | \partial_\mu \mH_{m+1}(\mu, s, r, t, x, z)(v) - \partial_\mu \mH_{m+1}(\mu', s, r, t, x, z)(v)| \nonumber \\
 & \leq K^{+}_\beta  W_2(\mu, \mu')^{\beta} \left\{ \frac{1}{(t-r)(r-s)^{\frac{1+\beta-\eta}{2}}} \wedge \frac{1}{(t-r)^{1- \frac{\eta}{2}}(r-s)^{\frac{1+\beta}{2}}} \right\}   \label{diff:mes:first:L:deriv:parametrix:kernel:pmp1} \\
 & \quad \quad  \times g(c(t-r), z-x),\nonumber
 \end{align}
 
 \begin{align}
 |  \Phi_{m}(\mu, s, r ,t , x ,z) -  \Phi_{m}(\mu', s,  r ,t , x ,z) |  \leq K  \frac{W_2(\mu, \mu')^{\beta}}{(t-r)^{1-\frac{\eta}{2}}(r-s)^{\frac{\beta}{2}}}  \, g(c(t-r), z-x), \label{gaussian:bound:diff:mes:Phim}
\end{align}

\noindent for any $\beta \in [0,(1+\eta)/2)$
\begin{align}
|& \partial_\mu [a_{i, j} (t, x, [X^{s_1, \xi, (m)}_t])  -  a_{i, j}(t, z, [X^{s_1, \xi, (m)}_t])](v) - \partial_\mu [a_{i, j} (t, x, [X^{s_2, \xi, (m)}_t])  -  a_{i, j}(t, z, [X^{s_2, \xi, (m)}_t])](v) | \notag \\
& \leq  K^{+}_\beta |s_1-s_2|^\beta \left\{ \frac{(|z-x|^\eta\wedge 1)}{(t-s_1 \vee s_2)^{\frac{1}{2}+\beta}} \wedge \frac{1}{(t-s_1 \vee s_2)^{\frac{1-\eta}{2} + \beta }} \right\}, \label{diff:time:first:L:deriv:diff:diff:coeff:holder:reg}
\end{align}

\noindent for any $\beta \in [0,1]$ if $n=0$ or any $\beta \in [0,\eta)$ if $n=1$ and for any $r\in [s, t)$
\begin{align}
& | \partial^{n}_v [\partial_\mu [a_{i, j}(t, x, [X^{s ,\xi, (m)}_t])]](v_1) -   \partial^{n}_v [\partial_\mu [a_{i, j}(t, x, [X^{s, \xi, (m)}_t])]]_{\mu=\mu'}(v_2)| \notag \\
& \quad \quad + | \partial^{n}_v [\partial_\mu [b_{i}(t, x, [X^{s, \xi , (m)}_t])]](v_1) -   \partial^{n}_v [\partial_\mu [b_{i}(t, x, [X^{s, \xi, (m)}_t])]]_{\mu=\mu'}(v_2)|  \label{diff:v:cross:deriv:diff:and:deriv:coeff} \\
& \quad \quad \quad \leq K^+_\beta \left\{|v_1-v_2|^\beta + W_2(\mu, \mu')^\beta\right\} \frac{1}{(t-s)^{\frac{1+n+\beta-\eta}{2}}}, \notag
\end{align}

\begin{align}
 \Big| &  \partial^{n}_v[\partial_\mu \widehat{p}^y_{m}(\mu, s, r,  t, x, z)](v_1) - \partial^{n}_v[ \partial_\mu \widehat{p}^{y}_{m}(\mu, s, r,  t, x, z)(v_2)] \Big|  \notag \\
 & \leq K \frac{|v_1-v_2|^\beta}{t-r} \int_r^t \frac{1}{(r'-s)^{\frac{1+n+\beta-\eta}{2}}} \,dr' \, g(c(t-r), z-x) ,  \label{gaussian:bound:diff:v:deriv:mes:hat:pm}
\end{align}

\noindent and for any $\beta \in [0,(1+\eta)/2)$ if $n=0$ or any $\beta \in [0,\eta/2)$ if $n=1$
\begin{align}
& | \partial^{n}_v [\partial_\mu [a_{i, j}(t, x, [X^{s_1 ,\xi, (m)}_t])]](v) -   \partial^{n}_v [\partial_\mu [a_{i, j}(t, x, [X^{s_2, \xi, (m)}_t])]](v)| \notag \\
& \quad \quad + | \partial^{n}_v [\partial_\mu [b_{i}(t, x, [X^{s_1, \xi , (m)}_t])]](v) -   \partial^{n}_v [\partial_\mu [b_{i}(t, x, [X^{s_2, \xi, (m)}_t])]](v)|  \label{diff:time:cross:deriv:diff:and:deriv:coeff} \\
& \quad \quad \quad \leq K^{+}_\beta  \frac{|s_1-s_2|^\beta}{(t-s_1\vee s_2)^{\frac{1+n-\eta}{2}+\beta}}.\notag
\end{align}

\end{lem}

\begin{proof} The estimates \eqref{regularity:time:estimate:v1:v2:v3:decoupling:mckean:prop:statement} up to \eqref{diff:time:cross:deriv:diff:and:deriv:coeff} together with their proof are provided in \cite{chaudruraynal:frikha} so we only prove \eqref{diff:time:deriv:mes:deriv:space:p:hat:m}. Let us first observe that if $|s_1-s_2| \geq t-s_1 \vee s_2$ then \eqref{diff:time:deriv:mes:deriv:space:p:hat:m} directly follows from \eqref{estimate:partial:mes:partial:space:pmhat}. We thus assume that $|s_1-s_2| \leq t-s_1 \vee s_2$ for the rest of the proof. We make use of the decomposition
\begin{align}
\partial_\mu [\partial_x \widehat{p}^{y}_m(\mu, s, t, x, z)](v) & = \partial_\mu \Big[H_1\Big(\int_s^t a(r, y, [X^{s, \xi, m}_{r}])\, dr, z-x\Big)\Big](v) \, \widehat{p}^{y}_m(\mu, s, t, x, z) \label{decomposition:partial:mes:partial:space:pm:hat}\\
& \quad +  H_1\Big(\int_s^t a(r, y, [X^{s, \xi, m}_{r}])\, dr, z-x\Big) \, \partial_\mu \widehat{p}^{y}_m(\mu, s, t, x, z)(v)  \notag
\end{align}
\noindent which directly stems from \eqref{eq:def:de:phat:m}. From \eqref{diff:time:drift:diff:coefficients} and the uniform boundedness of $a$, for any $\beta \in [0, 1]$, we obtain 
\begin{align}
\Big| & \left(\int_{s_1}^t  a(r, y, [X^{s_1 , \xi, (m)}_{r}])\, dr\right)^{-1} - \left(\int_{s_2}^t a(r, y, [X^{s_2, \xi, (m)}_{r}])\, dr\right)^{-1} \Big| \notag \\
& \quad \leq \frac{K}{(t-s_1 \vee s_2)^2} \Big[ |s_1-s_2| + \int_{s_1 \vee s_2}^t \max_{i, j} |a_{i, j}(r, y, [X^{s_1, \xi, (m)}_{r}]) - a_{i, j}(r, y, [X^{ s_2, \xi, (m)}_{r}])| \, dr \Big]. \notag \\
& \quad \leq  \frac{K}{(t-s_1 \vee s_2)^2} \Big[ |s_1-s_2|^\beta (t-s_1\vee s_2)^{1-\beta} + \int_{s_1 \vee s_2}^t \frac{|s_1-s_2|^\beta}{(r-s_1 \vee s_2)^{(\beta-\frac{\eta}{2})_+}} \, dr \Big] \notag \\
%& \quad \leq  K \frac{|s_1-s_2|^\beta}{(t-s_1 \vee s_2)^2}   (t-s_1\vee s_2)^{1-\beta}  \notag \\
& \quad \leq K \frac{|s_1-s_2|^\beta}{(t-s_1\vee s_2)^{1+\beta}}\label{diff:time:inverse:a}
\end{align}

\noindent so that
\begin{align}
\Big| H_1\Big(\int_{s_1}^t a(r, y, [X^{s_1, \xi, m}_{r}])\, dr, z-x\Big)  &- H_1\Big(\int_{s_2}^t a(r, y, [X^{s_2, \xi, m}_{r}])\, dr, z-x\Big) \Big| \label{diff:time:first:order:hermite:polyn:m}\\
& \quad \leq K \frac{|s_1-s_2|^\beta}{(t-s_1\vee s_2)^{1+\beta}} |z-x|. \notag
\end{align}

Similarly, from \eqref{deriv:mes:a:m} and \eqref{diff:time:cross:deriv:diff:and:deriv:coeff}, one has
\begin{align}
\Big| \partial_\mu \Big[H_1\Big(\int_s^t a(r, y, [X^{s, \xi, m}_{r}])\, dr, z-x\Big)\Big](v) \Big| \leq K\frac{|z-x|}{(t-s)^{1+\frac{1-\eta}{2}}} \label{bound:L:deriv:H1}
\end{align}

\noindent and for any $\beta \in [0, (1+\eta)/2)$
\begin{align*}
\Big| &\int_{s_1}^{t} \partial_\mu [a_{i, j}(r, y, [X^{s_1 , \xi, (m)}_{r}])](v)\, dr - \int_{s_2}^t \partial_\mu[a_{i, j}(r, y, [X^{s_2 , \xi, (m)}_{r}])](v) \, dr \Big| \\
& \quad \leq K^{+}_{\beta} \Big[ \int_{s_1 \wedge s_2}^{s_1 \vee s_2} \frac{1}{(r-s_1\wedge s_2)^{\frac{1-\eta}{2}}} \, dr + \int_{s_1\vee s_2}^{t} \frac{|s_1-s_2|^\beta}{(r-s_1\vee s_2)^{\frac{1-\eta}{2}+\beta}}\,dr\Big]\\
& \quad \leq K^{+}_\beta |s_1-s_2|^\beta (t-s_1\vee s_2)^{\frac{1+\eta}{2}-\beta}.
\end{align*}

Combining \eqref{diff:time:inverse:a} with the previous estimates and using again \eqref{deriv:mes:a:m}, after some standard computations, we obtain
\begin{align}
 \Big|  &  \partial_\mu \Big[H_1\left(\int_{s_1}^t a(r, y, [X^{s_1, \xi , (m)}_{r}]) \, dr, z-x\right)\Big](v) -  \partial_\mu \Big[H_1\left(\int_{s_2}^t a(r, y, [X^{s_2, \xi , (m)}_{r}]) \, dr, z-x\right)\Big](v)\Big| \label{diff:time:partial:mes:first:order:hermite:polyn} \\
  &  \leq K_\beta^{+}  |s_1-s_2|^\beta \frac{|z-x|}{(t-s_1\vee s_2)^{\frac{3-\eta}{2}+\beta}}. \notag 
\end{align}

We now come back to \eqref{decomposition:partial:mes:partial:space:pm:hat}. We combine \eqref{bound:L:deriv:H1} with \eqref{gaussian:bound:diff:time:hat:pm:same:time} (with $n=0$), \eqref{diff:time:partial:mes:first:order:hermite:polyn} with the Gaussian upper-bound on $\widehat{p}^y_m$,  \eqref{diff:time:L:deriv:p:hat:same:time} with \eqref{standard}, \eqref{diff:time:first:order:hermite:polyn:m} with \eqref{deriv:v:deriv:mes:pm:hat} (with $r=s$ and $n=0$) and finally use the space time inequality \eqref{space:time:inequality} and the inequality $|s_1-s_2| \leq t-s_1 \vee s_2$. We thus deduce that for any $\beta \in [0,(1+\eta)/2)$
\begin{align*}
\Big|\partial_\mu &  [\partial_x \widehat{p}^{y}_m(\mu, s_1, t, x, z)](v) - \partial_\mu [\partial_x \widehat{p}^{y}_m(\mu, s_2, t, x, z)](v) \Big| \\
& \leq K_\beta^{+} \left\{ \frac{|s_1-s_2|^{\beta}}{(t-s_1)^{1-\frac{\eta}{2}+\beta}} g(c(t-s_1), z-x) +  \frac{|s_1-s_2|^{\beta}}{(t-s_2)^{1-\frac{\eta}{2}+\beta}} g(c(t-s_2), z-x) \right\}
\end{align*}
\noindent which concludes the proof of \eqref{diff:time:deriv:mes:deriv:space:p:hat:m}. 

\end{proof}

Having the above technical estimates at hand, we now turn to the proof of \eqref{sensitivity:time:deriv:cross:space:mes:pm}. The strategy is clear inasmuch one has to quantify the H\"older regularity with respect to the variable $s$ of each term appearing in the identity \eqref{identity:partial:mes:partial:space:pm}. In particular, the estimate \eqref{diff:time:deriv:mes:deriv:space:p:hat:m} of the previous lemma allows to deal with the first term therein. In order to deal with the second term, we make use of the decomposition
$$
(\partial_x p_m \otimes \partial_\mu \mH_m(.)(v))(\mu, s_1 \vee s_2, t, x, z) - (\partial_x p_m \otimes \partial_\mu \mH_m(.)(v))(\mu, s_1\wedge s_2, t, x, z) = {\rm I} + {\rm II} + {\rm III}
$$  
\noindent with
\begin{align*}
{\rm I}& := \int_{s_1 \vee s_2}^{t} \int_{\mathbb{R}^d} [\partial_x p_m(\mu, s_1\vee s_2, r, x, y) - \partial_x p_m(\mu, s_1\wedge s_2, r, x, y)] \, \partial_\mu \mH_m(\mu, s_1 \vee s_2, r, t, y, z)(v) \, dy \, dr, \\
{\rm II}& := \int_{s_1 \vee s_2}^{t} \int_{\mathbb{R}^d}  \partial_x p_m(\mu, s_1\wedge s_2, r, x, y)  \, [\partial_\mu \mH_m(\mu, s_1 \vee s_2, r, t, y, z)(v) - \partial_\mu \mH_m(\mu, s_1 \wedge s_2, r, t, y, z)(v)]   \, dy \, dr, \\
{\rm III} & := -\int_{s_1 \wedge s_2}^{s_1 \vee s_2} \int_{\mathbb{R}^d} \partial_x p_m(\mu, s_1\wedge s_2, r, x, y) \, \partial_\mu \mH_m(\mu, s_1 \wedge s_2, r, t, y, z)(v) \, dy\, dr
\end{align*}
\noindent which directly stems from the very definition of the space time convolution operator $\otimes$. We now establish an appropriate estimate for each term. In order to deal with ${\rm I}$, we use \eqref{regularity:time:estimate:v1:v2:v3:decoupling:mckean:prop:statement} with $\beta \in [0,\eta/2)$ and \eqref{estimate:partial:deriv:mes:parametrix:kernel:iterated:m} with $k=1$, $n=0$ and $\beta= \beta' \in (0,1)$ small enough so that $\beta-(1-\beta')\eta/2< 0$ in order to ensure the integrability of the time singularity. We thus get
\begin{align*}
|{\rm I}| & \leq K_\beta  \int_{s_1\vee s_2}^{t} \frac{|s_1-s_2|^\beta}{(t-r)^{1-\frac{\beta'\eta}{2}}(r-s_1\vee s_2)^{1+\beta-(1-\beta')\frac{\eta}{2}}} \, dr \, \left\{ g(c(t-s_1), z-x) + g(c(t-s_2), z-x)\right\} \\
& \leq K_\beta \frac{|s_1-s_2|^\beta}{(t-s_1\vee s_2)^{1+\beta-\frac{\eta}{2}}} \, \left\{ g(c(t-s_1), z-x) + g(c(t-s_2), z-x)\right\}\\
& \leq K_\beta \left\{ \frac{|s_1-s_2|^\beta}{(t-s_1)^{1-\frac{\eta}{2}+\beta}} g(c(t-s_1), z-x) + \frac{|s_1-s_2|^\beta}{(t-s_2)^{1-\frac{\eta}{2}+\beta}} g(c(t-s_2), z-x) \right\} 
\end{align*}
\noindent where we used the inequality $(t-s_1\vee s_2)^{-1} \leq 2 (t-s_1 \wedge s_2)^{-1}$, recalling that $|s_1-s_2| \leq t-s_1\vee s_2$, for the last inequality.

In order to deal with ${\rm II}$, we use \eqref{bound:derivative:heat:kernel} and \eqref{diff:time:L:deriv:parametrix:kernel:pmp1}. In particular, we split the time interval $[s_1\vee s_2, t]$ into the disjoint two intervals $[s_1\vee s_2, (t+s_1\vee s_2)/2]$ and $((t+s_1\vee s_2)/2, t]$. On $[s_1\vee s_2, (t+s_1\vee s_2)/2]$, we bound $| [\partial_\mu \mH_m(\mu, s_1 \vee s_2, r, t, y, z)(v) - \partial_\mu \mH_m(\mu, s_1 \wedge s_2, r, t, y, z)(v)] | $ by $K_\beta^{+} |s_1-s_2|^\beta (t-r)^{-1} (r-s_1\vee s_2)^{-(1-\eta)/2 -\beta} g(c(t-r), z-y)$ while on $((t+s_1\vee s_2)/2, t]$ we bound it by $K_\beta^{+} |s_1-s_2|^\beta (t-r)^{-1+\eta/2} (r-s_1\vee s_2)^{-1/2 -\beta} g(c(t-r), z-y)$. After some standard computations, for any $\beta \in [0,\eta/2)$, we obtain
\begin{align*}
|{\rm II}| \leq K^{+}_\beta  \frac{|s_1-s_2|^\beta}{(t-s_1\vee s_2)^{1+\beta-\frac{\eta}{2}}} \, g(c(t-s_1\wedge s_2), z-x) \leq K^{+}_\beta  \frac{|s_1-s_2|^\beta}{(t-s_1\wedge s_2)^{1+\beta-\frac{\eta}{2}}} \, g(c(t-s_1\wedge s_2), z-x)
\end{align*}

\noindent where we again used the inequality $(t-s_1\vee s_2)^{-1} \leq 2 (t-s_1 \wedge s_2)^{-1}$ for the last inequality. 

We handle ${\rm III}$ by using \eqref{bound:derivative:heat:kernel} and \eqref{estimate:partial:deriv:mes:parametrix:kernel:iterated:m} with $k=1$, $n=0$ and any $\beta \in (0,1)$. We obtain
\begin{align*}
| {\rm III} | & \leq K_\beta \int_{s_1 \wedge s_2}^{s_1\vee s_2} \frac{1}{(r-s_1\wedge s_2)^{1- (1-\beta)\frac{\eta}{2}}(t-r)^{1-\beta\frac{\eta}{2}}} \, dr \, g(c(t-s_1\wedge s_2), z-x) \\
& \leq K_\beta \frac{|s_1-s_2|^{(1-\beta)\frac{\eta}{2}}}{(t-s_1\vee s_2)^{1-\beta\frac{\eta}{2}}} \, g(c(t-s_1\wedge s_2), z-x) \\
& \leq K_\beta \frac{|s_1-s_2|^{(1-\beta)\frac{\eta}{2}}}{(t-s_1\wedge s_2)^{1-\beta\frac{\eta}{2}}} \, g(c(t-s_1\wedge s_2), z-x)
\end{align*} 
\noindent for any $\beta \in (0,1)$. Note that the above estimate remains valid for $\beta=1$ since $|s_1-s_2| \leq t-s_1\vee s_2$. Gathering the above estimates on ${\rm I}$, ${\rm II}$ and ${\rm III}$, we thus deduce
\begin{align}
\Big| (\partial_x p_m & \otimes \partial_\mu \mH_m(.)(v))(\mu, s_1 \vee s_2, t, x, z) - (\partial_x p_m \otimes \partial_\mu \mH_m(.)(v))(\mu, s_1\wedge s_2, t, x, z) \Big| \notag \\
& \leq K^{+}_\beta \left\{ \frac{|s_1-s_2|^\beta}{(t-s_1)^{1-\frac{\eta}{2}+\beta}} g(c(t-s_1), z-x) + \frac{|s_1-s_2|^\beta}{(t-s_2)^{1-\frac{\eta}{2}+\beta}} g(c(t-s_2), z-x) \right\} \label{diff:time:partial:space:pm:partial:mes:parametrix:kernel}
\end{align}
\noindent for any $\beta \in [0,\eta/2)$.

We now turn our attention to the last term appearing in the right hand side of the identity \eqref{identity:partial:mes:partial:space:pm}. We employ a similar decomposition as for the previous term. Namely, we write
\begin{align*}
((\partial_\mu [\partial_x \widehat{p}_m](.)(v) & + \partial_x p_m \otimes \partial_\mu \mH_m(.)(v))\otimes \Phi_m)(\mu, s_1, t, x, z)  \\
& \quad - ((\partial_\mu [\partial_x \widehat{p}_m](.)(v) + \partial_x p_m \otimes \partial_\mu \mH_m(.)(v))\otimes \Phi_m)(\mu, s_2, t, x, z) \\
& = {\rm I} + {\rm II} + {\rm III}
\end{align*}
\noindent with 
\begin{align*}
{\rm I}& := \int_{s_1 \vee s_2}^{t} \int_{\mathbb{R}^d} [ (\partial_\mu [\partial_x \widehat{p}_m](.)(v)  + \partial_x p_m \otimes \partial_\mu \mH_m(.)(v))(\mu, s_1\vee s_2, r, x, y) \\
& \quad \quad - (\partial_\mu [\partial_x \widehat{p}_m](.)(v)  + \partial_x p_m \otimes \partial_\mu \mH_m(.)(v))(\mu, s_1\wedge s_2, r, x, y) ] \, \Phi_m(\mu, s_1 \vee s_2, r, t, y, z) \, dy \, dr, \\
{\rm II}& := \int_{s_1 \vee s_2}^{t} \int_{\mathbb{R}^d} (\partial_\mu [\partial_x \widehat{p}_m](.)(v)  + \partial_x p_m \otimes \partial_\mu \mH_m(.)(v))(\mu, s_1\wedge s_2, r, x, y) \\
& \quad \quad \times [\Phi_m(\mu, s_1 \vee s_2, r, t, y, z) - \Phi_m(\mu, s_1 \wedge s_2, r, t, y, z)]   \, dy \, dr, \\
{\rm III} & := -\int_{s_1 \wedge s_2}^{s_1 \vee s_2} \int_{\mathbb{R}^d} (\partial_\mu [\partial_x \widehat{p}_m](.)(v)  + \partial_x p_m \otimes \partial_\mu \mH_m(.)(v))(\mu, s_1\wedge s_2, r, x, y) \, \Phi_m(\mu, s_1 \wedge s_2, r, t, y, z) \, dy\, dr.
\end{align*}

We deal with ${\rm I}$ by using \eqref{diff:time:deriv:mes:deriv:space:p:hat:m}, \eqref{diff:time:partial:space:pm:partial:mes:parametrix:kernel} and \eqref{Gaussian:estimate:Phim}. After some standard computations, we obtain
\begin{align*}
| {\rm I} | \leq K^{+}_\beta \left\{ \frac{|s_1-s_2|^\beta}{(t-s_1)^{1-\eta+\beta}} g(c(t-s_1), z-x) + \frac{|s_1-s_2|^\beta}{(t-s_2)^{1-\eta+\beta}} g(c(t-s_2), z-x) \right\}.
\end{align*}

We handle ${\rm II}$ using \eqref{estimate:partial:mes:partial:space:pmhat}, \eqref{convolution:partial:space:pm:partial:mes:parametrix:kernel:m} and \eqref{gaussian:bound:diff:time:Phim}. For any $\beta \in [0,\eta/2)$, we get
\begin{align*}
| {\rm II} | \leq K \frac{|s_1-s_2|^\beta}{(t-s_1\vee s_2)^{1-\eta+\beta}} \, g(c(t-s_1\wedge s_2), z-x) \leq K \frac{|s_1-s_2|^\beta}{(t-s_1\wedge s_2)^{1-\eta+\beta}} \, g(c(t-s_1\wedge s_2), z-x)
\end{align*}
\noindent using the fact that $(t-s_1\vee s_2)^{-1} \leq 2 (t-s_1\wedge s_2)^{-1}$ for the last inequality.

We deal with ${\rm III}$ by using \eqref{estimate:partial:mes:partial:space:pmhat}, \eqref{convolution:partial:space:pm:partial:mes:parametrix:kernel:m} and \eqref{Gaussian:estimate:Phim} so that
\begin{align*}
|{\rm III}| & \leq \frac{K}{(t-s_1\vee s_2)^{1-\frac{\eta}{2}}} \int_{s_1 \wedge s_2}^{s_1 \vee s_2} \frac{1}{(r-s_1\vee s_2)^{1-\frac{\eta}{2}}} \, dr \, g(c(t-s_1\wedge s_2), z-x) \\
& \leq K \frac{|s_1-s_2|^\beta}{(t-s_1\vee s_2)^{1-\eta+\beta}} \, g(c(t-s_1\wedge s_2), z-x) \\
& \leq K   \frac{|s_1-s_2|^\beta}{(t-s_1\wedge s_2)^{1-\eta+\beta}} \, g(c(t-s_1\wedge s_2), z-x)
\end{align*}
\noindent for any $\beta \in [0,\eta/2]$.

We now collect the above estimates on ${\rm I}$, ${\rm II}$ and ${\rm III}$. We thus obtain
\begin{align*}
\Big| (\partial_\mu [\partial_x \widehat{p}_m](.)(v) & + \partial_x p_m \otimes \partial_\mu \mH_m(.)(v))\otimes \Phi_m)(\mu, s_1, t, x, z) \\
&\quad  - ((\partial_\mu [\partial_x \widehat{p}_m](.)(v) + \partial_x p_m \otimes \partial_\mu \mH_m(.)(v))\otimes \Phi_m)(\mu, s_2, t, x, z) \Big| \\
& \leq K^{+}_\beta \left\{ \frac{|s_1-s_2|^\beta}{(t-s_1)^{1-\eta+\beta}} g(c(t-s_1), z-x) + \frac{|s_1-s_2|^\beta}{(t-s_2)^{1-\eta+\beta}} g(c(t-s_2), z-x) \right\}
\end{align*}
\noindent for any $\beta \in [0,\eta/2)$. This last estimate concludes the proof of \eqref{sensitivity:time:deriv:cross:space:mes:pm}.
 
\subsection{Some preparatory technical results} \label{sec:technical:results:first:part}

To proceed with our induction procedure, we have to prove that the statements obtained in the base case $m=1$ indeed propagate at step $m+1$ provided they are satisfied at step $m$. Starting with the process $(X^{s,\xi,(m+1)}_t, t\in [s, T])$ with dynamics given by \eqref{iter:mckean} and coefficients frozen in their measure argument at the law of the Picard iteration scheme at step $m$, we importantly observe that the density function $z \mapsto p_{m+1}(\mu, s, t, z)$ of the random variable $X^{s, \xi, (m+1)}_t$ satisfies the relation \eqref{conv:relation:step:m} where $z\mapsto p_{m+1}(\mu, s, t, x, z)$ denotes the transition density of the decoupling SDE $(X^{s, x, \mu, (m+1)}_t , t\in [s, T])$. 

As already emphasized in \cite{chaudruraynal:frikha}, the key point is that this transition density satisfies a representation in infinite series given by \eqref{series:approx:mckean} which involves space-time iterated convolutions of the so-called parametrix kernel $\mathcal H_{m+1}$ given by \eqref{eq:def:de:mH} against the Gaussian type kernel $\widehat p_{m+1}$ given by \eqref{eq:def:de:phat:m}. These quantities in turn depend on the density $p_{m}$ built at the previous step of the Picard iteration scheme, so that, when investigating the $\mathcal{C}^{1, 2, 2}_f([0,t)\times \mathbb{R}^d \times \pp)$ smoothness of $p_{m+1}$ and its related estimates, we will naturally be lead to investigate the smoothness of these terms. In particular, as a preparatory step of our induction argument, we need to investigate the regularity properties and to establish some adequate estimates for the coefficients $b_i(t, x, [X^{s, \xi, (m)}_t]),\, a_{i, j}(t, x, [X^{s, \xi, (m)}_t])$, the Gaussian type kernel $\widehat p_{m+1}$, the parametrix kernel $\mH_{m+1}$ and its iterated space time convolution $\mH^{(k)}_{m+1},\, k\geq 1$, defined just after \eqref{eq:def:de:mH} in order to prove that the estimates in Proposition \ref{proposition:reg:density:recursive:scheme:mckean} indeed propagates from one step to another. 

This is the purpose of this section and the associated technical results are respectively given by Lemma \ref{lem:diff:and:control:deriv:coeff} and Corollaries \ref{cor:deriv:time:and:mes:phat}, \ref{cor:deriv:time:and:mes:parametrix:kernel}. As previously mentioned, though their proofs are rather intuitive, they are rather long and rely on technical Gaussian type computations. The reader may want to skip these derivations in a first reading. We thus decided to postpone them to some dedicated sections, see \ref{proof:lem:diff:and:control:deriv:coeff}, \ref{proof:cor:deriv:time:and:mes:phat} and \ref{proof:cor:deriv:time:and:mes:parametrix:kernel}.

\begin{lem}\label{lem:diff:and:control:deriv:coeff}
 For any fixed $(t, x) \in (0, T] \times \mathbb{R}^d$ and any $(i, j) \in \left\{1, \cdots, d\right\}^2$, the maps $(s, \mu) \mapsto b_i(t, x, [X^{s, \xi, (m)}_t]), \, a_{i, j}(t, x, [X^{s, \xi, (m)}_t])$ belong to $\mathcal{C}^{1, 2}_f([0,t)\times \pp)$ and satisfy the following estimates: for any $\beta \in [0,\eta)$, any $\beta' \in [0,1]$, any $(t, x) \in (0,T] \times \mathbb{R}^d$, any $(s, \mu, \mu' )\in [0,t) \times (\pp)^2$, any $\gv=(v, v')$, $\gv_1 = (v_1,v_1')$, $\gv_2 = (v_2,v_2')$ in $\mathbb{R}^{d} \times \mathbb{R}^d$ and any $(i, j) \in \left\{1, \cdots, d\right\}^2$
\begin{align}
\Big|\partial_\mu^2 & [b_i(t, x, [X^{s,\xi, (m)}_t])](\gv)\Big|  + \Big|\partial_\mu^2 [a_{i, j}(t, x, [X^{s,\xi, (m)}_t])](\gv)\Big| \nonumber \\
&  \leq K^+ \left\{ \frac{1}{(t-s)^{1-\frac{ \eta}{2}}} + \int_{(\mathbb{R}^d)^2} (|y-x'|^{\eta}\wedge 1)\, | \partial_\mu^2 p_{m}(\mu, s, t, x', y)(\gv)| \, \mu(dx') \, dy   \right\}, \label{recursive:bound:deriv:a:or:b}
\end{align}
\begin{align}
\Big|  & \partial^2_\mu \Big[a_{i, j}(t, x, [X^{s, \xi, (m)}_{t}]) - a_{i, j}(t, z, [X^{s, \xi, (m)}_t])\Big](\gv) \Big| \nonumber \\
 &\quad \quad \leq  K^+_{\beta'} |z-x|^{\beta' \eta} \left\{ \frac{1}{(t-s)^{1-\frac{(1-\beta')\eta}{2}}} \right. \label{recursive:bound:second:order:deriv:mes:holder:reg:space:a} \\
 & \left. \quad \quad \quad +  \int_{(\mathbb{R}^d)^2} (|y-x'|^{(1-\beta')\eta} \wedge 1) \, |\partial^2_\mu p_{m}(\mu, s, t, x', y)(\gv)| \, \mu(dx')\, dy\right\}, \nonumber
\end{align}
\begin{align}
\Big| & \partial^2_\mu [b_{i}(t, x, [X^{s, \xi, (m)}_{t}])](\gv_1)  - \partial^2_\mu [b_{i}(t, x, [X^{s, \xi, (m)}_{t}])](\gv_2) \Big| \nonumber \\
& \quad \quad + \Big|  \partial^2_\mu [a_{i, j}(t, x, [X^{s, \xi, (m)}_{t}])](\gv_1)  - \partial^2_\mu [ a_{i, j}(t, x, [X^{s, \xi, (m)}_{t}])](\gv_2)| \nonumber \\
& \leq K^+_\beta \left\{ \frac{|\gv_1-\gv_2|^\beta}{(t-s)^{1+\frac{\beta-\eta}{2}}} \right.  \label{recursive:bound:second:order:deriv:mes:reg:holder:v:a:or:b} \\
& \quad\quad \left. + \int_{(\mathbb{R}^d)^2} (|y-x'|^\eta \wedge 1) \, | \partial^2_\mu p_{m}(\mu, s, t, x', y)(\gv_1) - \partial^2_\mu p_{m}(\mu, s, t, x', y)(\gv_2)| \, \mu(dx') \, dy \right\}, \notag
\end{align}
%\begin{align}
%\Big| & \partial^2_\mu [b_{i}(t, x, [X^{s, \xi, (m)}_{t}])](\gv)  - [\partial^2_\mu [b_{i}(t, x, [X^{s, \xi, (m)}_{t}])](\gv)]_{\mu= \mu'} \Big| \nonumber \\
%& \quad \quad + \Big|  \partial^2_\mu [a_{i, j}(t, x, [X^{s, \xi, (m)}_{t}])](\gv)  -[\partial^2_\mu [ a_{i, j}(t, x, [X^{s, \xi, (m)}_{t}])](\gv)]_{\mu=\mu'}| \nonumber \\
%& \leq K^{++}_\beta \left\{ \frac{W_2^\beta(\mu, \mu')}{(t-s)^{1+\frac{\beta-\eta}{2}}} \right.  \label{recursive:bound:second:order:deriv:mes:reg:holder:mes:a:or:b} \\
%& \quad\quad \left. + \int_{(\mathbb{R}^d)^2} (|y-x'|^\eta \wedge 1) \, | \partial^2_\mu p_{m}(\mu, s, t, x', y)(\gv) - \partial^2_\mu p_{m}(\mu', s, t, x', y)(\gv)| \, \mu'(dx') \, dy \right\}, \notag
%\end{align}

\begin{align}
\Big| \partial^2_\mu & [a_{i, j}(t, x, [X^{s, \xi, (m)}_{t}])  - a_{i, j}(t, z, [X^{s, \xi, (m)}_t])](\gv_1) - \partial^2_\mu [a_{i, j}(t, x, [X^{s, \xi, (m)}_{t}]) - a_{i, j}(t, z, [X^{s, \xi, (m)}_t])](\gv_2) \Big| \nonumber \\
 &\quad  \leq  K^+_\beta \left\{\frac{(|z-x|^\eta \wedge 1)}{(t-s)^{1+\frac{\beta}{2}}} \wedge \frac{1}{(t-s)^{1+\frac{\beta-\eta}{2}}} \right\} \left\{ \Big. {|\gv_1-\gv_2|^\beta} \right.  \label{recursive:bound:deriv:mes:double:reg:holder:a} \\
 & \quad \quad \left. + (t-s)^{1+\frac{\beta-\eta}{2}} \int_{(\mathbb{R}^d)^2} (|y'-x'|^\eta \wedge 1) \Big| \partial^2_{\mu} p_{m}(\mu, s , t, x', y)(\gv_1) - \partial^2_{\mu}p_{m}(\mu, s , t, x', y)(\gv_2)| \, \mu(dx') \, dy \right. \nonumber \\
 & \quad \quad \left.  + (t-s)^{1+\frac{\beta}{2}}\int_{(\mathbb{R}^d)^2} \Big|\partial^2_{\mu}p_{m}(\mu, s , t, x', y)(\gv_1) -\partial^2_{\mu} p_{m}(\mu, s , t, x', y)(\gv_2)\Big| \, \mu(dx') \, dy \right\}. \nonumber
\end{align}

\end{lem}

\begin{remark} Note that if estimate \eqref{second:deriv:mes:induction:decoupling:mckean} holds at step $m$, then, from \eqref{recursive:bound:deriv:a:or:b} and the space time inequality \eqref{space:time:inequality}, it holds 
\begin{align}
|\partial_\mu^2 & [b_i(t, x, [X^{s,\xi, (m)}_t])](\gv) |  + |\partial_\mu^2 [a_{i, j}(t, x, [X^{s,\xi, (m)}_t])](\gv)| \nonumber \\
&  \leq K^+ \left\{ (t-s)^{-1+\frac{ \eta}{2}} + (t-s)^{-1+ \eta}\mathscr{C}_m(C^+, t-s) \right\} \label{recursive:bound:deriv:a:or:b:stepm}.
\end{align}
\end{remark}

\begin{cor}\label{cor:deriv:time:and:mes:phat}
Assume that the estimate \eqref{second:deriv:mes:induction:decoupling:mckean} is satisfied at step $m$ for some positive constant $C^{+}$. For any $t$ in $(0,T]$, any $(r, x, y, z) \in (0,t)\times (\mathbb{R}^{d})^3$, the maps $(s,\mu) \mapsto \widehat{p}^{y}_{m+1}(\mu, s, r, t, x, z)$, $\widehat{p}^{y}_{m+1}(\mu, s, t, x, z)$ belong to $C_f^{1, 2, 2}([0,r)\times \mathbb{R}^d \times \mathcal P_2(\mathbb R^d))$ and $C_f^{1,2}([0,t)\times \mathbb{R}^d \times \mathcal P_2(\mathbb R^d))$ respectively with continuous derivatives with respect to its entries. \\
Moreover, the second order L-derivative satisfy the following pointwise Gaussian estimates: there exist positive constants $K^+$ and $c$ such that for any $( \mu, x, y, z )\in \pp \times (\mathbb{R}^d)^3 $, any $\gv \in (\mathbb{R}^d)^2$ and any $0\leq s \leq r < t \leq T$
\begin{eqnarray}
&&|\partial^2_\mu  \widehat{p}^{y}_{m+1}(\mu, s, r, t, x, z)(\gv) |\label{bound:second:deriv:mes:hat:p} \le \frac{K^+}{t-r}\Big\{ \int_{r}^t \frac{1}{(r'-s)^{1-\frac{\eta}{2}}}\, dr'\\
&& +   \int_r^t \int_{(\mathbb{R}^d)^2} (|y-x'|^{\eta}\wedge 1) | \partial_\mu^2 p_{m}(\mu, s, r', x', y)(\gv)| \, \mu(dx') \, dy \, dr'  \Big\}   g(c(t-r), z-x).\notag
\end{eqnarray}

For any $\beta \in [0,1]$ and any $\beta' \in [0,\eta)$, there exist positive constants $K^+$, $K^+_{\beta'}$ and $c$ such that for any $(\mu, x, z, y)\in \pp \times (\mathbb{R}^d)^3$, any $\gv, \gv_1, \gv_2 \in (\mathbb{R}^d)^2$, any $0\leq s \leq r < t \leq T$ and any $x_1, x_2 \in \mathbb{R}^d$ 
\begin{align}
| &  \partial^2_\mu \widehat{p}^{y}_{m+1}(\mu, s, t, x_1, z)(\gv) -   \partial^2_\mu \widehat{p}^{y}_{m+1}(\mu, s, t, x_2, z)(\gv)| \notag \\
&   \leq K^+ \frac{|x_1-x_2|^\beta}{(t-s)^{\frac{\beta}{2}}}\left\{ \frac{1}{(t-s)^{1-\frac{\eta}{2}}} + \frac{1}{t-s} \int_s^t \int_{(\mathbb{R}^d)^2} (|y'-x'|^{\eta} \wedge 1)   |\partial^2_\mu p_{m}(\mu, s, r', x', y')(\gv)| \,  \mu(dx') \, dy' \, dr' \right\} \label{cross:mes:deriv:holder:p:hat:s:t}\\
& \quad \quad \times \left\{ g(c(t-s) , z-x_1) + g(c(t-s) , z-x_2)\right\}, \notag
\end{align}

\begin{align}
 | & \partial^2_\mu \widehat{p}^y_{m+1}(\mu, s, r, t, x, z)(\gv_1)   -  \partial^2_\mu \widehat{p}^y_{m+1}(\mu, s, r, t, x, z)(\gv_2) | \notag \\
& \leq \frac{K^+_{\beta'}}{t - r}  \left\{ \int_{r}^{t}  \frac{|\gv_1-\gv_2|^{\beta'}}{(r'-s)^{1+ \frac{\beta'-\eta}{2}}} \, dr'\right. \notag\\
&  \quad  \quad \left.+ \int_r^t \int_{(\mathbb{R}^d)^2} (|y'-x'|^{\eta}\wedge 1) |\partial^2_\mu p_{m}(\mu, s, r', x', y')(\gv_1) -  \partial^2_\mu p_{m}(\mu, s, r', x', y')(\gv_2) | \, \mu(dx')\, dy' \,  dr' \right\}   \label{cross:mes:deriv:reg:holder:terminal point:p:hat:s:r:t} \\
& \quad \times g(c(t-r), z- x). \nonumber
\end{align}
\end{cor}

\begin{remark} Note that if estimate \eqref{second:deriv:mes:induction:decoupling:mckean} holds at step $m$ for some positive constant $C^{+}$, then it follows from \eqref{bound:second:deriv:mes:hat:p}, the space time inequality \eqref{space:time:inequality} and the fact that $t\mapsto \mathscr{C}^{n, 0}_{m}(C^+, t)$ is non-decreasing that 
\begin{align}
|\partial_\mu^2 \widehat{p}^{y}_{m+1}(\mu, s, r, t, x, z)(\gv)|   \leq K^+ \left\{ (r-s)^{-1+\frac{ \eta}{2}} + (r-s)^{-1+ \eta} \mathscr{C}^{1, 0}_m(C^+, t-s) \right\} \label{recursive:bound:deriv:hat:p:stepm}.
\end{align}
\end{remark} 

\begin{cor}\label{cor:deriv:time:and:mes:parametrix:kernel}
Assume that the estimate \eqref{second:deriv:mes:induction:decoupling:mckean} is satisfied at step $m$ for some positive constant $C^{+}$. For any $(t, x, z)$ in $(0,T] \in (\mathbb{R}^d)^2$ and any $r \in (0,t)$, the maps $[0, r) \times \pp \ni (s,\mu) \mapsto \mH_{m+1}(\mu, s, r, t, x, z)$ is in $C_f^{1,2}([0,r)\times \mathcal P_2(\mathbb R^d))$ with derivatives $\partial_s \mH_{m+1}(\mu, s, r, t, x, z)$, $\partial_v^{n}[\partial_\mu \mH_{m+1}(\mu, s, r, t, x, z)](v)$, $n=0, 1$, $\partial^2_\mu \mH_{m+1}(\mu, s, r, t, x, z)(\gv)$ being continuous with respect to the variables $s$, $x$, $\mu$, $v$ and $\gv$. \\

 Moreover, the second order L-derivative satisfies the following Gaussian estimates: for any $\beta \in [0,1]$ and any $\beta' \in [0,\eta)$, there exist positive constants $K^+_{\beta}$, $K^+_{\beta'}$ and $c$ such that for any $(t, \mu, x, z)\in (0,T] \times \pp \times (\mathbb{R}^d)^2$, any $s \in [0, t)$, any $r\in (s, t)$ and any $\gv, \gv_1, \gv_2 \in (\mathbb{R}^d)^2$
\begin{eqnarray}\label{sec:mes:deriv:parametrix:kernel:s:r:t:with:beta}
&& \Big|  \partial_\mu^2 \mH_{m+1}(\mu, s, r, t, x, z)(\gv) \Big|  \\
& \leq & \frac{K^+_\beta}{(t-r)^{1-\beta\frac{\eta}{2}}(r-s)^{1-(1-\beta)\frac{\eta}{2}}}  \notag\\
&& \quad \times \Bigg\{(1 +  (r-s)^{1-(1-\beta)\frac{\eta}{2}} \int_{(\mathbb{R}^d)^2} (|y'-x'|^{(1-\beta)\eta}\wedge 1) |\partial_\mu^2 p_{m}(\mu, s, r, x', y')(\gv)| \,  \mu(dx') \, dy'  \notag \\
&& \quad \quad  + \frac{(r-s)^{1-(1-\beta)\frac{\eta}{2}}}{(t-r)^{1-(1-\beta)\frac{\eta}{2}}} \int_r^t \int_{(\mathbb{R}^d)^2} (|y'-x'|^{\eta}\wedge 1) |\partial_\mu^2 p_{m}(\mu, s, r', x', y')(\gv)| \, \mu(dx') \, dy'  \, dr'   \Bigg\}\notag \\
&& \quad \quad  \times g(c(t-r), z-x), \notag
\end{eqnarray}
\begin{align}
& \Big|  \partial^2_\mu \mH_{m+1}(\mu, s, r ,t ,x ,z)(\gv_1) - \partial^2_\mu \mH_{m+1}(\mu, s, r ,t ,x ,z)(\gv_2) \Big| \notag \\
& \leq K^+_{\beta'}  \left\{ \frac{1}{(t-r)^{1-\frac{\eta}{2}} (r-s)^{1+\frac{\beta'}{2}}} \wedge  \frac{1}{(t-r)(r-s)^{1+\frac{\beta'-\eta}{2}}} \right\}  \label{second:order:mes:deriv:parametrix:kernel:s:r:t:reg:holder:v:argument} \\
& \quad \times \left\{|\gv_1-\gv_2|^{\beta'} + (r-s)^{1+\frac{\beta'}{2}} \int_{(\mathbb{R}^d)^2} |\partial^2_{\mu}p_{m}(\mu, s , r, x', y')(\gv_1) - \partial^2_{\mu}p_{m}(\mu, s , r, x', y')(\gv_2)| \, \mu(dx') \, dy' \right.  \notag\\
& \quad \left. +(r-s)^{1+\frac{\beta'-\eta}{2}} \int_{(\mathbb{R}^d)^2} (|y'-x'|^\eta \wedge 1) |\partial^2_{\mu}p_{m}(\mu, s , r, x', y')(\gv_1) - \partial^2_{\mu}p_{m}(\mu, s , r, x', y')(\gv_2)| \, \mu(dx')\, dy'  \right. \notag \\
& \quad \left. +  \frac{(r-s)^{1 + \frac{\beta'-\eta}{2}}}{t-r}   \int_r^t \int_{(\mathbb{R}^d)^2} (|y'-x'|^\eta \wedge 1)  | \partial^2_\mu p_{m}(\mu, s, r', x', y')(\gv_1) - \partial^2_\mu p_{m}(\mu, s, r', x', y')(\gv_2)| \, \mu(dx') \, dy' \, dr' \right\} \notag \\
& \quad \quad \times g(c(t-r), z-x). \notag
\end{align}
\end{cor}

Let us importantly observe again that if the estimate \eqref{second:deriv:mes:induction:decoupling:mckean} is satisfied at step $m$ for some positive constant $C^{+}$ then from \eqref{sec:mes:deriv:parametrix:kernel:s:r:t:with:beta} with $\beta=0$ and $\beta=1$, we deduce 
\begin{eqnarray}\label{cor:sec:mes:deriv:parametrix:kernel:s:r:t:with:beta}
&& \Big|  \partial_\mu^2 \mH_{m+1}(\mu, s, r, t, x, z)(\gv) \Big|  \\
& \leq & K^{+} \left\{ \frac{1}{(t-r)^{1-\frac{\eta}{2}}(r-s)} \wedge  \frac{1}{(t-r)(r-s)^{1-\frac{\eta}{2}}}  \right\}  g(c(t-r), z-x)   \notag\\
  && \quad \times \Bigg\{(1 +  (r-s) \int_{(\mathbb{R}^d)^2} | \partial_\mu^2 p_{m}(\mu, s, r, x', y')(\gv)| \, \mu(dx')\, dy'   \notag \\
  & & \quad \quad +  (r-s)^{1-\frac{\eta}{2}} \int_{(\mathbb{R}^d)^2} (|y'-x'|^{\eta} \wedge 1) |\partial_\mu^2 p_{m}(\mu, s, r, x', y')(\gv)| \, \mu(dx') \, dy'  \notag\\
&& \quad \quad  + \frac{(r-s)^{1-\frac{\eta}{2}}}{t-r} \int_r^t \int_{(\mathbb{R}^d)^2} (|y'-x'|^\eta \wedge 1) |\partial_\mu^2 p_{m}(\mu, s, r', x', y')(\gv)| \, \mu(dx') \, dy' \, dr'   \Bigg\}, \notag 
\end{eqnarray}
\noindent where $K^{+}$ does not depend on $C^{+}$. Note also that taking $\beta=1/2$ in \eqref{sec:mes:deriv:parametrix:kernel:s:r:t:with:beta} and using the fact that $t\mapsto \mathscr{C}^{1, 0}_m(C^+,t)$ is non-decreasing and the space time inequality \eqref{space:time:inequality}
\begin{align}
\Big|  \partial_\mu^2 & \mH_{m+1}(\mu, s, r, t, x, z)(\gv) \Big| \label{cor:bis:sec:mes:deriv:parametrix:kernel:s:r:t:with:beta} \\
& \leq \frac{K^{+}}{(t-r)^{1-\frac{\eta}{4}}(r-s)^{1-\frac{\eta}{4}}} (1+ \mathscr{C}^{1,0}_m(C^+,t-s)) \, g(c(t-r), z-x). \notag
\end{align}

\subsection{First part of the induction step.}\label{sec:first:part:induction:step}

Our aim here is to prove the first part of the induction step of Proposition \ref{proposition:reg:density:recursive:scheme:mckean}. Namely, we prove that if the map $(s, x, \mu) \mapsto p_{m}(\mu, s, t, x, z)$ belongs to $\mathcal{C}_f^{1, 2,  2}([0,t)\times \mathbb{R}^d \times \pp)$ and if the pointwise Gaussian estimate \eqref{second:deriv:mes:induction:decoupling:mckean} is satisfied for some positive constants $C^+$ (the constant $C^{+}$ being the one appearing in the definition of the m-th partial sums $\mathscr{C}^{1,0}_m(C^+, t-s)$ therein) then $(s, x, \mu) \mapsto p_{m+1}(\mu, s, t, x, z) \in \mathcal{C}^{1, 2,  2}_f([0,t)\times \mathbb{R}^d \times \pp)$. Additionally, we prove that if the estimates \eqref{second:deriv:mes:induction:decoupling:mckean}, \eqref{second:deriv:mes:reg:space:deriv:arg:estimate:induction:decoupling:mckean} and \eqref{second:deriv:mes:reg:space:deriv:arg:estimate:induction:decoupling:mckean2} are satisfied at step $m$ for some adequate specification of the constants $C^+$ and $C^+_\beta$ (again the constants $C^{+}$ and $C^{+}_\beta$ being the one appearing in the definition of the m-th partial sums $\mathscr{C}^{1,0}_m(C^+, t-s)$ and $\mathscr{C}^{1,\beta}_m(C^+_\beta, t-s)$ therein), then they remain valid at step $m+1$.

\begin{prop}\label{prop:second:order:deriv:estimate:reg:holder:space:and:v}Assume that \eqref{second:deriv:mes:induction:decoupling:mckean} holds at step $m$ for some positive constant $C^+$.  For any $(t, x, z) \in (0,T] \times (\mathbb{R}^d)^2$, the map $[0,t) \times \mathbb{R}^d \times \pp \ni (s, x, \mu) \mapsto p_{m+1}(\mu, s, t, x, z)$ belongs to $\mathcal{C}^{1, 2,  2}_f([0,t)\times \mathbb{R}^d \times \pp)$ and for any $(s, x, \mu) \in [0,t)\times \mathbb{R}^d \times \pp$ and any $\gv=(v, v') \in \mathbb{R}^d \times \mathbb{R}^d$, it holds
\begin{eqnarray}
&&\partial^2_\mu p_{m+1}(\mu, s , t, x, z)(\gv) \label{representation:formula:lions:second:deriv:mes:dens}\\
&=& \sum_{k\geq0} \Big(\partial^2_\mu \widehat{p}_{m+1}(\cdot)(\gv) + p_{m+1} \otimes \partial^2_\mu \mH_{m+1}(\cdot)(\gv)+ (\partial_\mu p_{m+1}(\cdot)(v) \otimes \partial_\mu \mH_{m+1}(\cdot)(v'))\notag\\
&& \qquad \qquad+(\partial_\mu p_{m+1}(\cdot)(v') \otimes \partial_\mu \mH_{m+1}(.)(v))\Big) \otimes \mH^{(k)}_{m+1}(\mu, s, t, x, z) \notag
\end{eqnarray}
\noindent where we write $(\partial_\mu p_{m+1}(\cdot)(v) \otimes \partial_\mu \mH_{m+1}(\cdot)(v'))(\mu, s, t, x, z) =( ([\partial_\mu p_{m+1}(\cdot)]_i(v) \otimes [\partial_\mu \mH_{m+1}(\cdot)]_j(v'))(\mu, s, t, x, z))_{1\leq i , j \leq d}$ and $(\partial_\mu p_{m+1}(\cdot)(v') \otimes \partial_\mu \mH_{m+1}(.)(v)) = ([\partial_\mu p_{m+1}(\cdot)]_j(v')] \otimes [\partial_\mu \mH_{m+1}(.)]_i(v)])_{1\leq i, j \leq d}$.

Moreover, the following Gaussian estimates hold: for any $\beta \in [0,\eta)$, there exist positive constants $K^+_{\beta}$, $K^+$ and $c$ such that for any $(t, x, x', z)\in (0,T] \times (\mathbb{R}^d)^3$, any $(s, \mu, \gv) \in [0,t) \times \pp \times (\mathbb{R}^d)^2$, any $\gv_1, \gv_2 \in (\mathbb{R}^d)^2$ and any value of the positive constants $C^+$ and $C^+_{\beta}$ appearing in the definition of the m-th partial sums $\mathscr{C}^{1, 0}_m(C^+, t-s)$ and $\mathscr{C}^{1, \beta}_m(C^+_\beta, t-s)$ of \eqref{second:deriv:mes:induction:decoupling:mckean}, \eqref{second:deriv:mes:reg:space:deriv:arg:estimate:induction:decoupling:mckean} and \eqref{second:deriv:mes:reg:space:deriv:arg:estimate:induction:decoupling:mckean2}
\begin{align}
|\partial_\mu^2 p_{m+1}(\mu, s, t, x, z)(\gv) |  & \leq    \frac{K^+}{(t-s)^{1-\frac{\eta}{2}}} \left\{ 1+ \sum_{k=1}^{m} (C^+)^{k} (t-s)^{k \frac{\eta}{2}} \prod_{i=1}^{k} B\left(\frac{\eta}{2}, \frac{\eta}{2} +  (i-1)\frac{\eta}{2} \right)  \right\} \label{estimate:deriv:sec:mes:dens:stepmp1:cor}  \\
% \nonumber \\
& \quad \times  g(c(t-s), z-x),\nonumber
\end{align}
\begin{align}
| & \partial^2_\mu p_{m+1}(\mu, s, t, x, z)(\gv) - \partial^2_\mu p_{m+1}(\mu, s, t, x', z)(\gv) | \nonumber \\
& \leq   K^+_\beta \frac{|x-x'|^{\beta}}{(t-s)^{1+\frac{ \beta-\eta}{2}}} \left\{ 1 + \sum_{k=1}^{m} (C^+)^{k} (t-s)^{k \frac{\eta}{2}}  \prod_{i=1}^{k} B\left(\frac{\eta}{2}, \frac{\eta-\beta}{2} +  (i-1)\frac{\eta}{2} \right)  \right\} \label{estimate:deriv:sec:mes:holder:reg:space:dens:stepmp1:cor} \\
& \quad \times  \Big\{ g(c(t-s), z-x) + g(c(t-s), z-x') \Big\}, \nonumber
\end{align}
\noindent and
\begin{align}
| & \partial^2_\mu p_{m+1}(\mu, s, t, x, z)(\gv_1) - \partial^2_\mu p_{m+1}(\mu, s, t, x, z)(\gv_2) | \nonumber \\
& \leq   K^+_\beta \frac{|\gv_1-\gv_2|^{\beta}}{(t-s)^{1+\frac{ \beta-\eta}{2}}} \left\{ 1 + \sum_{k=1}^{m} (C^{+}_\beta)^{k} (t-s)^{k \frac{\eta}{2}}  \prod_{i=1}^{k} B\left(\frac{\eta}{2}, \frac{\eta-\beta}{2} +  (i-1)\frac{\eta}{2} \right)  \right\} \label{estimate:deriv:sec:mes:holder:reg:v:dens:stepmp1:cor} \\
& \quad \times  g(c(t-s), z-x). \nonumber
\end{align}

\end{prop}

\noindent \emph{Conclusion of the first part of the induction step:}\\
In view of the above result, it suffices to set the constants $C^{+}$ and $C^{+}_\beta$ involved in the $m$-th partial sums $\mathscr{C}^{1, 0}_m(C^+, t-s)$ and $\mathscr{C}^{1, \beta}_m(C^+_\beta, t-s)$ of \eqref{second:deriv:mes:induction:decoupling:mckean}, \eqref{second:deriv:mes:reg:space:deriv:arg:estimate:induction:decoupling:mckean} and \eqref{second:deriv:mes:reg:space:deriv:arg:estimate:induction:decoupling:mckean2} to be equal to the constants $K^{+}$ and $K^{+}_\beta$ appearing in \eqref{estimate:deriv:sec:mes:dens:stepmp1:cor}, \eqref{estimate:deriv:sec:mes:holder:reg:space:dens:stepmp1:cor} and \eqref{estimate:deriv:sec:mes:holder:reg:v:dens:stepmp1:cor} respectively. Indeed, doing so, by the very definition of $\mathscr{C}^{1, 0}_{m+1}(K^+, t-s)$ and $\mathscr{C}^{1, \beta}_{m+1}(K^+_\beta, t-s)$ and Proposition \ref{prop:second:order:deriv:estimate:reg:holder:space:and:v}, we deduce that the map $[0,t) \times \mathbb{R}^d \times \pp \ni (s, x, \mu) \mapsto p_{m+1}(\mu, s, t, x, z)$ belongs to $\mathcal{C}^{1, 2,  2}_f([0,t)\times \mathbb{R}^d \times \pp)$ and the estimates \eqref{estimate:deriv:sec:mes:dens:stepmp1:cor}, \eqref{estimate:deriv:sec:mes:holder:reg:space:dens:stepmp1:cor} and \eqref{estimate:deriv:sec:mes:holder:reg:v:dens:stepmp1:cor} directly yield the estimates \eqref{second:deriv:mes:induction:decoupling:mckean}, \eqref{second:deriv:mes:reg:space:deriv:arg:estimate:induction:decoupling:mckean} and \eqref{second:deriv:mes:reg:space:deriv:arg:estimate:induction:decoupling:mckean2} at step $m+1$. The first part of the induction step is thus satisfied. \\

From the above argument, we thus conclude that for any positive integer $m$ the map $ (s, x, \mu) \mapsto p_{m}(\mu, s, t, x, z) \in \mathcal{C}^{1, 2,  2}_f([0,t)\times \mathbb{R}^d \times \pp)$ and that the estimates \eqref{second:deriv:mes:induction:decoupling:mckean}, \eqref{second:deriv:mes:reg:space:deriv:arg:estimate:induction:decoupling:mckean} and \eqref{second:deriv:mes:reg:space:deriv:arg:estimate:induction:decoupling:mckean2} are satisfied.

\begin{proof}[Proof of Proposition \ref{prop:second:order:deriv:estimate:reg:holder:space:and:v}] \quad\\

\noindent \emph{Step 1: $(s, x, \mu) \mapsto p(\mu, s ,t, x, z)$ belongs to $\mathcal{C}^{1, 2,  2}_f([0,t)\times \mathbb{R}^d \times \pp)$.}\\ 

We recall that according to Proposition 5.1 in \cite{chaudruraynal:frikha}, for any positive integer $m$, the map $[0,t) \times \mathbb{R}^d \times \pp \ni (s, x, \mu) \mapsto p_{m}(\mu, s, t, x, z)$ belongs to $\mathcal{C}^{1, 2,  2}([0,t)\times \mathbb{R}^d \times \pp)$ so that it is sufficient to investigate the existence of the L-derivative of second order and its continuity with respect to $s$, $x$, $\mu$ and $\gv$. Again, according to Proposition A.1 in \cite{chaudruraynal:frikha}, the map $\mu \mapsto p_{m+1}(\mu, s, t, x, z)$ given by \eqref{series:approx:mckean} is continuously differentiable with a first order derivative satisfying (recalling \eqref{infinite:series:Phi:step:m}):
\begin{align}
\partial_\mu p_{m+1}(\mu, s, t, x, z)(v) & =\sum_{k\geq0} (\partial_\mu \widehat{p}_{m+1}(.)(v) + p_{m+1} \otimes \partial_\mu \mH_{m+1}(.)(v)) \otimes \mH^{(k)}_{m+1}(\mu, s, t, x, z) \notag \\
& = \partial_\mu \widehat{p}_{m+1}(\mu, s, t, x, z)(v) + (p_{m+1} \otimes \partial_\mu \mH_{m+1}(.)(v))(\mu, s, t, x, z) \label{representation:partial:mes:pmp1}\\
& \quad +  (\partial_\mu \widehat{p}_{m+1}(.)(v) + p_{m+1} \otimes \partial_\mu \mH_{m+1}(.)(v)) \otimes \Phi_{m+1}(\mu, s, t, x, z) \notag
\end{align}
\noindent and being continuous with respect to the variables $s$, $x$, $\mu$ and $v$. The previous identity allows to investigate the L-differentiability of the map $\mu \mapsto \partial_\mu p_{m+1}(\mu, s, t, x, z)(v)$. First, let us note that according to Corollary \ref{cor:deriv:time:and:mes:phat}, the map $\mu \mapsto \partial_\mu \widehat{p}_{m+1}(\mu, s, t, x, z)(v) $ is continuously L-differentiable, with a derivative being continuous in $s$, $x$, $\mu$ and $\gv$, and combining \eqref{bound:second:deriv:mes:hat:p} with \eqref{second:deriv:mes:induction:decoupling:mckean}, the space time inequality \eqref{space:time:inequality} and the fact that $t \mapsto \mathscr{C}^{1,0}_m(C^+, t)$ is non-decreasing, we deduce that 
\begin{align*}
 | \partial^2_\mu \widehat{p}_{m+1}(\mu, s, t, x, z)(\gv) | \leq K^+ \left\{\frac{1}{(t-s)^{1-\frac{\eta}{2}}} + \frac{\mathscr{C}^{1, 0}_{m}(C^+, t-s)}{(t-s)^{1-\eta}}\right\} \, g(c(t-s), z-x).
\end{align*} 

Similarly, it follows from the continuous L-differentiability of the two maps $\mu\mapsto p_{m+1}(\mu, s, t, x, z)$ and $\mu\mapsto  \partial_\mu \mH_{m+1}(\mu, s, r, t, x, z)(v)$ stemming from Corollary \ref{cor:deriv:time:and:mes:parametrix:kernel}, the estimates \eqref{estimate:partial:deriv:mes:parametrix:kernel:iterated:m} with $n=0$, $k=1$ and $\beta \in (0,1]$, \eqref{first:second:estimate:induction:decoupling:mckean}, \eqref{bound:derivative:heat:kernel} with $n=0$, \eqref{cor:bis:sec:mes:deriv:parametrix:kernel:s:r:t:with:beta} and the dominated convergence theorem that $\mu \mapsto  (p_{m+1} \otimes \partial_\mu \mH_{m+1}(.)(v))(\mu, s, t, x, z)$ is continuously differentiable with a derivative satisfying $\partial_\mu[(p_{m+1} \otimes \partial_\mu \mH_{m+1}(.)(v))(\mu, s, t, x, z)](v') = (\partial_\mu p_{m+1}(.)(v') \otimes \partial_\mu \mH_{m+1}(.)(v))(\mu, s, t, x, z) + (p_{m+1} \otimes \partial^2_\mu \mH_{m+1}(.)(\gv))(\mu, s, t, x, z)$, being continuous with respect to the variables $s$, $x$, $\mu$ and $\gv$ and such that 
\begin{align*}
|\partial_\mu[(p_{m+1} \otimes \partial_\mu \mH_{m+1}(.)(v))(\mu, s, t, x, z)](v') | \leq \frac{K^{+}}{(t-s)^{1-\frac{\eta}{2}}} (1+\mathscr{C}^{1,0}_m(C^+, t-s))\, g(c(t-s), z-x).
\end{align*} 

This in turn together with the continuous differentiability of $\mu \mapsto \Phi_{m+1}(\mu, s, r, t, x, z)$, the estimate \eqref{estimate:deriv:mes:Phim} (with $n=0$, $\beta \in (0,1]$), \eqref{deriv:v:deriv:mes:pm:hat},  \eqref{bound:derivative:heat:kernel} with $n=0$, \eqref{estimate:partial:deriv:mes:parametrix:kernel:iterated:m} (with $k=1$, $n=0$, $\beta \in (0,1]$) and the dominated convergence theorem imply that $\mu \mapsto (\partial_\mu \widehat{p}_{m+1}(.)(v) + p_{m+1} \otimes \partial_\mu \mH_{m+1}(.)(v)) \otimes \Phi_{m+1}(\mu, s, t, x, z)$ is continuously differentiable, with a derivative being continuous with respect to the variables $s$, $x$, $\mu$ and $\gv$ and satisfying
\begin{align*}
| \partial_\mu[(\partial_\mu \widehat{p}_{m+1}(.)(v) + p_{m+1} \otimes \partial_\mu \mH_{m+1}(.)(v)) \otimes \Phi_{m+1}(\mu, s, t, x, z)](v')| \leq  \frac{K^{+}}{(t-s)^{1-\eta}} (1+\mathscr{C}^{1,0}_m(C^+, t-s)) \, g(c(t-s), z-x).
\end{align*}

We now come back to the identity \eqref{representation:partial:mes:pmp1}. From the above discussion, we conclude that $\mu \mapsto \partial_\mu p_{m+1}(\mu, s, t, x, z)(v) $ is continuously L-differentiable, with a derivative being continuous in $s$, $x$, $\mu$ and $\gv$ and satisfying
$$
| \partial^2_\mu p_{m+1}(\mu, s, t, x, z)(\gv) | \leq  \frac{K^{+}}{(t-s)^{1-\frac{\eta}{2}}}(1+\mathscr{C}^{1,0}_m(C^+, t-s)) \, g(c(t-s), z-x).
$$

\noindent \emph{Step 2: proof of the representation formula \eqref{representation:formula:lions:second:deriv:mes:dens}.}\\

The above estimate together with \eqref{estimate:partial:deriv:mes:parametrix:kernel:iterated:m}, \eqref{first:second:estimate:induction:decoupling:mckean}, \eqref{cor:bis:sec:mes:deriv:parametrix:kernel:s:r:t:with:beta} and the dominated convergence theorem allow to differentiate twice with respect to the measure argument $\mu$ the relation 
$$
p_{m+1}(\mu, s, t, x, z) = \widehat{p}_{m+1}(\mu, s, t, x, z) + (p_{m+1}\otimes \mH_{m+1})(\mu, s , t, x, z)
$$
which yields
\begin{align*}
\partial^2_\mu p_{m+1}(\mu, s , t, x, z)(\gv) & = \partial^2_\mu \widehat{p}_{m+1}(\mu, s, t, x, z)(\gv) + (p_{m+1} \otimes \partial^2_\mu \mH_{m+1}(.)(\gv))(\mu, s ,t, x, z) \\
& \quad  +  (\partial_\mu p_{m+1}(\cdot)(v) \otimes \partial_\mu \mH_{m+1}(\cdot)(v') )(\mu, s ,t, x, z) \\
&\quad  + (\partial_\mu p_{m+1}(.)(v') \otimes \partial_\mu \mH_{m+1}(.)(v))(\mu, s ,t, x, z) \\
& \quad + (\partial^2_\mu p_{m+1}(.)(\gv) \otimes \mH_{m+1})(\mu, s, t, x, z).
\end{align*}

Now, using the estimates on the first fourth terms of the right-hand side of the above identity derived in the first step, we conclude that one may iterate this relation. This yields the representation in infinite series \eqref{representation:formula:lions:second:deriv:mes:dens} which is absolutely convergent. \\

\noindent \emph{Step 3: proof of the estimate \eqref{estimate:deriv:sec:mes:dens:stepmp1:cor}.} \\

In order to establish \eqref{estimate:deriv:sec:mes:dens:stepmp1:cor}, we start from the representation formula \eqref{representation:formula:lions:second:deriv:mes:dens} and estimate each term of the series. First, from \eqref{bound:second:deriv:mes:hat:p} of Corollary \ref{cor:deriv:time:and:mes:phat}, \eqref{second:deriv:mes:induction:decoupling:mckean} and the space-time inequality \eqref{space:time:inequality}, we get
\begin{eqnarray}
&&|  \partial^2_\mu \widehat{p}_{m+1}(\mu, s, t, x, z)(\gv)|  \label{deriv:p:hat:mu:bound} \\
& \leq& K^+ \left\{ \frac{1}{(t-s)^{1-\frac{\eta}{2}}} + \frac{1}{t-s} \int_s^t \frac{\mathscr C^{1,0}_{m}(C^+, r-s)}{(r-s)^{1-\eta}} \, dr \right\} \, g(c(t-s) , z-x) \nonumber \\
& \leq& K^+ \left\{ \frac{1}{(t-s)^{1-\frac{ \eta}{2}}} +  \int_s^t \frac{\mathscr C^{1,0}_{m}(C^+, r-s)}{(t-r)^{1-\frac{\eta}{2}} (r-s)^{1-\frac{\eta}{2}}} \, dr \right\} \, g(c(t-s) , z-x) \nonumber \\
&  \leq&  \frac{K^+}{(t-s)^{1-\frac{\eta}{2}}}  \left\{1 + \sum_{k=1}^{m} (C^+)^{k} (t-s)^{k \frac{\eta}{2}}  \prod_{i=1}^{k} B\left(\frac{\eta}{2}, \frac{\eta}{2} +  (i-1)\frac{\eta}{2} \right)  \right\}g(c(t-s) , z-x) \notag
\end{eqnarray}
\noindent recalling as well the definition \eqref{definition:series:Cmnbeta} of $\mathscr{C}^{1,0}_m(C, t)$ for the last inequality.
%Hence, by induction on $r$,
%\begin{align*}
%|(\partial^2_\mu \widehat{p}_{m+1} \otimes \mH^{(r)}_{m+1})(\mu, s, t, x, z)(y)| & \leq  K^{r} \left\{ 1 + \sum_{k=1}^{m} C^{k} \prod_{i=1}^{k+1} B\left(\frac{\eta}{2},  i\frac{\eta}{2} \right) (t-s)^{k \frac{\eta}{2}} \right\} \\
%&  \quad \times (t-s)^{-1+\frac{\eta}{2} + r \frac{\eta}{2}} \prod_{i=1}^{r} B\left(\frac{\eta}{2},  i \frac{\eta}{2}\right) g(c(t-s), z-x)
%\end{align*}
%
%\noindent which in turn implies
%\begin{align}
%\sum_{r\geq 0}& |(\partial^2_\mu \widehat{p}_{m+1} \otimes \mH^{(r)}_{m+1})(\mu, s, t, x, z)(y)| \nonumber \\
%& \leq  \frac{K}{(t-s)^{1-\frac{\eta}{2}}} \left\{ 1 + \sum_{k=1}^{m} C^{k} \prod_{i=1}^{k+1} B\left(\frac{\eta}{2},  i\frac{\eta}{2} \right) (t-s)^{k \frac{\eta}{2}} \right\}  g(c(t-s), z-x) \nonumber \\
%& \leq \frac{K}{(t-s)^{1-\frac{\eta}{2}}} \left\{ B\left(\frac{\eta}{2}, \frac{\eta}{2} \right) + \sum_{k=1}^{m} C^{k} \prod_{i=1}^{k+1} B\left(\frac{\eta}{2}, i\frac{\eta}{2} \right) (t-s)^{k \frac{\eta}{2}} \right\}   g(c(t-s), z-x). \label{deriv:hat:pm:mu:1st:term}
%\end{align}

We now establish an upper-bound for the quantity $p_{m+1} \otimes \partial^2_\mu \mH_{m+1}(\mu, r ,t , x, z)(\gv)$. We first observe that from \eqref{cor:sec:mes:deriv:parametrix:kernel:s:r:t:with:beta}, \eqref{second:deriv:mes:induction:decoupling:mckean} and the space-time inequality \eqref{space:time:inequality}, the following estimate holds
\begin{eqnarray}\label{sec:mes:deriv:parametrix:kernel:s:r:t:with:beta:proof:2} 
&&|  \partial_\mu^2 \mH_{m+1}(\mu, s, r, t, x, z)(\gv)|  \\
& \leq & K^{+} \left( \frac{1}{(t-r)^{1-\frac{\eta}{2}}(r-s)} \wedge  \frac{1}{(t-r)(r-s)^{1-\frac{\eta}{2}}}  \right)  \notag\\
  && \quad \times \Bigg\{1 +  \mathscr{C}^{1,0}_m(C^+, r-s) (r-s)^{\frac{\eta}{2}}  + \frac{(r-s)^{ \frac{\eta}{2}}}{t-r} \int_r^t \mathscr{C}^{1, 0}_m(C^+, r'-s) dr'   \Bigg\} g(c(t-r), z-x)  .\notag 
\end{eqnarray}

Assuming now that $r \in [s, (t+s)/2]$ so that $ (t-s)/2 \leq t-r \leq t-s$, we obtain, using \eqref{bound:derivative:heat:kernel} and \eqref{sec:mes:deriv:parametrix:kernel:s:r:t:with:beta:proof:2}
\begin{eqnarray*}
&&\int_{\mathbb{R}^d} |p_{m+1}(\mu, s, r, x, y)|\, |   \partial^2_\mu \mH_{m+1}(\mu, s, r ,t , y, z)(\gv)| \, dy \\
& \leq &\frac{K^{+}}{(t-s)(r-s)^{1-\frac{\eta}{2}}}\left(1+ \mathscr{C}^{1, 0}_{m}(C^+, r-s) (r-s)^{\frac{\eta}{2}} + \frac{(r-s)^{\frac{\eta}{2}}}{t-s} \int_r^t \mathscr{C}^{1, 0}_{m}(C^{+}, r'-s)  \, dr'  \right)\, g(c(t-s),z-x)
\end{eqnarray*}

\noindent so that, by Fubini's theorem
\begin{eqnarray*}
&&\int_s^{\frac{t+s}{2}} \int_{\mathbb{R}^d}  |p_{m+1}(\mu, s, r, x, y)|\,  | \partial^2_\mu \mH_{m+1}(\mu, s, r ,t , y, z)(\gv)| \, dy \, dr \\
& \leq& K^+ \left\{ \frac{1}{(t-s)^{1-\frac{\eta}{2}}}  + \frac{1}{t-s} \int_s^t \frac{ \mathscr{C}^{1, 0}_{m}(C^+, r-s)}{(r-s)^{1-\eta}} \, dr \right\} \, g(c(t-s) , z-x) \\
& \leq& K^{+} \left\{ \frac{1}{(t-s)^{1-\frac{\eta}{2}}}  +  \int_s^t \frac{\mathscr{C}^{1,0}_{m}(C^+, r-s)}{(t-r)^{1-\frac{\eta}{2}}(r-s)^{1-\frac{\eta}{2}}} \, dr \right\}  g(c(t-s) , z-x) \\
& \leq&  \frac{K^+}{(t-s)^{1-\frac{\eta}{2}}}  \left\{ 1 + \sum_{k=1}^{m} (C^+)^{k} (t-s)^{k \frac{\eta}{2}}  \prod_{i=1}^{k} B\left(\frac{\eta}{2}, \frac{\eta}{2} +  (i-1)\frac{\eta}{2} \right) \right\} \,  g(c(t-s) , z-x).
\end{eqnarray*}

Then, assuming that $r$ belongs to $[(t+s)/2, t]$ so that $(t-s)/2 \leq r-s \leq t-s$, we similarly get
\begin{eqnarray*}
&&\int_{\mathbb{R}^d} |p_{m+1}(\mu, s, r, x, y)| \, | \partial^2_\mu \mH_{m+1}(\mu, s, r ,t , y, z)(\gv)| dy \\
&  \leq& \frac{K^+}{(t-s) (t-r)^{1-\frac{\eta}{2}}} \left(1+ \mathscr{C}^{1,0}_{m}(C^{+}, r-s) (r-s)^{\frac{\eta}{2}} + \frac{1}{t-r}\int_r^t \mathscr{C}^{1,0}_{m}(C^+, r'-s) (r'-s)^{\frac{\eta}{2}} \, dr'\right) \,g(c(t-s),z-x)
\end{eqnarray*}  

\noindent which in turn, by Fubini's theorem, yields
\begin{eqnarray*}
&&\int_{\frac{t+s}{2}}^{t} \int_{\mathbb{R}^d}  |p_{m+1}(\mu, s, r, x, y)| \,  | \partial^2_\mu \mH_{m+1}(\mu, s, r ,t , y, z)(\gv)| \, dy \, dr \\
&  \leq& \frac{K^{+}}{t-s} \int_{\frac{t+s}{2}}^{t} \frac{1}{(t-r)^{1-\frac{\eta}{2}}} \Big[1+ \mathscr{C}^{1,0}_{m}(C^{+}, r-s) (r-s)^{\frac{\eta}{2}} + \frac{1}{t-r}\int_r^t \mathscr{C}^{1,0}_{m}(C^{+}, r'-s) (r'-s)^{\frac{\eta}{2}} \, dr'  \Big] \, dr \\
& & \quad \times g(c(t-s),z-x) \\
& \leq & K^{+} \left\{ \frac{1}{(t-s)^{1-\frac{\eta}{2}}}  +  \int_s^t \frac{\mathscr{C}^{1, 0}_{m}(C^{+}, r-s)}{(t-r)^{1-\frac{\eta}{2}} (r-s)^{1- \frac{\eta}{2}}} \, dr \right\} \, g(c(t-s) , z-x) \\
& \leq & \frac{K^{+}}{(t-s)^{1-\frac{\eta}{2}}}  \left\{ 1 + \sum_{k=1}^{m} (C^{+})^{k} (t-s)^{k \frac{\eta}{2}} \prod_{i=1}^{k} B\left(\frac{\eta}{2}, \frac{\eta}{2} +  (i-1)\frac{\eta}{2} \right)  \right\} \,  g(c(t-s) , z-x).
\end{eqnarray*}

Gathering the two previous cases, we clearly obtain
\begin{eqnarray}\label{estimate:convol:pmp1:cross:deriv:mes:Hmp1}
&&| p_{m+1} \otimes \partial^2_\mu \mH_{m+1}(.)(\gv)(\mu, s ,t , x, z) |   \\
&\leq &\frac{K^{+} }{(t-s)^{1-\frac{\eta}{2}}} \left\{ 1 + \sum_{k=1}^{m} (C^{+})^{k} (t-s)^{k \frac{\eta}{2}}  \prod_{i=1}^{k} B\left( \frac{\eta}{2}, \frac{\eta}{2} + (i-1) \frac{\eta}{2}\right) \right\} \, g(c(t-s),z-x).\nonumber
\end{eqnarray}

In order to handle the two last terms appearing in the right-hand side of \eqref{representation:formula:lions:second:deriv:mes:dens}, we employ similar computations to those used above. Namely, using \eqref{first:second:estimate:induction:decoupling:mckean} and \eqref{estimate:partial:deriv:mes:parametrix:kernel:iterated:m} with $n=0$, $k=1$ and $\beta \in (0,1]$, we obtain
\begin{align*}
&\int_{s}^{t} \int_{\mathbb{R}^d}  |\partial_\mu p_{m+1}(\mu, s, r, x, y)(v)| \,  | \partial_\mu \mH_{m+1}(\mu, s, r ,t , y, z)(v')| \, dy \, dr \\
& \quad + \int_{s}^{t} \int_{\mathbb{R}^d}  |\partial_\mu p_{m+1}(\mu, s, r, x, y)(v')| \,  | \partial_\mu \mH_{m+1}(\mu, s, r ,t , y, z)(v)| \, dy \, dr\\
& \quad \quad \leq \frac{K}{(t-s)^{1-\eta}}  g(c(t-s) , z-x).
\end{align*}

Hence, collecting the previous estimates, we clearly obtain
\begin{eqnarray}\label{estimate:convol:pmp1:cross:deriv:mes:Hmp1}
&&\Big|(\partial^2_\mu \widehat{p}_{m+1}(\cdot)(\gv) + p_{m+1} \otimes \partial^2_\mu \mH_{m+1}(\cdot)(\gv)+ \partial_\mu p_{m+1}(\cdot)(v) \otimes \partial_\mu \mH_{m+1}(\cdot)(v')\notag\\
&& \qquad \qquad+\partial_\mu p_{m+1}(\cdot)(v') \otimes \partial_\mu \mH_{m+1}(.)(v))(\mu, s, t, x, z)\Big|   \\
&\leq & \frac{K^+}{(t-s)^{1-\frac{\eta}{2}}} \left\{ 1 + \sum_{k=1}^{m} (C^{+})^{k} (t-s)^{k \frac{\eta}{2}} \prod_{i=1}^{k} B\left( \frac{\eta}{2}, \frac{\eta}{2} + (i-1) \frac{\eta}{2}\right)  \right\} \, g(c(t-s),z-x)\nonumber
\end{eqnarray}

\noindent which in turn yields
\begin{eqnarray*}
&&\sum_{k\geq0} \Big(\partial^2_\mu \widehat{p}_{m+1}(\cdot)(\gv) + p_{m+1} \otimes \partial^2_\mu \mH_{m+1}(\cdot)(\gv)+ \partial_\mu p_{m+1}(\cdot)(v) \otimes \partial_\mu \mH_{m+1}(\cdot)(v')\notag\\
&& \qquad \qquad + \partial_\mu p_{m+1}(\cdot)(v') \otimes \partial_\mu \mH_{m+1}(\cdot)(v)\Big) \otimes \mH^{(k)}_{m+1}(\mu, s, t, x, z)\\
& \leq &  \frac{K^+}{(t-s)^{1-\frac{\eta}{2}}} \left\{ 1 + \sum_{k=1}^{m} (C^{+})^{k} (t-s)^{k \frac{\eta}{2}} \prod_{i=1}^{k} B\left( \frac{\eta}{2}, \frac{\eta}{2} + (i-1) \frac{\eta}{2}\right)  \right\} \, g(c(t-s),z-x).
\end{eqnarray*}
This last estimate concludes the proof of \eqref{estimate:deriv:sec:mes:dens:stepmp1:cor}.\\

\noindent \emph{Step 4: proof of the estimate \eqref{estimate:deriv:sec:mes:holder:reg:space:dens:stepmp1:cor}.}\\

 It follows from \eqref{cross:mes:deriv:holder:p:hat:s:t}, \eqref{second:deriv:mes:induction:decoupling:mckean} and the space time inequality \eqref{space:time:inequality} that for any $\beta\in [0,\eta)$
\begin{eqnarray}
&&|  \partial^2_\mu \widehat{p}_{m+1}(\mu, s, t, x, z)(\gv) - \partial^2_\mu \widehat{p}_{m+1}(\mu, s, t, x', z)(\gv)|  \label{deriv:p:hat:mu:bound} \\
& \leq& K^+ |x-x'|^\beta\left\{ \frac{1}{(t-s)^{1+\frac{\beta-\eta}{2}}} + \frac{1}{(t-s)^{1+\frac{\beta}{2}}} \int_s^t \frac{\mathscr C^{1,0}_{m}(C^+, r-s)}{(r-s)^{1-\eta}} \, dr \right\} \nonumber \\
& & \quad \times \left\{ g(c(t-s) , z-x) + g(c(t-s), z-x') \right\} \nonumber \\
& \leq& K^+ |x-x'|^\beta \left\{ \frac{1}{(t-s)^{1 + \frac{\beta- \eta}{2}}} +  \int_s^t \frac{\mathscr C^{1,0}_{m}(C^+, r-s)}{(t-r)^{1-\frac{\eta}{2}} (r-s)^{1+\frac{\beta-\eta}{2}}} \, dr \right\}  \nonumber \\
& & \quad \times \left\{ g(c(t-s) , z-x) + g(c(t-s), z-x') \right\} \nonumber \\
&  \leq& K^{+} \frac{|x-x'|^\beta}{(t-s)^{1+\frac{\beta-\eta}{2}}}  \left\{1 + \sum_{k=1}^{m} (C^+)^{k} (t-s)^{k \frac{\eta}{2}}  \prod_{i=1}^{k} B\left(\frac{\eta}{2}, \frac{\eta-\beta}{2} +  (i-1)\frac{\eta}{2} \right)  \right\} \nonumber \\
&& \quad \times \left\{ g(c(t-s) , z-x) + g(c(t-s), z-x') \right\}. \nonumber 
\end{eqnarray}

In order to handle the difference $(p_{m+1} \otimes \partial^2_\mu  \mH_{m+1}(.)(\gv))(\mu, s, t, x, z) - (p_{m+1} \otimes \partial^2_\mu \mH_{m+1}(.)(\gv))(\mu, s, t, x', z)$, we use \eqref{reg:heat:kernel:deriv} with $n=0$ and  \eqref{sec:mes:deriv:parametrix:kernel:s:r:t:with:beta:proof:2}. We also split the time integral of the space-time convolution operator into the two intervals $[s, (t+s)/2)$ and $[(t+s)/2, t]$ as we did in the previous step and perform similar computations. Skipping some technical details, we deduce that for any $\beta \in [0,\eta)$
\begin{align*}
| (p_{m+1}&  \otimes \partial^2_\mu  \mH_{m+1}(.)(\gv))(\mu, s, t, x, z) - (p_{m+1} \otimes \partial^2_\mu \mH_{m+1}(.)(\gv))(\mu, s, t, x', z) | \\
& \leq K^{+}_\beta \frac{|x-x'|^\beta}{(t-s)^{1+\frac{\beta-\eta}{2}}}\left\{1 + \sum_{k=1}^{m} (C^+)^{k} (t-s)^{k \frac{\eta}{2}}  \prod_{i=1}^{k} B\left(\frac{\eta}{2}, \frac{\eta-\beta}{2} +  (i-1)\frac{\eta}{2} \right)  \right\} \nonumber \\
& \quad \times \left\{ g(c(t-s) , z-x) + g(c(t-s), z-x') \right\}.
\end{align*}

In order to investigate the H\"older regularity of the two maps $x\mapsto (\partial_\mu p_{m+1}(\cdot)(v) \otimes \partial_\mu \mH_{m+1}(\cdot)(v'))(\mu, s, t, x, z)$, $x\mapsto (\partial_\mu p_{m+1}(\cdot)(v') \otimes \partial_\mu \mH_{m+1}(.)(v))(\mu, s, t, x, z)$, we first claim that there exist positive constants $K$ and $c$ such that for any $\beta \in [0,1]$
\begin{align}
|\partial_\mu p_{m+1}(\mu, s, t, x, z) & - \partial_\mu p_{m+1}(\mu, s, t, x', z)| \label{holder:reg:partial:deriv:mes:pmp1} \\
\quad & \leq K \frac{|x-x'|^\beta}{(t-s)^{\frac{1+\beta-\eta}{2}}} \left\{ g(c(t-s), z-x) + g(c(t-s), z- x')\right\}. \notag 
\end{align}

The above estimate directly follows from \eqref{first:second:estimate:induction:decoupling:mckean} with $n=0$ if $|x-x'| \geq (t-s)^{1/2}$ while it is a consequence of the mean-value theorem, \eqref{cross:deriv:mes:space:induction:decoupling:mckean} and the fact that for any $\lambda \in [0,1]$ and any $0 < c' < c$
\begin{equation}
\label{diagonal:regime:heat:kernel}
 \exp\left(-c \frac{|\lambda x + (1-\lambda) x' - z|^2}{t-s}\right) \leq K \exp\left(-c' \frac{|z-x|^2}{t-s}\right)
\end{equation}
\noindent in the diagonal case $|x-x'|\leq (t-s)^{1/2}$.

Then, using \eqref{holder:reg:partial:deriv:mes:pmp1} with $\beta \in [0,\eta)$ and \eqref{estimate:partial:deriv:mes:parametrix:kernel:iterated:m} with $n=0$, $k=1$ and $\beta\in (0,1]$, after some standard computations, we obtain
\begin{align}
|(\partial_\mu p_{m+1}(\cdot)(v) & \otimes \partial_\mu \mH_{m+1}(\cdot)(v'))(\mu, s, t, x, z)  - (\partial_\mu p_{m+1}(\cdot)(v) \otimes \partial_\mu \mH_{m+1}(\cdot)(v'))(\mu, s, t, x', z)| \label{holder:reg:space:partial:mes:pmp1:convol:partial:mes:parametrix:kernel}\\
&  + |(\partial_\mu p_{m+1}(\cdot)(v')  \otimes \partial_\mu \mH_{m+1}(.)(v))(\mu, s, t, x, z)  - (\partial_\mu p_{m+1}(\cdot)(v') \otimes \partial_\mu \mH_{m+1}(.)(v))(\mu, s, t, x', z)| \nonumber \\
& \leq K \frac{|x-x'|^\beta}{(t-s)^{1+\frac{\beta}{2}-\eta}} \left\{ g(c(t-s), z-x) + g(c(t-s), z- x')\right\}. \nonumber
\end{align}

Gathering the estimates \eqref{deriv:p:hat:mu:bound}, \eqref{holder:reg:partial:deriv:mes:pmp1} and \eqref{holder:reg:space:partial:mes:pmp1:convol:partial:mes:parametrix:kernel} yields
\begin{eqnarray*}
& &\Big|(\partial^2_\mu \widehat{p}_{m+1}(\cdot)(\gv) + p_{m+1} \otimes \partial^2_\mu \mH_{m+1}(\cdot)(\gv)+ (\partial_\mu p_{m+1}(\cdot)(v) \otimes \partial_\mu \mH_{m+1})(\cdot)(v')\notag\\
&& \qquad \qquad+(\partial_\mu p_{m+1}(\cdot)(v') \otimes \partial_\mu \mH_{m+1})(v))(\mu, s, t, x, z) \\
& & - \Big((\partial^2_\mu \widehat{p}_{m+1}(\cdot)(\gv) + p_{m+1} \otimes \partial^2_\mu \mH_{m+1}(\cdot)(\gv)+ (\partial_\mu p_{m+1}(\cdot)(v) \otimes \partial_\mu \mH_{m+1})(\cdot)(v')\notag\\
&& \qquad \qquad+(\partial_\mu p_{m+1}(\cdot)(v') \otimes \partial_\mu \mH_{m+1})(v))(\mu, s, t, x', z)\Big) \Big|\\
& & \leq K^{+}_\beta \frac{|x-x'|^\beta}{(t-s)^{1+\frac{\beta-\eta}{2}}}\left\{1 + \sum_{k=1}^{m} (C^+)^{k} (t-s)^{k \frac{\eta}{2}}  \prod_{i=1}^{k} B\left(\frac{\eta}{2}, \frac{\eta-\beta}{2} +  (i-1)\frac{\eta}{2} \right)  \right\} \nonumber \\
& & \quad \times \left\{ g(c(t-s) , z-x) + g(c(t-s), z-x') \right\}
\end{eqnarray*}
\noindent which in turn by \eqref{iter:parametrix:kernel} implies 
\begin{eqnarray*}
& & \sum_{k\geq 0} \Big|(\partial^2_\mu \widehat{p}_{m+1}(\cdot)(\gv) + p_{m+1} \otimes \partial^2_\mu \mH_{m+1}(\cdot)(\gv)+ (\partial_\mu p_{m+1}(\cdot)(v) \otimes \partial_\mu \mH_{m+1})(\cdot)(v')\notag\\
&& \qquad \qquad+(\partial_\mu p_{m+1}(\cdot)(v') \otimes \partial_\mu \mH_{m+1})(v))\otimes \mH^{(k)}_{m+1}(\mu, s, t, x, z) \\
& & - \Big((\partial^2_\mu \widehat{p}_{m+1}(\cdot)(\gv) + p_{m+1} \otimes \partial^2_\mu \mH_{m+1}(\cdot)(\gv)+ (\partial_\mu p_{m+1}(\cdot)(v) \otimes \partial_\mu \mH_{m+1})(\cdot)(v')\notag\\
&& \qquad \qquad+(\partial_\mu p_{m+1}(\cdot)(v') \otimes \partial_\mu \mH_{m+1})(v))\otimes \mH^{(k)}_{m+1}(\mu, s, t, x', z)\Big) \Big|\notag \\
& & \leq K^{+}_\beta \frac{|x-x'|^\beta}{(t-s)^{1+\frac{\beta-\eta}{2}}}\left\{1 + \sum_{k=1}^{m} (C^+)^{k} (t-s)^{k \frac{\eta}{2}}  \prod_{i=1}^{k} B\left(\frac{\eta}{2}, \frac{\eta-\beta}{2} +  (i-1)\frac{\eta}{2} \right)  \right\} \nonumber \\
& & \quad \times \left\{ g(c(t-s) , z-x) + g(c(t-s), z-x') \right\}.
\end{eqnarray*}

\noindent Coming back to the identity \eqref{representation:formula:lions:second:deriv:mes:dens} and using the previous estimate allow to conclude the proof of \eqref{estimate:deriv:sec:mes:holder:reg:space:dens:stepmp1:cor}.\\

\noindent \emph{Step 5: proof of the estimate \eqref{estimate:deriv:sec:mes:holder:reg:v:dens:stepmp1:cor}. \\}

In order to control the difference $\partial^2_\mu p_{m+1}(\mu, s, t, x, z)(\gv_1) - \partial^2_\mu p_{m+1}(\mu, s, t, x, z)(\gv_2)$, we start again from the representation formula \eqref{representation:formula:lions:second:deriv:mes:dens} and investigate the H\"older regularity of each term of the series with respect to the variable $\gv$. 

First, from the estimate \eqref{cross:mes:deriv:reg:holder:terminal point:p:hat:s:r:t} in Corollary \ref{cor:deriv:time:and:mes:phat}, \eqref{second:deriv:mes:reg:space:deriv:arg:estimate:induction:decoupling:mckean} and the space time inequality \eqref{space:time:inequality}, we get that for any $\beta \in [0,\eta)$
\begin{eqnarray}
&&|  \partial^2_\mu \widehat{p}_{m+1}(\mu, s, t, x, z)(\gv_1) -   \partial^2_\mu \widehat{p}_{m+1}(\mu, s, t, x, z)(\gv_2)|  \label{holder:reg:v:argument:sec:order:partial:deriv:mes:p:hat:mu} \\
& \leq& K^+_\beta |\gv_1 -\gv_2|^\beta \left\{ \frac{1}{(t-s)^{1+\frac{\beta-\eta}{2}}} + \frac{1}{t-s} \int_s^t \frac{ \mathscr{C}^{1,0}_{m}(C^+_\beta, r-s)}{(r-s)^{1+\frac{\beta}{2}-\eta}} \, dr \right\} \, g(c(t-s) , z-x) \nonumber \\
& \leq& K^+_\beta |\gv_1- \gv_2|^\beta \left\{ \frac{1}{(t-s)^{1+\frac{\beta- \eta}{2}}} +  \int_s^t \frac{\mathscr C^{1,0}_{m}(C^+_\beta, r-s)}{(t-r)^{1-\frac{\eta}{2}} (r-s)^{1+\frac{\beta-\eta}{2}}} \, dr \right\} \, g(c(t-s) , z-x) \nonumber \\
&  \leq& K^+_\beta \frac{|\gv_1-\gv_2|^\beta}{(t-s)^{1+\frac{\beta-\eta}{2}}}  \left\{1 + \sum_{k=1}^{m} (C^+_\beta)^{k} (t-s)^{k \frac{\eta}{2}}  \prod_{i=1}^{k} B\left(\frac{\eta}{2}, \frac{\eta-\beta}{2} +  (i-1)\frac{\eta}{2} \right)  \right\}g(c(t-s) , z-x). \notag
\end{eqnarray}
 
In order to handle the difference $(p_{m+1} \otimes \partial^2_\mu  \mH_{m+1}(.)(\gv_1))(\mu, s, t, x, z) - (p_{m+1} \otimes \partial^2_\mu \mH_{m+1}(.)(\gv_2))(\mu, s, t, x, z)$, we first remark that  \eqref{second:order:mes:deriv:parametrix:kernel:s:r:t:reg:holder:v:argument} together with \eqref{second:deriv:mes:reg:space:deriv:arg:estimate:induction:decoupling:mckean} and the space time inequality \eqref{space:time:inequality} implies that for any $\beta \in [0,\eta)$
\begin{align*}
& \Big|  \partial^2_\mu \mH_{m+1}(\mu, s, r ,t ,x ,z)(\gv_1) - \partial^2_\mu \mH_{m+1}(\mu, s, r ,t ,x ,z)(\gv_2) \Big| \notag \\
& \leq K^+_{\beta} |\gv_1-\gv_2|^\beta  \left\{ \frac{1}{(t-r)^{1-\frac{\eta}{2}} (r-s)^{1+\frac{\beta}{2}}} \wedge  \frac{1}{(t-r)(r-s)^{1+\frac{\beta-\eta}{2}}} \right\}  \notag \\
& \quad \times \left\{ 1 + \mathscr{C}^{1,\beta}_{m}(C^+_\beta, r-s)  (r-s)^{\frac{\eta}{2}} +  \frac{(r-s)^{\frac{\eta}{2}}}{t-r}   \int_r^t \mathscr{C}^{1,\beta}_{m}(C^{+}_\beta, r'-s) \, dr' \right\}\, g(c(t-r), z-x).
\end{align*}

We then use the pointwise Gaussian estimate \eqref{bound:derivative:heat:kernel} with $n=0$ and the previous upper-bound. We also split the time integral of the space-time convolution operator into the two disjoint intervals $[s, (t+s)/2]$ and $((t+s)/2, t]$ as we already did in the previous steps and perform similar computations. Skipping some technical details, we deduce that for any $\beta \in [0,\eta)$
\begin{align}
& | (p_{m+1}  \otimes \partial^2_\mu  \mH_{m+1}(.)(\gv_1))(\mu, s, t, x, z) - (p_{m+1} \otimes \partial^2_\mu \mH_{m+1}(.)(\gv_2))(\mu, s, t, x, z) | \label{holder:reg:v:pmp1:sec:partial:mes:parametrix:kernel} \\
& \leq K^{+}_\beta \frac{|\gv_1-\gv_2|^\beta}{(t-s)^{1+\frac{\beta-\eta}{2}}}\left\{1 + \sum_{k=1}^{m} (C^+_\beta)^{k} (t-s)^{k \frac{\eta}{2}}  \prod_{i=1}^{k} B\left(\frac{\eta}{2}, \frac{\eta-\beta}{2} +  (i-1)\frac{\eta}{2} \right)  \right\} \, g(c(t-s) , z-x). \nonumber
\end{align}

For the last two terms of the series \eqref{representation:formula:lions:second:deriv:mes:dens}, we first claim that for any $\beta \in [0,1]$
$$
| \partial_\mu p_{m+1}(\mu, s, t, x, z)(v_1) - \partial_\mu p_{m+1}(\mu, s, t, x, z)(v_2) | \leq K \frac{|v_1-v_2|^\beta}{(t-s)^{\frac{1+\beta-\eta}{2}}} \, g(c(t-s), z-x).
$$
The previous estimate is a direct consequence of \eqref{first:second:estimate:induction:decoupling:mckean} with $n=0$ if $|v_1-v_2| \geq (t-s)^{1/2}$ or the mean-value theorem combined with \eqref{first:second:estimate:induction:decoupling:mckean} with $n=1$ if $|v_1-v_2| < (t-s)^{1/2}$. Similarly, separating the computations into the two cases $|v_1-v_2|\geq (r-s)^{1/2}$ and $|v_1-v_2|< (r-s)^{1/2}$ and using \eqref{estimate:partial:deriv:mes:parametrix:kernel:iterated:m} with $k=1$ and $n=0$ or $n=1$, we obtain 
$$
| \partial_\mu \mH_{m+1}(\mu, s, r, t, x, z)(v_1) - \partial_\mu \mH_{m+1}(\mu, s,  r, t, x, z)(v_2) | \leq K_{\beta'} \frac{|v_1-v_2|^\beta}{(r-s)^{\frac{1+\beta}{2}-(1-\beta')\frac{\eta}{2}}(t-r)^{1 - \beta' \frac{\eta}{2}}} \, g(c(t-s), z-x)
$$
\noindent for any $\beta \in [0,1]$ and any $\beta' \in (0,1]$.

Now, it follows from the two previous estimates as well as \eqref{first:second:estimate:induction:decoupling:mckean} with $n=0$, \eqref{estimate:partial:deriv:mes:parametrix:kernel:iterated:m} with $k=1$, $n=0$, $\beta=1$ and some standard computations that for any $\beta \in [0,\eta)$
\begin{align}
|(\partial_\mu p_{m+1}& (\cdot)(v_1)  \otimes \partial_\mu \mH_{m+1}(\cdot)(v_1'))(\mu, s, t, x, z)  - (\partial_\mu p_{m+1}(\cdot)(v_2) \otimes \partial_\mu \mH_{m+1}(\cdot)(v_2'))(\mu, s, t, x, z)| \label{holder:reg:partial:mes:pmp1:convol:partial:mes:parametrix:kernel}\\
& + |(\partial_\mu p_{m+1}(\cdot)(v_1')  \otimes \partial_\mu \mH_{m+1}(.)(v_1))(\mu, s, t, x, z)  - (\partial_\mu p_{m+1}(\cdot)(v_2') \otimes \partial_\mu \mH_{m+1}(.)(v_2))(\mu, s, t, x, z)| \nonumber \\
&  \qquad\leq K \frac{|\gv_1-\gv_2|^\beta}{(t-s)^{1+\frac{\beta}{2}-\eta}} g(c(t-s), z-x). \nonumber
\end{align}

Gathering the estimates \eqref{holder:reg:v:argument:sec:order:partial:deriv:mes:p:hat:mu}, \eqref{holder:reg:v:pmp1:sec:partial:mes:parametrix:kernel} and \eqref{holder:reg:partial:mes:pmp1:convol:partial:mes:parametrix:kernel} yields
\begin{eqnarray*}
& &\Big|(\partial^2_\mu \widehat{p}_{m+1}(\cdot)(\gv_1) + p_{m+1} \otimes \partial^2_\mu \mH_{m+1}(\cdot)(\gv_1)+ \partial_\mu p_{m+1}(\cdot)(v_1) \otimes \partial_\mu \mH_{m+1}(\cdot)(v_1')\notag\\
&& \qquad \qquad+ \partial_\mu p_{m+1}(\cdot)(v_1') \otimes \partial_\mu \mH_{m+1}(.)(v_1))(\mu, s, t, x, z) \\
& & - \Big((\partial^2_\mu \widehat{p}_{m+1}(\cdot)(\gv_2) + p_{m+1} \otimes \partial^2_\mu \mH_{m+1}(\cdot)(\gv_2)+ \partial_\mu p_{m+1}(\cdot)(v_2) \otimes \partial_\mu \mH_{m+1}(\cdot)(v_2')\notag\\
&& \qquad \qquad + \partial_\mu p_{m+1}(\cdot)(v_2') \otimes \partial_\mu \mH_{m+1}(.)(v_2))(\mu, s, t, x, z)\Big) \Big|\\
& & \leq K^{+}_\beta \frac{|\gv_1-\gv_2|^\beta}{(t-s)^{1+\frac{\beta-\eta}{2}}}\left\{1 + \sum_{k=1}^{m} (C^+_\beta)^{k} (t-s)^{k \frac{\eta}{2}}  \prod_{i=1}^{k} B\left(\frac{\eta}{2}, \frac{\eta-\beta}{2} +  (i-1)\frac{\eta}{2} \right)  \right\} \nonumber \\
& & \quad \times g(c(t-s) , z-x) 
\end{eqnarray*}
\noindent which in turn by \eqref{iter:parametrix:kernel} implies
\begin{eqnarray*}
& & \sum_{k\geq 0} \Big( (\partial^2_\mu \widehat{p}_{m+1}(\cdot)(\gv_1) + p_{m+1} \otimes \partial^2_\mu \mH_{m+1}(\cdot)(\gv_1)+  \partial_\mu p_{m+1}(\cdot)(v_1) \otimes \partial_\mu \mH_{m+1}(\cdot)(v_1')\notag\\
&& \qquad \qquad+ \partial_\mu p_{m+1}(\cdot)(v_1') \otimes \partial_\mu \mH_{m+1}(.)(v_1)) \\
& & \qquad - (\partial^2_\mu \widehat{p}_{m+1}(\cdot)(\gv_2) + p_{m+1} \otimes \partial^2_\mu \mH_{m+1}(\cdot)(\gv_2)+ \partial_\mu p_{m+1}(\cdot)(v_2) \otimes \partial_\mu \mH_{m+1}(\cdot)(v_2')\notag\\
&& \qquad \qquad+ \partial_\mu p_{m+1}(\cdot)(v_2') \otimes \partial_\mu \mH_{m+1}(.)(v_2) \Big)\otimes \mH^{(k)}_{m+1}(\mu, s, t, x, z) \\
& & \leq K^{+}_\beta \frac{|\gv_1-\gv_2|^\beta}{(t-s)^{1+\frac{\beta-\eta}{2}}}\left\{1 + \sum_{k=1}^{m} (C^+_\beta)^{k} (t-s)^{k \frac{\eta}{2}}  \prod_{i=1}^{k} B\left(\frac{\eta}{2}, \frac{\eta-\beta}{2} +  (i-1)\frac{\eta}{2} \right)  \right\} \, g(c(t-s) , z-x).
\end{eqnarray*}

\noindent Coming back to the identity \eqref{representation:formula:lions:second:deriv:mes:dens} and using the previous estimate allow to conclude the proof of \eqref{estimate:deriv:sec:mes:holder:reg:v:dens:stepmp1:cor}.\\

\end{proof}

\subsection{Proof of the second part of the induction step.}\label{sec:proof:second:part:induction:step}
We here prove the second part of the induction step, that is, the estimates \eqref{second:deriv:mes:reg:mes:estimate:induction:decoupling:mckean} and \eqref{regularity:time:estimate:secon:deriv:decoupling:mckean} at step $m+1$ under the additional assumption \HRpp. We importantly emphasize that in what follows we will use the results established in the first step. In particular, we will use the estimates \eqref{cross:deriv:mes:space:induction:decoupling:mckean} to \eqref{second:deriv:mes:reg:space:deriv:arg:estimate:induction:decoupling:mckean2} which now hold for any positive integer $m$. Moreover, the estimate \eqref{estimate:convol:pmp1:cross:deriv:mes:Hmp1} (with the choice $C^{+}= K^{+}$ as discussed just after the statement of Proposition \ref{prop:second:order:deriv:estimate:reg:holder:space:and:v}) established in the third step of the proof of Proposition \ref{prop:second:order:deriv:estimate:reg:holder:space:and:v} now writes
\begin{align}
\label{estimate:convol:pmp1:cross:deriv:mes:Hmp1:with:first:induction:step:done}
&\Big|(\partial^2_\mu \widehat{p}_{m+1}(\cdot)(\gv) + p_{m+1} \otimes \partial^2_\mu \mH_{m+1}(\cdot)(\gv)+ (\partial_\mu p_{m+1}(\cdot)(v) \otimes \partial_\mu \mH_{m+1})(\cdot)(v')\\
& \qquad \qquad+(\partial_\mu p_{m+1}(\cdot)(v') \otimes \partial_\mu \mH_{m+1})(v))(\mu, s, t, x, z)\Big| \nonumber   \\
&\qquad \leq \frac{\mathscr{C}^{1,0}_{m+1}(C^+, t-s)}{(t-s)^{1-\frac{\eta}{2}}} \, g(c(t-s),z-x)\nonumber
\end{align}

\noindent and recall that $\mathscr{C}^{1,0}_{m+1}(C^+, t-s) \leq K^{+}:= \lim_{m\rightarrow \infty}\mathscr{C}^{1,0}_{m}(C^+, t-s)< \infty$.
As in the previous step, we first need to establish some technical auxiliary estimates regarding the regularity of the coefficients, the Gaussian kernel $\widehat{p}_{m+1}$ and the parametrix kernel $\mH_{m+1}$ with respect to the initial measure $\mu$ and starting time $s$. We thus start with the following lemmas whose proofs are postponed to subsections \ref{proof:lem:reg:mes:deriv:sec:mes:coeff:parametrix:density:kernel} and \ref{proof:lem:reg:time:deriv:sec:mes:coeff:parametrix:density:kernel}. The reader may want to skip their proofs in a first reading and jump to the more natural and intuitive result stated in Proposition \ref{prop:second:part:induction:step} which actually corresponds to the heart of the proof of the second part of the induction step.

\begin{lem}\label{lem:reg:mes:deriv:sec:mes:coeff:parametrix:density:kernel}
For any $\beta \in [0,\eta)$, there exists a positive constant $K^{++}_\beta$ such that for any $(s, x, \gv) \in [0,t) \times \mathbb{R}^d \times (\mathbb{R}^d)^2$, any $r \in [s, t)$, any $\mu, \mu' \in \pp$ and any $(i , j) \in \left\{1, \cdots, d \right\}^2$
\begin{align}
&|\partial^2_\mu [a_{i, j}(t, x,  [X^{s, \xi, (m)}_t])](\gv) - \partial^2_\mu [a_{i, j}(t, x,  [X^{s, \xi, (m)}_t])]_{|\mu=\mu'}(\gv)| \notag\\
&\quad + |\partial^2_\mu [b_i(t, x,  [X^{s, \xi, (m)}_t])](\gv) - \partial^2_\mu [b_i(t, x,  [X^{s, \xi, (m)}_t])]_{|\mu =\mu'}(\gv)|  \label{diff:mes:second:deriv:mes:a:b}\\
&  \quad \leq K^{++}_\beta \left\{ \frac{W_2(\mu, \mu')^{\beta}}{(t-s)^{1+\frac{\beta-\eta}{2}}}  + \int_{(\mathbb{R}^d)^2} (|y'-x'|^\eta \wedge 1) |\partial^2_\mu p_m(\mu, s, t, x', y')(\gv) - \partial^2_\mu p_m(\mu', s, t, x', y')(\gv)|  \,  \mu(dx') \, dy'\right\},\nonumber
\end{align}

\begin{align}
& | \partial^2_\mu \widehat{p}_{m+1}(\mu, s, r, t, x, z)(\gv) - \partial^2_\mu \widehat{p}_{m+1}(\mu', s, r, t, x, z)(\gv)  | \notag\\
& \quad \leq  K^{++}_\beta  \Bigg\{ W_2(\mu, \mu')^{\beta} \Bigg(\frac{1}{(t-s)^{ 1+\frac{\beta-\eta}{2}}} \I_\seq{r=s} + \frac{1}{(r-s)^{ 1+\frac{\beta-\eta}{2}}} \mathbf \I_\seq{r>s} \Bigg ) \label{diff:lions:deriv:pm:mu:mup} \\
& \qquad \qquad  + \frac{1}{t-r}   \int_r^t \int_{(\mathbb{R}^d)^2} (|y'-x'|^\eta \wedge 1)|\partial^2_\mu p_m(\mu, s, r', x', y')(\gv) - \partial^2_\mu p_m(\mu', s, r', x', y')(\gv)|   \,  \mu(dx') \, dy' dr' \Bigg\}  \notag\\
& \qquad \qquad \qquad \times  g(c(t-r), z-x),\notag 
\end{align}

\begin{align}
& | \partial^2_\mu \big[a_{i, j}(t, x, [X^{s, \xi , (m)}_t]) - a_{i, j}(t, z, [X^{s, \xi , (m)}_t])\big](\gv) -  \partial^2_\mu \big[a_{i, j}(t, x, [X^{s, \xi , (m)}_t]) - a_{i, j}(t, z, [X^{s, \xi , (m)}_t])\big]_{| \mu= \mu'}(\gv) |\notag\\
&\leq  K^{++}_\beta \Bigg\{ W_2(\mu, \mu')^{\beta} \Bigg(\frac{|x-z|^{\eta} \wedge 1}{(t-s)^{1+\frac{\beta}{2}}} \wedge \frac{1}{(t-s)^{1+\frac{\beta-\eta}{2}}}\Bigg) \label{diff:lions:deriv:a:mu:mua}  \\ 
& \quad +   \Bigg((|x-z|^{\eta} \wedge 1)\int_{(\mathbb{R}^d)^2} | \partial_\mu^2 p_m(\mu, s, t, x', y')(\gv) -  \partial_\mu^2 p_m(\mu', s, t, x', y')(\gv) | \, \mu(dx') \, dy' \notag\\
&\quad\quad\quad\quad\quad\quad \wedge \int_{(\mathbb{R}^d)^2} (|y'-x'|^{\eta} \wedge 1)  | \partial_\mu^2 p_m(\mu, s, t, x', y')(\gv) -  \partial_\mu^2 p_m(\mu', s, t, x', y')(\gv) | \, \mu(dx') \, dy' \Bigg) \Bigg\},\notag
\end{align}

\begin{align}
 & |\partial_\mu^2 \mH_{m+1}(\mu, s, r, t, x, z)(\gv) - \partial_\mu^2 \mH_{m+1}(\mu', s, r, t, x, z)(\gv) | \nonumber \\
 & \quad \leq K^{++}_\beta \left( W_2(\mu, \mu')^{\beta} \left\{ \frac{1}{(t-r)(r-s)^{1+\frac{\beta-\eta}{2}}} \wedge \frac{1}{(t-r)^{1- \frac{\eta}{2}}(r-s)^{1+\frac{\beta}{2}}} \right\} \right. \nonumber \\
 & \left. \quad + \left\{ \frac{1}{(t-r)^{1-\frac{\eta}{2}}} \int_{(\mathbb{R}^d)^2}  | \partial_\mu^2 p_m(\mu, s, r, x', y')(\gv) -  \partial_\mu^2 p_m(\mu', s, r, x', y')(\gv) |  \, \mu(dx') \, dy'   \right. \right. \label{diff:mes:L:deriv:parametrix:kernel:pmp1} \\
& \quad \quad \left. \left. \wedge \, \frac{1}{t-r} \int_{(\mathbb{R}^d)^2} (|y'-x'|^\eta \wedge 1)  | \partial_\mu^2 p_m(\mu, s, r, x', y')(\gv) -  \partial_\mu^2 p_m(\mu', s, r, x', y')(\gv) |  \,  \mu(dx') \, dy'  \right\} \right. \nonumber  \\
 & \left. \quad +  \frac{1}{(t-r)^{2-\frac{\eta}{2}}}   \int_r^t \int_{(\mathbb{R}^d)^2} (|y'-x'|^{\eta} \wedge1) | \partial_\mu^2 p_m(\mu, s, r', x', y')(\gv) -  \partial_\mu^2 p_m(\mu', s, r', x', y')(\gv) | \, \mu(dx') \, dy' \, dr' \right) \nonumber \\
 & \quad \quad  \times g(c(t-r), z-x).\nonumber
\end{align}

\end{lem}

\begin{lem}\label{lem:reg:time:deriv:sec:mes:coeff:parametrix:density:kernel}
For any $\beta \in [0,\eta)$, there exist positive constants $K^{++}_\beta$ and $c$ such that for any $t\in (0,T]$, any $(\mu, s_1, s_2, x, y,  z) \in \pp \times [0,t)^2 \times (\mathbb{R}^d)^3$, any $\gv \in (\mathbb{R}^d)^2$, any $r\in (s_1\vee s_2, t)$ and any $(i, j) \in \left\{1, \cdots, d\right\}^2$
\begin{align}
&|\partial^2_\mu [a_{i, j}(t, x,  [X^{s_1, \xi, (m)}_t])](\gv) - \partial^2_\mu [a_{i, j}(t, x,  [X^{s_2, \xi, (m)}_t])](\gv)| \notag\\
&\quad + |\partial^2_\mu [b_i(t, x,  [X^{s_1, \xi, (m)}_t])](\gv) - \partial^2_\mu [b_i(t, x,  [X^{s_2, \xi, (m)}_t])](\gv)|  \label{diff:time:second:deriv:mes:a:b}\\
&  \quad \leq K^{++}_\beta \left\{ \frac{|s_1-s_2|^{\frac{\beta}{2}}}{(t-s_1\vee s_2)^{1+\frac{\beta-\eta}{2}}}  + \int_{(\mathbb{R}^d)^2} (|y'-x'|^\eta \wedge 1) |\partial^2_\mu p_m(\mu, s_1, t, x', y')(\gv) - \partial^2_\mu p_m(\mu, s_2, t, x', y')(\gv)|  \, \mu(dx') \,  dy'  \right\}, \notag
\end{align}

\begin{align}
& | \partial^2_\mu \widehat{p}^y_{m+1}(\mu, s_1, r, t, x, z)(\gv) - \partial^2_\mu \widehat{p}^y_{m+1}(\mu, s_2, r, t, x, z)(\gv)  | \notag\\
& \leq K^{++}_\beta \left\{ \frac{|s_1-s_2|^{\frac{\beta}{2}}}{(r-s_1\vee s_2)^{1+\frac{\beta-\eta}{2}}} \right. \label{diff:time:second:L:deriv:p:hat} \\
 &    \Big. \left.   + \frac{1}{t-r} \int_{r}^{t} \int_{(\mathbb{R}^d)^2} (|y'-x'|^{\eta}\wedge 1)  | \partial^2_\mu p_{m}(\mu, s_1, r', x', y')(\gv) - \partial^2_\mu p_{m}(\mu, s_2, r', x', y')(\gv) |   \,  \mu(dx') \, dy' \, dr' \right\} \notag \\
 & \quad \quad \times g(c(t-r), z-x), \notag 
\end{align}

 \begin{align}
  |& \partial^2_\mu \widehat{p}^y_{m+1}(\mu, s_1, t, x, z)(\gv) - \partial^2_\mu \widehat{p}^y_{m+1}(\mu, s_2, t, x, z)(\gv)  | \notag \\
  & \leq K^{++}_\beta \bigg\{ \frac{|s_1-s_2|^{\frac{\beta}{2}}}{(t-s_1)^{1+\frac{\beta-\eta}{2}}} g(c(t-s_1), z-x) +  \frac{|s_1-s_2|^{\frac{\beta}{2}}}{(t-s_2)^{1+\frac{\beta-\eta}{2}}} g(c(t-s_2), z-x)  \label{diff:time:second:L:deriv:p:hat:same:time} \\
 &       + \frac{1}{t-s_1 \vee s_2} \int_{s_1 \vee s_2}^{t} \int_{(\mathbb{R}^d)^2} (|y'-x'|^{\eta}\wedge 1)  | \partial^2_\mu p_{m}(\mu, s_1, r', x', y')(\gv) - \partial^2_\mu p_{m}(\mu, s_2, r', x', y')(\gv)  |   \, \mu(dx')\, dy' \, dr'  \notag \\
 &  \Big.  \quad \times g(c(t-s_1 \vee s_2), z-x) \Big. \bigg\} , \notag 
 \end{align}
 
\begin{align}
& \bigg| \partial^2_\mu \big[a_{i, j}(t, x, [X^{s_1, \xi , (m)}_t]) - a_{i, j}(t, z, [X^{s_1, \xi , (m)}_t])\big](\gv) -  \partial^2_\mu \big[a_{i, j}(t, x, [X^{s_2, \xi , (m)}_t]) - a_{i, j}(t, z, [X^{s_2, \xi , (m)}_t])\big](\gv) \bigg|\notag\\
&\leq  K^{++}_\beta \Bigg\{ |s_1-s_2|^{\frac{\beta}{2}} \Bigg(\frac{|x-z|^{\eta} \wedge 1}{(t-s_1\vee s_2)^{1+\frac{\beta}{2}}} \wedge \frac{1}{(t-s_1\vee s_2)^{1+\frac{\beta-\eta}{2}}}\Bigg) \label{diff:time:second:deriv:mes::diff:space:a:mu:mua}  \\ 
& \quad +   \Bigg((|x-z|^{\eta} \wedge 1)\int_{(\mathbb{R}^d)^2} | \partial_\mu^2 p_m(\mu, s_1, t, x', y')(\gv) -  \partial_\mu^2 p_m(\mu, s_2, t, x', y')(\gv) | \,  \mu(dx') \, dy' \notag\\
&\quad\quad\quad\quad\quad\quad \wedge \int_{(\mathbb{R}^d)^2} (|y'-x'|^{\eta} \wedge 1)  | \partial_\mu^2 p_m(\mu, s_1, t, x', y')(\gv) -  \partial_\mu^2 p_m(\mu, s_2, t, x', y')(\gv) | \,  \mu(dx')\, dy'\Bigg) \Bigg\},\notag
\end{align}

\begin{align}
 & |\partial_\mu^2 \mH_{m+1}(\mu, s_1, r, t, x, z)(\gv) - \partial_\mu^2 \mH_{m+1}(\mu, s_2, r, t, x, z)(\gv) | \nonumber \\
 & \quad \leq K^{++}_\beta \left( |s_1-s_2|^{\frac{\beta}{2}} \left\{ \frac{1}{(t-r)(r-s_1\vee s_2)^{1+\frac{\beta-\eta}{2}}} \wedge \frac{1}{(t-r)^{1- \frac{\eta}{2}}(r-s_1\vee s_2)^{1+\frac{\beta}{2}}} \right\} \right. \nonumber \\
 & \left. \quad + \left\{ \frac{1}{(t-r)^{1-\frac{\eta}{2}}} \int_{(\mathbb{R}^d)^2}  | \partial_\mu^2 p_m(\mu, s_1, r, x', y')(\gv) -  \partial_\mu^2 p_m(\mu, s_2, r, x', y')(\gv) |  \,  \mu(dx')\, dy'  \right. \right. \label{diff:time:second:L:deriv:parametrix:kernel:pmp1} \\
& \quad \quad \left. \left. \wedge \, \frac{1}{t-r} \int_{(\mathbb{R}^d)^2} (|y'-x'|^\eta \wedge 1)  | \partial_\mu^2 p_m(\mu, s_1 , r, x', y')(\gv) -  \partial_\mu^2 p_m(\mu', s_2 , r, x', y')(\gv) |  \,  \mu(dx') \, dy'  \right\} \right. \nonumber  \\
 & \left. \quad +  \frac{1}{(t-r)^{2-\frac{\eta}{2}}}   \int_r^t \int_{(\mathbb{R}^d)^2} (|y'-x'|^{\eta} \wedge1) | \partial_\mu^2 p_m(\mu, s_1, r', x', y')(\gv) -  \partial_\mu^2 p_m(\mu, s_2, r', x', y')(\gv) | \,  \mu(dx')\, dy'  \, dr' \right) \nonumber \\
 & \quad \quad  \times g(c(t-r), z-x).\nonumber
\end{align}

\end{lem}

\begin{prop}\label{prop:second:part:induction:step}For any $\beta  \in [0,\eta)$ there exist positive constants $K^{++}_\beta$ and $c$ such that for any $(t, x, z)\in (0,T] \times (\mathbb{R}^d)^2$, any $s, s_1, s_2 \in [0,t)$, any $\mu, \mu' \in \pp$, any $\gv \in (\mathbb{R}^d)^2$ and any value of the constant $C_\beta^{++}$ appearing in the right-hand side of the estimates \eqref{second:deriv:mes:reg:mes:estimate:induction:decoupling:mckean} and \eqref{regularity:time:estimate:secon:deriv:decoupling:mckean} 
\begin{align}
|& \partial_\mu^2 p_{m+1}(\mu, s, t, x, z)(\gv) - \partial_\mu^2 p_{m+1}(\mu', s, t, x, z)(\gv) | \nonumber \\
& \leq   K^{++}_\beta \frac{W_2(\mu, \mu')^{\beta}}{(t-s)^{1+\frac{\beta-\eta}{2}}} \left\{ 1 + \sum_{k=1}^{m} (C_\beta^{++})^{k} (t-s)^{k \frac{\eta}{2}}  \prod_{i=1}^{k} B\left(\frac{\eta}{2}, \frac{\eta-\beta}{2} +  (i-1)\frac{\eta}{2} \right)  \right\} \label{estimate:deriv:mes:holder:reg:mes:mes:dens:stepmp1:cor} \\
& \quad \times  g(c(t-s), z-x) \nonumber
\end{align}
\noindent and
\begin{align}
| & \partial^2_\mu p_{m+1}(\mu, s_1, t, x, z)(\gv) - \partial^2_\mu p_{m+1}(\mu, s_2, t, x, z)(\gv) | \nonumber \\
& \leq   K^{++}_{\beta}  \left\{ 1 + \sum_{k=1}^{m} (C_\beta^{++})^{k} (t-s_1 \vee s_2)^{k \frac{\eta}{2}}  \prod_{i=1}^{k} B\left(\frac{\eta}{2}, \frac{\eta-\beta}{2} +  (i-1)\frac{\eta}{2} \right)  \right\} \label{estimate:deriv:mes:holder:reg:time:dens:stepmp1:cor} \\
& \quad \times  \left\{ \frac{|s_1-s_2|^{\frac{\beta}{2}}}{(t-s_1)^{1+\frac{\beta-\eta}{2}}} g(c(t-s_1), z-x) + \frac{|s_1-s_2|^{\frac{\beta}{2}}}{(t-s_2)^{1+\frac{\beta-\eta}{2}}} g(c(t-s_2), z-x) \right\}. \nonumber
\end{align}

\end{prop}

\noindent \emph{Conclusion of the second part of the induction step:}\\
Similarly to the conclusion of the first part of the induction step, we set the constant $C_\beta^{++}$ in the mth partial sums $\mathscr{C}^{1, \beta}_m(C^{++}_{\beta}, t-s)$ appearing in the statement of the Gaussian estimates \eqref{second:deriv:mes:reg:mes:estimate:induction:decoupling:mckean} and \eqref{regularity:time:estimate:secon:deriv:decoupling:mckean} to be equal to the constant $K^{++}_\beta$ appearing in the right-hand side of the Gaussian estimates \eqref{estimate:deriv:mes:holder:reg:mes:mes:dens:stepmp1:cor} and \eqref{estimate:deriv:mes:holder:reg:time:dens:stepmp1:cor}. In doing so, from the above result and by the very definition of $\mathscr{C}^{1, \beta}_{m+1}(K^{++}_\beta, t-s)$, we conclude that the estimates \eqref{estimate:deriv:mes:holder:reg:mes:mes:dens:stepmp1:cor} and \eqref{estimate:deriv:mes:holder:reg:time:dens:stepmp1:cor} directly yield the desired estimates \eqref{second:deriv:mes:reg:mes:estimate:induction:decoupling:mckean} and \eqref{regularity:time:estimate:secon:deriv:decoupling:mckean} at step $m+1$. We thus conclude that the Gaussian estimates \eqref{second:deriv:mes:reg:mes:estimate:induction:decoupling:mckean} and \eqref{regularity:time:estimate:secon:deriv:decoupling:mckean} hold for any positive integer $m$. This completes the proof of the second of part of Proposition \ref{proposition:reg:density:recursive:scheme:mckean}.

\begin{proof}[Proof of Proposition \ref{prop:second:part:induction:step}.]

\noindent \emph{Step 1: proof of the estimate \eqref{estimate:deriv:mes:holder:reg:mes:mes:dens:stepmp1:cor}.\\}

From the identities \eqref{representation:formula:lions:second:deriv:mes:dens} and \eqref{infinite:series:Phi:step:m}, it holds
\begin{align}
\partial^2_\mu & p_{m+1}(\mu, s , t, x, z)(\gv) \notag \\
& = \partial^2_\mu \widehat{p}_{m+1}(\mu, s, t, x, z)(\gv) + (p_{m+1} \otimes \partial^2_\mu \mH_{m+1}(.)(\gv))(\mu, s, t, x, z) \notag \\
& \quad + (\partial_\mu p_{m+1}(\cdot)(v) \otimes \partial_\mu \mH_{m+1}(\cdot)(v'))(\mu, s, t, x, z)+ (\partial_\mu p_{m+1}(\cdot)(v') \otimes \partial_\mu \mH_{m+1}(.)(v))(\mu, s, t, x, z) \label{eq:decom:deltamumu':inter}\\
& + \Big(\Big(\partial^2_\mu \widehat{p}_{m+1}()(\gv) + p_{m+1} \otimes \partial^2_\mu \mH_{m+1}(\cdot)(\gv) \notag \\
& \quad + \partial_\mu p_{m+1}(\cdot)(v) \otimes \partial_\mu \mH_{m+1}(\cdot)(v') + \partial_\mu p_{m+1}(\cdot)(v') \otimes \partial_\mu \mH_{m+1}(.)(v) \Big) \otimes \Phi_{m+1}\Big)(\mu, s, t, x, z).
 \notag
\end{align}

We now investigate the H\"older regularity with respect to the variable $\mu $ of each term in the above decomposition. First, from \eqref{diff:lions:deriv:pm:mu:mup}, \eqref{second:deriv:mes:reg:mes:estimate:induction:decoupling:mckean} at step $m$ and the space time inequality \eqref{space:time:inequality}, we obtain
\begin{eqnarray}
&&| \partial^2_\mu \widehat{p}_{m+1}(\mu, s, t, x, z)(\gv) - \partial^2_\mu \widehat{p}_{m+1}(\mu', s, t, x, z)(\gv) |  \label{esti:1:param:deltamu}\\
&\leq& K_\beta^{++} \left\{ \frac{1}{(t-s)^{1+\frac{\beta-\eta}{2}}} + \frac{1}{t-s} \int_s^t  \frac{\mathscr C_{m}^{1, \beta}(C_\beta^{++},r-s)}{(r-s)^{1+\frac{\beta}{2}-\eta}} dr \right\} \, W_2(\mu, \mu')^{\beta} \,  g(c(t-s), z-x) \notag\\
& \leq& \frac{K_\beta^{++}}{(t-s)^{1+\frac{\eta-\beta}{2}}} \left\{ 1 +\sum_{k=1}^{m} (C^{++}_\beta)^{k} (t-s)^{k \frac{\eta}{2}}  \prod_{i=1}^{k} B\left(\frac{\eta}{2}, \frac{\eta-\beta}{2} +  (i-1)\frac{\eta}{2} \right) \right\}\notag\\
&& \quad \times W_2(\mu, \mu')^{\beta} \, g(c(t-s), z-x)\notag
\end{eqnarray}
for any $\beta$ in $[0,\eta)$. Next, separating the time integral of the space time convolution operator into the two disjoint intervals $[s, (t+s)/2]$ and $((t+s)/2, t]$, we obtain from \eqref{sec:mes:deriv:parametrix:kernel:s:r:t:with:beta} (with $\beta = 0$ on $[s, (t+s)/2]$ and $\beta=1$ on $((t+s)/2, t]$) combined with \eqref{second:deriv:mes:induction:decoupling:mckean}, the space time inequality \eqref{space:time:inequality} and \eqref{regularity:measure:estimate:v1:v2:v3:decoupling:mckean} with $n=0$ that
 \begin{align}
 \int_s^t \int_{\mathbb{R}^d} |p_{m+1}(\mu, s, r, x, y) & - p_{m+1}(\mu', s, r, x, y)|   |\partial^2_\mu \mH_{m+1}(\mu, s, r, t, y, z)(\gv)|\, dr \, dy \label{esti:2:param:deltamu} \\
 &  \leq \frac{K_\beta^{+}}{(t-s)^{1+ \frac{\beta-\eta}{2}}} \, W_2(\mu, \mu')^{\beta} \, g(c(t-s), z-x) \notag
\end{align}

\noindent for any $\beta$ in $[0,\eta)$. Then, by splitting again the time interval as previously done, we can obtain with \eqref{bound:derivative:heat:kernel}, \eqref{diff:mes:L:deriv:parametrix:kernel:pmp1} combined with \eqref{second:deriv:mes:reg:mes:estimate:induction:decoupling:mckean} at step $m$ and the space time inequality \eqref{space:time:inequality}, after some standard computations that we omit, that
 \begin{eqnarray}\label{esti:3:param:deltamu}
&&  \int_s^t \int_{\mathbb{R}^d}  p_{m+1}(\mu', s, r, x, y)  \,  | \partial^2_\mu \mH_{m+1}(\mu, s, r,  t, y, z)(\gv) -  \partial^2_\mu \mH_{m+1}(\mu', s, r, t, y, z)(\gv) |\,  dy \, dr \, \\
 & \leq& \frac{K_\beta^{++}}{(t-s)^{1+\frac{\beta-\eta}{2}}} \left( 1  + \sum_{k=1}^{m} (C_\beta^{++})^{k}  (t-s)^{k \frac{\eta}{2}} \prod_{i=1}^{k} B\left( \frac{\eta}{2}, \frac{\eta - \beta}{2} + (i-1) \frac{\eta}{2}\right) \  \right)\notag \\ 
 && \quad \times W_2(\mu, \mu')^{\beta} \, g(c(t-s), z-x).\notag
 \end{eqnarray}

By symmetry, the third and fourth term appearing on the right-hand side of \eqref{eq:decom:deltamumu':inter} are handled by similar arguments. In particular, from \eqref{diff:mes:first:L:deriv:parametrix:kernel:pmp1}, \eqref{regularity:measure:estimate:v1:v2:decoupling:mckean}, \eqref{first:second:estimate:induction:decoupling:mckean} and \eqref{estimate:partial:deriv:mes:parametrix:kernel:iterated:m} with $k=1$, $n=0$ and $\beta=1$, after some standard computations that we omit, for any $\beta \in [0,\eta)$, we get
\begin{align}
\Big| & (\partial_\mu p_{m+1}(\cdot)(v) \otimes \partial_\mu \mH_{m+1}(\cdot)(v'))(\mu, s, t, x, z) - (\partial_\mu p_{m+1}(\cdot)(v) \otimes \partial_\mu \mH_{m+1}(\cdot)(v'))(\mu', s, t, x, z) \Big| \label{esti:4:param:delta:mu} \\
& + \Big|  (\partial_\mu p_{m+1}(\cdot)(v') \otimes \partial_\mu \mH_{m+1}(v))(\mu, s, t, x, z) - (\partial_\mu p_{m+1}(\cdot)(v') \otimes \partial_\mu \mH_{m+1}(v))(\mu', s, t, x, z) \Big| \notag \\
& \quad \leq  \frac{K_\beta^{+}}{(t-s)^{1+ \frac{\beta}{2}-\eta}} \, W_2(\mu, \mu')^{\beta} \, g(c(t-s), z-x).\notag
\end{align}

Gathering the estimates \eqref{esti:1:param:deltamu}, \eqref{esti:2:param:deltamu}, \eqref{esti:3:param:deltamu} and \eqref{esti:4:param:delta:mu}, we deduce
 \begin{eqnarray*}
&&\Big| \Big( \partial^2_\mu \widehat{p}_{m+1}(\cdot)(\gv)+ p_{m+1} \otimes \partial^2_\mu \mH_{m+1}(\cdot)(\gv) \notag\\
&&\qquad + \partial_\mu p_{m+1}(\cdot)(v) \otimes \partial_\mu \mH_{m+1}(\cdot)(v') + \partial_\mu p_{m+1}(\cdot)(v') \otimes \partial_\mu \mH_{m+1}(\cdot)(v)\Big)(\mu, s, t, x, z)\\
&&  - \Big( \partial^2_\mu \widehat{p}_{m+1}(\cdot)(\gv)+ p_{m+1} \otimes \partial^2_\mu \mH_{m+1}(\cdot)(\gv) \notag\\
&&\qquad + \partial_\mu p_{m+1}(\cdot)(v) \otimes \partial_\mu \mH_{m+1}(\cdot)(v') + \partial_\mu p_{m+1}(\cdot)(v') \otimes \partial_\mu \mH_{m+1}(\cdot)(v) \Big)(\mu', s, t, x, z)\Big| \\
& \qquad \leq& \frac{K^{++}_\beta}{(t-s)^{1+\frac{\beta-\eta}{2}}} \left\{ 1+ \sum_{k=1}^{m} (C_\beta^{++})^{k} (t-s)^{k \frac{\eta}{2}} \prod_{i=1}^{k} B\left(\frac{\eta}{2}, \frac{\eta-\beta}{2} +  (i-1)\frac{\eta}{2} \right)  \right\}\\
&& \quad \times W_2(\mu, \mu')^{\beta} \, g(c(t-s), z-x)
 \end{eqnarray*}

\noindent which in turn combined with \eqref{Gaussian:estimate:Phim} and then using \eqref{estimate:convol:pmp1:cross:deriv:mes:Hmp1:with:first:induction:step:done} with \eqref{gaussian:bound:diff:mes:Phim} imply
  
\begin{eqnarray*}
&&\Big| \Big[\partial^2_\mu \widehat{p}_{m+1}(\cdot)(\gv)+ p_{m+1} \otimes \partial^2_\mu \mH_{m+1}(\cdot)(\gv)+ (\partial_\mu p_{m+1}(\cdot)(v) \otimes \partial_\mu \mH_{m+1})(\cdot)(v')\notag\\
&&\qquad \qquad+(\partial_\mu p_{m+1}(\cdot)(v') \otimes \partial_\mu \mH_{m+1})(v)\Big] \otimes \Phi_{m+1}(\mu, s, t, x, z) \\
& & - \Big[\partial^2_\mu \widehat{p}_{m+1}(\cdot)(\gv)+ p_{m+1} \otimes \partial^2_\mu \mH_{m+1}(\cdot)(\gv)+ (\partial_\mu p_{m+1}(\cdot)(v) \otimes \partial_\mu \mH_{m+1})(\cdot)(v')\notag\\
&&\qquad \qquad+(\partial_\mu p_{m+1}(\cdot)(v') \otimes \partial_\mu \mH_{m+1})(v)\Big] \otimes \Phi_{m+1}(\mu', s, t, x, z)\Big|\\
&\quad \leq &\frac{K^{++}_\beta}{(t-s)^{1+\frac{\beta}{2}-\eta}} \left\{ 1+ \sum_{k=1}^{m} (C_\beta^{++})^{k} (t-s)^{k \frac{\eta}{2}} \prod_{i=1}^{k} B\left(\frac{\eta}{2}, \frac{\eta-\beta}{2} +  (i-1)\frac{\eta}{2} \right)  \right\}\\
&& \quad \times W_2(\mu, \mu')^{\beta} \, g(c(t-s), z-x)
\end{eqnarray*}
 
\noindent for any $\beta \in [0,\eta)$. Coming back to the decomposition \eqref{eq:decom:deltamumu':inter} and gathering the two previous estimates conclude the proof of \eqref{estimate:deriv:mes:holder:reg:mes:mes:dens:stepmp1:cor}. \\

\noindent \emph{Step 2: proof of the estimate \eqref{estimate:deriv:mes:holder:reg:time:dens:stepmp1:cor}.\\}

Let us first observe that if $|s_1-s_2| \geq t-s_1\vee s_2$ then \eqref{estimate:deriv:mes:holder:reg:time:dens:stepmp1:cor} directly follows from \eqref{second:deriv:mes:induction:decoupling:mckean}. We thus assume that $|s_1-s_2| \leq t-s_1\vee s_2$ for the rest of the proof. We again start from the identity \eqref{eq:decom:deltamumu':inter} and now investigate the H\"older regularity of each term with respect to the time variable $s$. For the first term, we combine \eqref{diff:time:second:L:deriv:p:hat:same:time} with \eqref{regularity:time:estimate:secon:deriv:decoupling:mckean} at step $m$ and the space time inequality \eqref{space:time:inequality}. We thus obtain that for any $\beta \in [0,\eta)$
\begin{align}
  |& \partial^2_\mu \widehat{p}^y_{m+1}(\mu, s_1, t, x, z)(\gv) - \partial^2_\mu \widehat{p}^y_{m+1}(\mu, s_2, t, x, z)(\gv)  | \notag \\
  & \leq K^{++}_\beta \bigg\{ \frac{|s_1-s_2|^{\frac{\beta}{2}}}{(t-s_1)^{1+\frac{\beta-\eta}{2}}} g(c(t-s_1), z-x) +  \frac{|s_1-s_2|^{\frac{\beta}{2}}}{(t-s_2)^{1+\frac{\beta-\eta}{2}}} g(c(t-s_2), z-x)  \notag\\
 &       + \frac{|s_1-s_2|^{\frac{\beta}{2}}}{t-s_1 \vee s_2} \int_{s_1 \vee s_2}^{t} \frac{\mathscr{C}_m^{1,\beta}(C^{++}_\beta, r-s_1\vee s_2)}{(r-s_1\vee s_2)^{1+\frac{\beta}{2}-\eta}} dr \,  g(c(t-s_1 \vee s_2), z-x) \Big. \bigg\}  \notag \\
 & \leq K^{++}_\beta \bigg\{ \frac{|s_1-s_2|^{\frac{\beta}{2}}}{(t-s_1)^{1+\frac{\beta-\eta}{2}}} g(c(t-s_1), z-x) +  \frac{|s_1-s_2|^{\frac{\beta}{2}}}{(t-s_2)^{1+\frac{\beta-\eta}{2}}} g(c(t-s_2), z-x) \notag \\
 &       +  |s_1-s_2|^{\frac{\beta}{2}} \int_{s_1 \vee s_2}^{t} \frac{\mathscr{C}_m^{1,\beta}(C^{++}_\beta, r-s_1\vee s_2)}{(t-r)^{1-\frac{\eta}{2}}(r-s_1\vee s_2)^{1+\frac{\beta-\eta}{2}}} dr \,  g(c(t-s_1 \vee s_2), z-x) \Big. \bigg\}  \notag \\
 & \leq K^{++}_\beta  \left\{ 1 + \sum_{k=1}^{m} (C_\beta^{++})^{k} (t-s_1 \vee s_2)^{k \frac{\eta}{2}}  \prod_{i=1}^{k} B\left(\frac{\eta}{2}, \frac{\eta-\beta}{2} +  (i-1)\frac{\eta}{2} \right)  \right\} \notag\\
& \quad \times  \left\{ \frac{|s_1-s_2|^{\frac{\beta}{2}}}{(t-s_1)^{1+\frac{\beta-\eta}{2}}} g(c(t-s_1), z-x) + \frac{|s_1-s_2|^{\frac{\beta}{2}}}{(t-s_2)^{1+\frac{\beta-\eta}{2}}} g(c(t-s_2), z-x) \right\}. \label{diff:time:second:L:deriv:p:hat:same:time:with:pm:estimate} 
 \end{align} 
  
We now investigate the uniform H\"older regularity of the map $s\mapsto (p_{m+1} \otimes \partial^2_\mu \mH_{m+1}(.)(\gv))(\mu, s, t, x, z)$. We use the following decomposition 
$$
 (p_{m+1} \otimes \partial^2_\mu \mH_{m+1}(.)(\gv))(\mu, s_1\vee s_2, t, x, z) - (p_{m+1} \otimes \partial^2_\mu \mH_{m+1}(.)(\gv))(\mu, s_1 \wedge s_2, t, x, z)= {\rm I} + {\rm II} + {\rm III},
$$

\noindent with
\begin{align*}
{\rm I} & := \int_{s_1 \vee s_2}^t \int_{\mathbb{R}^d} [ p_{m+1}(\mu, s_1\vee s_2, r, x, y)  -  p_{m+1}(\mu, s_1\wedge s_2, r, x, y)] \,  \partial^2_\mu \mH_{m+1}(\mu, s_1\vee s_2, r, t, y, z)(\gv) \, dy \, dr, 
\end{align*}

\begin{align*}
{\rm II} & := \int_{s_1 \vee s_2}^t \int_{\mathbb{R}^d}  p_{m+1}(\mu, s_1\wedge s_2, r, x, y) \, [\partial^2_\mu \mH_{m+1}(\mu, s_1\vee s_2, r, t, y, z)(\gv) - \partial^2_\mu \mH_{m+1}(\mu, s_1\wedge s_2, r, t, y, z)(\gv)] \, dy \, dr,
\end{align*}
\noindent and
\begin{align*}
{\rm III} & := - \int_{s_1 \wedge s_2}^{s_1 \vee s_2} \int_{\mathbb{R}^d}   p_{m+1}(\mu, s_1 \wedge s_2, r, x, y)  \, \partial^2_\mu \mH_{m+1}(\mu, s_1\wedge s_2, r, t, y, z)(\gv) \, dy \, dr.
\end{align*}

In order to deal with ${\rm I}$, we first split the time integral into the two intervals $[s_1\vee s_2, (t+s_1\vee s_2)/2)$ and $ [(t+s_1\vee s_2)/2, t]$ to balance the time singularity, then use the estimates \eqref{regularity:time:estimate:v1:v2:v3:decoupling:mckean} with $n=0$ and \eqref{sec:mes:deriv:parametrix:kernel:s:r:t:with:beta} (with $\beta = 0$ on $[s_1\vee s_2, (t+s_1\vee s_2)/2]$ and $\beta=1$ on $((t+s_1\vee s_2)/2, t]$) combined with \eqref{second:deriv:mes:induction:decoupling:mckean} and the space time inequality \eqref{space:time:inequality} so that
$$
| {\rm I} | \leq K^{+}_\beta \left\{\frac{|s_1-s_2|^{\frac{\beta}{2}}}{(t-s_1)^{1+\frac{\beta-\eta}{2}}} g(c(t-s_1), z-x) + \frac{|s_1-s_2|^{\frac{\beta}{2}}}{(t-s_2)^{1+\frac{\beta-\eta}{2}}} g(c(t-s_2) , z-x) \right\}.
$$

To deal with ${\rm II}$, we use \eqref{diff:time:second:L:deriv:parametrix:kernel:pmp1}, \eqref{bound:derivative:heat:kernel} and \eqref{regularity:time:estimate:secon:deriv:decoupling:mckean} at step $m$. To be more specific, we again split the time integral into the two disjoint intervals $[s_1\vee s_2, (t+s_1\vee s_2)/2)$ and $ [(t+s_1\vee s_2)/2, t]$ as previously done. For the time integral on $[s_1\vee s_2, (t+s_1\vee s_2)/2)$, we bound the first term appearing on the right-hand side of \eqref{diff:time:second:L:deriv:parametrix:kernel:pmp1} which writes as a minimum by $K^{++}_\beta |s_1-s_2|^{\frac{\beta}{2}} (t-r)^{-1} (r-s_1\vee s_2)^{-1- \frac{(\beta-\eta)}{2}} g(c(t-r), z-x)$ while for the second term which also writes as a minimum, we bound it by $K^{++}_\beta (t-r)^{-1} \int_{(\mathbb{R}^d)^2} (|y'-x'|^\eta \wedge 1) | \partial^2_\mu p_{m}(\mu, s_1, r, x', y')(\gv) -  \partial^2_\mu p_{m}(\mu, s_2, r, x', y')(\gv)|  \, dy'  \mu(dx') \, g(c(t-r), z-x)$. For the time integral on $[(t+s_1\vee s_2)/2), t]$, we bound the first term appearing on the right-hand side of \eqref{diff:time:second:L:deriv:parametrix:kernel:pmp1} by $K^{++}_\beta |s_1-s_2|^{\frac{\beta}{2}} (t-r)^{-1+\frac{\eta}{2}} (r-s_1\vee s_2)^{-1-\frac{\beta}{2}} g(c(t-r), z-x)$ while for the second term, we bound it by $K^{++}_\beta (t-r)^{-1+\frac{\eta}{2}} \int_{(\mathbb{R}^d)^2} | \partial^2_\mu p_{m}(\mu, s_1, r, x', y')(\gv) - \partial^2_\mu p_{m}(\mu, s_2, r, x', y')(\gv) |  \, dy'  \mu(dx') \, g(c(t-r), z-x)$. For the third term, in both cases, we use Fubini's theorem. After some standard computations that we omit, we obtain
\begin{align*}
|{\rm II}| & \leq K^{++}_\beta \left\{\frac{1}{(t-s_1\vee s_2)^{1+\frac{\beta-\eta}{2}}} +  \int_{s_1\vee s_2}^{t} \frac{\mathscr{C}^{1,  \beta}_{m}(C^{++}_\beta , r-s_1\vee s_2)}{(t-r)^{1-\frac{\eta}{2}}(r-s_1\vee s_2)^{1+\frac{\beta-\eta}{2}}} \, dr \, \right\} \\
& \quad \times |s_1-s_2|^{\beta} g(c(t-s_1\wedge s_2), z-x)
\end{align*}

\noindent for any $\beta \in [0,\eta)$.

We eventually deal with ${\rm III}$ by using  \eqref{bound:derivative:heat:kernel} and \eqref{sec:mes:deriv:parametrix:kernel:s:r:t:with:beta} with $\beta = 0$ combined with \eqref{second:deriv:mes:induction:decoupling:mckean} and the space time inequality \eqref{space:time:inequality}. We get 
\begin{align*}
|{\rm III}| & \leq K^+ \int_{s_1 \wedge s_2}^{s_1 \vee s_2} \frac{1}{(t-r)(r-s_1 \wedge s_2)^{1-\frac{\eta}{2}}} \, dr \, g(c(t-s_1 \wedge s_2), z-x)\\
& \leq K^+ \frac{|s_1-s_2|^{\frac{\eta}{2}}}{t-s_1\vee s_2} \, g(c(t- s_1\wedge s_2), z-x) \\
& \leq K^+ \frac{|s_1-s_2|^{\frac{\beta}{2}}}{(t-s_1\wedge s_2)^{1+\frac{\beta-\eta}{2}}} \, g(c(t-s_1\wedge s_2) , z-x)
\end{align*}

\noindent for any $\beta \in [0, \eta]$ where we used the fact that $t-s_1 \wedge s_2 \leq 2 (t-s_1 \vee s_2)$ for the last inequality. Gathering the three previous estimates eventually yields
\begin{align}
\Big| (p_{m+1} & \otimes \partial^2_\mu \mH_{m+1}(.)(\gv))(\mu, s_1, t, x, z) - (p_{m+1} \otimes \partial^2_\mu \mH_{m+1}(.)(\gv))(\mu, s_2, t, x, z) \Big|\nonumber \\
&  \leq K^{++}_\beta \left\{\frac{|s_1-s_2|^{\frac{\beta}{2}}}{(t-s_1)^{1+\frac{\beta-\eta}{2}}} g(c(t-s_1), z-x) + \frac{|s_1-s_2|^{\frac{\beta}{2}}}{(t-s_2)^{1+\frac{\beta-\eta}{2}}} g(c(t-s_2) , z-x) \right\} \nonumber \\
& + K^{++}_\beta |s_1-s_2|^{\frac{\beta}{2}} \int_{s_1\vee s_2}^{t} \frac{\mathscr{C}^{1, \beta}_{m}(C^{++}_\beta, r-s_1 \vee s_2)}{(t-r)^{1-\frac{\eta}{2}}(r-s_1\vee s_2)^{1+\frac{\beta-\eta}{2}}} \, dr \, g(c(t-s_1\wedge s_2), z-x) \nonumber \\
& \leq  K^{++}_\beta  \left\{ 1 + \sum_{k=1}^{m} (C^{++}_\beta)^{k} (t-s_1 \vee s_2)^{k \frac{\eta}{2}}  \prod_{i=1}^{k} B\left(\frac{\eta}{2}, \frac{\eta-\beta}{2} +  (i-1)\frac{\eta}{2} \right)  \right\} \label{bound:delta:time:convol:pmp1:and:L:deriv:H} \\
& \quad \times  \left\{ \frac{|s_1-s_2|^{\frac{\beta}{2}}}{(t-s_1)^{1+\frac{\beta -\eta}{2}}} g(c(t-s_1), z-x) + \frac{|s_1-s_2|^{\frac{\beta}{2}}}{(t-s_2)^{1+\frac{\beta -\eta}{2}}} g(c(t-s_2), z-x) \right\}. \nonumber
\end{align}

As for the previous estimate, the third and fourth term appearing on the right-hand side of \eqref{eq:decom:deltamumu':inter} are handled by similar arguments. Namely, we first employ the decomposition
\begin{align*}
(\partial_\mu p_{m+1}(\cdot)(v) \otimes \partial_\mu \mH_{m+1}(\cdot)& (v'))(\mu, s_1 \vee s_2, t, x, z) - (\partial_\mu p_{m+1}(\cdot)(v) \otimes \partial_\mu \mH_{m+1}(\cdot)(v'))(\mu, s_1 \wedge s_2, t, x, z) \\
& = {\rm I} + {\rm II} + {\rm III}
\end{align*}
\noindent with
\begin{align*}
{\rm I} & := \int_{s_1 \vee s_2}^t \int_{\mathbb{R}^d} [ \partial_\mu p_{m+1}(\mu, s_1\vee s_2, r, x, y)(v)  -  \partial_\mu p_{m+1}(\mu, s_1\wedge s_2, r, x, y)(v)] \\
& \quad   \partial_\mu \mH_{m+1}(\mu, s_1\vee s_2, r, t, y, z)(v') \, dy \, dr, \\
{\rm II} & := \int_{s_1 \vee s_2}^t \int_{\mathbb{R}^d}  \partial_\mu p_{m+1}(\mu, s_1\wedge s_2, r, x, y)(v) \\
& \quad  [\partial_\mu \mH_{m+1}(\mu, s_1\vee s_2, r, t, y, z)(v') - \partial_\mu \mH_{m+1}(\mu, s_1\wedge s_2, r, t, y, z)(v')] \, dy \, dr,\\
{\rm III} & := - \int_{s_1 \wedge s_2}^{s_1 \vee s_2} \int_{\mathbb{R}^d}   \partial_\mu p_{m+1}(\mu, s_1 \wedge s_2, r, x, y)(v)  \, \partial_\mu \mH_{m+1}(\mu, s_1\wedge s_2, r, t, y, z)(v') \, dy \, dr.
\end{align*}

Then, it follows from \eqref{regularity:time:estimate:v1:v2:decoupling:mckean} with $n=0$, \eqref{first:second:estimate:induction:decoupling:mckean}, \eqref{diff:time:L:deriv:parametrix:kernel:pmp1}, \eqref{estimate:partial:deriv:mes:parametrix:kernel:iterated:m} with $k=1$, $n=0$, $\beta=1$, the inequality $t-s_1 \wedge s_2 \leq 2 (t-s_1 \vee s_2)$ and some standard computations that we omit that for any $\beta \in [0,\eta)$
$$
|{\rm I}| \leq K^+_{\beta}  \left\{ \frac{|s_1-s_2|^{\frac{\beta}{2}}}{(t-s_1)^{1+\frac{\beta}{2}-\eta}} g(c(t-s_1), z-x) + \frac{|s_1-s_2|^{\frac{\beta}{2}}}{(t-s_2)^{1+\frac{\beta}{2}-\eta}} g(c(t-s_2), z-x) \right\},
$$
$$
|{\rm II}| \leq K^+_\beta   \frac{|s_1-s_2|^{\frac{\beta}{2}}}{(t-s_1 \wedge s_2)^{1+\frac{\beta}{2}-\eta}} g(c(t-s_1 \wedge s_2), z-x)
$$
\noindent and
$$
|{\rm III}| \leq K  \frac{|s_1-s_2|^{\frac{\beta}{2}}}{(t-s_1 \wedge s_2)^{1+\frac{\beta}{2}-\eta}} g(c(t-s_1 \wedge s_2), z-x)
$$
\noindent so that
\begin{align}
\Big| (\partial_\mu p_{m+1}(\cdot)(v)&  \otimes \partial_\mu \mH_{m+1}(\cdot) (v'))(\mu, s_1 \vee s_2, t, x, z) - (\partial_\mu p_{m+1}(\cdot)(v) \otimes \partial_\mu \mH_{m+1}(\cdot)(v'))(\mu, s_1 \wedge s_2, t, x, z) \Big| \notag \\
& \leq K^+_{\beta}  \left\{ \frac{|s_1-s_2|^{\frac{\beta}{2}}}{(t-s_1)^{1+\frac{\beta }{2}-\eta}} g(c(t-s_1), z-x) + \frac{|s_1-s_2|^{\frac{\beta}{2}}}{(t-s_2)^{1+\frac{\beta}{2}-\eta}} g(c(t-s_2), z-x) \right\}.\label{bound:delta:time:convol:partial:mes:pmp1:and:L:deriv:H:part1}
\end{align}

Following similar lines of reasonings, it holds
\begin{align}
\Big| (\partial_\mu p_{m+1}(\cdot)(v')&  \otimes \partial_\mu \mH_{m+1}(\cdot)(v))(\mu, s_1 \vee s_2, t, x, z) - (\partial_\mu p_{m+1}(\cdot)(v') \otimes \partial_\mu \mH_{m+1}(\cdot)(v))(\mu, s_1 \wedge s_2, t, x, z) \Big|\notag \\
& \leq K^+_{\beta}  \left\{ \frac{|s_1-s_2|^{\frac{\beta}{2}}}{(t-s_1)^{1+\frac{\beta }{2}-\eta}} g(c(t-s_1), z-x) + \frac{|s_1-s_2|^{\frac{\beta}{2}}}{(t-s_2)^{1+\frac{\beta}{2}-\eta}} g(c(t-s_2), z-x) \right\}.\label{bound:delta:time:convol:partial:mes:pmp1:and:L:deriv:H:part2}
\end{align}

Gathering \eqref{diff:time:second:L:deriv:p:hat:same:time:with:pm:estimate}, \eqref{bound:delta:time:convol:pmp1:and:L:deriv:H}, \eqref{bound:delta:time:convol:partial:mes:pmp1:and:L:deriv:H:part1} and \eqref{bound:delta:time:convol:partial:mes:pmp1:and:L:deriv:H:part2} yields
\begin{align}
\Big| \Big(\partial^2_\mu & \widehat{p}_{m+1}(\cdot)(\gv) + p_{m+1} \otimes \partial^2_\mu \mH_{m+1}(\cdot)(\gv) \notag \\
& \quad + \partial_\mu p_{m+1}(\cdot)(v) \otimes \partial_\mu \mH_{m+1}(\cdot)(v') + \partial_\mu p_{m+1}(\cdot)(v') \otimes \partial_\mu \mH_{m+1}(.)(v)\Big)(\mu, s_1, t, x, z) \notag \\
& \quad - \Big( \partial^2_\mu  \widehat{p}_{m+1}(\cdot)(\gv) + p_{m+1} \otimes \partial^2_\mu \mH_{m+1}(\cdot)(\gv) \notag \\
& \quad + \partial_\mu p_{m+1}(\cdot)(v) \otimes \partial_\mu \mH_{m+1}(\cdot)(v') + \partial_\mu p_{m+1}(\cdot)(v') \otimes \partial_\mu \mH_{m+1}(.)(v)\Big)(\mu, s_2, t, x, z) \Big| \notag \\
& \leq K^{++}_\beta  \left\{ 1 + \sum_{k=1}^{m} (C^{++}_\beta)^{k} (t-s_1 \vee s_2)^{k \frac{\eta}{2}}  \prod_{i=1}^{k} B\left(\frac{\eta}{2}, \frac{\eta-\beta}{2} +  (i-1)\frac{\eta}{2} \right)  \right\} \label{bound:delta:time:first:term:second:L:derivative} \\
& \quad \times  \left\{ \frac{|s_1-s_2|^{\frac{\beta}{2}}}{(t-s_1)^{1+\frac{\beta -\eta}{2}}} g(c(t-s_1), z-x) + \frac{|s_1-s_2|^{\frac{\beta}{2}}}{(t-s_2)^{1+\frac{\beta -\eta}{2}}} g(c(t-s_2), z-x) \right\}. \notag
\end{align}

Finally, from a similar decomposition to the one employed previously, using the previous bound together with \eqref{Gaussian:estimate:Phim}, then \eqref{estimate:convol:pmp1:cross:deriv:mes:Hmp1:with:first:induction:step:done} with \eqref{gaussian:bound:diff:time:Phim} and the inequality $t-s_1 \wedge s_2 \leq 2 (t-s_1 \vee s_2)$, after some standard computations that we omit, for any $\beta \in [0,\eta)$, we obtain
\begin{align}
\Big| \Big(\partial^2_\mu & \widehat{p}_{m+1}(\cdot)(\gv) + p_{m+1} \otimes \partial^2_\mu \mH_{m+1}(\cdot)(\gv) \notag \\
& \quad + \partial_\mu p_{m+1}(\cdot)(v) \otimes \partial_\mu \mH_{m+1}(\cdot)(v') + \partial_\mu p_{m+1}(\cdot)(v') \otimes \partial_\mu \mH_{m+1}(.)(v)\Big)\otimes \Phi_{m+1}(\mu, s_1, t, x, z) \notag \\
& \quad - \Big( \partial^2_\mu  \widehat{p}_{m+1}(\cdot)(\gv) + p_{m+1} \otimes \partial^2_\mu \mH_{m+1})(\cdot)(\gv) \notag \\
& \quad + \partial_\mu p_{m+1}(\cdot)(v) \otimes \partial_\mu \mH_{m+1}(\cdot)(v')+ \partial_\mu p_{m+1}(\cdot)(v') \otimes \partial_\mu \mH_{m+1}(.)(v)\Big)\otimes \Phi_{m+1}(\mu, s_2, t, x, z) \Big| \notag \\
& \leq K^{++}_\beta  \left\{ 1 + \sum_{k=1}^{m} (C^{++}_\beta)^{k} (t-s_1 \vee s_2)^{k \frac{\eta}{2}}  \prod_{i=1}^{k} B\left(\frac{\eta}{2}, \frac{\eta-\beta}{2} +  (i-1)\frac{\eta}{2} \right)  \right\} \label{bound:delta:time:last:term:second:L:derivative} \\
& \quad \times  \left\{ \frac{|s_1-s_2|^{\frac{\beta}{2}}}{(t-s_1)^{1+\frac{\beta -\eta}{2}}} g(c(t-s_1), z-x) + \frac{|s_1-s_2|^{\frac{\beta}{2}}}{(t-s_2)^{1+\frac{\beta -\eta}{2}}} g(c(t-s_2), z-x) \right\}. \notag
\end{align}

Coming back to the identity \eqref{eq:decom:deltamumu':inter} and gathering the estimates \eqref{bound:delta:time:first:term:second:L:derivative} and \eqref{bound:delta:time:last:term:second:L:derivative} conclude the proof of \eqref{estimate:deriv:mes:holder:reg:time:dens:stepmp1:cor}.
\end{proof}

 \section{Proof of the technical estimates of \ref{proof:main:prop}.}\label{sec:proof:technical:estimates}

 \subsection{Proof of Lemma \ref{lem:diff:and:control:deriv:coeff}.\\}\label{proof:lem:diff:and:control:deriv:coeff}

\noindent \emph{Step 1: regularity of the maps $[0,t) \times \pp \ni (s, \mu) \mapsto  b_i (t, x, [X^{s,\xi, (m)}_t]),\,   a_{i, j} (t, x, [X^{s,\xi, (m)}_t])$. }\\

We apply Proposition \ref{structural:class} with the density function $(s, x, \mu) \mapsto p_m(\mu, s, t, x, z) \in \mathcal{C}^{1, 2, 2}_f([0,t)\times \rr^d \times \pp) $ and to both maps $h(.) = b_i(t, x, .)$ and $h(.) = a_{i, j}(t, x, .)$ respectively. Note that the regularity property established in \ref{subsubsection:regularity:derivatives:heat:kernel:approx} and the estimates \eqref{first:second:estimate:induction:decoupling:mckean}, \eqref{time:derivative:induction:decoupling:mckean}, \eqref{second:deriv:mes:induction:decoupling:mckean}, \eqref{cross:deriv:mes:space:induction:decoupling:mckean} and \eqref{bound:derivative:heat:kernel} ensure that the map $[0,t) \times \mathbb{R}^d \times \pp \ni (s, x, \mu) \mapsto p_m(\mu, s, t, x, z)$ satisfies the conditions of Proposition \ref{structural:class}. In particular the estimate \eqref{integrability:condition} is satisfied. We thus deduce that $(s, \mu) \mapsto  b_i (t, x, [X^{s,\xi, (m)}_t]), \, a_{i, j}(t, x, [X^{s,\xi, (m)}_t]) \in \mathcal{C}^{1,2}_f([0,t) \times \pp)$. \\ 

\noindent \emph{Step 2: proof of the estimate \eqref{recursive:bound:deriv:a:or:b}.}\\

We start from the identity \eqref{full:expression:second:deriv} in Proposition \ref{structural:class} applied to $h(.) = b_i(t, x, .)$ and $h(.) = a_{i, j}(t, x, .)$. We deduce from the estimates \eqref{first:second:estimate:induction:decoupling:mckean} (with $n=0$), \eqref{bound:derivative:heat:kernel} (with $n=1$) and \eqref{cross:deriv:mes:space:induction:decoupling:mckean}, the uniform boundedness and $\eta$-H\"older regularity of the maps $[\delta b_i/\delta m](t, x, m)(.)$ and  $[\delta a_{i, j}/\delta m](t, x, m)(.)$, $[\delta^2 b_i/\delta m^2](t, x, m)(v,.)$ and  $[\delta^2 a_{i, j}/\delta m^2](t, x, m)(v,.)$ and the space time inequality \eqref{space:time:inequality} that the estimate \eqref{recursive:bound:deriv:a:or:b} holds. \\

\noindent \emph{Step 3: proof of the estimate \eqref{recursive:bound:second:order:deriv:mes:holder:reg:space:a}.}\\

The identity \eqref{full:expression:second:deriv} applied to $h(.) = a_{i, j}(t, x, .)$ and $a_{i, j}(t, z, .)$ gives the decomposition 
{\small
$$
\partial^2_\mu \Big[a_{i, j}(t, x, [X^{s, \xi, (m)}_{t}]) - a_{i, j}(t, z, [X^{s, \xi, (m)}_t])\Big](\gv) = \sum_{\ell=1}^{7} \delta a^\ell_{i, j} (\gv, \mu), \quad \gv=(v,v'),
$$
\noindent with 
\begin{align}
\delta a^1_{i, j} (\gv, \mu) & :=  \int_{(\mathbb{R}^d)^2} \Big\{\frac{\delta^2 a_{i, j}}{\delta m^2}(t, x, [X^{s, \xi, (m)}_{t}])(z', z'') - \frac{\delta^2 a_{i, j}}{\delta m^2}(t, x, [X^{s, \xi, (m)}_{t}])(z',v') \nonumber\\
&  - \Big[\frac{\delta^2 a_{i, j}}{\delta m^2}(t, z, [X^{s, \xi, (m)}_{t}])(z', z'') - \frac{\delta^2 a_{i, j}}{\delta m^2}(t, z, [X^{s, \xi, (m)}_{t}])(z',v')\Big]\Big\}\, \nonumber\\
& \quad \quad \, \partial_{x} p_m(\mu, s, t, v, z') \otimes \partial_{x} p_m(\mu, s, t, v',z'') \, dz' dz'', \nonumber \\
\delta a^{2}_{i, j}(\gv, \mu) & := \int_{(\mathbb{R}^d)^3} \Big\{\frac{\delta^2 a_{i, j}}{\delta m^2}(t, x, [X^{s, \xi, (m)}_{t}])(z', z'') - \frac{\delta^2 a_{i, j}}{\delta m^2}(t, x, [X^{s, \xi, (m)}_{t}])(z', x') \nonumber\\
& - \Big[\frac{\delta^2 a_{i, j}}{\delta m^2}(t, z, [X^{s, \xi, (m)}_{t}])(z', z'') - \frac{\delta^2 a_{i, j}}{\delta m^2}(t, z, [X^{s, \xi, (m)}_{t}])(z', x') \Big] \Big\} \nonumber \\
& \quad \quad \,\partial_x  p_m(\mu, s, t, v, z')\otimes  \partial_\mu p_m(\mu, s, t, x', z'')(v')  \, dz'dz''\, \mu(dx'), \nonumber 
\end{align}
\begin{align}
\delta a^{3}_{i, j}(\gv, \mu) & :=\int_{\mathbb{R}^d}  \Big\{\frac{\delta a_{i, j}}{\delta m}(t, x, [X^{s, \xi, (m)}_{t}])(z')- \frac{\delta a_{i, j}}{\delta m}(t, x, [X^{s, \xi, (m)}_{t}])(v) \nonumber \\
& \quad - \Big[\frac{\delta a_{i, j}}{\delta m}(t, z, [X^{s, \xi, (m)}_{t}])(z')- \frac{\delta a_{i, j}}{\delta m}(t, z, [X^{s, \xi, (m)}_{t}])(v) \Big]\Big\} \,  \partial_\mu [\partial_x p_m(\mu, s, t, v ,z')](v')dz',\nonumber \\
\delta a^4_{i, j}(\gv, \mu) &:= \int_{\mathbb{R}^d}  \Big\{\frac{\delta a_{i, j}}{\delta m}(t, x, [X^{s, \xi, (m)}_{t}])(z') - \frac{\delta a_{i, j}}{\delta m} (t, x, [X^{s, \xi, (m)}_{t}])(v')\nonumber\\
& \quad \quad - \Big[\frac{\delta a_{i, j}}{\delta m}(t, z, [X^{s, \xi, (m)}_{t}])(z') - \frac{\delta a_{i, j}}{\delta m} (t, z, [X^{s, \xi, (m)}_{t}])(v')\Big]\Big\} \partial_{x}[\partial_\mu p_m(\mu, s, t, v', z')(v)] \, dz', \nonumber 
\end{align}
\begin{align}
\delta a^5_{i, j}(\gv, \mu) & := \int_{(\mathbb{R}^d)^3} \Big\{\frac{\delta^2 a_{i, j}}{\delta m^2}(t, x, [X^{s, \xi, (m)}_{t}])(z', z'') - \frac{\delta^2 a_{i, j}}{\delta m^2}(t, x, [X^{s, \xi, (m)}_{t}])(z', v') \nonumber \\
& \quad - \Big[\frac{\delta^2 a_{i, j}}{\delta m^2}(t, z, [X^{s, \xi, (m)}_{t}])(z', z'') - \frac{\delta^2 a_{i, j}}{\delta m^2}(t, z, [X^{s, \xi, (m)}_{t}])(z', v')\Big]\Big\} \nonumber \\
& \quad   \quad \quad \partial_\mu p_m(\mu, s, t, x', z')(v) \otimes \partial_x p_m(\mu, s, t, v',z'')\, dz'\, dz''\, \mu(dx'), \nonumber \\
\delta a^6_{i, j}(\gv, \mu) & :=  \int_{(\mathbb{R}^d)^4}\Big\{\frac{\delta^2 a_{i, j}}{\delta m^2}(t, x, [X^{s, \xi, (m)}_{t}])(z', z'')-\frac{\delta^2 a_{i, j}}{\delta m^2}(t, x, [X^{s, \xi, (m)}_{t}])(z', x'') \nonumber \\
& \quad- \Big[\frac{\delta^2 a_{i, j}}{\delta m^2}(t , z, [X^{s, \xi, (m)}_{t}])(z', z'')-\frac{\delta^2 a_{i, j}}{\delta m^2}(t, z, [X^{s, \xi, (m)}_{r}])(z', x'') \Big] \Big\} \nonumber \\
& \quad \quad  \partial_\mu p_m(\mu, s, r, x', z')(v)\otimes \partial_\mu p_m(\mu, s, r, x'', z'')(v')\,dz'\, dz''\, \mu(dx'')\, \mu(dx'), \nonumber \\
\delta a^{7}_{i, j}(\gv, \mu) & := \int_{(\mathbb{R}^d)^2} \Big\{\frac{\delta a_{i, j}}{\delta m}(t, x, [X^{s, \xi, (m)}_{t}])(z')- \frac{\delta a_{i, j}}{\delta m}(t, x, [X^{s, \xi, (m)}_{t}])(x') \nonumber \\
& \quad - \Big[\frac{\delta a_{i, j}}{\delta m}(t, z, [X^{s, \xi, (m)}_{t}])(z')- \frac{\delta a_{i, j}}{\delta m}(t, z, [X^{s, \xi, (m)}_{t}])(x')\Big]\Big\}\, \partial_\mu^2 p_m(\mu, s, t, x', z')(v, v')\,  dz'\, \mu(dx').\nonumber
\end{align}
}
Now, the uniform $\eta$-H\"older regularity of the maps $ [\delta a_{i, j}/\delta m](t, ., m)(.)$ and $ [\delta^2 a_{i, j}/\delta m^2](t, ., m)(v, .)$ yields 
\begin{align*}
\Big| \frac{\delta a_{i, j}}{\delta m}(t, x, [X^{s, \xi, (m)}_{t}])(z') &  - \frac{\delta a_{i, j}}{\delta m}(t, x, [X^{s, \xi, (m)}_{t}])(v) \\
& \quad - \Big[\frac{\delta a_{i, j}}{\delta m}(t, z, [X^{s, \xi, (m)}_{t}])(z')- \frac{\delta a_{i, j}}{\delta m}(t, z, [X^{s, \xi, (m)}_{t}])(v) \Big]\Big| \\
& \quad \quad\leq K^+ (|z-x|^\eta \wedge |z'-v|^\eta \wedge 1) \\
& \quad \quad \leq K^{+} |z-x|^{\beta' \eta} (|z'-v|^{(1-\beta')\eta} \wedge 1)
\end{align*}
\noindent and
\begin{align*}
\Big| \frac{\delta^2 a_{i, j}}{\delta m^2}(t, x, [X^{s, \xi, (m)}_{t}])(z', z'')  & - \frac{\delta^2 a_{i, j}}{\delta m^2}(t, x, [X^{s, \xi, (m)}_{t}])(z', x'') \\
& \quad- \Big[\frac{\delta^2 a_{i, j}}{\delta m^2}(t, z, [X^{s, \xi, (m)}_{t}])(z', z'')- \frac{\delta^2 a_{i, j}}{\delta m^2}(t, z, [X^{s, \xi, (m)}_{t}])(z',x'') \Big]\Big| \\
& \quad \quad \leq K^+ |z-x|^{\beta' \eta} (|z''-x''|^{(1-\beta')\eta} \wedge 1)
\end{align*}
\noindent for any $\beta' \in [0,1]$. It follows from the two previous inequalities as well as the estimates \eqref{first:second:estimate:induction:decoupling:mckean} (with $n=0$), \eqref{bound:derivative:heat:kernel} (with $n=1$), \eqref{cross:deriv:mes:space:induction:decoupling:mckean}, the space-time inequality \eqref{space:time:inequality} and standard computations based on Gaussian kernels that
\begin{equation}
\label{bound:sum:1:to:6:delta:a:ell}
\sum_{\ell=1}^{6} |\delta a^\ell_{i, j} (\gv, \mu) | \leq K^{+}_{\beta'}  \frac{|z-x|^{\beta' \eta}}{(t-s)^{1-\frac{(1-\beta')\eta}{2}}}
\end{equation}
\noindent and
$$
|\delta a^{7}_{i, j}(\gv, \mu)| \leq K^{+} |z-x|^{\beta' \eta} \int_{(\mathbb{R}^d)^2} (|z'-x'|^{(1-\beta')\eta} \wedge 1)  |\partial_\mu^2 p_m(\mu, s, t, x', z')(\gv)| \,  \mu(dx') \, dz'.
$$

Gathering the above estimates concludes the proof of \eqref{recursive:bound:second:order:deriv:mes:holder:reg:space:a}. \\

\noindent \emph{Step 4: proof of the estimate \eqref{recursive:bound:second:order:deriv:mes:reg:holder:v:a:or:b}.}\\

We first apply the identity \eqref{full:expression:second:deriv} to the map $m \mapsto b_i(t, x, m)$ as already done in the two previous steps in order to decompose the difference $\partial^2_\mu [b_{i}(t, x, [X^{s, \xi, (m)}_{t}])](\gv_1)  - \partial^2_\mu [b_{i}(t, x, [X^{s, \xi, (m)}_{t}])](\gv_2)$ as the sum of the terms $\delta b^{\ell}_i(\gv_1, \gv_2)$, $\ell=1, \cdots, 7 $, defined by
{\small
\begin{align*}
\delta b^{1}_i(\gv_1, \gv_2) & = \int_{(\mathbb{R}^d)^2} \Big\{\frac{\delta^2 b_i}{\delta m^2}(t, x, [X^{s, \xi, (m)}_{t}])(z, z') - \frac{\delta^2 b_i}{\delta m^2}(t, x, [X^{s, \xi, (m)}_t])(z, v_1')\Big\} \\
& \quad \Big[\partial_x p_m(\mu, s, t, v_1, z)\otimes \partial_{x} p_m(\mu, s, t, v_1',z')  - \partial_x p_m(\mu, s, t, v_2, z)\otimes \partial_{x} p_m(\mu, s, t, v_2',z')  \Big] \, dz\, dz', 
\end{align*}
\begin{align*}
\delta b^{2}_i(\gv_1, \gv_2) & = \int_{(\mathbb{R}^d)^3}  \frac{\delta^2 b_i}{\delta m^2}(t, x, [X^{s, \xi , (m)}_t])(z, z') \\
& \quad  \Big[\partial_x  p_m(\mu, s, t, v_1, z) \otimes \partial_\mu p_m(\mu, s, t, x', z')(v_1') - \partial_x  p_m(\mu, s, t, v_2, z) \otimes \partial_\mu p_m(\mu, s, t, x', z')(v_2')\Big] \, dz\, dz'\, \mu(dx'), 
\end{align*}
\begin{align*}
\delta b^{3}_i(\gv_1, \gv_2) & = \int_{\mathbb{R}^d}  \frac{\delta b_i}{\delta m}(t, x, [X^{s, \xi, (m)}_t])(z)\, \Big[ \partial_\mu\partial_x p_m(\mu, s, t, v_1 ,z)(v_1') -  \partial_\mu\partial_x p_m(\mu, s, t, v_2 ,z)(v_2')  \Big]\, dz, 
\end{align*}
\begin{align*}
\delta b^{4}_i(\gv_1, \gv_2) & =  \int_{\mathbb{R}^d}  \frac{\delta b_i}{\delta m} (t, x, [X_t^{s, \xi , (m)}])(z) \Big[ \partial_{x}\partial_\mu p_m(\mu, s, t, v_1', z)(v_1) - \partial_{x}\partial_\mu p_m(\mu, s, t, v_2', z)(v_2) \Big]\, dz,
\end{align*}
\begin{align*}
\delta b^{5}_i(\gv_1, \gv_2) & =  \int_{(\mathbb{R}^d)^3}  \frac{\delta^2 b_i}{\delta m^2} (t, x, [X_t^{s, \xi , (m)}])(z, z')\\
& \quad \Big[ \partial_\mu p_m(\mu, s, t, x', z)(v_1)\otimes  \partial_{x}p_m(\mu, s, t, v_1',z') -  \partial_\mu p_m(\mu, s, t, x', z)(v_2)\otimes  \partial_{x}p_m(\mu, s, t, v_2',z')\Big]\, dz \, dz' \, \mu(dx'),
\end{align*}
\begin{align*}
\delta b^{6}_i(\gv_1, \gv_2) & =  \int_{(\mathbb{R}^d)^4}  \frac{\delta^2 b_i}{\delta m^2} (t, x, [X_t^{s, \xi , (m)}])(z, z')   \Big[ \partial_\mu p_m(\mu, s, t, x', z)(v_1)\otimes  \partial_{\mu}p_m(\mu, s, t, x'',z')(v_1') \\
& \quad -   \partial_\mu p_m(\mu, s, t, x', z)(v_2)\otimes  \partial_{\mu}p_m(\mu, s, t, x'',z')(v_2')\Big]\, dz \, dz' \, \mu(dx') \mu(dx''),
\end{align*}
\begin{align*}
\delta b^{7}_i(\gv_1, \gv_2) & =  \int_{(\mathbb{R}^d)^2}  \Big\{ \frac{\delta b_i}{\delta m} (t, x, [X_t^{s, \xi , (m)}])(z) -  \frac{\delta b_i}{\delta m} (t, x, [X_t^{s, \xi , (m)}])(x') \Big\} \\
& \quad \quad \Big[  \partial_\mu^2 p_m(\mu, s, t, x', z)(\gv_1) - \partial_\mu^2 p_m(\mu, s, t, x', z)(\gv_2) \Big]\, \,  dz\, \mu(dx').
\end{align*}
}
We then remark that the boundedness and uniform $\eta$-H\"older regularity of $[\delta b_i/\delta m](t, x, m)(.)$ directly yields
\begin{align}
|\delta b^{7}_i(\gv_1, \gv_2)| \leq K  \int_{(\mathbb{R}^d)^2} (|z-x'|^\eta \wedge 1) | \partial_\mu^2 p_m(\mu, s, t, x', z)(\gv_1) - \partial_\mu^2 p_m(\mu, s, t, x', z)(\gv_2)| \,  \mu(dx') \, dz. \label{delta:b7i:v}
\end{align}

Observe now that if $|\gv_1 -\gv_2|\geq (t-s)^{1/2}$ then we use the estimates \eqref{bound:derivative:heat:kernel}, \eqref{first:second:estimate:induction:decoupling:mckean} with $n=0$, rewrite $\delta b_i^{1}(\gv_1, \gv_2)$ as follows
{\small
\begin{align*}
\delta b^{1}_i(\gv_1, \gv_2) & = \int_{(\mathbb{R}^d)^2} \Big\{\frac{\delta^2 b_i}{\delta m^2}(t, x, [X^{s, \xi, (m)}_{t}])(z, z') - \frac{\delta^2 b_i}{\delta m^2}(t, x, [X^{s, \xi, (m)}_t])(z, v_1')\Big\}  \partial_x p_m(\mu, s, t, v_1, z)\otimes \partial_{x} p_m(\mu, s, t, v_1',z')   \, dz\, dz' \\
&  - \int_{(\mathbb{R}^d)^2} \Big\{\frac{\delta^2 b_i}{\delta m^2}(t, x, [X^{s, \xi, (m)}_{t}])(z, z') - \frac{\delta^2 b_i}{\delta m^2}(t, x, [X^{s, \xi, (m)}_t])(z, v_2')\Big\}  \partial_x p_m(\mu, s, t, v_2, z)\otimes \partial_{x} p_m(\mu, s, t, v_2',z')  \, dz \, dz'
\end{align*}
}

\noindent and then use the uniform $\eta$-H\"older regularity of $[\delta^2 b_i/\delta m^2](t, x, m)(z, .)$ and the space-time inequality \eqref{space:time:inequality}. We thus obtain  
$$
\sum_{\ell=1}^{6} |\delta b^{\ell}_i(\gv_1, \gv_2)| \leq \frac{K^+}{(t-s)^{1-\frac{\eta}{2}}} \leq K^+ \frac{|\gv_1-\gv_2|^\beta}{(t-s)^{1+\frac{\beta-\eta}{2}}}
$$
\noindent for any $\beta \in [0,1]$.

We eventually conclude that \eqref{recursive:bound:second:order:deriv:mes:reg:holder:v:a:or:b} is satisfied in the case $|\gv_1 -\gv_2|\geq (t-s)^{1/2}$ by combining the previous estimate with \eqref{delta:b7i:v}. Let us now assume that $|\gv_1- \gv_2|\leq (t-s)^{1/2}$. 
It follows from the uniform boundedness and $\eta$-H\"older regularity of the maps $ [\delta b_i/\delta m](t, x, m)(.), [\delta^2 b_i/\delta m^2](t, x, m)(z, .)$, the estimates \eqref{reg:heat:kernel:deriv} with $n=1$, \eqref{bound:derivative:heat:kernel}, \eqref{sensitivity:mes:deriv:cross:space:mes:pm}, \eqref{first:second:estimate:induction:decoupling:mckean} with $n=0$ and $n=1$ together with the fact that $|\gv_1 -\gv_2|\leq (t-s)^{1/2}$ which in particular implies that $|z'-v_1'|^\eta \leq |z'-v_2'|^\eta + |v_1'-v_2'|^\eta \leq |z'-v_2'|^\eta  + (t-s)^{\eta/2}$ and the space time inequality \eqref{space:time:inequality}, that
$$
\sum_{\ell=1}^{6} |\delta b^{\ell}_i(\gv_1, \gv_2) |\leq K_\beta \frac{|\gv_1-\gv_2|^\beta}{(t-s)^{1+\frac{\beta-\eta}{2}}}.
$$
%\noindent and
%$$
% |\delta b^{7}_i(\gv_1, \gv_2)| \leq K  \int_{(\mathbb{R}^d)^2}  |z-x'| \Big|\partial_\mu^2 p_m(\mu, s, t, x', z)(\gv_1) - \partial_\mu^2 p_m(\mu, s, t, x', z)(\gv_2) \Big|\, \,  dz\, \mu(dx').
%$$
The above estimate together with \eqref{delta:b7i:v} concludes the proof of \eqref{recursive:bound:second:order:deriv:mes:reg:holder:v:a:or:b} for $\partial^2_\mu [b_{i}(t, x, [X^{s, \xi, (m)}_{t}])](\gv_1)  - \partial^2_\mu [b_{i}(t, x, [X^{s, \xi, (m)}_{t}])]](\gv_2)$. The proof for $ \partial^2_\mu [a_{i, j}(t, x, [X^{s, \xi, (m)}_{t}])](\gv_1)  - \partial^2_\mu [ a_{i, j}(t, x, [X^{s, \xi, (m)}_{t}])](\gv_2) $ follows from completely analogous arguments and is thus omited.\\

\noindent \emph{Step 5: proof of the estimate \eqref{recursive:bound:deriv:mes:double:reg:holder:a}.}\\

We start from the decomposition of \emph{Step 3} and establish an appropriate estimate for $ \delta a^\ell_{i, j} (\gv_1, \mu) -  \delta a^\ell_{i, j} (\gv_2, \mu)$, $\ell=1,\cdots, 7$. We first apply \eqref{bound:sum:1:to:6:delta:a:ell} for $\beta'=0$ and $\beta'=1$ so that 
$$
\sum_{\ell=1}^{6} |\delta a^\ell_{i, j} (\gv_1, \mu) | + |\delta a^\ell_{i, j} (\gv_2, \mu) |  \leq K^{+} \left\{\frac{|z-x|^{\eta} \wedge 1}{t-s} \wedge \frac{1}{(t-s)^{1-\frac{\eta}{2}}} \right\}
$$
\noindent and remark that from the uniform boundedness and $\eta$-H\"older regularity of $\delta a_{i, j}(t, ., m)(.)$ 
\begin{align}
| & \delta a^{7}_{i, j}(\gv_1, \mu)  - \delta a^{7}_{i, j}(\gv_2, \mu) | \notag \\
& \leq K \int_{(\mathbb{R}^d)^2} (|z-x|^\eta \wedge |z'-x'|^\eta \wedge 1) \Big| \partial_\mu^2 p_m(\mu, s, t, x', z')(\gv_1) - \partial_\mu^2 p_m(\mu, s, t, x', z')(\gv_2) \Big| \,  \mu(dx') \, dz' . \label{diff:delta:a7:v:argument}
\end{align}

Combining the two previous estimates yields \eqref{recursive:bound:deriv:mes:double:reg:holder:a} in the case $|\gv_1-\gv_2|\geq (t-s)^{1/2}$. Assuming now that $|\gv_1-\gv_2|\leq (t-s)^{1/2}$, we write
\begin{align*}
\delta a^1_{i, j} (\gv_1, \mu)& -\delta a^1_{i, j}(\gv_2, \mu) \\
& :=  \int_{(\mathbb{R}^d)^2} \Big\{\frac{\delta^2 a_{i, j}}{\delta m^2}(t, x, [X^{s, \xi, (m)}_{t}])(z', z'') - \frac{\delta^2 a_{i, j}}{\delta m^2}(t, x, [X^{s, \xi, (m)}_{t}])(z',v_1') \\
&  \quad \quad  - \Big[\frac{\delta^2 a_{i, j}}{\delta m^2}(t, z, [X^{s, \xi, (m)}_{t}])(z', z'') - \frac{\delta^2 a_{i, j}}{\delta m^2}(t, z, [X^{s, \xi, (m)}_{t}])(z',v_1')\Big]\Big\}\,\\
& \quad \quad \, \Big[ \partial_{x} p_m(\mu, s, t, v_1, z') \otimes \partial_{x} p_m(\mu, s, t, v_1',z'') - \partial_{x} p_m(\mu, s, t, v_2, z') \otimes \partial_{x} p_m(\mu, s, t, v_2',z'') \Big] \, dz' dz'',
\end{align*}
\noindent then use the uniform boundedness and $\eta$-H\"older regularity of the map $[\delta^2 a_{i, j}/\delta m^2](t, ., m)(z',.)$ together with the estimates \eqref{bound:derivative:heat:kernel}, \eqref{reg:heat:kernel:deriv} with $n=1$, the inequality $|z''-v_1'|^\eta\leq |z"-v_2'|^\eta + (t-s)^{\eta/2}$ and the space time inequality \eqref{space:time:inequality} so that for any $\beta \in [0,1]$
$$
|\delta a^1_{i, j} (\gv_1, \mu) -\delta a^1_{i, j}(\gv_2, \mu)| \leq K^{+}_\beta  \left\{\frac{|z-x|^{\eta} \wedge 1}{(t-s)^{1+\frac{\beta}{2}}} \wedge \frac{1}{(t-s)^{1+\frac{\beta-\eta}{2}}} \right\} |\gv_1-\gv_2|^\beta.
$$

The upper-bounds on the remaining terms, namely $|\delta a^\ell_{i, j} (\gv_1, \mu) -\delta a^\ell_{i, j}(\gv_2, \mu)|$, $\ell=2, \cdots, 6$, can be derived following similar lines of reasonings. Namely, using the estimates \eqref{reg:heat:kernel:deriv}, \eqref{bound:derivative:heat:kernel}, \eqref{sensitivity:mes:deriv:cross:space:mes:pm}, \eqref{first:second:estimate:induction:decoupling:mckean} with $n=0$ and $n=1$, recalling that $\mathscr{C}^{1,0}_m(C, t-s) \leq \mathscr{C}^{1,0}_\infty(C, t-s) = \lim_{m\rightarrow \infty} \mathscr{C}^{1,0}_m(C, t-s)< \infty$, as well as the uniform $\eta$-H\"older regularity of the maps $[\delta a_{i, j}/\delta m](t, ., m)(z)$ and $[\delta^2 a_{i, j}/\delta m^2](t, ., m)(z, z')$, omitting some technical details, we get
$$
\sum_{\ell=2}^{6} |\delta a^\ell_{i, j} (\gv_1, \mu) -\delta a^\ell_{i, j}(\gv_2, \mu)| \leq K^+_\beta \frac{|z-x|^{\eta} \wedge 1}{(t-s)^{1+\frac{\beta-\eta}{2}}} |\gv_1-\gv_2|^\beta.
$$
Combining \eqref{diff:delta:a7:v:argument} with the two previous estimates allows to derive \eqref{recursive:bound:deriv:mes:double:reg:holder:a} in the case $|\gv_1-\gv_2|\leq (t-s)^{1/2}$. The proof is now complete.

\subsection{Proof of Corollary \ref{cor:deriv:time:and:mes:phat}.\\} \label{proof:cor:deriv:time:and:mes:phat}

\noindent \emph{Step 1: smoothness of the maps $(s, x, \mu) \mapsto \widehat{p}^{y}_{m+1}(\mu, s, r,  t , x, z)$, $\widehat{p}^{y}_{m+1}(\mu, s,  t , x, z)$.}\\

From Corollary A.1 in \cite{chaudruraynal:frikha}, the two maps $(s, x, \mu) \mapsto \widehat{p}^{y}_{m+1}(\mu, s, r, t, x, z)$ and $\widehat{p}^{y}_{m+1}(\mu, s, t, x, z)$ are in $C^{1,2,2}([0,r)\times \mathbb{R}^d \times \mathcal P_2(\mathbb R^d))$ and $C^{1,2, 2}([0,t)\times \mathbb{R}^d \times \mathcal P_2(\mathbb R^d))$ respectively. According to Definition \ref{def:space:c122}, it thus only remains to investigate the L-derivative of second order and its joint continuity with respect to the variables $s$, $x$, $\mu$ and $\gv$ to obtain that these maps belong to $C_f^{1,2, 2}([0,r)\times \mathcal P_2(\mathbb R^d))$ and $C_f^{1,2}([0,t)\times \mathbb{R}^d \times \mathcal P_2(\mathbb R^d))$ respectively. This can be deduced from Lemma \ref{lem:diff:and:control:deriv:coeff}, the estimates \eqref{deriv:mes:a:m}, \eqref{recursive:bound:deriv:a:or:b:stepm} together with the dominated convergence theorem.\\

\noindent \emph{Step 2: proof of the estimate \eqref{bound:second:deriv:mes:hat:p}.}\\

From \eqref{eq:def:de:phat:m}, \eqref{deriv:mes:a:m}, the dominated convergence theorem and Jacobi's formula, for any $r\in [s, t)$, it holds 
{\small
\begin{align}
& \partial_\mu \widehat{p}^{y}_{m+1}(\mu, s, r, t, x, z)(v)  \notag\\
& = -\frac12 \Bigg[  \tr\left(\left(\int_r^t a(r', y, [X^{s, \xi, (m)}_{r'}])\, dr'\right)^{-1} \int_r^t \partial_\mu [a(r', y, [X^{s, \xi, (m)}_{r'}])](v) \, dr' \right)  - (z-x)^{*} \left(\int_r^t a(r', y, [X^{s, \xi, (m)}_{r'}])\, dr'\right)^{-1}  \notag\\
&   \int_r^t \partial_\mu [a(r', y, [X^{s, \xi, (m)}_{r'}])](v) \, dr'  \left(\int_r^t a(r', y, [X^{s, \xi, (m)}_{r'}])\, dr'\right)^{-1} (z-x) \Bigg]  \widehat{p}^{y}_{m+1}(\mu, s, r, t, x, z)  \label{representation:formula:deriv:mes:p:hat} 
\end{align}

\noindent where

\begin{align*}
& \tr\left(\left(\int_r^t a(r', y, [X^{s, \xi, (m)}_{r'}])\, dr'\right)^{-1}  \int_r^t \partial_\mu [a(r', y, [X^{s, \xi, (m)}_{r'}])](v) \, dr' \right)  \\
& = \sum_{k, \ell =1}^{d} \Big[\left(\int_r^t a(r', y, [X^{s, \xi, (m)}_{r'}])\, dr'\right)^{-1}\Big]_{k, \ell}  \int_r^t \partial_\mu [a_{\ell, k}(r', y, [X^{s, \xi, (m)}_{r'}])](v) \, dr'
\end{align*}

\noindent and
\begin{align*}
& (z-x)^{*} \left(\int_r^t a(r', y, [X^{s, \xi, (m)}_{r'}])\, dr'\right)^{-1} \int_r^t  \partial_\mu [a(r', y, [X^{s, \xi, (m)}_{r'}])](v) \, dr'  \left(\int_r^t a(r', y, [X^{s, \xi, (m)}_{r'}])\, dr'\right)^{-1} (z-x) \\
& = \sum_{k, \ell=1}^d \Big[(z-x)^{*} \left(\int_r^t a(r', y, [X^{s, \xi, (m)}_{r'}])\, dr'\right)^{-1}\Big]_{k} \int_r^t  \partial_\mu [a_{k, \ell}(r', y, [X^{s, \xi, (m)}_{r'}])](v) \, dr' \\
& \quad \quad \times\Big[\left(\int_r^t a(r', y, [X^{s, \xi, (m)}_{r'}])\, dr'\right)^{-1} (z-x)\Big]_{\ell}.
\end{align*}
}
From the previous identities, \eqref{recursive:bound:deriv:a:or:b:stepm} and the dominated convergence theorem, we obtain
{\small
\begin{align}
& \partial^2_\mu \widehat{p}^{y}_{m+1}(\mu, s, r, t, x, z)(\gv)  \notag\\
& = -\frac12\left\{ \tr\Big( \Big[\int_r^t \partial_\mu [a(r', y, [X^{s, \xi, (m)}_{r'}])](v) \, dr' \Big] \otimes \partial_\mu\Big[\left(\int_r^t a(r', y, [X^{s, \xi, (m)}_{r'}])\, dr'\right)^{-1} \Big](v') \right. \notag \\
& \quad \left. + \left(\int_r^t a(r', y, [X^{s, \xi, (m)}_{r'}])\, dr'\right)^{-1} \int_r^t \partial^2_\mu [a(r', y, [X^{s, \xi, (m)}_{r'}])](\gv) \, dr' \Big) \right. \notag\\
& \quad \left. - (z-x)^{*} \int_r^t \partial_\mu [a(r', y, [X^{s, \xi, (m)}_{r'}])](v) \, dr'  \otimes \Big[ \partial_\mu\Big[\left(\int_r^t a(r', y, [X^{s, \xi, (m)}_{r'}])\, dr'\right)^{-1}\Big](v') \Big] \right. \label{representation:formula:second:order:deriv:mes:p:hat}\\ 
& \quad \quad \left. \left(\int_r^t a(r', y, [X^{s, \xi, (m)}_{r'}])\, dr'\right)^{-1} (z-x) \right. \notag \\
& \quad \left. - (z-x)^{*} \left(\int_r^t a(r', y, [X^{s, \xi, (m)}_{r'}])\, dr'\right)^{-1} \int_r^t \partial^2_\mu [a(r', y, [X^{s, \xi, (m)}_{r'}])](\gv) \, dr'  \right. \notag \\
& \quad \quad \left. \left(\int_r^t a(r', y, [X^{s, \xi, (m)}_{r'}])\, dr'\right)^{-1} (z-x) \right. \notag \\
& \quad \left. - (z-x)^{*} \left(\int_r^t a(r', y, [X^{s, \xi, (m)}_{r'}])\, dr'\right)^{-1} \int_r^t \partial_\mu [a(r', y, [X^{s, \xi, (m)}_{r'}])](v) \, dr'  \right. \notag \\
& \quad \left. \quad \otimes \partial_\mu \Big[\left(\int_r^t a(r', y, [X^{s, \xi, (m)}_{r'}])\, dr'\right)^{-1}\Big](v') (z-x) \right\} \widehat{p}^{y}_{m+1}(\mu, s, r, t, x, z) \notag\\
& \quad -\frac12\left( \tr\left(\left(\int_r^t a(r', y, [X^{s, \xi, (m)}_{r'}])\, dr'\right)^{-1} \int_r^t \partial_\mu [a(r', y, [X^{s, \xi, (m)}_{r'}])](v) \, dr' \right) \right. \notag\\
& \quad \left. - (z-x)^{*} \left(\int_r^t a(r', y, [X^{s, \xi, (m)}_{r'}])\, dr'\right)^{-1} \int_r^t \partial_\mu [a(r', y, [X^{s, \xi, (m)}_{r'}])](v) \, dr'  \right. \notag\\ 
& \quad \left. \quad \times \left(\int_r^t a(r', y, [X^{s, \xi, (m)}_{r'}])\, dr'\right)^{-1} (z-x) \right) \otimes \partial_\mu \widehat{p}^{y}_{m+1}(\mu, s, r, t, x, z)(v'). \notag
\end{align}
}
\noindent with the notations
{\small
\begin{align*}
& \tr\Big( \Big[\int_r^t \partial_\mu [a(r', y, [X^{s, \xi, (m)}_{r'}])](v) \, dr' \Big] \otimes \partial_\mu\Big[\left(\int_r^t a(r', y, [X^{s, \xi, (m)}_{r'}])\, dr'\right)^{-1} \Big](v')\\
& + \left(\int_r^t a(r', y, [X^{s, \xi, (m)}_{r'}])\, dr'\right)^{-1} \int_r^t \partial^2_\mu [a(r', y, [X^{s, \xi, (m)}_{r'}])](\gv) \, dr' \Big)\\
& =  \sum_{k, \ell =1}^{d}  \int_r^t \partial_\mu [a_{\ell, k}(r', y, [X^{s, \xi, (m)}_{r'}])](v) \, dr' \otimes \partial_\mu \Big[\left(\int_r^t a(r', y, [X^{s, \xi, (m)}_{r'}])\, dr'\right)^{-1}\Big]_{k, \ell} (v') \\
& \quad + \Big[\left(\int_r^t a(r', y, [X^{s, \xi, (m)}_{r'}])\, dr'\right)^{-1}\Big]_{k, \ell}  \int_r^t \partial^2_\mu [a_{\ell, k}(r', y, [X^{s, \xi, (m)}_{r'}])](\gv) \, dr',
\end{align*}
\noindent 
\begin{align*}
&(z-x)^{*} \int_r^t \partial_\mu [a(r', y, [X^{s, \xi, (m)}_{r'}])](v) \, dr'  \otimes \Big[ \partial_\mu\Big[\left(\int_r^t a(r', y, [X^{s, \xi, (m)}_{r'}])\, dr'\right)^{-1}\Big](v') \Big]  \\ 
&   \left(\int_r^t a(r', y, [X^{s, \xi, (m)}_{r'}])\, dr'\right)^{-1} (z-x) \\
&= \sum_{j, k, \ell=1}^d (z-x)_j  \int_r^t  \partial_\mu [a_{k, \ell}(r', y, [X^{s, \xi, (m)}_{r'}])](v) \, dr' \otimes \partial_\mu \Big[\left(\int_r^t a(r', y, [X^{s, \xi, (m)}_{r'}])\, dr'\right)^{-1} \Big]_{j, k}(v')\\
& \quad \quad \times\Big[\left(\int_r^t a(r', y, [X^{s, \xi, (m)}_{r'}])\, dr'\right)^{-1} (z-x)\Big]_{\ell},
\end{align*}
\begin{align*}
& (z-x)^{*} \left(\int_r^t a(r', y, [X^{s, \xi, (m)}_{r'}])\, dr'\right)^{-1} \int_r^t \partial_\mu [a(r', y, [X^{s, \xi, (m)}_{r'}])](v) \, dr'  \notag \\
& \quad  \quad \otimes \partial_\mu \Big[\left(\int_r^t a(r', y, [X^{s, \xi, (m)}_{r'}])\, dr'\right)^{-1}\Big](v') (z-x) \\
& =  \sum_{j, k, \ell =1}^d \Big[(z-x)^{*} \left(\int_r^t a(r', y, [X^{s, \xi, (m)}_{r'}])\, dr'\right)^{-1}\Big]_{k} \\
& \quad  \int_r^t  \partial_\mu [a_{k, \ell}(r', y, [X^{s, \xi, (m)}_{r'}])](v) \, dr' \otimes \partial_\mu \Big[\left(\int_r^t a(r', y, [X^{s, \xi, (m)}_{r'}])\, dr'\right)^{-1} \Big]_{\ell, j}(v') (z-x)_j
\end{align*}
\noindent and
\begin{align*}
&  (z-x)^{*} \left(\int_r^t a(r', y, [X^{s, \xi, (m)}_{r'}])\, dr'\right)^{-1} \int_r^t \partial^2_\mu [a(r', y, [X^{s, \xi, (m)}_{r'}])](\gv) \, dr'   \notag \\
& \quad \quad  \left(\int_r^t a(r', y, [X^{s, \xi, (m)}_{r'}])\, dr'\right)^{-1} (z-x) \\
& = \sum_{k, \ell=1}^d \Big[(z-x)^{*} \left(\int_r^t a(r', y, [X^{s, \xi, (m)}_{r'}])\, dr'\right)^{-1}\Big]_{k} \int_r^t  \partial^2_\mu [a_{k, \ell}(r', y, [X^{s, \xi, (m)}_{r'}])](\gv) \, dr' \\
& \quad \quad \times\Big[\left(\int_r^t a(r', y, [X^{s, \xi, (m)}_{r'}])\, dr'\right)^{-1} (z-x)\Big]_{\ell}.
\end{align*}
}
%Concerning the estimates, we have for any fixed $0 \leq s \leq r < t$, $(x, y, z)\in \rr^d \times \rr^d \times \rr^d$, $v,v' \in \rr^d$, $\mu \in \pp$
%\begin{eqnarray}
% \label{cross:mes:deriv:p:hat}&&\partial^2_\mu  \widehat{p}^{y}_{m+1}(\mu, s, r, t, x, z)(v,v') \\
%&&= -\frac12 (H_2 .g)\Big(\int_r^t  a(r', y, [X^{s,\xi, (m)}_{r'}])(y)\, dr', z-x\Big) . \int_r^t \partial^2_\mu  a(r', y, [X^{s,\xi, (m)}_{r'}])(v,v') \, dr' \nonumber \\
%&&  +\frac14  \left(\int_r^t \partial_\mu  a(r', y, [X^{s,\xi, (m)}_{r'}])(v') \, dr'\right)^{\dagger}  (H_4.g)\Big(\int_r^t a(r', y, [X^{s, \xi , (m)}_{r'}]) \, dr', z-x \Big) \notag\\
%&&\qquad . \int_r^t \partial_\mu  a(r', y, [X^{s,\xi, (m)}_{r'}])(v) \, dr'. \notag
%\end{eqnarray}
The estimate \eqref{bound:second:deriv:mes:hat:p} now follows by combining the previous identity with \eqref{recursive:bound:deriv:a:or:b}, \eqref{deriv:v:deriv:mes:pm:hat} with $n=0$, \eqref{deriv:mes:a:m} and the space-time inequality \eqref{space:time:inequality}. \\
 
%The estimate thus follows from the space time inequality \eqref{space:time:inequality}, estimates in Lemma \ref{lem:diff:and:control:deriv:coeff} and estimates \textcolor{blue}{[METTRE REF AUTRE PAPIER ESTIMEE COEFF PREMIERE DERIVE LIONS OU AJOUTER ICI]}, as dominated convergence theorem can be used thanks to \eqref{recursive:bound:deriv:a:or:b:stepm}.

\noindent \emph{Step 3: proof of the estimate \eqref{cross:mes:deriv:holder:p:hat:s:t}.\\}

Let us first observe that if $|x_1-x_2|\geq (t-s)^{1/2}$ then the result directly follows from \eqref{bound:second:deriv:mes:hat:p} with $r=s$. Assuming now that $|x_1-x_2| \leq (t-s)^{1/2}$, the estimate \eqref{cross:mes:deriv:holder:p:hat:s:t} follows from the identity \eqref{representation:formula:second:order:deriv:mes:p:hat} combined with \eqref{recursive:bound:deriv:a:or:b}, \eqref{estimate:partial:mes:partial:space:pmhat}, the estimate $|\partial_x \widehat{p}^{y}_{m+1}(\mu, s, r, t, x, z)|\leq K (t-r)^{-1/2} g(c(t-r), z-x)$ stemming for \eqref{standard}, \eqref{first:second:estimate:induction:decoupling:mckean} with $n=0$, \eqref{deriv:mes:a:m} and the space-time inequality \eqref{space:time:inequality}. The remaining technical details are omitted. \\

\noindent \emph{Step 4: proof of the estimate \eqref{cross:mes:deriv:reg:holder:terminal point:p:hat:s:r:t}.\\}

The result follows again from the identity \eqref{representation:formula:second:order:deriv:mes:p:hat} combined with the estimates \eqref{deriv:mes:a:m}, \eqref{diff:v:cross:deriv:diff:and:deriv:coeff} with $n=0$, \eqref{deriv:v:deriv:mes:pm:hat}, \eqref{gaussian:bound:diff:v:deriv:mes:hat:pm} with $n=0$, \eqref{recursive:bound:second:order:deriv:mes:reg:holder:v:a:or:b} and the space-time inequality \eqref{space:time:inequality}.

\subsection{Proof of Corollary \ref{cor:deriv:time:and:mes:parametrix:kernel}.\\} \label{proof:cor:deriv:time:and:mes:parametrix:kernel}

\noindent \emph{Step 1: $(s,\mu) \mapsto \mH_{m+1}(\mu, s, r, t, x, z) \in C_f^{1,2}([0,r)\times \mathcal P_2(\mathbb R^d))$.}\\

It follows from Corollary A.2 in \cite{chaudruraynal:frikha} that the map $(s,\mu) \mapsto  \mH_{m+1}(\mu, s, r, t, x, z)$ is in $C^{1,2}([0,r)\times \mathcal P_2(\mathbb R^d))$. It thus only remains to focus on the L-derivative of second order and its joint continuity with respect to the variables $s$, $x$, $\mu$ and $\gv$ to obtain that this map belongs to $C_f^{1,2}([0,r)\times \mathcal P_2(\mathbb R^d))$. It can be deduced from Corollary \ref{cor:deriv:time:and:mes:phat}, Lemma \ref{lem:diff:and:control:deriv:coeff}, the estimates \eqref{second:deriv:mes:induction:decoupling:mckean} and \eqref{recursive:bound:deriv:a:or:b:stepm} together with the dominated convergence theorem that each term appearing in the expression of $\mH_{m+1}(\mu, s, r, t, x, z)$ given by \eqref{eq:def:de:mH} belongs to $\mathcal{C}_f^{1, 2}([0,r)\times \pp)$ with a second order L-derivative being continuous with respect to the variables $s$, $x$, $\mu$ and $\gv$. We thus conclude that the map $(s, \mu) \mapsto \mH_{m+1}(\mu, s, r,  t , x, z) \in \mathcal{C}^{1, 2}_f([0,r) \times \pp)$ with continuous derivatives with respect to the variables $s$, $x$, $\mu$, $v$ and $\gv$.\\

\noindent \emph{Step 2: proof of the estimate \eqref{sec:mes:deriv:parametrix:kernel:s:r:t:with:beta}. \\}

Then, in order to compute the L-derivative of second order of the parametrix kernel $\mu \mapsto  \mH_{m+1}(\mu, s, r, t, x, z)$ at step $m+1$, we recall the following identity taken from the proof of Corollary A.2 in \cite{chaudruraynal:frikha}, namely, for any $v \in \mathbb{R}^d$, it holds
\begin{align}
\partial_\mu \mH_{m+1}(\mu, s, r ,t ,x ,z)(v) &  = : \big({\rm I}(\mu, s)(v)+ {\rm II}(\mu, s)(v)\big)\widehat{p}_{m+1}(\mu, s, r, t, x, z)  \label{deriv:mu:H:mp1} \\
& +{\rm III }(\mu, s)\, \partial_\mu\widehat{p}_{m+1}(\mu, s, r, t, x, z)(v), \nonumber
\end{align}
\noindent with
\begin{align*}
{\rm I }(\mu, s) (v) &: =  \left\{ - \sum_{i=1}^d \partial_\mu[b_i(r, x, [X^{s, \xi, (m)}_r])](v) H^{i}_1\left(\int_r^{t} a(r', z, [X^{s, \xi, (m)}_{r'}]) dr', z-x\right)   \right\}\\ \, 
%\widehat{p}_{m+1}(\mu, s, r, t, x, z)  \\
& + \Bigg\{ - \sum_{i=1}^d b_i(r, x, [X^{s, \xi, (m)}_r])  \partial_\mu  \Big[H^{i}_1\left(\int_r^{t} a(r',z, [X^{s, \xi, (m)}_{r'}]) dr', z-x\right)\Big](v)   \Bigg\}, \\
%\widehat{p}_{m+1}(\mu, s, r, t, x, z) \\
%& =: {\rm I}_1(v)+{\rm I}_2(v),  \\
{\rm II}(\mu, s)(v) & : =  \left\{ \frac12 \sum_{i, j=1}^d   \partial_\mu \Big[ a_{i, j}(r, x, [X^{s, \xi, (m)}_{r}]) - a_{i, j}(r, z, [X^{s, \xi, (m)}_r])\Big](v) \right.  \\
&\quad \left. \times H^{i, j}_2\left(\int_r^{t} a(r', z , [X^{s, \xi, (m)}_{r'}]) dr', z-x\right)  \right\}  \\
% \widehat{p}_{m+1}(\mu, s, r, t, x, z) \\
&+  \Bigg\{ \frac12 \sum_{i, j=1}^d  \left(a_{i, j}(r, x, [X^{s, \xi, (m)}_{r}]) - a_{i, j}(r, z, [X^{s, \xi, (m)}_r])\right)  \\ 
& \quad  \times  \partial_\mu \Big[ H^{i, j}_2\left(\int_r^{t} a(r', z, [X^{s, \xi, (m)}_{r'}]) dr', z-x\right)\Big](v) \Bigg\}, \\
%\,  \widehat{p}_{m+1}(\mu, s, r, t, x, z) \\
%& =: {\rm II}_1(v) + {\rm II}_2(v),  \\ \\
{\rm III}(\mu, s) &:= \left\{- \sum_{i=1}^d b_i(r, x, [X^{s, \xi, (m)}_r]) H^{i}_1\left(\int_r^{t} a(r', z, [X^{s, \xi, (m)}_{r'}]) dr', z-x\right)\right.\\
& +\left.\frac12 \sum_{i, j=1}^d (a_{i, j}(r, x, [X^{s, \xi, (m)}_{r}]) - a_{i, j}(r, z, [X^{s, \xi, (m)}_r])) H^{i, j}_2\left(\int_r^{t} a(r', z, [X^{s, \xi, (m)}_{r'}]) dr', z-x\right) \right\} \\
%& \quad \times \partial_\mu\widehat{p}_{m+1}(\mu, s, r, t, x, z)](v).
\end{align*}
\noindent so that for any $\gv=(v, v') \in \mathbb{R}^d \times \mathbb{R}^d$
\begin{align}
\partial^2_\mu \mH_{m+1}(\mu, s, r ,t ,x ,z)(\gv) & = \big(\partial_\mu {\rm I}(\mu, s)(\gv)+ \partial_\mu {\rm II}(\mu, s)(\gv)\big)\widehat{p}_{m+1}(\mu, s, r, t, x, z)\label{derivsec:mu:parametrix:kernel:mp1}  \\
& \quad  +\big( {\rm I}(\mu, s)(v)+  {\rm II}(\mu, s)(v)\big) \otimes \partial_\mu \widehat{p}_{m+1}(\mu, s, r, t, x, z)(v')\notag\\
& \quad  +  \partial_\mu\widehat{p}_{m+1}(\mu, s, r, t, x, z)(v) \otimes \partial_\mu {\rm III }(\mu, s)(v') \notag \\
& \quad + {\rm III }(\mu, s)\, \partial_\mu^2\widehat{p}_{m+1}(\mu, s, r, t, x, z)(\gv). \notag
\end{align}
We now estimate each of these terms.\\

\noindent (i) \emph{ Estimate on $|\big( {\rm I}(\mu, s)(v)+  {\rm II}(\mu, s)(v)\big) \otimes \partial_\mu \widehat{p}_{m+1}(\mu, s, r, t, x, z)(v')| + | \partial_\mu\widehat{p}_{m+1}(\mu, s, r, t, x, z)(v) \otimes \partial_\mu {\rm III }(\mu, s)(v')|$ in \eqref{derivsec:mu:parametrix:kernel:mp1}}.\\
 Note carefully that the estimate on the term $\big( {\rm I}(\mu, s)(v)+  {\rm II}(\mu, s)(v)\big)\partial_\mu \widehat{p}_{m+1}(\mu, s, r, t, x, z)(v')$ only involves first order L-derivatives and can be derived easily by following similar lines of reasonings as those employed to deal with the two terms ${\rm I}^{n}(v)$ and ${\rm II}^{n}(v)$ appearing in the decomposition of $\partial^n_v[\partial_\mu \mH_{m+1}(\mu, s, r, t, x, z)](v)$ in the proof of Corollary A.2 in \cite{chaudruraynal:frikha}, the only difference lying into the fact that it is multiplied by $\partial_\mu \widehat{p}_{m+1}(\mu, s, r, t, x, z)(v)$ instead of $\widehat{p}_{m+1}(\mu, s, r, t, x, z)$. Hence, using the estimate \eqref{deriv:v:deriv:mes:pm:hat} instead of the standard Gaussian estimate on $\widehat{p}_{m+1}(\mu, s, r, t, x, z)$, the same estimate holds up to a modification of the constant (which does not depend on $m$) and a time singularity of order $(r-s)^{(1-\eta)/2}$ coming from the additional L-derivative acting on $\widehat p_{m+1}$. We thus conclude that for any $\beta \in [0,1]$
 \begin{align}
 \Big| ( {\rm I}(\mu, s)(v) & +  {\rm II}(\mu, s)(v)) \otimes \partial_\mu \widehat{p}_{m+1}(\mu, s, r, t, x, z)(v') \Big| \label{estimate:I:II:partial:mes:p:hat:mp1} \\
 & \leq ( |{\rm I}(\mu, s)(v)|  +  |{\rm II}(\mu, s)(v)|) |\partial_\mu \widehat{p}_{m+1}(\mu, s, r, t, x, z)(v')|  \notag \\
 & \leq \frac{K_\beta}{(t-r)^{1-\frac{\beta\eta}{2}}(r-s)^{\frac{1-(1-\beta)\eta}{2}}} \frac{1}{(r-s)^{\frac{1-\eta}{2}}}g(c(t-r),z-x). \notag
 \end{align}
 
 We then deal with the term $ \partial_\mu\widehat{p}_{m+1}(\mu, s, r, t, x, z)(v) \otimes \partial_\mu {\rm III }(\mu, s)(v')$ by first observing that from its very definition one has 
\begin{equation}
\label{identity:partial:mes:pmp1:convolution:partial:mes:parametrix:kernel}
 \partial_\mu\widehat{p}_{m+1}(\mu, s, r, t, x, z)(v) \otimes \partial_\mu {\rm III }(\mu, s)(v') = \partial_\mu\widehat{p}_{m+1}(\mu, s, r, t, x, z)(v) \otimes \big({\rm I}(\mu, s)(v')+ {\rm II}(\mu, s)(v')\big),
\end{equation}

\noindent and then by using the fact that since our estimate are here uniform in the variables $v$ and $v'$, these variables do not play any particular role here. By doing so, we readily get that for any $v, \, v' \in \mathbb{R}^d$ and any $\beta \in [0,1]$
\begin{eqnarray}\label{sec:mes:deriv:parametrix:kernel:s:r:t:with:beta:parti}
&& | \partial_\mu\widehat{p}_{m+1}(\mu, s, r, t, x, z)(v) \otimes \partial_\mu {\rm III }(\mu, s)(v')|\le \frac{K_\beta}{(t-r)^{1-\frac{\beta\eta}{2}}(r-s)^{\frac{1-(1-\beta)\eta}{2}}} \frac{1}{(r-s)^{\frac{1-\eta}{2}}}g(c(t-r),z-x).\notag
\end{eqnarray}
%where we used estimate \textcolor{blue}{[METTRE REF AUTRE PAPIER OU ICI]} with $n=0$ with space-time inequality \eqref{space:time:inequality}.\\

\noindent (ii) \emph{Estimate on $\partial_\mu {\rm I}(\mu, s)(\gv)\widehat{p}_{m+1}(\mu, s, r, t, x, z)$ in \eqref{derivsec:mu:parametrix:kernel:mp1}}. It holds

\begin{eqnarray*}
\partial_\mu {\rm I }(\mu, s) (\gv) &: = &  - \sum_{i=1}^d H^{i}_1\left(\int_r^{t} a(r', z, [X^{s, \xi, (m)}_{r'}]) dr', z-x\right)   \partial_\mu^2 [b_i(r, x, [X^{s, \xi, (m)}_r])](\gv)   \\
&&  - \sum_{i=1}^d  \partial_\mu \left[b_i(r, x, [X^{s, \xi, (m)}_r])\right](v) \otimes \partial_\mu \left[H^{i}_1\left(\int_r^{t} a(r', z, [X^{s, \xi, (m)}_{r'}]) dr', z-x\right)\right](v')  \\
&&  - \sum_{i=1}^d b_i(r, x, [X^{s, \xi, (m)}_r])  \partial_\mu^2 \left[H^{i}_1\left(\int_r^{t} a(r',z, [X^{s, \xi, (m)}_{r'}]) dr', z-x\right)\right](\gv) \\
&& - \sum_{i=1}^d   \partial_\mu \left[H^{i}_1\left(\int_r^{t} a(r',z, [X^{s, \xi, (m)}_{r'}]) dr', z-x\right)\right](v) \otimes \partial_\mu \left[ b_i(r, x, [X^{s, \xi, (m)}_r])\right](v').
\end{eqnarray*}
Note that for any map $\mathcal{P}_2(\mathbb{R}^d) \ni  \mu \mapsto \Sigma(\mu)$ taking values in the set of positive definite matrix and being two times continuously L-differentiable, one has 
\begin{align}
\partial_\mu (\Sigma^{-1}(\mu))_{i, j}(v)  & = - (\Sigma^{-1}(\mu) \partial_\mu \Sigma(\mu) \Sigma^{-1}(\mu))_{i, j}(v) \label{first:order:diff:mat:mes} \\
&  = - \sum_{k_1, k_2} (\Sigma^{-1}(\mu))_{i, k_1} \partial_\mu(\Sigma(\mu))_{{k_1,k_2}}(v) (\Sigma^{-1}(\mu))_{k_2, j}\notag
\end{align}
\noindent so that
\begin{align}
\partial^2_\mu&  (\Sigma^{-1}(\mu))_{i, j}(\gv) \label{second:diff-mat-mes} \\
& =  \sum_{k_1, k_2, k_3, k_4} (\Sigma^{-1}(\mu))_{i, k_3} \partial_\mu(\Sigma(\mu))_{{k_1,k_2}}(v) \otimes  \partial_\mu(\Sigma(\mu))_{{k_3,k_4}}(v')   (\Sigma^{-1}(\mu))_{k_4, k_1} (\Sigma^{-1}(\mu))_{k_2, j} \notag \\
& \quad - \sum_{k_1, k_2} (\Sigma^{-1}(\mu))_{i, k_1} \partial^2_\mu(\Sigma(\mu))_{{k_1,k_2}}(\gv) (\Sigma^{-1}(\mu))_{k_2, j}\notag \\
& \quad  + \sum_{k_1, k_2, k_3, k_4} (\Sigma^{-1}(\mu))_{i, k_1} \partial_\mu(\Sigma(\mu))_{{k_1,k_2}}(v) \otimes  \partial_\mu(\Sigma(\mu))_{{k_3,k_4}}(v') (\Sigma^{-1}(\mu))_{k_2, k_3}  (\Sigma^{-1}(\mu))_{k_4, j}.\notag
\end{align}

The above identity together with \eqref{deriv:mes:a:m} and \eqref{recursive:bound:deriv:a:or:b} yield
\begin{align}
\Big| \partial_\mu^2 & \left[H^{i}_1\left(\int_r^{t} a(r',z, [X^{s, \xi, (m)}_{r'}]) dr', z-x\right)\right](\gv) \Big| \label{estimate:second:partial:deriv:H1}\\
 & \leq K^{+}  \left\{ \frac{|z-x|}{(t-r)(r-s)^{1-\frac{\eta}{2}}} +\frac{|z-x|}{(t-r)^2} \int_r^t  \int_{(\mathbb{R}^d)^2} (|y'-x'|^{\eta}\wedge 1) | \partial_\mu^2 p_{m}(\mu, s, r', x', y')(\gv)| \, \mu(dx') \, dy' \, dr'   \right\}.
\end{align}

The previous estimate combined again with \eqref{deriv:mes:a:m}, \eqref{recursive:bound:deriv:a:or:b} and the space time inequality \eqref{space:time:inequality} eventually imply
\begin{eqnarray}\label{sec:mes:deriv:parametrix:kernel:s:r:t:with:beta:partii}
&&|\partial_\mu {\rm I}(\mu, s)(\gv)\widehat{p}_{m+1}(\mu, s, r, t, x, z)| \\
&\le &  \frac{K^+}{(t-r)^{\frac{1}{2}}}\left\{ \frac{1}{(r-s)^{1-\frac{\eta}{2}}} +  \int_{(\rr^d)^2} (|y'-x'|^{\eta}\wedge 1) | \partial_\mu^2 p_{m}(\mu, s, r, x', y')(\gv)| \, \mu(dx') \, dy' \, \right.\notag\\
&& \left.  +\frac{1}{t-r} \int_r^t  \int_{(\mathbb{R}^d)^2} (|y'-x'|^{\eta}\wedge 1) | \partial_\mu^2 p_{m}(\mu, s, r', x', y')(\gv)| \, \mu(dx') \, dy' \, dr'   \right\} g(c(t-r), z-x).\notag
\end{eqnarray}

\noindent (iii) \emph{Estimate on $|\partial_\mu {\rm II}(\mu, s)(\gv) \widehat{p}_{m+1}(\mu, s, r, t, x, z)|$ in \eqref{derivsec:mu:parametrix:kernel:mp1}.} We write

\begin{eqnarray*}
&&\partial_\mu {\rm II}(\mu, s)(\gv) \\
& : = & \frac12 \sum_{i, j=1}^d    H^{i, j}_2\left(\int_r^{t} a(r', z , [X^{s, \xi, (m)}_{r'}]) dr', z-x\right)  \partial^2_\mu \left(a_{i, j}(r, x, [X^{s, \xi, (m)}_{r}]) - a_{i, j}(r, z, [X^{s, \xi, (m)}_r])\right)(\gv)  \\
&&+ \frac12 \sum_{i, j=1}^d   \partial_\mu \left(a_{i, j}(r, x, [X^{s, \xi, (m)}_{r}]) - a_{i, j}(r, z, [X^{s, \xi, (m)}_r])\right)(v) \otimes \partial_\mu \Big[H^{i, j}_2\left(\int_r^{t} a(r', z , [X^{s, \xi, (m)}_{r'}]) dr', z-x\right)\Big](v')  \\
&&+   \frac12 \sum_{i, j=1}^d \partial_\mu \Big[H^{i, j}_2\left(\int_r^{t} a(r', z, [X^{s, \xi, (m)}_{r'}]) dr', z-x\right)\Big](v) \otimes \partial_\mu\left(a_{i, j}(r, x, [X^{s, \xi, (m)}_{r}]) - a_{i, j}(r, z, [X^{s, \xi, (m)}_r])\right)(v')   \\ 
&&+   \frac12 \sum_{i, j=1}^d  \left(a_{i, j}(r, x, [X^{s, \xi, (m)}_{r}]) - a_{i, j}(r, z, [X^{s, \xi, (m)}_r])\right)  \partial_\mu^2 \Big[H^{i, j}_2\left(\int_r^{t} a(r', z, [X^{s, \xi, (m)}_{r'}]) dr', z-x\right)\Big](\gv) \\
& =: & {\rm II}_1(\mu, s)(\gv) + {\rm II}_2(\mu, s)(\gv) +  {\rm II}_3(\mu, s)(\gv) +  {\rm II}_4(\mu, s)(\gv).
\end{eqnarray*}

We deal with ${\rm II}_1(\mu, s)(\gv)$ by using \eqref{recursive:bound:second:order:deriv:mes:holder:reg:space:a} and the space time inequality \eqref{space:time:inequality} so that
\begin{align}
& |{\rm II}_1(\mu, s)(\gv)|  \widehat{p}_{m+1}(\mu, s, r, t, x, z) \label{eq:esti:de:II:1:dans:H} \\
 &\quad  \leq K^{+}_{\beta} |z-x|^{\beta \eta} \left\{ \frac{1}{(r-s)^{1-(1-\beta)\frac{\eta}{2}}}  +  \int_{(\mathbb{R}^d)^2} (|y-x'|^{(1-\beta)\eta} \wedge 1) \, |\partial^2_\mu p_{m}(\mu, s, r, x', y)(\gv)| \, \mu(dx')\, dy\right\} \notag \\
 & \quad \quad \times\left\{ \frac{|z-x|^2}{(t-r)^2} + \frac{1}{t-r}\right\} g(c(t-r), z-x) \notag \\
 & \quad \leq   \frac{K^{+}_{\beta}}{(t-r)^{1-\beta \frac{\eta}{2}}(r-s)^{1-(1-\beta)\frac{\eta}{2}}} \notag\\
 & \quad \times \left\{ 1 + (r-s)^{1-(1-\beta)\frac{\eta}{2}}  \int_{(\mathbb{R}^d)^2} (|y-x'|^{(1-\beta)\eta} \wedge 1) \, |\partial^2_\mu p_{m}(\mu, s, r, x', y)(v, v')| \, \mu(dx')\, dy\right\} \, g(c(t-r), z-x).\notag
 \end{align}

Observe that the two terms ${\rm II}_2(\mu, s)(\gv)$ and ${\rm II}_3(\mu, s)(\gv)$ can be handled in a similar manner following the computations provided in the proof of the estimate (A.15) of Corollary A.2 in \cite{chaudruraynal:frikha}. In comparison with, we have to take into account the additional time singularity of order $(r-s)^{(1-\eta)/2}$ coming from the first order L-derivative of the second order Hermite polynomial $H^{i, j}_2$ which is estimated using \eqref{first:order:diff:mat:mes} and \eqref{deriv:mes:a:m}. We thus derive that for all $\beta \in [0,1]$ 
\begin{eqnarray*}
&&|{\rm II}_2(\mu, s)(\gv)| + |{\rm II}_3(\mu, s)(\gv)|\\
&\le & K^{+}_\beta \frac{|z-x|^{\beta \eta}}{(r-s)^{\frac{1-(1-\beta)\eta}{2}}} \left(\frac{|z-x|^2}{(t-r)^{2}} + \frac{1}{t-r} \right) \frac{1}{(r-s)^{\frac{1-\eta}{2}}}
%&&+  \int_{(\mathbb{R}^d)^2} (|y'-x'|^{(1-\beta)\eta} \wedge 1) |[\partial_\mu^2 p_{m}(\mu, s, r, x', y')](v)+[\partial_\mu^2 p_{m}(\mu, s, r, x', y')](v')| \, dy' \mu(dx')   \Bigg\} \\
%&& \times K \Bigg\{\left(\frac{|z-x|^2}{(t-r)^{2}} + \frac{1}{t-r} \right) \frac{1}{(r-s)^{\frac{1-\eta}{2}}}\Bigg\},\\
\end{eqnarray*}
\noindent which in turn by the space-time inequality \eqref{space:time:inequality} yields
\begin{align}
(|{\rm II}_2(\mu, s)(\gv)| &+ |{\rm II}_3(\mu, s)(\gv)|)  \widehat{p}_{m+1}(\mu, s, r, t, x, z) \notag  \\
&  \leq K^{+}_\beta \frac{1}{(t-r)^{1-\beta \frac{\eta}{2}}(r-s)^{1-(2-\beta)\frac{\eta}{2}}} \, g(c(t-r), z-x).
\end{align}

We now deal with the last term ${\rm II}_4(\mu)(\gv)$. From the very definition of $H^{i, j}_2$ and \eqref{second:diff-mat-mes} we obtain
\begin{eqnarray} 
&&\Big| \partial^2_\mu  \left[ H^{i, j}_2\left(\int_r^{t} a(r', z, [X^{s, \xi, (m)}_v]) dr', z-x\right)\right](\gv) \Big|\notag\\ 
&  \leq&   K \Bigg\{ \left( \frac{|z-x|^2}{(t-r)^3} + \frac{1}{(t-r)^2} \right)   \int_r^{t} \max_{i, j} \Big|\partial^2_\mu [a_{i, j}(r', z, [X^{s, \xi, (m)}_v])](\gv)\Big| dr'   \notag\\
&& + \left( \frac{|z-x|^2}{(t-r)^4} + \frac{1}{(t-r)^3} \right)   \Big(\int_r^{t} \max_{i, j} \Big|\partial_\mu [a_{i, j}(r', z, [X^{s, \xi, (m)}_v])](v)\Big| dr'\Big)\notag\\ 
&& \quad \times \Big(\int_r^t \max_{i, j} \Big|\partial_\mu [a_{i, j}(r', z, [X^{s, \xi, (m)}_v])](v')\Big|  dr'\Big) \Bigg\}. \label{bound:second:deriv:mes:H2}
\end{eqnarray}
Hence, using the estimate \eqref{recursive:bound:deriv:a:or:b} of Lemma \ref{lem:diff:and:control:deriv:coeff} for the first term appearing in the above right-hand side and \eqref{deriv:mes:a:m} for the second term, the uniform $\eta$-H\"older regularity of $a(t, .,m)$ and the space time inequality \eqref{space:time:inequality} give
\begin{eqnarray*} 
&&\Big| {\rm II}_4(\mu, s)(\gv) \Big| \widehat{p}_{m+1}(\mu, s, r, t, x, z)  \notag\\ 
&\le &  K^+ \Big\{ \frac{1}{(t-r)^{1-\frac{\eta}{2}}(r-s)^{1-\frac{\eta}{2}}} + \frac{1}{(t-r)^{2-\frac{\eta}{2}}} \int_r^t \int_{(\mathbb{R}^d)^2} (|y'-x'|^{(1-\beta)\eta} \wedge 1) |\partial_\mu^2 p_{m}(\mu, s, r', x', y')(\gv)| \, \mu(dx') \, dy' \,  dr'\Big\} \\
&& \quad \times  g(c(t-r), z-x).
\end{eqnarray*}
%Therefore, thanks to the uniform $\eta-$H\"older regularity of $x \mapsto a_{i,j}(\cdot,x,\cdot)$ we obtain
%\begin{eqnarray*}
%|{\rm II}_4(\mu)(v,v')| &\le&  K|z-x|^{\eta} \Bigg\{\left( \frac{|z-x|^2}{(t-r)^4} + \frac{1}{(t-r)^3} \right) \frac{(t-r)^2}{(r-s)^{1-\eta}}\\
%&&+  \left( \frac{|z-x|^2}{(t-r)^3} + \frac{1}{(t-r)^2} \right) \int_r^t \int_{(\mathbb{R}^d)^2} (|y'-x'|^{(1-\beta)\eta} \wedge 1) |[\partial_\mu^2 p_{m}(\mu, s, r', x', y')](v,v')| \, dy' \mu(dx') dr'.
%\end{eqnarray*}

Finally, gathering the various estimates on the terms ${\rm II}_\ell(\mu)(\gv)$, $\ell=1,2,3,4$, we obtain
\begin{eqnarray}\label{sec:mes:deriv:parametrix:kernel:s:r:t:with:beta:partiii}
&&|\partial_\mu {\rm II}(\mu, s)(\gv)| \widehat{p}_{m+1}(\mu, s, r, t, x, z) \leq \frac{K^+_\beta}{(t-r)^{1-\beta\frac{\eta}{2}}(r-s)^{1-(1-\beta)\frac{\eta}{2}}}g(c(t-r), z-x)\\
 && \quad \times \Bigg\{(1 +  (r-s)^{1-(1-\beta)\frac{\eta}{2}} \int_{(\mathbb{R}^d)^2} (|y'-x'|^{(1-\beta)\eta}\wedge 1) | \partial_\mu^2 p_{m}(\mu, s, r, x', y') (\gv)| \,  \mu(dx') \, dy'  \notag \\
&& \quad \quad  + \frac{(r-s)^{1-(1-\beta)\frac{\eta}{2}}}{(t-r)^{1+(\beta-1)\frac{\eta}{2}}} \int_r^t \int_{(\mathbb{R}^d)^2} (|y'-x'|^{\eta}\wedge 1) |\partial_\mu^2 p_{m}(\mu, s, r', x', y')(\gv)| \,  \mu(dx')\, dy' \, dr'   \Bigg\}.\notag
\end{eqnarray}

\noindent (iv) \emph{Estimate on $|{\rm III}(\mu, s)\, \partial_\mu^2 \widehat{p}_{m+1}(\mu, s, r, t, x, z)(\gv)|$ in \eqref{derivsec:mu:parametrix:kernel:mp1}.} It follows from the pointwise estimate \eqref{bound:second:deriv:mes:hat:p} of Corollary \ref{cor:deriv:time:and:mes:phat}, the boundedness of $b_i$ and the uniform $\eta$-H\"older regularity of $ a_{i, j}(t, ., m)$ as well as the space time inequality \eqref{space:time:inequality} that
\begin{eqnarray}\label{sec:mes:deriv:parametrix:kernel:s:r:t:with:beta:partiv}
&&|{\rm III}(\mu, s)\, \partial_\mu^2 \widehat{p}_{m+1}(\mu, s, r, t, x, z)(\gv)| \\
& & \leq  \frac{K^+}{(t-r)^{1-\frac{\eta}{2}}(r-s)^{1-\frac \eta2}} \notag\\
&&  \times  \Bigg\{1+ \frac{(r-s)^{1-\frac \eta2}}{t-r} \int_r^t \int_{(\mathbb{R}^d)^2} (|y'-x'|^{\eta}\wedge 1) |\partial_\mu^2 p_{m}(\mu, s, r', x', y')(\gv)| \, \mu(dx') \, dy'  \, dr' \Bigg\}  g(c(t-r),z-x).\notag
\end{eqnarray}

Gathering the above estimates \eqref{estimate:I:II:partial:mes:p:hat:mp1}, \eqref{sec:mes:deriv:parametrix:kernel:s:r:t:with:beta:parti}, \eqref{sec:mes:deriv:parametrix:kernel:s:r:t:with:beta:partii}, \eqref{sec:mes:deriv:parametrix:kernel:s:r:t:with:beta:partiii}, \eqref{sec:mes:deriv:parametrix:kernel:s:r:t:with:beta:partiv} and plugging them into \eqref{derivsec:mu:parametrix:kernel:mp1} eventually gives \eqref{sec:mes:deriv:parametrix:kernel:s:r:t:with:beta}. \\

\noindent \emph{Step 3: proof of the estimate \eqref{second:order:mes:deriv:parametrix:kernel:s:r:t:reg:holder:v:argument}. \\}

We come back to the identity \eqref{derivsec:mu:parametrix:kernel:mp1} and investigate the H\"older regularity of each term of the decomposition with respect to the variable $\gv$.\\

\noindent (i) \emph{H\"older regularity of the two maps $\gv\mapsto \big( {\rm I}(\mu, s)(v) +  {\rm II}(\mu, s)(v)\big) \otimes \partial_\mu \widehat{p}_{m+1}(\mu, s, r, t, x, z)(v')$ and $\gv\mapsto \partial_\mu\widehat{p}_{m+1}(\mu, s, r, t, x, z)(v) \otimes \partial_\mu {\rm III }(\mu, s)(v')$.}\\

 Note that it can deduced from the estimate on the two terms ${\rm I}^{n}(v)$ and ${\rm II}^{n}(v)$ appearing in the decomposition of $\partial^n_v[\partial_\mu \mH_{m+1}(\mu, s, r, t, x, z)](v)$ in the proof of Corollary A.2 in \cite{chaudruraynal:frikha} that $v\mapsto {\rm I}(\mu, s)(v) + {\rm II}(\mu, s)(v)$ is continuously differentiable with 
 \begin{align}
 \label{partial:deriv:v:first:term:second:order:partial:mes:parametrix:kernel}
| \partial_v[{\rm I}(\mu, s)(v) + {\rm II}(\mu, s)(v)] & \otimes \partial_\mu \widehat{p}_{m+1}(\mu, s, r, t, x, z)(v') |  \\
& \leq \frac{K_\beta}{(t-r)^{1-\beta\frac{\eta}{2}}(r-s)^{1-(1-\beta)\frac{\eta}{2}+\frac{1-\eta}{2}}}\, g(c(t-r), z-x).\notag
 \end{align}
 
 \noindent where we also used the estimate \eqref{deriv:v:deriv:mes:pm:hat} instead of the standard Gaussian estimate on $\widehat{p}_{m+1}(\mu, s, r, t, x, z)$ for the last but one inequality. Similarly, using \eqref{deriv:v:deriv:mes:pm:hat} with $n=1$ and the computations employed in Corollary A.2 in \cite{chaudruraynal:frikha} to estimate ${\rm I}(\mu, s)(v)$ and ${\rm II}(\mu, s)(v)$, it follows that $v'\mapsto \big( {\rm I}(\mu, s)(v) +  {\rm II}(\mu, s)(v)\big) \otimes \partial_\mu \widehat{p}_{m+1}(\mu, s, r, t, x, z)(v')$ is continuously differentiable with
 \begin{align}
  \label{partial:deriv:v:first:term:bis:second:order:partial:mes:parametrix:kernel}
 \Big| [{\rm I}(\mu, s)(v) + {\rm II}(\mu, s)(v)] & \otimes  \partial_v[\partial_\mu \widehat{p}_{m+1}(\mu, s, r, t, x, z)](v') \Big| \\
 &  \leq  \frac{K_\beta}{(t-r)^{1-\beta \frac{\eta}{2}} (r-s)^{\frac{1-(1-\beta)\eta}{2}+1-\frac{\eta}{2}}} \, g(c(t-r), z-x).\notag
 \end{align}
 
 We now distinguish the two cases $|\gv_1-\gv_2| \geq (r-s)^{1/2}$ and $|\gv_1-\gv_2|\leq (r-s)^{1/2}$. In the first case, it directly follows from \eqref{estimate:I:II:partial:mes:p:hat:mp1} that for any $\beta, \, \beta' \in [0,1]$ 
 \begin{align*}
 | \big( {\rm I}(\mu, s)(v_1) +  {\rm II}(\mu, s)(v_1)\big)&  \otimes \partial_\mu \widehat{p}_{m+1}(\mu, s, r, t, x, z)(v_1') -  \big( {\rm I}(\mu, s)(v_2) +  {\rm II}(\mu, s)(v_2)\big) \otimes \partial_\mu \widehat{p}_{m+1}(\mu, s, r, t, x, z)(v_2') | \\
 & \leq \frac{K_{\beta'}}{(t-r)^{1-\beta'\frac{\eta}{2}}(r-s)^{\frac12 - (1-\beta')\frac{\eta}{2}}} \frac{1}{(r-s)^{\frac{1-\eta}{2}}}g(c(t-r),z-x)\\
 & \leq K_{\beta'} \frac{|\gv_1-\gv_2|^{\beta}}{(t-r)^{1-\beta' \frac{\eta}{2}}(r-s)^{1+\frac{\beta-\eta}{2}-\frac{(1-\beta')\eta}{2}}} \, g(c(t-r),z-x).
 \end{align*}
 
 In the second case, using the mean-value theorem together with \eqref{partial:deriv:v:first:term:second:order:partial:mes:parametrix:kernel} and \eqref{partial:deriv:v:first:term:bis:second:order:partial:mes:parametrix:kernel} imply that for any $\beta, \, \beta' \in [0,1]$
  \begin{align*}
 | \big( {\rm I}(\mu, s)(v_1) +  {\rm II}(\mu, s)(v_1)\big)&  \otimes \partial_\mu \widehat{p}_{m+1}(\mu, s, r, t, x, z)(v_1') -  \big( {\rm I}(\mu, s)(v_2) +  {\rm II}(\mu, s)(v_2)\big) \otimes \partial_\mu \widehat{p}_{m+1}(\mu, s, r, t, x, z)(v_2') | \\
 & \leq K_{\beta'} \frac{|\gv_1-\gv_2|}{(t-r)^{1-\beta'\frac{\eta}{2}}(r-s)^{\frac{3-\eta}{2} - (1-\beta')\frac{\eta}{2}}} \, g(c(t-r),z-x)\\
 & \leq K_{\beta'} \frac{|\gv_1-\gv_2|^{\beta}}{(t-r)^{1-\beta' \frac{\eta}{2}}(r-s)^{1+\frac{\beta-\eta}{2}-\frac{(1-\beta')\eta}{2}}} \, g(c(t-r),z-x).
 \end{align*}
 
 Gathering the two previous estimates, we conclude that for any $\beta, \, \beta' \in [0,1]$
  \begin{align}
 | \big( {\rm I}(\mu, s)(v_1) +  {\rm II}(\mu, s)(v_1)\big)&  \otimes \partial_\mu \widehat{p}_{m+1}(\mu, s, r, t, x, z)(v_1')  \label{holder:reg:v:arg:sec:mes:deriv:parametrix:kernel:s:r:t:with:beta:part:i1}  \\
 & \quad -  \big( {\rm I}(\mu, s)(v_2) +  {\rm II}(\mu, s)(v_2)\big) \otimes \partial_\mu \widehat{p}_{m+1}(\mu, s, r, t, x, z)(v_2') | \notag  \\
 & \leq K_{\beta'} \frac{|\gv_1-\gv_2|}{(t-r)^{1-\beta'\frac{\eta}{2}}(r-s)^{\frac{3-\eta}{2} - (1-\beta')\frac{\eta}{2}}} \, g(c(t-r),z-x) \notag\\
 & \leq K_{\beta'} \frac{|\gv_1-\gv_2|^{\beta}}{(t-r)^{1-\beta' \frac{\eta}{2}}(r-s)^{1+\frac{\beta-\eta}{2}-\frac{(1-\beta')\eta}{2}}} \, g(c(t-r),z-x). \notag
 \end{align}

The H\"older regularity of the map $\gv \mapsto \partial_\mu\widehat{p}_{m+1}(\mu, s, r, t, x, z)(v) \otimes \partial_\mu {\rm III }(\mu, s)(v')$ is a consequence of the previous estimates recalling the relation \eqref{identity:partial:mes:pmp1:convolution:partial:mes:parametrix:kernel}. We thus obtain that for $\beta, \, \beta' \in [0,1]$
\begin{align}
 | \partial_\mu\widehat{p}_{m+1}(\mu, s, r, t, x, z)(v_1) \otimes & \partial_\mu {\rm III }(\mu, s)(v_1') - \partial_\mu\widehat{p}_{m+1}(\mu, s, r, t, x, z)(v_2) \otimes \partial_\mu {\rm III }(\mu, s)(v_2') | \label{holder:reg:v:arg:sec:mes:deriv:parametrix:kernel:s:r:t:with:beta:part:i2}  \\
 & \leq K_{\beta'} \frac{|\gv_1-\gv_2|^{\beta}}{(t-r)^{1-\beta' \frac{\eta}{2}}(r-s)^{1+\frac{\beta-\eta}{2}-\frac{(1-\beta')\eta}{2}}} \, g(c(t-r),z-x). \notag
\end{align}

\noindent (ii) \emph{H\"older regularity of the map $\gv \mapsto \partial_\mu {\rm I}(\mu, s)(\gv) \widehat{p}_{m+1}(\mu, s, r, t, x, z)$}.\\

From \eqref{recursive:bound:second:order:deriv:mes:reg:holder:v:a:or:b}, we directly obtain
\begin{align*}
\Big| & H^{i}_1\left(\int_r^{t} a(r', z, [X^{s, \xi, (m)}_{r'}]) dr', z-x\right)  \Big[ \partial_\mu^2 [b_i(r, x, [X^{s, \xi, (m)}_r])](\gv_1) - \partial_\mu^2 [b_i(r, x, [X^{s, \xi, (m)}_r])](\gv_2)\Big] \Big| \\
& \leq K^+_\beta \frac{|z-x|}{t-r} \left\{ \frac{|\gv_1-\gv_2|^\beta}{(r-s)^{1+\frac{\beta-\eta}{2}}} + \int_{(\mathbb{R}^d)^2} (|y-x'|^\eta \wedge 1) \, | \partial^2_\mu p_{m}(\mu, s, r, x', y)(\gv_1) - \partial^2_\mu p_{m}(\mu, s, r, x', y)(\gv_2)| \, \mu(dx') \, dy \right\},
\end{align*}
 From the identity \eqref{first:order:diff:mat:mes}, \eqref{diff:v:cross:deriv:diff:and:deriv:coeff} with $n=0$ and \eqref{deriv:mes:a:m} we get
 \begin{align*}
 \Big| \partial_\mu \left[b_i(r, x, [X^{s, \xi, (m)}_r])\right](v_1) & \otimes \partial_\mu \left[H^{i}_1\left(\int_r^{t} a(r', z, [X^{s, \xi, (m)}_{r'}]) dr', z-x\right)\right](v_1') \\
  & -  \partial_\mu \left[b_i(r, x, [X^{s, \xi, (m)}_r])\right](v_2) \otimes \partial_\mu \left[H^{i}_1\left(\int_r^{t} a(r', z, [X^{s, \xi, (m)}_{r'}]) dr', z-x\right)\right](v_2') \Big| \\
  & \leq K^+_\beta| \gv_1-\gv_2|^\beta \frac{|z-x|}{(t-r)(r-s)^{1+\frac{\beta}{2}-\eta}}
 \end{align*}
 \noindent and similarly
 \begin{align*}
\Big| \partial_\mu&  \left[H^{i}_1\left(\int_r^{t} a(r',z, [X^{s, \xi, (m)}_{r'}]) dr', z-x\right)\right](v_1) \otimes \partial_\mu \left[ b_i(r, x, [X^{s, \xi, (m)}_r])\right](v_1') \\
& - \partial_\mu \left[H^{i}_1\left(\int_r^{t} a(r',z, [X^{s, \xi, (m)}_{r'}]) dr', z-x\right)\right](v_2) \otimes \partial_\mu \left[ b_i(r, x, [X^{s, \xi, (m)}_r])\right](v_2') \Big| \\
& \leq   K^{+}_\beta| \gv_1-\gv_2|^\beta \frac{|z-x|}{(t-r)(r-s)^{1+\frac{\beta}{2}-\eta}}.
 \end{align*}

It follows from the relation \eqref{second:diff-mat-mes}, \eqref{diff:v:cross:deriv:diff:and:deriv:coeff} with $n=0$, \eqref{deriv:mes:a:m}, \eqref{recursive:bound:second:order:deriv:mes:reg:holder:v:a:or:b} and the uniform boundedness of $b_i$
\begin{align*}
\Big| & b_i(r, x, [X^{s, \xi, (m)}_r])   \Big[\partial_\mu^2 \Big[H^{i}_1\left(\int_r^{t} a(r',z, [X^{s, \xi, (m)}_{r'}]) dr', z-x\right)\Big](\gv_1) - \partial_\mu^2 \Big[H^{i}_1\left(\int_r^{t} a(r',z, [X^{s, \xi, (m)}_{r'}]) dr', z-x\right)\Big](\gv_2)\Big] \Big|\\
& \leq K^+_\beta \left\{ |\gv_1 - \gv_2|^\beta \frac{|z-x|}{(t-r)(r-s)^{1+\frac{\beta}{2}-\eta}} \right. \\
& \quad \left. + \frac{|z-x|}{(t-r)^2} \int_r^t | \partial_\mu^2[a_{i, j}(r', z, [X_{r'}^{s, \xi, (m)}])](\gv_1) - \partial_\mu^2[a_{i, j}(r', z, [X_{r'}^{s, \xi, (m)}])](\gv_2)| \, dr' \right\}\\
& \leq K^+_\beta \left\{ |\gv_1 - \gv_2|^\beta \frac{|z-x|}{(t-r)(r-s)^{1+\frac{\beta-\eta}{2}}}  \right. \\
& \quad \left. + \frac{|z-x|}{(t-r)^2} \int_r^t \int_{\mathbb{R}^d} (|y-x'|^\eta \wedge 1) \, | \partial^2_\mu p_{m}(\mu, s, r', x', y)(\gv_1) - \partial^2_\mu p_{m}(\mu, s, r', x', y)(\gv_2)| \, \mu(dx') \, dy  \, dr' \right\}.
\end{align*}

Gathering the above estimates and using the space time inequality \eqref{space:time:inequality}, we conclude
\begin{align}
\Big| \big[&  \partial_\mu {\rm I }(\mu, s) (\gv_1) - \partial_\mu {\rm I }(\mu, s) (\gv_2)  \big]\widehat{p}_{m+1}(\mu, s, r, t, x, z)\Big|  \label{holder:reg:v:arg:sec:mes:deriv:parametrix:kernel:s:r:t:with:beta:part:ii}  \\
& \leq  \frac{K^+_\beta}{(t-r)^{\frac12}(r-s)^{1+\frac{\beta-\eta}{2}}} \left\{ |\gv_1 - \gv_2|^\beta + \int_{(\mathbb{R}^d)^2} (|y-x'|^\eta \wedge 1) \, | \partial^2_\mu p_{m}(\mu, s, r, x', y)(\gv_1) - \partial^2_\mu p_{m}(\mu, s, r, x', y)(\gv_2)| \, \mu(dx') \, dy\right. \notag \\
&  \left. +  \frac{(r-s)^{1+\frac{\beta-\eta}{2}}}{t-r} \int_r^t \int_{\mathbb{R}^d} (|y-x'|^\eta \wedge 1) \, | \partial^2_\mu p_{m}(\mu, s, r', x', y)(\gv_1) - \partial^2_\mu p_{m}(\mu, s, r', x', y)(\gv_2)| \, \mu(dx') \, dy  \, dr' \right\} g(c(t-r), z-x).\notag
\end{align}

\noindent (iii) \emph{H\"older regularity of the map $\gv \mapsto \partial_\mu {\rm II}(\mu, s)(\gv) \widehat{p}_{m+1}(\mu, s, r, t, x, z)$}.\\

We investigate the H\"older regularity of each map ${\rm II}_{\ell}(\mu, s)(.)$, $\ell=1, \cdots, 4$, of the decomposition of $\partial_\mu {\rm II}(\mu, s)(\gv)$ appearing in \emph{Step 2}.

From \eqref{recursive:bound:deriv:mes:double:reg:holder:a}, for any $\beta \in [0,\eta)$, we obtain
\begin{align*}
\Big| &{\rm II}_{1}(\mu, s)(\gv_1) - {\rm II}_1(\mu, s)(\gv_2) \Big| \leq K^+_\beta \left\{ \frac{|z-x|^2}{(t-r)^2} + \frac{1}{t-r}\right\} \left\{\frac{(|z-x|^\eta \wedge 1)}{(r-s)^{1+\frac{\beta}{2}}} \wedge \frac{1}{(r-s)^{1+\frac{\beta-\eta}{2}}} \right\}\\
&\quad \times \left\{ \Big. {|\gv_1-\gv_2|^\beta} + (r-s)^{1+\frac{\beta}{2}}\int_{(\mathbb{R}^d)^2} \Big|\partial^2_{\mu}p_{m}(\mu, s , r, x', y)(\gv_1) -\partial^2_{\mu} p_{m}(\mu, s , r, x', y)(\gv_2)\Big| \, \mu(dx') \, dy   \right. \nonumber \\
 & \quad \quad \left.  +(r-s)^{1+\frac{\beta-\eta}{2}} \int_{(\mathbb{R}^d)^2} (|y'-x'|^\eta \wedge 1) \Big| \partial^2_{\mu} p_{m}(\mu, s , r, x', y)(\gv_1) - \partial^2_{\mu}p_{m}(\mu, s , r, x', y)(\gv_2) \Big| \, \mu(dx') \, dy \right\}. 
\end{align*}

From the relation \eqref{first:order:diff:mat:mes}, \eqref{diff:v:cross:deriv:diff:and:deriv:coeff} with $n=0$, \eqref{recursive:bound:deriv:mes:holder:reg:a} with $n=0$ and $\beta=1$, we obtain
\begin{align*}
\Big| &  \partial_\mu \Big[H^{i, j}_2\left(\int_r^{t} a(r', z , [X^{s, \xi, (m)}_{r'}]) dr', z-x\right)\Big](v_1) -  \partial_\mu \Big[H^{i, j}_2\left(\int_r^{t} a(r', z , [X^{s, \xi, (m)}_{r'}]) dr', z-x\right)\Big](v_2) \Big| \\
& \times  |\partial_\mu \left(a_{i, j}(r, x, [X^{s, \xi, (m)}_{r}]) - a_{i, j}(r, z, [X^{s, \xi, (m)}_r])\right)(v_2)|  \\
& \leq K^{+}_\beta \left\{\frac{|z-x|^{2+\eta}}{(t-r)^2}  +  \frac{|z-x|^\eta}{(t-r)}\right\} \frac{|v_1-v_2|^\beta}{(r-s)^{1+\frac{\beta-\eta}{2}}}.
\end{align*}

\noindent From \eqref{recursive:bound:deriv:mes:holder:reg:a} with $\beta=1$ and $n=0$ if $|v_1-v_2|\geq (r-s)^{1/2}$ or with $n=1$ together with the mean-value theorem if $|v_1-v_2|< (r-s)^{1/2}$, we deduce that for any $\beta \in [0,1]$
\begin{align*}
\Big| \partial_\mu & \left(a_{i, j}(r, x, [X^{s, \xi, (m)}_{r}])  - a_{i, j}(r, z, [X^{s, \xi, (m)}_r])\right)(v_1) - \partial_\mu \left(a_{i, j}(r, x, [X^{s, \xi, (m)}_{r}]) - a_{i, j}(r, z, [X^{s, \xi, (m)}_r])\right)(v_2) \Big|\\
& \quad \leq K_\beta \frac{|z-x|^\eta}{(r-s)^{\frac{1+\beta}{2}}} |v_1-v_2|^\beta
\end{align*}
\noindent which in turn, combined with \eqref{first:order:diff:mat:mes} and \eqref{deriv:mes:a:m}, yield
\begin{align*}
\partial_\mu & \left(a_{i, j}(r, x, [X^{s, \xi, (m)}_{r}])  - a_{i, j}(r, z, [X^{s, \xi, (m)}_r])\right)(v_1) - \partial_\mu \left(a_{i, j}(r, x, [X^{s, \xi, (m)}_{r}]) - a_{i, j}(r, z, [X^{s, \xi, (m)}_r])\right)(v_2) \\
& \quad \times \Big|   \partial_\mu \Big[H^{i, j}_2\left(\int_r^{t} a(r', z , [X^{s, \xi, (m)}_{r'}]) dr', z-x\right)\Big](v_1) \Big|\\
& \quad \leq K_\beta \, \left\{\frac{|z-x|^{2+\eta}}{(t-r)^2}  +  \frac{|z-x|^\eta}{(t-r)}\right\} \frac{|v_1-v_2|^\beta}{(r-s)^{1+\frac{\beta-\eta}{2}}}.
\end{align*}

Gathering the two previous estimates, we obtain
\begin{align*}
|{\rm II}_2(\mu, s)(\gv_1) - {\rm II}_2(\mu, s)(\gv_2)| & + |{\rm II}_3(\mu, s)(\gv_1) - {\rm II}_3(\mu, s)(\gv_2)| \\
& \leq K_\beta\,  \left\{\frac{|z-x|^{2+\eta}}{(t-r)^2}  +  \frac{|z-x|^\eta}{(t-r)}\right\} \frac{|\gv_1-\gv_2|^\beta}{(r-s)^{1+\frac{\beta-\eta}{2}}}.
\end{align*}

From the relation \eqref{second:diff-mat-mes}, \eqref{deriv:mes:a:m}, \eqref{diff:v:cross:deriv:diff:and:deriv:coeff} with $n=0$, \eqref{recursive:bound:second:order:deriv:mes:reg:holder:v:a:or:b} and then the uniform $\eta$-H\"older regularity of $a(t, ., m)$, we get
\begin{align*}
| & {\rm II}_4(\mu, s)(\gv_1) -  {\rm II}_4(\mu, s)(\gv_2) |\notag  \\
& \leq K^{+}_\beta \left\{\frac{|z-x|^{2+\eta}}{(t-r)^{2}} + \frac{|z-x|^\eta}{t-r}\right\}  \frac{ |\gv_1-\gv_2|^\beta}{(r-s)^{1+\frac{\beta}{2}-\eta}} \notag \\
&  + K \left\{ \frac{|z-x|^{2+\eta}}{(t-r)^3}+\frac{|z-x|^\eta}{(t-r)^2}\right\} \int_r^t \max_{i, j}| \partial^2_\mu [a_{i, j}(r', z, [X^{s, \xi, (m)}_{r'}])](\gv_1) - \partial^2_\mu [a_{i, j}(r', z, [X^{s, \xi, (m)}_{r'}])](\gv_2) | \, dr' \notag\\
& \leq K^+_\beta \left\{\frac{|z-x|^{2+\eta}}{(t-r)^{2}} + \frac{|z-x|^\eta}{t-r}\right\}  \frac{ |\gv_1-\gv_2|^\beta}{(r-s)^{1+\frac{\beta-\eta}{2}}} \notag \\
&  + K^+_\beta \left\{ \frac{|z-x|^{2+\eta}}{(t-r)^3}+\frac{|z-x|^\eta}{(t-r)^2}\right\} \int_r^t \int_{\mathbb{R}^d}  (|y-x'|^\eta \wedge 1) \, | \partial^2_\mu p_{m}(\mu, s, r', x', y)(\gv_1) - \partial^2_\mu p_{m}(\mu, s, r', x', y)(\gv_2)| \, \mu(dx') \, dy  \, dr'. \notag
\end{align*}

Collecting the estimates on $| {\rm II}_\ell(\mu, s)(\gv_1) - {\rm II}_\ell(\mu, s)(\gv_2) |$, $\ell=1, \cdots, 4$ and using the space time inequality \eqref{space:time:inequality}, we obtain
\begin{align}
\Big| \big[&  \partial_\mu {\rm II }(\mu, s) (\gv_1) - \partial_\mu {\rm II }(\mu, s) (\gv_2)  \big] \widehat{p}_{m+1}(\mu, s, r, t, x, z)\Big| \label{holder:reg:v:arg:sec:mes:deriv:parametrix:kernel:s:r:t:with:beta:part:iii} \\
& \leq K^+_\beta \left\{\frac{1}{(t-r)^{1-\frac{\eta}{2}}(r-s)^{1+\frac{\beta}{2}}} \wedge \frac{1}{(t-r)(r-s)^{1+\frac{\beta-\eta}{2}}} \right\} \notag \\
&\quad  \times \left\{ \Big. {|\gv_1-\gv_2|^\beta} + (r-s)^{1+\frac{\beta}{2}}\int_{(\mathbb{R}^d)^2} \Big|\partial^2_{\mu}p_{m}(\mu, s , r, x', y)(\gv_1) -\partial^2_{\mu} p_{m}(\mu, s , r, x', y)(\gv_2)\Big| \, \mu(dx') \, dy   \right. \notag \\
 & \quad \quad \left.  +(r-s)^{1+\frac{\beta-\eta}{2}} \int_{(\mathbb{R}^d)^2} (|y'-x'|^\eta \wedge 1) \Big| \partial^2_{\mu} p_{m}(\mu, s , r, x', y)(\gv_1) - \partial^2_{\mu}p_{m}(\mu, s , r, x', y)(\gv_2)| \, \mu(dx') \, dy \right. \notag\\
 & \quad \quad \left. + \frac{(r-s)^{1+ \frac{\beta-\eta}{2}}}{t-r} \int_r^t \int_{\mathbb{R}^d}  (|y-x'|^\eta \wedge 1) \, | \partial^2_\mu p_{m}(\mu, s, r', x', y)(\gv_1) - \partial^2_\mu p_{m}(\mu, s, r', x', y)(\gv_2)| \, \mu(dx') \, dy  \, dr'\right\} \notag\\
 & \quad \times g(c(t-r), z-x).\notag
\end{align}

\noindent (iv) \emph{H\"older regularity of the map $\gv \mapsto |{\rm III}(\mu, s)\, \partial_\mu^2 \widehat{p}_{m+1}(\mu, s, r, t, x, z)(\gv)|$.} It follows from the H\"older regularity estimate \eqref{cross:mes:deriv:reg:holder:terminal point:p:hat:s:r:t} of Corollary \ref{cor:deriv:time:and:mes:phat}, the boundedness of $b_i$ and the uniform $\eta$-H\"older regularity of $ a_{i, j}(t, ., m)$ as well as the space time inequality \eqref{space:time:inequality} that for any $\beta \in [0,\eta)$
\begin{align}\label{holder:reg:v:arg:sec:mes:deriv:parametrix:kernel:s:r:t:with:beta:part:iv}
&|{\rm III}(\mu, s)\, [ \partial_\mu^2 \widehat{p}_{m+1}(\mu, s, r, t, x, z)(\gv_1) -  \partial_\mu^2 \widehat{p}_{m+1}(\mu, s, r, t, x, z)(\gv_2) ] | \\
& \leq  \frac{K^+_\beta}{(t-r)^{1-\frac{\eta}{2}}(r-s)^{1+\frac{\beta- \eta}{2}}}  \notag\\
& \times \Bigg\{1+ \frac{(r-s)^{1+\frac{\beta-\eta}{2}}}{t-r} \int_r^t \int_{(\mathbb{R}^d)^2} (|y'-x'|^{\eta}\wedge 1) |\partial_\mu^2 p_{m}(\mu, s, r', x', y')(\gv_1) - \partial_\mu^2 p_{m}(\mu, s, r', x', y')(\gv_2)| \, \mu(dx')\, dy'  \, dr' \Bigg\} \notag \\
& \quad \times g(c(t-r),z-x). \notag
\end{align}

Gathering the above estimates \eqref{holder:reg:v:arg:sec:mes:deriv:parametrix:kernel:s:r:t:with:beta:part:i1}, \eqref{holder:reg:v:arg:sec:mes:deriv:parametrix:kernel:s:r:t:with:beta:part:i2} (both with $\beta'=1$), \eqref{holder:reg:v:arg:sec:mes:deriv:parametrix:kernel:s:r:t:with:beta:part:ii}, \eqref{holder:reg:v:arg:sec:mes:deriv:parametrix:kernel:s:r:t:with:beta:part:iii}, \eqref{holder:reg:v:arg:sec:mes:deriv:parametrix:kernel:s:r:t:with:beta:part:iv} and plugging them into \eqref{deriv:mu:H:mp1} eventually gives \eqref{second:order:mes:deriv:parametrix:kernel:s:r:t:reg:holder:v:argument}. 

\subsection{Proof of Lemma \ref{lem:reg:mes:deriv:sec:mes:coeff:parametrix:density:kernel}.\\}\label{proof:lem:reg:mes:deriv:sec:mes:coeff:parametrix:density:kernel}

\noindent \emph{Step 1: proof of the estimate \eqref{diff:mes:second:deriv:mes:a:b}.}\\

We first observe that if $W_2(\mu, \mu') \geq (t-s)^{1/2}$ then the result directly follows from \eqref{recursive:bound:deriv:a:or:b} combined with \eqref{second:deriv:mes:induction:decoupling:mckean} and the space time inequality \eqref{space:time:inequality}. We thus assume that $W_2(\mu, \mu') \leq (t-s)^{1/2}$ for the rest of the proof. It now follows from the identity \eqref{full:expression:second:deriv} applied to the map $m \mapsto b_i(t, x, m)$ that the difference $\partial^2_\mu [b_{i}(t, x, [X^{s, \xi, (m)}_{t}])](\gv)  - \partial^2_\mu [b_{i}(t, x, [X^{s, \xi, (m)}_{t}])]_{|\mu= \mu'}(\gv)$ writes as the sum of the terms $\delta b^{\ell}_i(\mu, \mu')$, $\ell=1, \cdots, 7 $, defined by
{\small
\begin{align*}
\delta b^{1}_i(\mu, \mu') & = \int_{(\mathbb{R}^d)^2} \Big[\frac{\delta^2 b_i}{\delta m^2}(t, x, [X^{s, \xi, (m)}_{t}])(z, z') - \frac{\delta^2 b_i}{\delta m^2}(t, x, [X^{s, \xi', (m)}_t])(z, z')\Big] \\
& \quad  \quad \partial_x p_m(\mu, s, t, v, z)\otimes \partial_{x} p_m(\mu, s, t, v',z') \, dz\, dz' \\
& + \int_{(\mathbb{R}^d)^2} \Big[\frac{\delta^2 b_i}{\delta m^2}(t, x, [X^{s, \xi', (m)}_{t}])(z, z') - \frac{\delta^2 b_i}{\delta m^2}(t, x, [X^{s, \xi', (m)}_t])(z, v')\Big] \\
& \quad \Big[\partial_x p_m(\mu, s, t, v, z)\otimes \partial_{x} p_m(\mu, s, t, v',z')  - \partial_x p_m(\mu', s, t, v, z)\otimes \partial_{x} p_m(\mu', s, t, v',z')  \Big] \, dz\, dz',
\end{align*}
\begin{align*}
\delta b^{2}_i(\mu, \mu') & = \int_{(\mathbb{R}^d)^3} \Big[ \frac{\delta^2 b_i}{\delta m^2}(t, x, [X^{s, \xi , (m)}_t])(z, z') - \frac{\delta^2 b_i}{\delta m^2}(t, x, [X^{s, \xi' , (m)}_t])(z, z') \Big] \\
& \quad  \partial_x  p_m(\mu, s, t, v, z) \otimes \partial_\mu p_m(\mu, s, t, x', z')(v')  \, dz\, dz'\, \mu(dx')\\
& +  \int_{(\mathbb{R}^d)^3}  \frac{\delta^2 b_i}{\delta m^2}(t, x, [X^{s, \xi' , (m)}_t])(z, z') \\
& \quad  \Big[\partial_x  p_m(\mu, s, t, v, z) \otimes \partial_\mu p_m(\mu, s, t, x', z')(v') - \partial_x  p_m(\mu', s, t, v, z) \otimes \partial_\mu p_m(\mu', s, t, x', z')(v')\Big] \, dz\, dz'\, \mu(dx') \\
&  +  \int_{(\mathbb{R}^d)^3}  \frac{\delta^2 b_i}{\delta m^2}(t, x, [X^{s, \xi' , (m)}_t])(z, z') \\
& \quad  \partial_x  p_m(\mu', s, t, v, z) \otimes \partial_\mu p_m(\mu', s, t, x', z')(v') \, dz\, dz'\, [\mu-\mu'](dx'),
\end{align*}

\begin{align*}
\delta b^{3}_i(\mu, \mu') & = \int_{\mathbb{R}^d}  \Big[\frac{\delta b_i}{\delta m}(t, x, [X^{s, \xi, (m)}_t])(z) - \frac{\delta b_i}{\delta m}(t, x, [X^{s, \xi', (m)}_t])(z) \Big]\, \partial_\mu\partial_x p_m(\mu, s, t, v ,z)(v')  dz \\
& \quad +  \int_{\mathbb{R}^d}  \frac{\delta b_i}{\delta m}(t, x, [X^{s, \xi', (m)}_t])(z) \, \Big[ \partial_\mu\partial_x p_m(\mu, s, t, v ,z)(v') -  \partial_\mu\partial_x p_m(\mu', s, t, v ,z)(v')  \Big]\, dz,
\end{align*}
\begin{align*}
\delta b^{4}_i(\mu, \mu') & =  \int_{\mathbb{R}^d}  \Big[ \frac{\delta b_i}{\delta m} (t, x, [X_t^{s, \xi , (m)}])(z) -  \frac{\delta b_i}{\delta m} (t, x, [X_t^{s, \xi' , (m)}])(z)\Big]   \partial_{x}\partial_\mu p_m(\mu, s, t, v', z)(v) \, dz\\
& \quad + \int_{\mathbb{R}^d}   \frac{\delta b_i}{\delta m} (t, x, [X_t^{s, \xi' , (m)}])(z)  \Big[ \partial_{x}\partial_\mu p_m(\mu, s, t, v', z)(v) - \partial_{x}\partial_\mu p_m(\mu', s, t, v', z)(v) \Big]\, dz,
\end{align*}
\begin{align*}
\delta b^{5}_i(\mu, \mu') & =  \int_{(\mathbb{R}^d)^3} \Big[ \frac{\delta^2 b_i}{\delta m^2} (t, x, [X_t^{s, \xi , (m)}])(z, z') - \frac{\delta^2 b_i}{\delta m^2} (t, x, [X_t^{s, \xi' , (m)}])(z, z')  \Big] \\
& \qquad   \partial_\mu p_m(\mu, s, t, x', z)(v)\otimes  \partial_{x}p_m(\mu, s, t, v',z') \, dz \, dz' \, \mu(dx')\\
& \quad  + \int_{(\mathbb{R}^d)^3} \frac{\delta^2 b_i}{\delta m^2} (t, x, [X_t^{s, \xi' , (m)}])(z, z')  \\
& \quad \quad \Big[ \partial_\mu p_m(\mu, s, t, x', z)(v)\otimes  \partial_{x}p_m(\mu, s, t, v',z') - \partial_\mu p_m(\mu', s, t, x', z)(v)\otimes  \partial_{x}p_m(\mu', s, t, v',z') \Big]\, dz \, dz' \, \mu(dx') \\ 
& \quad + \int_{(\mathbb{R}^d)^3} \frac{\delta^2 b_i}{\delta m^2} (t, x, [X_t^{s, \xi' , (m)}])(z, z')  \partial_\mu p_m(\mu', s, t, x', z)(v)\otimes  \partial_{x}p_m(\mu', s, t, v',z') \, dz \, dz' \, [\mu-\mu'](dx'),
\end{align*}
\begin{align*}
\delta b^{6}_i(\mu, \mu') & =  \int_{(\mathbb{R}^d)^4} \Big[ \frac{\delta^2 b_i}{\delta m^2} (t, x, [X_t^{s, \xi , (m)}])(z, z') - \frac{\delta^2 b_i}{\delta m^2} (t, x, [X_t^{s, \xi' , (m)}])(z, z')\Big] \\
& \quad  \partial_\mu p_m(\mu, s, t, x', z)(v)\otimes  \partial_{\mu}p_m(\mu, s, t, x'',z')(v') \, dz \, dz' \, \mu(dx') \mu(dx'')\\
& + \int_{(\mathbb{R}^d)^4} \frac{\delta^2 b_i}{\delta m^2} (t, x, [X_t^{s, \xi' , (m)}])(z, z')  \Big[ \partial_\mu p_m(\mu, s, t, x', z)(v)\otimes  \partial_{\mu}p_m(\mu, s, t, x'',z')(v') \\
& \quad-   \partial_\mu p_m(\mu', s, t, x', z)(v)\otimes  \partial_{\mu}p_m(\mu', s, t, x'',z')(v')\Big]\, dz \, dz' \, \mu(dx') \, \mu(dx'')\\
&  + \int_{(\mathbb{R}^d)^4}  \frac{\delta^2 b_i}{\delta m^2} (t, x, [X_t^{s, \xi' , (m)}])(z, z') \\
& \quad  \partial_\mu p_m(\mu', s, t, x', z)(v)\otimes  \partial_{\mu}p_m(\mu', s, t, x'',z')(v') \, dz \, dz' \, [\mu(dx') \mu(dx'') - \mu'(dx') \mu'(dx'')],
\end{align*}
\begin{align*}
\delta b^{7}_i(\mu, \mu') & =  \int_{(\mathbb{R}^d)^2}  \Big[ \frac{\delta b_i}{\delta m} (t, x, [X_t^{s, \xi , (m)}])(z) -  \frac{\delta b_i}{\delta m} (t, x, [X_t^{s, \xi' , (m)}])(z) \Big] \, \partial_\mu^2 p_m(\mu, s, t, x', z)(\gv)  \,  dz\, \mu(dx') \\
& + \int_{(\mathbb{R}^d)^2}   \Big[ \frac{\delta b_i}{\delta m} (t, x, [X_t^{s, \xi' , (m)}])(z) -  \frac{\delta b_i}{\delta m} (t, x, [X_t^{s, \xi' , (m)}])(x') \Big] \\
& \quad \Big[  \partial_\mu^2 p_m(\mu, s, t, x', z)(\gv) - \partial_\mu^2 p_m(\mu', s, t, x', z)(\gv) \Big] \, \mu(dx') \,  dz \\
& + \int_{(\mathbb{R}^d)^2}  \frac{\delta b_i}{\delta m} (t, x, [X_t^{s, \xi' , (m)}])(z) \,  \partial_\mu^2 p_m(\mu', s, t, x', z)(\gv)  \,  dz\, [\mu-\mu'](dx').
\end{align*}
}

Let us recall from the proof of the estimate (A.45) of Lemma A.2 in \cite{chaudruraynal:frikha} that if $h: \mathcal{P}(\mathbb{R}^d) \mapsto \mathbb{R}$ has a bounded and continuous linear functional derivative such that $[\delta h/\delta m](m)(.)$ is uniformly $\eta$-H\"older continuous then for any $\beta \in [\eta, 1]$ one has 
$$
|h([X^{s, \xi, (m)}_t]) - h([X_t^{s, \xi', (m)}])| \leq K \frac{W_2(\mu, \mu')^{\beta}}{(t-s)^{\frac{\beta-\eta}{2}}}.
$$

The above estimate is established in \cite{chaudruraynal:frikha} for the map $m\mapsto a_{i, j}(t, x, m)$ but the argument works, \emph{mutatis mutandis}, in this general form. In particular, for any $\beta \in [\eta, 1]$, it holds
\begin{align}
\Big| & \frac{\delta b_i}{\delta m}(t, x, [X^{s, \xi, (m)}_{t}])(z) - \frac{\delta b_i}{\delta m}(t, x, [X^{s, \xi', (m)}_t])(z) \Big| \label{delta:mes:linear:functional:deriv:coeff}\\
& \quad + \Big|\frac{\delta^2 b_i}{\delta m^2}(t, x, [X^{s, \xi, (m)}_{t}])(z, z') - \frac{\delta^2 b_i}{\delta m^2}(t, x, [X^{s, \xi', (m)}_t])(z, z')\Big| \leq K^{++} \frac{W_2(\mu, \mu')^{\beta}}{(t-s)^{\frac{\beta-\eta}{2}}}. \notag
\end{align}
 
It should be noted, however, that since $W_2(\mu, \mu') \leq (t-s)^{1/2}$, the above estimate holds for any $\beta \in [0,1]$. Now it follows from \eqref{delta:mes:linear:functional:deriv:coeff} together with \eqref{regularity:measure:estimate:v1:v2:v3:decoupling:mckean}, \eqref{bound:derivative:heat:kernel}, the uniform $\eta$-H\"older regularity of $[\delta^2 b_i/\delta m^2](t, x, m)(z, .)$ and the space time inequality \eqref{space:time:inequality} that
 $$
 |\delta b^{1}_i(\mu, \mu')| \leq K^{++} \frac{W_2(\mu, \mu')^{\beta}}{(t-s)^{1+\frac{\beta-\eta}{2}}}.
 $$
 
 Similarly, from the boundedness of $[\delta^2 b_i/\delta m^2]$, \eqref{delta:mes:linear:functional:deriv:coeff}, \eqref{bound:derivative:heat:kernel}, \eqref{first:second:estimate:induction:decoupling:mckean}, \eqref{regularity:measure:estimate:v1:v2:v3:decoupling:mckean}, \eqref{regularity:measure:estimate:v1:v2:decoupling:mckean}, the inequality
 \begin{align*}
\Big| \int_{(\mathbb{R}^d)^2}  & \frac{\delta^2 b_i}{\delta m^2}(t, x, [X^{s, \xi' , (m)}_t])(z, z')  \partial_x  p_m(\mu', s, t, v, z) \otimes [\partial_\mu p_m(\mu', s, t, x', z')(v') - \partial_\mu p_m(\mu', s, t, y, z')(v')]  \, dz\, dz' \Big| \\
& \quad \leq K \frac{|x'-y|^\beta}{(t-s)^{1+\frac{\beta-\eta}{2}}}, \quad \beta \in [0,1],
 \end{align*}
 \noindent stemming from the estimates \eqref{equicontinuity:second:third:estimate:decoupling:mckean} and \eqref{bound:derivative:heat:kernel}, we deduce
 $$
 |\delta b^{2}_i(\mu, \mu')| \leq K^{++} \frac{W_2(\mu, \mu')^{\beta}}{(t-s)^{1+\frac{\beta-\eta}{2}}}.
 $$
 
 From \eqref{delta:mes:linear:functional:deriv:coeff}, \eqref{cross:deriv:mes:space:induction:decoupling:mckean} and \eqref{sensitivity:mes:deriv:cross:space:mes:pm}, we obtain
 $$
  |\delta b^{3}_i(\mu, \mu')| +  |\delta b^{4}_i(\mu, \mu')| \leq K^{++}_\beta \frac{W_2(\mu, \mu')^{\beta}}{(t-s)^{1+\frac{\beta-\eta}{2}}}.
 $$
 
 We deal with $\delta b^5_i(\mu, \mu')$ using \eqref{delta:mes:linear:functional:deriv:coeff}, \eqref{bound:derivative:heat:kernel}, \eqref{first:second:estimate:induction:decoupling:mckean}, \eqref{regularity:measure:estimate:v1:v2:v3:decoupling:mckean}, \eqref{regularity:measure:estimate:v1:v2:decoupling:mckean} and the inequality
 \begin{align*}
 \Big| \int_{(\mathbb{R}^d)^3} \frac{\delta^2 b_i}{\delta m^2} (t, x, [X_t^{s, \xi' , (m)}])(z, z')  & \Big[\partial_\mu p_m(\mu', s, t, x', z)(v) - \partial_\mu p_m(\mu', s, t, y, z)(v)\Big]\otimes  \partial_{x}p_m(\mu', s, t, v',z') \, dz \, dz' \Big|\\
 & \leq K^{+} \frac{|x'-y|^\beta}{(t-s)^{1+\frac{\beta -\eta}{2}}}, \quad \beta \in [0,1],
 \end{align*}
 \noindent stemming from \eqref{bound:derivative:heat:kernel} and \eqref{equicontinuity:second:third:estimate:decoupling:mckean}. We obtain
$$
   |\delta b^{5}_i(\mu, \mu')| \leq K^{++}_\beta \frac{W_2(\mu, \mu')^{\beta}}{(t-s)^{1+\frac{\beta-\eta}{2}}}.
$$
 
We handle $\delta b_i^{6}(\mu, \mu')$ in a similar manner. First, it follows from \eqref{first:second:estimate:induction:decoupling:mckean} and \eqref{equicontinuity:second:third:estimate:decoupling:mckean} that the last term appearing in the decomposition of $\delta b_i^{6}(\mu, \mu')$ satisfies 
\begin{align*}
\Big|&  \int_{(\mathbb{R}^d)^4}  \frac{\delta^2 b_i}{\delta m^2} (t, x, [X_t^{s, \xi' , (m)}])(z, z')  \\
& \quad \partial_\mu p_m(\mu', s, t, x', z)(v)\otimes  \partial_{\mu}p_m(\mu', s, t, x'',z')(v') \, dz \, dz' \, [\mu(dx') \mu(dx'') - \mu'(dx') \mu'(dx'')] \Big| \\
& \leq K^{+}_\beta \frac{W_2(\mu, \mu')^{\beta}}{(t-s)^{1+\frac{\beta}{2}-\eta}}.
\end{align*}

\noindent for any $\beta \in [0,1]$.
Then, using the above estimate together with \eqref{delta:mes:linear:functional:deriv:coeff}, \eqref{first:second:estimate:induction:decoupling:mckean} and \eqref{regularity:measure:estimate:v1:v2:decoupling:mckean}, we obtain
$$
|\delta b_i^{6}(\mu, \mu')| \leq K^{++}_\beta \frac{W_2(\mu, \mu')^{\beta}}{(t-s)^{1+\frac{\beta-\eta}{2}}}.
$$

Finally, from \eqref{second:deriv:mes:reg:space:deriv:arg:estimate:induction:decoupling:mckean2}, recalling that $\mathscr{C}^{1, \beta}_{m}(C^{+}, t-s) \leq K^+:=\lim_{m\uparrow \infty} \mathscr{C}^{1, \beta}_{m}(C^{+}, t-s) < \infty$, we deduce that the last term appearing in the decomposition of $\delta b_i^{7}(\mu, \mu')$ satisfies
$$
\Big|  \int_{(\mathbb{R}^d)^2}  \frac{\delta b_i}{\delta m} (t, x, [X_t^{s, \xi' , (m)}])(z) \,  \partial_\mu^2 p_m(\mu', s, t, x', z)(\gv)  \,  dz\, [\mu-\mu'](dx') \Big| \leq K^{+} \frac{W_2(\mu, \mu')^{\beta}}{(t-s)^{1+\frac{\beta-\eta}{2}}}
$$
\noindent for any $\beta \in [0,\eta)$, while the first term is handled using \eqref{delta:mes:linear:functional:deriv:coeff} and \eqref{second:deriv:mes:induction:decoupling:mckean}, recalling again that $\mathscr{C}^{1, 0}_{m}(C^{+}, t-s)\leq K^{+} < \infty$, so that
$$
\Big| \int_{(\mathbb{R}^d)^2}  \Big[ \frac{\delta b_i}{\delta m} (t, x, [X_t^{s, \xi , (m)}])(z) -  \frac{\delta b_i}{\delta m} (t, x, [X_t^{s, \xi' , (m)}])(z) \Big] \, \partial_\mu^2 p_m(\mu, s, t, x', z)(\gv)  \,  dz\, \mu(dx') \Big| \leq K^{++} \frac{W_2(\mu, \mu')^\beta}{(t-s)^{1+\frac{\beta}{2}-\eta}}.
$$

We finally remark that the boundedness and uniform $\eta$-H\"older regularity of $[\delta b_i/\delta m](t, x, m)(.)$ directly implies that the second term satisfies
\begin{align*}
|\int_{(\mathbb{R}^d)^2}  &  \Big[ \frac{\delta b_i}{\delta m} (t, x, [X_t^{s, \xi' , (m)}])(z) -  \frac{\delta b_i}{\delta m} (t, x, [X_t^{s, \xi' , (m)}])(x') \Big]   \Big[  \partial_\mu^2 p_m(\mu, s, t, x', z)(\gv) - \partial_\mu^2 p_m(\mu', s, t, x', z)(\gv) \Big]\, \mu(dx')  \,  dz| \\
& \quad  \leq  K  \int_{(\mathbb{R}^d)^2} (|z-x'|^\eta \wedge 1) | \partial_\mu^2 p_m(\mu, s, t, x', z)(\gv) - \partial_\mu^2 p_m(\mu', s, t, x', z)(\gv)| \, \mu(dx')  \, dz.
\end{align*}

 We conclude the proof of \eqref{diff:mes:second:deriv:mes:a:b} for $ \partial^2_\mu [b_{i}(t, x, [X^{s, \xi, (m)}_{t}])](\gv)  - \partial^2_\mu [ b_{i}(t, x, [X^{s, \xi, (m)}_{t}])]_{|\mu=\mu'}(\gv)$ by gathering the above estimates. The difference $\partial^2_\mu [a_{i, j}(t, x, [X^{s, \xi, (m)}_{t}])](\gv)  - \partial^2_\mu [ a_{i, j}(t, x, [X^{s, \xi, (m)}_{t}])]_{|\mu=\mu'}(\gv)$ can be handled in a completely analogous manner. The proof is thus omitted.\\

\noindent \emph{Step 2: proof of the estimate \eqref{diff:lions:deriv:pm:mu:mup}.}\\

The strategy is clear inasmuch one starts from the identity \eqref{representation:formula:second:order:deriv:mes:p:hat} and has to quantify the regularity with respect to the variable $\mu$ of each term. We first note that from the mean-value theorem one has for any $\beta \in [0,1]$
\begin{align}
\Big| \left(\int_r^t a(r', y, [X^{s, \xi, (m)}_{r'}])\, dr'\right)^{-1} &  - \left(\int_r^t a(r', y, [X^{s, \xi', (m)}_{r'}])\, dr'\right)^{-1} \Big| \notag \\
& \leq  \frac{K}{(t-r)^2} \int_r^t \max_{i, j} |a_{i, j}(r', y, [X^{s, \xi, (m)}_{r'}]) - a_{i, j}(r', y, [X^{s, \xi', (m)}_{r'}])| \, dr' \notag \\
& \leq K \frac{W_2(\mu, \mu')^\beta}{(t-r)^{2}} \int_r^t \frac{1}{(r' -s)^{\frac{\beta}{2}}}\, dr' \notag \\
& \leq  K W_2(\mu, \mu')^{\beta} \left(\frac{1}{(t-s)^{1+\frac{\beta}{2}}} \I_\seq{r=s} + \frac{1}{(t-r)(r-s)^{\frac{\beta}{2}}} \I_\seq{r>s} \right) \label{difference:a:mes:var:diff:time}
\end{align}
\noindent where we used the estimate \eqref{diff:mes:drift:diff:coefficients} inside the time integral if $W_2(\mu, \mu') \leq r'-s$ or the boundedness of $a_{i, j}$ otherwise. Similarly, from the identity \eqref{first:order:diff:mat:mes}, the above estimate, \eqref{deriv:mes:a:m} and \eqref{diff:v:cross:deriv:diff:and:deriv:coeff} with $n=0$, we get that for any $\beta \in [0,1]$
\begin{align*}
\Big| \partial_\mu\Big[\left(\int_r^t a(r', y, [X^{s, \xi, (m)}_{r'}])\, dr'\right)^{-1}& \Big](v)  - \partial_\mu\Big[\left(\int_r^t a(r', y, [X^{s, \xi, (m)}_{r'}])\, dr'\right)^{-1}\Big]_{|\mu=\mu'}(v) \Big| \\
& \leq K^+_{\beta} W_2(\mu, \mu')^\beta \left(\frac{1}{(t-s)^{\frac{3+\beta-\eta}{2}}} \I_\seq{r=s} + \frac{1}{(t-r)(r-s)^{\frac{1+\beta-\eta}{2}}}\I_\seq{r>s} \right).
\end{align*}

Then, we use the identity \eqref{representation:formula:second:order:deriv:mes:p:hat} together with the two above estimates, \eqref{diff:mes:second:deriv:mes:a:b}, \eqref{recursive:bound:deriv:a:or:b} (combined with \eqref{second:deriv:mes:induction:decoupling:mckean} and the space time inequality \eqref{space:time:inequality}), \eqref{deriv:mes:a:m}, \eqref{diff:L:deriv:pm:mu:mup} and \eqref{diff:v:cross:deriv:diff:and:deriv:coeff} with $n=0$. After some standard computations that we omit, we obtain
\begin{align*}
& | \partial^2_\mu \widehat{p}_{m+1}(\mu, s, r, t, x, z)(\gv) - \partial^2_\mu \widehat{p}_{m+1}(\mu', s, r, t, x, z)(\gv)  | \notag\\
& \quad \leq  K^+_\beta  \Bigg\{ W_2(\mu, \mu')^{\beta} \Bigg(\frac{1}{(t-s)^{ 1+\frac{\beta-\eta}{2}}} \I_\seq{r=s} + \frac{1}{(r-s)^{ 1+\frac{\beta-\eta}{2}}} \mathbf \I_\seq{r>s} \Bigg ) \notag \\
& \qquad \qquad  + \frac{1}{t-r}   \int_r^t \int_{(\mathbb{R}^d)^2} (|y'-x'|^\eta \wedge 1)|\partial^2_\mu p_m(\mu, s, r', x', y')(\gv) - \partial^2_\mu p_m(\mu', s, r', x', y')(\gv)|   \, \mu(dx') \, dy' \, dr' \Bigg\}  \notag\\
& \qquad \qquad \qquad \times  g(c(t-r), z-x).\notag 
\end{align*}

\noindent \emph{Step: 3: proof of the estimate \eqref{diff:lions:deriv:a:mu:mua}.}\\

We first remark that if $W_2(\mu, \mu') > (t-s)^{1/2}$ then the result directly follows from \eqref{recursive:bound:second:order:deriv:mes:holder:reg:space:a} with $\beta'=0$ and $\beta'=1$ combined with \eqref{second:deriv:mes:induction:decoupling:mckean} and the space time inequality \eqref{space:time:inequality}. We thus assume that $W_2(\mu, \mu') \leq (t-s)^{1/2}$ for the rest of the proof. Similarly to the proof of \eqref{diff:mes:second:deriv:mes:a:b}, we apply the identity \eqref{full:expression:second:deriv} to the map $m \mapsto a_{i, j}(t, x, m)$ and write the difference $\partial^2_\mu [a_{i, j}(t, x, [X^{s, \xi, (m)}_{t}])  - a_{i, j}(t, z, [X^{s, \xi, (m)}_{t}])](\gv)- \partial^2_\mu[ a_{i, j}(t, x, [X^{s, \xi, (m)}_{t}])  -  a_{i, j}(t, z, [X^{s, \xi, (m)}_{t}])]_{|\mu= \mu'}(\gv)$ as the sum of the following terms
 {\small
\begin{align*}
T^{1}_{i, j}   & :=  \int_{(\mathbb{R}^d)^2} \Big\{\frac{\delta^2 a_{i, j}}{\delta m^2}(t, x, [X^{s, \xi , (m)}_t])(z', z'') - \frac{\delta^2 a_{i, j}}{\delta m^2}(t, z, [X^{s, \xi , (m)}_t])(z', z'') \\
& \quad  - \Big(\frac{\delta^2 a_{i, j}}{\delta m^2}(t, x, [X^{s, \xi' , (m)}_t])(z', z'') - \frac{\delta^2 a_{i, j}}{\delta m^2}(t, z, [X^{s, \xi' , (m)}_t])(z', z'') \Big) \Big\}\\
& \quad \quad  \partial_{x} p_m(\mu, s, t, v, z')\otimes \partial_{x} p_m(\mu, s, t, v',z'')\, dz'\, dz''\\
& \quad +  \int_{(\mathbb{R}^d)^2} \Big\{\frac{\delta^2 a_{i, j}}{\delta m^2}(t, x, [X^{s, \xi' , (m)}_t])(z', z'') - \frac{\delta^2 a_{i, j}}{\delta m^2}(t, z, [X^{s, \xi' , (m)}_t])(z', z'')  \Big\}\\
& \quad \quad  \Big[ \partial_{x} p_m(\mu, s, t, v, z')\otimes \partial_{x} p_m(\mu, s, t, v',z'') - \partial_{x} p_m(\mu', s, t, v, z')\otimes \partial_{x} p_m(\mu', s, t, v',z'')  \Big]\, dz'\, dz''\\
T^{2}_{i, j}  & :=  \int_{(\mathbb{R}^d)^3}  \Big\{ \frac{\delta^2 a_{i, j}}{\delta m^2}(t, x, [X^{s, \xi , (m)}_t])(z', z'') - \frac{\delta^2 a_{i, j}}{\delta m^2}(t, z, [X^{s, \xi , (m)}_t])(z', z'') \\
& \quad - \Big( \frac{\delta^2 a_{i, j}}{\delta m^2}(t, x, [X^{s, \xi' , (m)}_t])(z', z'') - \frac{\delta^2 a_{i, j}}{\delta m^2}(t, z, [X^{s, \xi' , (m)}_t])(z', z'') \Big)\Big\}\\
& \quad \quad  \partial_{x}  p_m(\mu, s, t, v, z') \otimes \partial_\mu p_m(\mu, s, t, x', z'')(v') \, dz'\, dz''\, \mu(dx')\\
& \quad  +   \int_{(\mathbb{R}^d)^3} \Big\{ \frac{\delta^2 a_{i, j}}{\delta m^2}(t, x, [X^{s, \xi' , (m)}_t])(z', z'') - \frac{\delta^2 a_{i, j}}{\delta m^2}(t, z, [X^{s, \xi' , (m)}_t])(z', z'')  \Big\}\\
& \quad  \quad   \Big[\partial_{x}  p_m(\mu, s, t, v, z') \otimes \partial_\mu p_m(\mu, s, t, x', z'')(v')  - \partial_{x}  p_m(\mu', s, t, v, z') \otimes \partial_\mu p_m(\mu', s, t, x', z'')(v') \Big]  \, dz'\, dz''\, \mu(dx')\\
& \quad + \int_{(\mathbb{R}^d)^3} \Big\{ \frac{\delta^2 a_{i, j}}{\delta m^2}(t, x, [X^{s, \xi' , (m)}_t])(z', z'') - \frac{\delta^2 a_{i, j}}{\delta m^2}(t, z, [X^{s, \xi' , (m)}_t])(z', z'')  \Big\}\\
& \quad  \quad    \partial_{x}  p_m(\mu', s, t, v, z') \otimes \partial_\mu p_m(\mu', s, t, x', z'')(v')   \, dz'\, dz''\, [\mu-\mu'](dx'),
\end{align*}
\begin{align*}
T^{3}_{i, j} & := \int_{\mathbb{R}^d}  \Big\{\frac{\delta a_{i, j}}{\delta m}(t, x, [X^{s, \xi , (m)}_t])(z') - \frac{\delta a_{i, j}}{\delta m}(t, z, [X^{s, \xi , (m)}_t])(z') \\
& \quad - \Big(\frac{\delta a_{i, j}}{\delta m}(t, x, [X^{s, \xi' , (m)}_t])(z') - \frac{\delta a_{i, j}}{\delta m}(t, z, [X^{s, \xi' , (m)}_t])(z')\Big) \Big\} \partial_\mu[\partial_x p_m(\mu, s, t, v ,z')](v') \, dz' , \\
& \quad +  \int_{\mathbb{R}^d}  \left\{ \frac{\delta a_{i, j}}{\delta m}(t, x, [X^{s, \xi' , (m)}_t])(z') - \frac{\delta a_{i, j}}{\delta m}(t, z, [X^{s, \xi' , (m)}_t])(z') \right\} \\
& \quad \Big( \partial_\mu[\partial_x p_m(\mu, s, t, v ,z')](v') -\partial_\mu[\partial_x p_m(\mu', s, t, v ,z')](v')  \Big) \, dz' , \\
T^{4}_{i, j} & := \int_{\mathbb{R}^d}  \Big\{\frac{\delta a_{i, j}}{\delta m}(t, x, [X^{s, \xi , (m)}_t])(z') - \frac{\delta a_{i, j}}{\delta m}(t, z, [X^{s, \xi , (m)}_t])(z') \\
& \quad - \Big(\frac{\delta a_{i, j}}{\delta m}(t, x, [X^{s, \xi' , (m)}_t])(z') - \frac{\delta a_{i, j}}{\delta m}(t, z, [X^{s, \xi' , (m)}_t])(z')\Big) \Big\} \partial_x[\partial_\mu p_m(\mu, s, t, v' ,z')](v) \, dz' , \\
& \quad +  \int_{\mathbb{R}^d}  \left\{ \frac{\delta a_{i, j}}{\delta m}(t, x, [X^{s, \xi' , (m)}_t])(z') - \frac{\delta a_{i, j}}{\delta m}(t, z, [X^{s, \xi' , (m)}_t])(z') \right\} \\
& \quad \Big( \partial_x[\partial_\mu p_m(\mu, s, t, v' ,z')(v)] -\partial_x[\partial_\mu p_m(\mu', s, t, v' ,z')(v)]  \Big) \, dz' ,
\end{align*}
\begin{align*}
T^{5}_{i, j} & :=  \int_{(\mathbb{R}^d)^3} \Big\{ \frac{\delta^2 a_{i, j}}{\delta m^2}(t, x, [X^{s, \xi , (m)}_t])(z', z'') - \frac{\delta^2 a_{i, j}}{\delta m^2}(t, z, [X^{s, \xi , (m)}_t])(z', z'') \\
& \quad - \Big(\frac{\delta^2 a_{i, j}}{\delta m^2}(t, x, [X^{s, \xi' , (m)}_t])(z', z'') - \frac{\delta^2 a_{i, j}}{\delta m^2}(t, z, [X^{s, \xi' , (m)}_t])(z', z'')\Big) \Big\}\\
& \quad \quad  \partial_\mu p_m(\mu, s, t, x', z')(v) \otimes \partial_{x}p_m(\mu, s, t, v',z'')\, dz'\, dz''\, \mu(dx') \\
& \quad + \int_{(\mathbb{R}^d)^3} \Big\{ \frac{\delta^2 a_{i, j}}{\delta m^2}(t, x, [X^{s, \xi' , (m)}_t])(z', z'') - \frac{\delta^2 a_{i, j}}{\delta m^2}(t, z, [X^{s, \xi' , (m)}_t])(z', z'') \Big\} \\
& \quad  \quad  \Big( \partial_\mu p_m(\mu, s, t, x', z')(v) \otimes \partial_{x}p_m(\mu, s, t, v',z'') - \partial_\mu p_m(\mu', s, t, x', z')(v) \otimes \partial_{x}p_m(\mu', s, t, v',z'')\Big) \, dz'\, dz''\, \mu(dx') \\
& \quad +  \int_{(\mathbb{R}^d)^3} \Big\{ \frac{\delta^2 a_{i, j}}{\delta m^2}(t, x, [X^{s, \xi' , (m)}_t])(z', z'') - \frac{\delta^2 a_{i, j}}{\delta m^2}(t, z, [X^{s, \xi' , (m)}_t])(z', z'') \Big\} \\
& \quad  \quad \partial_\mu p_m(\mu', s, t, x', z')(v) \otimes \partial_{x}p_m(\mu', s, t, v',z'') \, dz'\, dz''\, [\mu-\mu'](dx'),\\
T^{6}_{i, j} & := \int_{(\mathbb{R}^d)^4}\Big\{\frac{\delta^2 a_{i, j}}{\delta m^2}(t, x, [X^{s, \xi , (m)}_t])(z', z'')-\frac{\delta^2 a_{i, j}}{\delta m^2}(t, z, [X^{s, \xi , (m)}_t])(z', z'') \\
& \quad - \Big(\frac{\delta^2 a_{i, j}}{\delta m^2}(t, x, [X^{s, \xi' , (m)}_t])(z', z'')-\frac{\delta^2 a_{i, j}}{\delta m^2}(t, z, [X^{s, \xi' , (m)}_t])(z', z'') \Big) \Big\} \\
& \quad \quad  \partial_\mu p_m(\mu, s, t, x', z)(v) \otimes  \partial_\mu p_m(\mu, s, t, x'', z')(v')\,dz\, dz'\, \mu(dx'')\, \mu(dx') \\
& \quad + \int_{(\mathbb{R}^d)^4}\Big\{\frac{\delta^2 a_{i, j}}{\delta m^2}(t, x, [X^{s, \xi' , (m)}_t])(z', z'')-\frac{\delta^2 a_{i, j}}{\delta m^2}(t, z, [X^{s, \xi' , (m)}_t])(z', z'') \Big\} \\
& \quad \quad  \Big( \partial_\mu p_m(\mu, s, t, x', z)(v) \otimes  \partial_\mu p_m(\mu, s, t, x'', z')(v') - \partial_\mu p_m(\mu', s, t, x', z)(v) \otimes  \partial_\mu p_m(\mu', s, t, x'', z')(v') \Big)\,dz\, dz'\, \mu(dx'')\, \mu(dx') \\
& \quad +  \int_{(\mathbb{R}^d)^4}\Big\{\frac{\delta^2 a_{i, j}}{\delta m^2}(t, x, [X^{s, \xi' , (m)}_t])(z', z'')-\frac{\delta^2 a_{i, j}}{\delta m^2}(t, z, [X^{s, \xi' , (m)}_t])(z', z'') \Big\} \\
& \quad  \partial_\mu p_m(\mu, s, t, x', z)(v) \otimes  \partial_\mu p_m(\mu, s, t, x'', z')(v') \,dz\, dz'\,  [\mu(dx'') \mu(dx') - \mu'(dx'')\mu'(dx')],
\end{align*}
\begin{align*}
T^{7}_{i, j}  & := \int_{(\mathbb{R}^d)^2}  \Big\{ \frac{\delta a_{i, j}}{\delta m}(t, x, [X^{s, \xi , (m)}_t])(z')- \frac{\delta a_{i, j}}{\delta m}(t, z, [X^{s, \xi , (m)}_t])(z') \\
& \quad \quad - \Big(  \frac{\delta a_{i, j}}{\delta m}(t, x, [X^{s, \xi' , (m)}_t])(z')- \frac{\delta a_{i, j}}{\delta m}(t, z, [X^{s, \xi' , (m)}_t])(z')\Big)\Big\}\, \partial_\mu^2 p_m(\mu, s, t, x', z')(\gv)\,  dz'\, \mu(dx')\\
& \quad + \int_{(\mathbb{R}^d)^2}  \Big\{ \frac{\delta a_{i, j}}{\delta m}(t, x, [X^{s, \xi' , (m)}_t])(z')- \frac{\delta a_{i, j}}{\delta m}(t, z, [X^{s, \xi' , (m)}_t])(z')  \\
& \quad \quad - \Big( \frac{\delta a_{i, j}}{\delta m}(t, x, [X^{s, \xi' , (m)}_t])(x')- \frac{\delta a_{i, j}}{\delta m}(t, z, [X^{s, \xi' , (m)}_t])(x') \Big)\Big\} \\
& \quad \quad \Big(\partial_\mu^2 p_m(\mu, s, t, x', z')(\gv) - \partial_\mu^2 p_m(\mu', s, t, x', z')(\gv)\Big)\,  dz'\, \mu(dx')\\
& \quad +\int_{(\mathbb{R}^d)^2}  \Big\{ \frac{\delta a_{i, j}}{\delta m}(t, x, [X^{s, \xi' , (m)}_t])(z')- \frac{\delta a_{i, j}}{\delta m}(t, z, [X^{s, \xi' , (m)}_t])(z') \Big\}  \partial_\mu^2 p_m(\mu', s, t, x', z')(\gv)\,  dz'\, [\mu(dx') - \mu'(dx')].
\end{align*} 
 }

 As previously done, we quantify the contribution of each term in the above decomposition. We let $\Theta^{(m)}_{\lambda, t}:= (1-\lambda)[X^{s, \xi, (m)}_t] + \lambda [X^{s, \xi', (m)}_t]$ and $h \in \left\{ \delta a_{i, j}/\delta m,  \delta^2 a_{i, j}/\delta m^2 \right\}$. We write
 \begin{align*}
h(t, x, [X^{s, \xi, (m)}_t])(.) & - h(t, x, [X^{s, \xi', (m)}_t])(.) \\
& = \int_0^1\int_{\mathbb{R}^d} \frac{\delta h}{\delta m}(t, x, \Theta^{(m)}_{t, \lambda})(.)(z') (p_m(\mu, s, t, z') - p_{m}(\mu', s, t, z')) \, dz' \, d\lambda
\end{align*}

\noindent so that, from similar arguments as those used to derive (A.49) of Lemma A.2 in \cite{chaudruraynal:frikha}, we get that for any $\alpha \in [0,\eta]$ and any $\beta \in [\alpha, 1]$
\begin{equation}
\label{holder:property:h:delta:mes}
|h(x) - h(z)| \leq K^{++} (|z-x|^{\eta-\alpha} \wedge 1) \frac{ W_2(\mu, \mu')^{\beta}}{(t-s)^{\frac{\beta-\alpha}{2}}}.
\end{equation}

Since $W_2(\mu, \mu') \leq (t-s)^{1/2}$, the above estimate remains valid for any $\beta \in [0,1]$. We now consider the first term $T^{1}_{i, j}$ which can be decomposed as the sum of $T^{1,1}_{i, j}$ and $T^{1,2}_{i, j}$ as written above. From \eqref{bound:derivative:heat:kernel} and \eqref{holder:property:h:delta:mes} with $\alpha=0$ and $\alpha=\eta$, we deduce that for any $\beta \in [0,1]$
$$
| T^{1,1}_{i, j} | \leq K^{++} W_2(\mu, \mu')^{\beta}\left\{ \frac{|z-x|^\eta\wedge 1}{(t-s)^{1+\frac{\beta}{2}}} \wedge \frac{1}{(t-s)^{1+\frac{\beta-\eta}{2}}} \right\}.
$$

In order to deal with $T^{1, 2}_{i, j}$, we use a centering argument. Namely, we write
\begin{align*}
 \Big| T^{1,2}_{i, j} \Big| & = \Big| \int_{(\mathbb{R}^d)^2} \Big\{\frac{\delta^2 a_{i, j}}{\delta m^2}(t, x, [X^{s, \xi' , (m)}_t])(z', z'') - \frac{\delta^2 a_{i, j}}{\delta m^2}(t, z, [X^{s, \xi' , (m)}_t])(z', z'') \\
 & \quad - \Big(\frac{\delta^2 a_{i, j}}{\delta m^2}(t, x, [X^{s, \xi' , (m)}_t])(z', v') - \frac{\delta^2 a_{i, j}}{\delta m^2}(t, z, [X^{s, \xi' , (m)}_t])(z', v')\Big) \Big\}\\
& \quad \quad  \Big[ \partial_{x} p_m(\mu, s, t, v, z')\otimes \partial_{x} p_m(\mu, s, t, v',z'') - \partial_{x} p_m(\mu', s, t, v, z')\otimes \partial_{x} p_m(\mu', s, t, v',z'')  \Big]\, dz'\, dz''  \Big|\\
& \leq K^{+} \frac{W_2(\mu, \mu')^\beta}{(t-s)^{1+\frac{\beta}{2}}} \int_{(\mathbb{R}^d)^2} (|z-x|^\eta \wedge |z''-v'|^\eta\wedge 1) g(c(t-s), z'-v) \, g(c(t-s), z''-v') \, dz' \, dz''\\
& \leq K^{+} W_2(\mu, \mu')^\beta\left\{ \frac{|z-x|^\eta\wedge 1}{(t-s)^{1+\frac{\beta}{2}}} \wedge \frac{1}{(t-s)^{1+\frac{\beta-\eta}{2}}} \right\}
\end{align*}
\noindent for any $\beta \in [0,1]$, where we used  the uniform $\eta$-H\"older regularity of $[\delta^2 a_{i, j}/\delta m^2](t, ., m)(z',.)$, \eqref{regularity:measure:estimate:v1:v2:v3:decoupling:mckean}, \eqref{bound:derivative:heat:kernel} and eventually the space time inequality \eqref{space:time:inequality}. Gathering the above estimates, we thus conclude that $T^{1}_{i, j}$ is bounded by the first term appearing on the right-hand side of \eqref{diff:lions:deriv:a:mu:mua}.

The other terms $T^{\ell}_{i, j}$, $\ell=1, \cdots, 6$, can be dealt in a similar manner so we will be short and omit some technical details. We use \eqref{holder:property:h:delta:mes} with $\alpha=0$, \eqref{bound:derivative:heat:kernel}, \eqref{first:second:estimate:induction:decoupling:mckean}, the boundedness and uniform $\eta$-H\"older regularity of $[\delta^2 a_{i, j}/\delta m^2](t,.,m)(z', z'')$, \eqref{regularity:measure:estimate:v1:v2:v3:decoupling:mckean}, \eqref{regularity:measure:estimate:v1:v2:decoupling:mckean} and \eqref{equicontinuity:second:third:estimate:decoupling:mckean}. Hence, for any $\beta \in [0,\eta)$, we obtain
$$
 \Big| T^{2}_{i, j} \Big| \leq K^{++}_\beta W_2(\mu, \mu')^{\beta} \frac{|z-x|^\eta\wedge 1}{(t-s)^{1+\frac{\beta-\eta}{2}}}
$$
\noindent for any $\beta \in [0,1]$. By symmetry, the two terms $T^{3}_{i, j}$ and $T^{4}_{i, j}$ can be handled by similar arguments. Namely, we use \eqref{holder:property:h:delta:mes} with $\alpha=0$, \eqref{cross:deriv:mes:space:induction:decoupling:mckean}, \eqref{sensitivity:mes:deriv:cross:space:mes:pm} and the uniform $\eta$-H\"older regularity of $[\delta a_{i, j}/\delta m](t,.,m)(z')$ so that for any $\beta \in [0,\eta)$
$$
\Big| T^{3}_{i, j} \Big| + \Big| T^{4}_{i, j}\Big| \leq K^{++}_\beta W_2(\mu, \mu')^{\beta} \frac{|z-x|^\eta\wedge 1}{(t-s)^{1+\frac{\beta-\eta}{2}}}.
$$

Similarly, using \eqref{holder:property:h:delta:mes} with $\alpha=0$, \eqref{bound:derivative:heat:kernel}, \eqref{first:second:estimate:induction:decoupling:mckean}, \eqref{regularity:measure:estimate:v1:v2:v3:decoupling:mckean}, \eqref{regularity:measure:estimate:v1:v2:decoupling:mckean}, \eqref{equicontinuity:second:third:estimate:decoupling:mckean} and the uniform $\eta$-H\"older regularity of $[\delta^2 a_{i, j}/\delta m^2](t,.,m)(z', z'')$. For any $\beta \in [0,1]$, we get
$$
\Big| T^{5}_{i, j} \Big| + \Big| T^{6}_{i, j} \Big|   \leq  K^{++}_\beta W_2(\mu, \mu')^{\beta} \frac{|z-x|^\eta\wedge 1}{(t-s)^{1+\frac{\beta-\eta}{2}}}.
$$

 We now deal with the last term, namely, $T_{i, j}^{7}$. We decompose this term as the sum of the three terms $T_{i, j}^{7, 1}$, $T_{i, j}^{7, 2}$ and $T_{i, j}^{7, 3}$ as written above. The first and the third terms are handled using  \eqref{holder:property:h:delta:mes} with $\alpha=0$, \eqref{second:deriv:mes:induction:decoupling:mckean}, \eqref{second:deriv:mes:reg:space:deriv:arg:estimate:induction:decoupling:mckean2} and the uniform $\eta$-H\"older regularity of $[\delta a_{i, j}/\delta m](t, ., m)(z')$ so that for any $\beta \in [0,\eta)$
 $$
 |T^{7, 1}_{i, j}| + |T^{7, 3}_{i, j}| \leq K^{++}_\beta  W_2(\mu, \mu')^{\beta} \frac{|z-x|^\eta\wedge 1}{(t-s)^{1+\frac{\beta-\eta}{2}}}.
 $$
 
For the second one, it follows from the uniform $\eta$-H\"older regularity of $[\delta a_{i, j}/\delta m](t, ., m)(.)$ that
\begin{align*}
| T^{7, 2}_{i, j} |&  \leq K^{+} \int_{(\mathbb{R}^d)^2} (|z-x|^\eta \wedge |z'- x'|^\eta \wedge 1) |\partial_\mu^2 p_m(\mu, s, t, x', z')(\gv) - \partial_\mu^2 p_m(\mu', s, t, x', z')(\gv)| \  \mu(dx')\, dz' \\
& \leq K^{+} (|z-x|^{\eta} \wedge 1)\int_{(\mathbb{R}^d)^2} | \partial_\mu^2 p_m(\mu, s, t, x', z')(\gv) -  \partial_\mu^2 p_m(\mu', s, t, x', z')(\gv) | \, \mu(dx') \, dz' \notag\\
&\quad\quad\quad\quad\quad \wedge \int_{(\mathbb{R}^d)^2} (|z'-x'|^{\eta} \wedge 1)  | \partial_\mu^2 p_m(\mu, s, t, x', z')(\gv) -  \partial_\mu^2 p_m(\mu', s, t, x', z')(\gv) | \,  \mu(dx')\, dz'.
\end{align*}

Collecting the above estimates allows to conclude the proof of \eqref{diff:lions:deriv:a:mu:mua}.\\

\noindent \emph{Step 4: proof of the estimate \eqref{diff:mes:L:deriv:parametrix:kernel:pmp1}.}\\

We start from the decomposition of $\partial^2_\mu \mH_{m+1}(\mu, s, r ,t ,x ,z)(\gv)$ given by \eqref{derivsec:mu:parametrix:kernel:mp1} and investigate the regularity of each term with respect to the variable $\mu$.\\
 
 \noindent (i) \emph{Regularity of the map $\mu \mapsto \big(\partial_\mu {\rm I}(\mu, s)(\gv)+ \partial_\mu {\rm II}(\mu, s)(\gv)\big)\widehat{p}_{m+1}(\mu, s, r, t, x, z)$.}\\
 
 First, it follows from the identity \eqref{second:diff-mat-mes}, \eqref{difference:a:mes:var:diff:time} (with $r>s$), \eqref{diff:v:cross:deriv:diff:and:deriv:coeff},  \eqref{diff:mes:second:deriv:mes:a:b}, \eqref{deriv:mes:a:m} and \eqref{recursive:bound:deriv:a:or:b} combined with \eqref{second:deriv:mes:induction:decoupling:mckean} (recalling that $\mathscr{C}^{1, 0}_{m}(C^+, t-s) \leq K^{+} := \lim_{m\rightarrow \infty} \mathscr{C}^{1, 0}_{m}(C^+, t-s) < \infty$) that for any $\beta \in [0,\eta)$
 \begin{align*}
 \Big| \partial^2_\mu & \Big[H^{i}_1\left(\int_r^{t} a(r',z, [X^{s, \xi, (m)}_{r'}]) dr', z-x\right)\Big](\gv) -  \partial^2_\mu  \Big[H^{i}_1\left(\int_r^{t} a(r',z, [X^{s, \xi, (m)}_{r'}]) dr', z-x\right)\Big]_{\mu=\mu'}(\gv) \Big| \\
 & \leq K^{++}_\beta \frac{|z-x|}{t-r}\left\{  \frac{W_2(\mu, \mu')^\beta}{(r-s)^{1+\frac{\beta-\eta}{2}}}  \right. \\
 & \quad \left. +  \frac{1}{t-r} \int_r^t \int_{(\mathbb{R}^d)^2} (|y'-x'|^\eta \wedge 1) |\partial^2_\mu p_m(\mu, s, r', x', y')(\gv) - \partial^2_\mu p_m(\mu', s, r', x', y')(\gv)|  \, \mu(dx') \, dy' \, dr' \right\}
 \end{align*}
 \noindent and
  \begin{align*}
 \Big| \partial^2_\mu & \Big[H^{i, j}_2\left(\int_r^{t} a(r',z, [X^{s, \xi, (m)}_{r'}]) dr', z-x\right)\Big](\gv) -  \partial^2_\mu  \Big[H^{i, j}_2\left(\int_r^{t} a(r',z, [X^{s, \xi, (m)}_{r'}]) dr', z-x\right)\Big]_{\mu=\mu'}(\gv) \Big| \\
 & \leq K^{++}_\beta \left\{\frac{|z-x|^2}{(t-r)^2} + \frac{1}{t-r}\right\} \left\{  \frac{W_2(\mu, \mu')^\beta}{(r-s)^{1+\frac{\beta-\eta}{2}}}  \right. \\
 & \quad \left. +  \frac{1}{t-r} \int_r^t \int_{(\mathbb{R}^d)^2} (|y'-x'|^\eta \wedge 1) |\partial^2_\mu p_m(\mu, s, r', x', y')(\gv) - \partial^2_\mu p_m(\mu', s, r', x', y')(\gv)|  \, \mu(dx')\, dy' \, dr'\right\}.
 \end{align*}
 
 Then, using the aforementioned estimates together with the two above estimates, \eqref{diff:mes:L:deriv:diff:diff:coeff:holder:reg}, \eqref{diff:lions:deriv:a:mu:mua}, the uniform boundedness of $b_i$, the uniform $\eta$-H\"older regularity of $a_{i, j}(t,.,m)$ and the space time inequality \eqref{space:time:inequality}, we obtain that for any $\beta \in [0,\eta)$
 \begin{align*}
 \Big| [\partial_\mu & {\rm I}(\mu, s)(\gv) - \partial_\mu {\rm I}(\mu', s)(\gv)] \, \widehat{p}_{m+1}(\mu, s, r, t, x, z)  \Big| \\
 & \leq K^{++}_\beta \left\{  \frac{W_2(\mu, \mu')^\beta}{(t-r)^{\frac12}(r-s)^{1+\frac{\beta-\eta}{2}}} \right.   \\
 & \quad \left. + \frac{1}{(t-r)^{\frac12}} \int_{(\mathbb{R}^d)^2} (|y'-x'|^\eta \wedge 1) |\partial^2_\mu p_m(\mu, s, r, x', y')(\gv) - \partial^2_\mu p_m(\mu', s, r, x', y')(\gv)|  \,  \mu(dx') \, dy' \right. \\
 & \quad \left. +  \frac{1}{(t-r)^{\frac32}} \int_r^t \int_{(\mathbb{R}^d)^2} (|y'-x'|^\eta \wedge 1) |\partial^2_\mu p_m(\mu, s, r', x', y')(\gv) - \partial^2_\mu p_m(\mu', s, r', x', y')(\gv)|  \, \mu(dx')\, dy'  \, dr'\right\}\\
 & \quad \quad g(c(t-r), z-x),
 \end{align*}
  \begin{align*}
 \Big| [\partial_\mu & {\rm II}(\mu, s)(\gv) -  \partial_\mu {\rm II}(\mu', s)(\gv)] \widehat{p}_{m+1}(\mu, s, r, t, x, z) \Big| \\
 & \leq K^{++}_\beta \left\{ W_2(\mu,\mu')^\beta \Bigg(\frac{1}{(t-r)^{1-\frac{\eta}{2}}(r-s)^{1+\frac{\beta}{2}}} \wedge \frac{1}{(t-r)(r-s)^{1+\frac{\beta-\eta}{2}}}\Bigg) \right. \\
 &\quad \left. + \frac{1}{(t-r)^{1-\frac{\eta}{2}}} \int_{(\mathbb{R}^d)^2}  | \partial_\mu^2 p_m(\mu, s, r, x', y')(\gv) -  \partial_\mu^2 p_m(\mu', s, r, x', y')(\gv) |  \,  \mu(dx') \, dy'  \right.   \\
& \quad \quad \left.  \wedge \, \frac{1}{t-r} \int_{(\mathbb{R}^d)^2} (|y'-x'|^\eta \wedge 1)  | \partial_\mu^2 p_m(\mu, s, r, x', y')(\gv) -  \partial_\mu^2 p_m(\mu', s, r, x', y')(\gv) |  \,  \mu(dx') \, dy'  \right. \\
 & \quad \left. +  \frac{1}{(t-r)^{2-\frac{\eta}{2}}} \int_r^t \int_{(\mathbb{R}^d)^2} (|y'-x'|^\eta \wedge 1) |\partial^2_\mu p_m(\mu, s, r', x', y')(\gv) - \partial^2_\mu p_m(\mu', s, r', x', y')(\gv)|  \, \mu(dx') \, dy' \, dr' \right\}\\
 & \quad \quad g(c(t-r), z-x)
 \end{align*}
  \noindent and, using \eqref{sec:mes:deriv:parametrix:kernel:s:r:t:with:beta:partii} and \eqref{sec:mes:deriv:parametrix:kernel:s:r:t:with:beta:partiii} (combined with \eqref{second:deriv:mes:induction:decoupling:mckean} and the space time inequality \eqref{space:time:inequality}) the latter being used both with $\beta=0$ and $\beta=1$, replacing therein the standard Gaussian estimate on $\widehat{p}_{m+1}$ by the estimate \eqref{gaussian:bound:diff:mes:hat:pm:same:time} on $\widehat{p}_{m+1}(\mu, s, r, t, x, z)-\widehat{p}_{m+1}(\mu', s, r, t, x, z)$, we get 
 \begin{align*}
 \Big|  (\partial_\mu & {\rm I}(\mu', s)(\gv)+  \partial_\mu  {\rm II}(\mu', s)(\gv)) (\widehat{p}_{m+1}(\mu, s, r, t, x, z)-\widehat{p}_{m+1}(\mu', s, r, t, x, z)) \Big|\\
 & \leq K^+ W_2(\mu,\mu')^\beta \Bigg(\frac{1}{(t-r)^{1-\frac{\eta}{2}}(r-s)^{1+\frac{\beta}{2}}} \wedge \frac{1}{(t-r)(r-s)^{1+\frac{\beta-\eta}{2}}}\Bigg) \, g(c(t-r), z-x).
 \end{align*}
  
Hence the difference $\big(\partial_\mu {\rm I}(\mu, s)(\gv)+ \partial_\mu {\rm II}(\mu, s)(\gv)\big)\widehat{p}_{m+1}(\mu, s, r, t, x, z) - \big(\partial_\mu {\rm I}(\mu', s)(\gv)+ \partial_\mu {\rm II}(\mu', s)(\gv)\big)\widehat{p}_{m+1}(\mu', s, r, t, x, z)$ is bounded by the right-hand side of \eqref{diff:mes:L:deriv:parametrix:kernel:pmp1} by gathering the three previous estimates.\\
 
 \noindent (ii) \emph{Regularity of the maps $\mu \mapsto \big({\rm I}(\mu, s)(v)+  {\rm II}(\mu, s)(v)\big) \otimes \partial_\mu \widehat{p}_{m+1}(\mu, s, r, t, x, z)(v'), \,  \partial_\mu\widehat{p}_{m+1}(\mu, s, r, t, x, z)(v) \otimes \partial_\mu {\rm III }(\mu, s)(v')$.}\\
Note that these two terms only involve first order L-derivative so that they can be handled using the regularity results established in \cite{chaudruraynal:frikha}. To be more specific, if $W_2(\mu, \mu') > (r-s)^{1/2}$, one simply uses \eqref{estimate:I:II:partial:mes:p:hat:mp1} with $\beta=1$ to conclude that
\begin{align*}
\Big| \big({\rm I}(\mu, s)(v)+  {\rm II}(\mu, s)(v)\big) & \otimes \partial_\mu \widehat{p}_{m+1}(\mu, s, r, t, x, z)(v') \\
& - \big({\rm I}(\mu', s)(v)+  {\rm II}(\mu', s)(v)\big) \otimes \partial_\mu \widehat{p}_{m+1}(\mu', s, r, t, x, z)(v')\Big|\\
& \leq K \frac{W_2(\mu, \mu')^\beta}{(t-r)^{1-\frac{\eta}{2}}(r-s)^{1+\frac{\beta-\eta}{2}}}.
\end{align*}

If $W_2(\mu, \mu') \leq (r-s)^{1/2}$, one uses the Lipschitz regularity of the maps $\mu \mapsto {\rm I}(\mu, s), \, {\rm II}(\mu, s)$ provided by the estimates \eqref{sec:mes:deriv:parametrix:kernel:s:r:t:with:beta:partii}, \eqref{sec:mes:deriv:parametrix:kernel:s:r:t:with:beta:partiii} with $\beta=1$ combined with \eqref{second:deriv:mes:induction:decoupling:mckean}, in which we replace the Gaussian estimate on $\widehat{p}_{m+1}$ by the estimate \eqref{deriv:v:deriv:mes:pm:hat} with $n=0$. Hence, we get
\begin{align*}
\Big| [{\rm I}(\mu, s)(v) - {\rm I}(\mu', s)(v) & + {\rm II}(\mu, s)(v) - {\rm II}(\mu', s)(v)] \otimes \partial_\mu \widehat{p}_{m+1}(\mu, s, r, t, x, z)(v') \Big| \\
& \leq K^{+}  \frac{W_2(\mu, \mu')}{(t-r)^{1-\frac{\eta}{2}}(r-s)^{\frac{3-\eta}{2}}} \, g(c(t-r), z-x) \\
& \leq K^{+}  \frac{W_2(\mu, \mu')^\beta}{(t-r)^{1-\frac{\eta}{2}}(r-s)^{1+\frac{\beta-\eta}{2}}} \, g(c(t-r), z-x).
\end{align*}

Then, one uses \eqref{estimate:I:II:partial:mes:p:hat:mp1} with $\beta=1$ replacing therein the estimate on $\partial_\mu \widehat{p}_{m+1}$ by \eqref{diff:L:deriv:pm:mu:mup} with $r>s$. We thus obtain
\begin{align*}
 \Big| [{\rm I}(\mu', s)(v) + {\rm II}(\mu', s)(v)] & \otimes [ \partial_\mu \widehat{p}_{m+1}(\mu, s, r, t, x, z)(v')  -\partial_\mu \widehat{p}_{m+1}(\mu', s, r, t, x, z)(v')] \Big| \\
  & \leq K^{+}_\beta  \frac{W_2(\mu, \mu')^\beta}{(t-r)^{1-\frac{\eta}{2}}(r-s)^{1+\frac{\beta-\eta}{2}}} \, g(c(t-r), z-x). 
\end{align*}

Gathering the above estimates, we get that $\big|\big[ {\rm I}(\mu, s)(v) + {\rm II}(\mu, s)(v)\big]  \otimes \partial_\mu \widehat{p}_{m+1}(\mu, s, r, t, x, z)(v')  -  \big[ {\rm I}(\mu', s)(v) + {\rm II}(\mu', s)(v)\big]  \otimes \partial_\mu \widehat{p}_{m+1}(\mu', s, r, t, x, z)(v') \big|$ is bounded by the first term appearing on the right-hand side of \eqref{diff:mes:L:deriv:parametrix:kernel:pmp1}.
%\begin{align*}
% \Big| \big[ {\rm I}(\mu, s)(v) + {\rm II}(\mu, s)(v)\big] & \otimes \partial_\mu \widehat{p}_{m+1}(\mu, s, r, t, x, z)(v')  -  \big[ {\rm I}(\mu', s)(v) + {\rm II}(\mu', s)(v)\big]  \otimes \partial_\mu \widehat{p}_{m+1}(\mu', s, r, t, x, z)(v') \Big| \\
%& \leq W_2^{\beta}(\mu, \mu') \left\{ \frac{1}{(t-r)(r-s)^{1+\frac{\beta-\eta}{2}}} \wedge \frac{1}{(t-r)^{1- \frac{\eta}{2}}(r-s)^{1+\frac{\beta}{2}}} \right\} \, g(c(t-r), z-x).
%\end{align*}

From the symmetry identity \eqref{identity:partial:mes:pmp1:convolution:partial:mes:parametrix:kernel}, the same conclusion holds for $\big|\partial_\mu \widehat{p}_{m+1}(\mu, s, r, t, x, z)(v)  \otimes \partial_\mu {\rm III }(\mu, s)(v') - \partial_\mu\widehat{p}_{m+1}(\mu', s, r, t, x, z)(v) \otimes \partial_\mu {\rm III }(\mu', s)(v')\big|$.\\
%\begin{align*}
%\Big| \partial_\mu & \widehat{p}_{m+1}(\mu, s, r, t, x, z)(v)  \otimes \partial_\mu {\rm III }(\mu, s)(v') - \partial_\mu\widehat{p}_{m+1}(\mu', s, r, t, x, z)(v) \otimes \partial_\mu {\rm III }(\mu', s)(v') \Big|\\
%& \leq W_2^{\beta}(\mu, \mu') \left\{ \frac{1}{(t-r)(r-s)^{1+\frac{\beta-\eta}{2}}} \wedge \frac{1}{(t-r)^{1- \frac{\eta}{2}}(r-s)^{1+\frac{\beta}{2}}} \right\} \, g(c(t-r), z-x).
%\end{align*}
% 
 
\noindent (iii) \emph{Regularity of the maps $\mu \mapsto {\rm III }(\mu, s)\, \partial_\mu^2\widehat{p}_{m+1}(\mu, s, r, t, x, z)(\gv)$.}\\

We remark that $\partial_\mu {\rm III}(\mu, s)(v) = {\rm I}(\mu, s)(v) + {\rm II}(\mu, s)(v)$ so that using the uniform estimate \eqref{estimate:I:II:partial:mes:p:hat:mp1} with $\beta=1$ in which we replace the Gaussian estimate on $\partial_\mu \widehat{p}_{m+1}(\mu, s, r, t, x, z)(v')$ by \eqref{second:deriv:mes:induction:decoupling:mckean} recalling that $\mathscr{C}^{1,0}_m(C^+, t-s)\leq K^+:=\lim_{m\rightarrow \infty} \mathscr{C}^{1,0}_m(C^+, t-s)$, we obtain
\begin{align*}
\Big| \Big({\rm III }(\mu, s) -{\rm III}(\mu', s) \Big)  \partial_\mu^2\widehat{p}_{m+1}(\mu, s, r, t, x, z)(\gv)\Big| & \leq K^+ \frac{W_2(\mu, \mu')}{(t-r)^{1-\frac{\eta}{2}}(r-s)^{\frac{3-\eta}{2}}} \, g(c(t-r), z-x) \\
& \leq K^+  \frac{W_2(\mu, \mu')^\beta}{(t-r)^{1-\frac{\eta}{2}}(r-s)^{1+\frac{\beta-\eta}{2}}} \, g(c(t-r), z-x)
\end{align*}
\noindent if $W_2(\mu, \mu') \leq (r-s)^{1/2}$. Assuming now that $W_2(\mu, \mu') \geq (r-s)^{1/2}$, we directly use \eqref{sec:mes:deriv:parametrix:kernel:s:r:t:with:beta:partiv} combined with \eqref{second:deriv:mes:induction:decoupling:mckean} so that we get the previous estimate.

%From \eqref{diff:mes:with:holder:reg:space:drift:diff:coefficients} with $\alpha=0$, \eqref{diff:mes:drift:diff:coefficients} if $W_2(\mu, \mu') \leq (t-s)^{1/2}$ or the boundedness of $b_i$ and the uniform $\eta$-H\"older regularity of $a_{i, j}(t, ., m)$, \eqref{bound:second:deriv:mes:hat:p} combined with \eqref{second:deriv:mes:induction:decoupling:mckean} and the space time inequality \eqref{space:time:inequality} recalling that $\mathscr{C}^{1,0}_m(C^+, t-s)\leq K^+:=\lim_{m\rightarrow \infty} \mathscr{C}^{1,0}_m(C^+, t-s)$ and again the space time inequality \eqref{space:time:inequality}, we deduce that for any $\beta \in [0,\eta)$
%\begin{align*}
%\Big| \Big({\rm III }(\mu, s) -{\rm III}(\mu', s) \Big)  \partial_\mu^2\widehat{p}_{m+1}(\mu, s, r, t, x, z)(\gv)\Big| \leq K^+  \frac{W_2^\beta(\mu, \mu')}{(t-r)^{1-\frac{\eta}{2}}(r-s)^{1+\frac{\beta-\eta}{2}}} \, g(c(t-r), z-x)
%\end{align*} 
%  
%\noindent and, 
Now, from the uniform boundedness of $b_i$ and the uniform $\eta$-H\"older regularity of $a_{i, j}(t, ., m)$, \eqref{diff:lions:deriv:pm:mu:mup} and the space time inequality \eqref{space:time:inequality}
\begin{align*}
\Big| {\rm III}(\mu', s) & \Big(\partial_\mu^2\widehat{p}_{m+1}(\mu, s, r, t, x, z)(\gv) - \partial_\mu^2\widehat{p}_{m+1}(\mu', s, r, t, x, z)(\gv) \Big) \Big| \\
& \leq K^{++}_\beta  \left\{ \frac{W_2(\mu, \mu')^\beta}{(t-r)^{1-\frac{\eta}{2}}(r-s)^{1+\frac{\beta-\eta}{2}}} \right. \\ 
& \quad \left. +\frac{1}{(t-r)^{2-\frac{\eta}{2}}}   \int_r^t \int_{(\mathbb{R}^d)^2} (|y'-x'|^{\eta} \wedge1) | \partial_\mu^2 p_m(\mu, s, r', x', y')(\gv) -  \partial_\mu^2 p_m(\mu', s, r', x', y')(\gv) | \,  \mu(dx') \, dy'  \, dr' \right\} \\
& \quad \quad \times g(c(t-r), z-x).
\end{align*}

Collecting the above estimates concludes the proof of \eqref{diff:mes:L:deriv:parametrix:kernel:pmp1}.

\subsection{Proof of Lemma \ref{lem:reg:time:deriv:sec:mes:coeff:parametrix:density:kernel}.\\}\label{proof:lem:reg:time:deriv:sec:mes:coeff:parametrix:density:kernel}

\emph{Step 1: proof of the estimate \eqref{diff:time:second:deriv:mes:a:b}.}\\

We first observe that if $|s_1-s_2| \geq t-s_1\vee s_2$ then the result directly follows from \eqref{recursive:bound:deriv:a:or:b} combined with \eqref{second:deriv:mes:induction:decoupling:mckean} and the space time inequality \eqref{space:time:inequality}, recalling that $\mathscr{C}^{1,0}_m(C^+, t-s) \leq K^+:= \lim_{m\rightarrow \infty} \mathscr{C}^{1,0}_m(C^+, t-s)< \infty$. We thus assume that $|s_1-s_2| \leq t-s_1\vee s_2$ for the rest of the proof. As in the proof of \eqref{diff:mes:second:deriv:mes:a:b}, we make use of the identity \eqref{full:expression:second:deriv} applied to the map $m \mapsto b_i(t, x, m)$ so that the difference $\partial^2_\mu [b_{i}(t, x, [X^{s_1, \xi, (m)}_{t}])](\gv)  - \partial^2_\mu [b_{i}(t, x, [X^{s_2, \xi, (m)}_{t}])](\gv)$ writes as the sum of the terms $\delta b^{\ell}_i(s_1, s_2)$, $\ell=1, \cdots, 7 $, defined by
{\small
\begin{align*}
\delta b^{1}_i(s_1, s_2) & = \int_{(\mathbb{R}^d)^2} \Big[\frac{\delta^2 b_i}{\delta m^2}(t, x, [X^{s_1, \xi, (m)}_{t}])(z, z') - \frac{\delta^2 b_i}{\delta m^2}(t, x, [X^{s_2, \xi, (m)}_t])(z, z')\Big] \\
& \quad  \quad \partial_x p_m(\mu, s_1, t, v, z)\otimes \partial_{x} p_m(\mu, s_1, t, v',z') \, dz\, dz' \\
& + \int_{(\mathbb{R}^d)^2} \Big[\frac{\delta^2 b_i}{\delta m^2}(t, x, [X^{s_2, \xi, (m)}_{t}])(z, z') - \frac{\delta^2 b_i}{\delta m^2}(t, x, [X^{s_2, \xi, (m)}_t])(z, v')\Big] \\
& \quad \Big[\partial_x p_m(\mu, s_1, t, v, z)\otimes \partial_{x} p_m(\mu, s_1, t, v',z')  - \partial_x p_m(\mu, s_2, t, v, z)\otimes \partial_{x} p_m(\mu, s_2, t, v',z')  \Big] \, dz\, dz',
\end{align*}
\begin{align*}
\delta b^{2}_i(s_1, s_2) & = \int_{(\mathbb{R}^d)^3} \Big[ \frac{\delta^2 b_i}{\delta m^2}(t, x, [X^{s_1, \xi , (m)}_t])(z, z') - \frac{\delta^2 b_i}{\delta m^2}(t, x, [X^{s_2, \xi , (m)}_t])(z, z') \Big] \\
& \quad  \partial_x  p_m(\mu, s_1, t, v, z) \otimes \partial_\mu p_m(\mu, s_1, t, x', z')(v')  \, dz\, dz'\, \mu(dx')\\
& +  \int_{(\mathbb{R}^d)^3}  \frac{\delta^2 b_i}{\delta m^2}(t, x, [X^{s_2, \xi' , (m)}_t])(z, z') \\
& \quad  \Big[\partial_x  p_m(\mu, s_1, t, v, z) \otimes \partial_\mu p_m(\mu, s_1, t, x', z')(v') - \partial_x  p_m(\mu, s_2, t, v, z) \otimes \partial_\mu p_m(\mu, s_2, t, x', z')(v')\Big] \, dz\, dz'\, \mu(dx'),
\end{align*}

\begin{align*}
\delta b^{3}_i(s_1, s_2) & = \int_{\mathbb{R}^d}  \Big[\frac{\delta b_i}{\delta m}(t, x, [X^{s_1, \xi, (m)}_t])(z) - \frac{\delta b_i}{\delta m}(t, x, [X^{s_2, \xi, (m)}_t])(z) \Big]\, \partial_\mu\partial_x p_m(\mu, s_2, t, v ,z)(v') \,  dz \\
& \quad +  \int_{\mathbb{R}^d}  \frac{\delta b_i}{\delta m}(t, x, [X^{s_2, \xi, (m)}_t])(z) \, \Big[ \partial_\mu\partial_x p_m(\mu, s_1, t, v ,z)(v') -  \partial_\mu\partial_x p_m(\mu, s_2, t, v ,z)(v')  \Big] \, dz,
\end{align*}
\begin{align*}
\delta b^{4}_i(s_1, s_2) & =  \int_{\mathbb{R}^d}  \Big[ \frac{\delta b_i}{\delta m} (t, x, [X_t^{s_1, \xi , (m)}])(z) -  \frac{\delta b_i}{\delta m} (t, x, [X_t^{s_2, \xi , (m)}])(z)\Big]   \partial_{x}\partial_\mu p_m(\mu, s_1, t, v', z)(v) \, dz\\
& \quad + \int_{\mathbb{R}^d}   \frac{\delta b_i}{\delta m} (t, x, [X_t^{s_2, \xi , (m)}])(z)  \Big[ \partial_{x}\partial_\mu p_m(\mu, s_1, t, v', z)(v) - \partial_{x}\partial_\mu p_m(\mu, s_2, t, v', z)(v) \Big]\, dz,
\end{align*}
\begin{align*}
\delta b^{5}_i(s_1, s_2) & =  \int_{(\mathbb{R}^d)^3} \Big[ \frac{\delta^2 b_i}{\delta m^2} (t, x, [X_t^{s_1, \xi , (m)}])(z, z') - \frac{\delta^2 b_i}{\delta m^2} (t, x, [X_t^{s_2, \xi , (m)}])(z, z')  \Big] \\
& \quad \quad  \partial_\mu p_m(\mu, s_1, t, x', z)(v)\otimes  \partial_{x}p_m(\mu, s_1, t, v',z') \, dz \, dz' \, \mu(dx')\\
& \quad  + \int_{(\mathbb{R}^d)^3} \frac{\delta^2 b_i}{\delta m^2} (t, x, [X_t^{s_2, \xi , (m)}])(z, z')  \\
& \quad \quad \Big[ \partial_\mu p_m(\mu, s_1, t, x', z)(v)\otimes  \partial_{x}p_m(\mu, s_1, t, v',z') - \partial_\mu p_m(\mu, s_2, t, x', z)(v)\otimes  \partial_{x}p_m(\mu, s_2, t, v',z') \Big]\, dz \, dz' \, \mu(dx'), 
\end{align*}
\begin{align*}
\delta b^{6}_i(s_1, s_2) & =  \int_{(\mathbb{R}^d)^4} \Big[ \frac{\delta^2 b_i}{\delta m^2} (t, x, [X_t^{s_1, \xi , (m)}])(z, z') - \frac{\delta^2 b_i}{\delta m^2} (t, x, [X_t^{s_2, \xi , (m)}])(z, z')\Big] \\
& \quad  \partial_\mu p_m(\mu, s_1, t, x', z)(v)\otimes  \partial_{\mu}p_m(\mu, s_1, t, x'',z')(v') \, dz \, dz' \, \mu(dx') \mu(dx'')\\
& + \int_{(\mathbb{R}^d)^4} \frac{\delta^2 b_i}{\delta m^2} (t, x, [X_t^{s_2, \xi , (m)}])(z, z')  \Big[ \partial_\mu p_m(\mu, s_1, t, x', z)(v)\otimes  \partial_{\mu}p_m(\mu, s_1, t, x'',z')(v') \\
& \quad -   \partial_\mu p_m(\mu, s_2, t, x', z)(v)\otimes  \partial_{\mu}p_m(\mu, s_2, t, x'',z')(v')\Big]\, dz \, dz' \, \mu(dx') \, \mu(dx''),
\end{align*}
\begin{align*}
\delta b^{7}_i(s_1, s_2) & =  \int_{(\mathbb{R}^d)^2}  \Big[ \frac{\delta b_i}{\delta m} (t, x, [X_t^{s_1, \xi , (m)}])(z) -  \frac{\delta b_i}{\delta m} (t, x, [X_t^{s_2, \xi , (m)}])(z) \Big] \, \partial_\mu^2 p_m(\mu, s_1, t, x', z)(\gv)  \,  dz\, \mu(dx') \\
& + \int_{(\mathbb{R}^d)^2}   \Big[ \frac{\delta b_i}{\delta m} (t, x, [X_t^{s_2, \xi , (m)}])(z) -  \frac{\delta b_i}{\delta m} (t, x, [X_t^{s_2, \xi , (m)}])(x') \Big] \\
& \quad \Big[  \partial_\mu^2 p_m(\mu, s_1, t, x', z)(\gv) - \partial_\mu^2 p_m(\mu, s_2, t, x', z)(\gv) \Big]\,  \mu(dx') \,  dz.
\end{align*}
}

We recall from the proof of the estimate (A.62) of Lemma A.3 in \cite{chaudruraynal:frikha} that if $h: \mathcal{P}(\mathbb{R}^d) \mapsto \mathbb{R}$ has a bounded and continuous linear functional derivative such that $[\delta h/\delta m](m)(.)$ is uniformly $\eta$-H\"older continuous then for any $\beta \in [0, 1]$ one has 
$$
|h([X^{s_1, \xi, (m)}_t]) - h([X_t^{s_2, \xi, (m)}])| \leq K |s_1-s_2|^\beta \left\{ \frac{1}{(t-s_1)^{\beta-\frac{\eta}{2}}} + \frac{1}{(t-s_2)^{\beta-\frac{\eta}{2}}} \right\}.
$$

The above estimate is established in \cite{chaudruraynal:frikha} for the map $m\mapsto a_{i, j}(t, x, m)$ but the argument works, \emph{mutatis mutandis}, in this general form. In particular, under our current assumptions, for any $\beta \in [0, 1]$, it holds
\begin{align}
\Big| & \frac{\delta b_i}{\delta m}(t, x, [X^{s_1, \xi, (m)}_{t}])(z) - \frac{\delta b_i}{\delta m}(t, x, [X^{s_2, \xi, (m)}_t])(z) \Big| \label{delta:time:linear:functional:deriv:coeff}\\
& \quad + \Big|\frac{\delta^2 b_i}{\delta m^2}(t, x, [X^{s_1, \xi, (m)}_{t}])(z, z') - \frac{\delta^2 b_i}{\delta m^2}(t, x, [X^{s_2, \xi, (m)}_t])(z, z') \Big| \notag \\
& \quad \quad \leq K^{++} |s_1-s_2|^\beta \left\{ \frac{1}{(t-s_1)^{\beta-\frac{\eta}{2}}} + \frac{1}{(t-s_2)^{\beta-\frac{\eta}{2}}} \right\}. \notag
\end{align}
 
The above estimate together with \eqref{bound:derivative:heat:kernel}, \eqref{regularity:time:estimate:v1:v2:v3:decoupling:mckean}, the uniform $\eta$-H\"older regularity of $[\delta^2 b_i/\delta m^2](t, x, m)(z, .)$, the space time inequality \eqref{space:time:inequality} and the fact that $|s_1-s_2|\leq t-s_1\vee s_2$ give that for any $\beta \in [0,\eta)$
 $$
 |\delta b^{1}_i(s_1, s_2)| \leq K^{++}_\beta \frac{|s_1-s_2|^{\frac{\beta}{2}}}{(t-s_1\vee s_2)^{1+\frac{\beta-\eta}{2}}}.
 $$
 
 Similarly, from \eqref{delta:time:linear:functional:deriv:coeff}, \eqref{bound:derivative:heat:kernel}, \eqref{first:second:estimate:induction:decoupling:mckean} with $n=0$, \eqref{regularity:time:estimate:v1:v2:v3:decoupling:mckean} with $n=1$, \eqref{regularity:time:estimate:v1:v2:decoupling:mckean} with $n=0$, we deduce
 $$
 |\delta b^{2}_i(s_1, s_2)| \leq K^{++}_\beta \frac{|s_1-s_2|^{\frac{\beta}{2}}}{(t-s_1\vee s_2)^{1+\frac{\beta-\eta}{2}}}.
 $$
 
 The two next terms, namely, $\delta b^{3}_i(s_1, s_2)$ and $\delta b^{4}_i(s_1, s_2)$ can be handled in a similar manner. From \eqref{delta:time:linear:functional:deriv:coeff}, \eqref{cross:deriv:mes:space:induction:decoupling:mckean} and \eqref{sensitivity:time:deriv:cross:space:mes:pm}, we obtain
 $$
  |\delta b^{3}_i(s_1, s_2)| +  |\delta b^{4}_i(s_1, s_2)| \leq K^{++}_\beta \frac{|s_1-s_2|^{\frac{\beta}{2}}}{(t-s_1\vee s_2)^{1+\frac{\beta-\eta}{2}}}.
 $$
 
 We deal with $\delta b^5_i(s_1, s_2)$ using \eqref{delta:time:linear:functional:deriv:coeff}, \eqref{bound:derivative:heat:kernel}, \eqref{first:second:estimate:induction:decoupling:mckean}, \eqref{regularity:time:estimate:v1:v2:v3:decoupling:mckean} and \eqref{regularity:time:estimate:v1:v2:decoupling:mckean}. We obtain
$$
   |\delta b^{5}_i(s_1, s_2)| \leq K^{++} \frac{|s_1-s_2|^{\frac{\beta}{2}}}{(t-s_1\vee s_2)^{1+\frac{\beta-\eta}{2}}}.
$$
 
We handle $\delta b_i^{6}(s_1, s_2)$ in a similar manner. It follows from \eqref{delta:time:linear:functional:deriv:coeff}, \eqref{first:second:estimate:induction:decoupling:mckean}, \eqref{regularity:time:estimate:v1:v2:decoupling:mckean} with $n=0$ that
$$
|\delta b_i^{6}(s_1, s_2)| \leq K^{++} \frac{|s_1-s_2|^{\frac{\beta}{2}}}{(t-s_1\vee s_2)^{1+\frac{\beta}{2}- \eta}}.
$$

Finally, from \eqref{delta:time:linear:functional:deriv:coeff}, \eqref{second:deriv:mes:induction:decoupling:mckean}, the boundedness and uniform $\eta$-H\"older regularity of $[\delta b_i/\delta m](t, x, m)(.)$, we deduce that the last term satisfies
\begin{align*}
|\delta b_i^{7}(s_1, s_2)| & \leq K^{++} \left\{ \frac{|s_1-s_2|^{\frac{\beta}{2}}}{(t-s_1\vee s_2)^{1+\frac{\beta}{2}-\eta}}  \right. \\
& \left. \quad +   \int_{(\mathbb{R}^d)^2} (|z-x'|^\eta \wedge 1) | \partial_\mu^2 p_m(\mu, s_1, t, x', z)(\gv) - \partial_\mu^2 p_m(\mu, s_2, t, x', z)(\gv)| \,  \mu(dx') \,dz \right\}.
\end{align*}

 We conclude the proof of \eqref{diff:time:second:deriv:mes:a:b} for $\Big|  \partial^2_\mu [b_{i}(t, x, [X^{s_1, \xi, (m)}_{t}])](\gv)  - \partial^2_\mu [ b_{i}(t, x, [X^{s, \xi, (m)}_{t}])](\gv) \Big|$ by collecting the above estimates. The proof of the upper-bound for $\Big|  \partial^2_\mu [a_{i, j}(t, x, [X^{s_1, \xi, (m)}_{t}])](\gv)  - \partial^2_\mu [ a_{i, j}(t, x, [X^{s, \xi, (m)}_{t}])](\gv) \Big|$ follows from completely analogous arguments and is thus omitted.\\

\noindent \emph{Step 2: proof of the estimate \eqref{diff:time:second:L:deriv:p:hat}.}\\

We proceed as in the proof of \eqref{diff:lions:deriv:pm:mu:mup}. Namely, we start from the identity \eqref{representation:formula:second:order:deriv:mes:p:hat} and quantify the regularity with respect to the variable $s$ of each term. We first note that from the mean-value theorem, for any $\beta \in [0,1]$, it holds
\begin{align}
\Big| & \left(\int_r^t a(r', y, [X^{s_1, \xi, (m)}_{r'}])\, dr'\right)^{-1}   - \left(\int_r^t a(r', y, [X^{s_2, \xi, (m)}_{r'}])\, dr'\right)^{-1} \Big|   \label{diff:time:inverse:diff:matrix} \\
& \qquad \leq  \frac{K}{(t-r)^2} \int_r^t \max_{i, j} |a_{i, j}(r', y, [X^{s_1, \xi, (m)}_{r'}]) - a_{i, j}(r', y, [X^{s_2, \xi, (m)}_{r'}])| \, dr' \notag \\
&\qquad \leq K \frac{|s_1-s_2|^\beta}{(t-r)^{2}} \int_r^t \frac{1}{(r' -s_1\vee s_2)^{\beta}}\, dr' \notag \\
&\qquad \leq  K \frac{|s_1-s_2|^\beta}{(t-r)(r-s_1\vee s_2)^{\beta}}\notag
\end{align}
\noindent where we used the estimate \eqref{diff:time:drift:diff:coefficients} inside the time integral if $|s_1-s_2| \leq r'-s_1\vee s_2$ or the boundedness of $a_{i, j}$ otherwise. From the identity \eqref{first:order:diff:mat:mes}, the previous estimate, \eqref{deriv:mes:a:m} and \eqref{diff:time:cross:deriv:diff:and:deriv:coeff} both with $n=0$, we get that for any $\beta \in [0,1]$
\begin{align}
\Big| \partial_\mu\Big[\left(\int_r^t a(r', y, [X^{s_1, \xi, (m)}_{r'}])\, dr'\right)^{-1}& \Big](v')  - \partial_\mu\Big[\left(\int_r^t a(r', y, [X^{s_2, \xi, (m)}_{r'}])\, dr'\right)^{-1}\Big](v') \Big| \label{diff:L:deriv:invers:diff:coeff}\\
& \leq K^{+}_{\beta} \frac{|s_1-s_2|^\beta}{(t-r)(r-s_1\vee s_2)^{\frac{1-\eta}{2}+\beta}}. \notag
\end{align}

Then, we use the identity \eqref{representation:formula:second:order:deriv:mes:p:hat} together with the two above estimates, \eqref{diff:time:second:deriv:mes:a:b}, \eqref{recursive:bound:deriv:a:or:b} (combined with \eqref{second:deriv:mes:induction:decoupling:mckean} and the space time inequality \eqref{space:time:inequality}), \eqref{deriv:mes:a:m}, \eqref{gaussian:bound:diff:time:hat:pm:different:time}, \eqref{diff:time:L:deriv:p:hat:different:time} and \eqref{diff:time:cross:deriv:diff:and:deriv:coeff} with $n=0$. After some standard computations that we omit, for any $\beta \in [0,\eta)$, we obtain
\begin{align*}
& | \partial^2_\mu \widehat{p}^y_{m+1}(\mu, s_1, r, t, x, z)(\gv) - \partial^2_\mu \widehat{p}^y_{m+1}(\mu, s_2, r, t, x, z)(\gv)  | \notag\\
& \quad \leq  K^{++}_\beta  \Bigg\{ \frac{|s_1-s_2|^{\frac{\beta}{2}}}{(r-s_1\vee s_2)^{1+\frac{\beta-\eta}{2}}} \notag \\
& \qquad \qquad  + \frac{1}{t-r}   \int_r^t \int_{(\mathbb{R}^d)^2} (|y'-x'|^\eta \wedge 1)|\partial^2_\mu p_m(\mu, s_1, r', x', y')(\gv) - \partial^2_\mu p_m(\mu, s_2, r', x', y')(\gv)|   \,  \mu(dx') \, dy' \, dr' \Bigg\}  \notag\\
& \qquad \qquad \qquad \times  g(c(t-r), z-x)\notag 
\end{align*}
\noindent which completes the proof of the estimate \eqref{diff:time:second:L:deriv:p:hat}.\\

\noindent \emph{Step 3: proof of the estimate \eqref{diff:time:second:L:deriv:p:hat:same:time}.}\\

We first remark that if $|s_1-s_2| \geq t-s_1\vee s_2$ then the announced estimate directly follows from \eqref{bound:second:deriv:mes:hat:p} (with $r=s$) combined with \eqref{second:deriv:mes:induction:decoupling:mckean} recalling that $\mathscr{C}^{n, 0}_{m}(C^+, t-s)\leq K^+ :=\mathscr{C}^{n, 0}_{\infty}(C^+, t-s)=\lim_{m\uparrow \infty} \mathscr{C}^{n,0}_{m}(C^+, t-s) < \infty$. From now on and for the rest of the proof, we assume that $|s_1 - s_2| \leq  t-s_1\vee s_2 $. The proof being quite similar to the previous one, we will be short on some arguments and will omit some technical details. Observe that from the identity \eqref{first:order:diff:mat:mes}, the estimates \eqref{diff:time:inverse:a}, \eqref{deriv:mes:a:m} and \eqref{diff:time:cross:deriv:diff:and:deriv:coeff} with $n=0$, after some standard computations that we omit, we get that for any $\beta \in [0,(1+\eta)/2)$
\begin{align*}
\Big| \partial_\mu\Big[\left(\int_{s_1}^t a(r', y, [X^{s_1, \xi, (m)}_{r'}])\, dr'\right)^{-1}& \Big](v')  - \partial_\mu\Big[\left(\int_{s_2}^t a(r', y, [X^{s_2, \xi, (m)}_{r'}])\, dr'\right)^{-1}\Big](v') \Big| \\
& \leq K^{+}_{\beta} \frac{|s_1-s_2|^\beta}{(t-s_1\vee s_2)^{\frac{3-\eta}{2}+\beta}}.
\end{align*}

Now, we again use the identity \eqref{representation:formula:second:order:deriv:mes:p:hat} (with $r=s$) together with the above estimate, \eqref{diff:time:inverse:a}, \eqref{diff:time:second:deriv:mes:a:b}, \eqref{recursive:bound:deriv:a:or:b} (combined with \eqref{second:deriv:mes:induction:decoupling:mckean} and the space time inequality \eqref{space:time:inequality}), \eqref{deriv:mes:a:m}, \eqref{diff:time:L:deriv:p:hat:same:time}, \eqref{diff:time:cross:deriv:diff:and:deriv:coeff} with $n=0$ and the fact that $|s_1-s_2| \leq t-s_1\vee s_2$. After some standard computations that we omit, for any $\beta \in [0,\eta)$, we obtain
\begin{align*}
& | \partial^2_\mu \widehat{p}_{m+1}(\mu, s_1, t, x, z)(\gv) - \partial^2_\mu \widehat{p}_{m+1}(\mu, s_2, t, x, z)(\gv)  | \notag\\
& \quad \leq  K^{++}_\beta  \Bigg\{ \frac{|s_1-s_2|^{\frac{\beta}{2}}}{(t-s_1\vee s_2)^{1+\frac{\beta-\eta}{2}}}  g(c(t-s_1\vee s_2), z-x) +\frac{|s_1-s_2|^{\frac{\beta}{2}}}{(t-s_1\wedge s_2)^{1+\frac{\beta-\eta}{2}}}  g(c(t-s_1\wedge s_2), z-x) \notag \\
& \qquad \qquad  + \frac{1}{t-s_1\vee s_2}   \int_{s_1\vee s_2}^t \int_{(\mathbb{R}^d)^2} (|y'-x'|^\eta \wedge 1)|\partial^2_\mu p_m(\mu, s_1, r', x', y')(\gv) - \partial^2_\mu p_m(\mu, s_2, r', x', y')(\gv)|   \,\mu(dx')\,  dy' \, dr'  \notag\\
& \qquad \qquad \qquad \times  g(c(t-s_1\vee s_2), z-x) \Bigg\}.\notag 
\end{align*}

The proof of the estimate \eqref{diff:time:second:L:deriv:p:hat:same:time} is now complete.\\

\noindent \emph{Step 4: proof of the estimate \eqref{diff:time:second:deriv:mes::diff:space:a:mu:mua}.}\\

We proceed as in the proof of \eqref{diff:lions:deriv:a:mu:mua}. We first remark that if $|s_1-s_2|>  t-s_1\vee s_2 $ then the result directly follows from \eqref{recursive:bound:second:order:deriv:mes:holder:reg:space:a} with $\beta'=0$ and $\beta'=1$ combined with \eqref{second:deriv:mes:induction:decoupling:mckean} and the space time inequality \eqref{space:time:inequality}. We thus assume that $|s_1-s_2| \leq  t-s_1\vee s_2 $ for the rest of the proof. We now apply the identity \eqref{full:expression:second:deriv} to the map $m \mapsto a_{i, j}(t, x, m)$ so that the difference $\partial^2_\mu [a_{i, j}(t, x, [X^{s_1, \xi, (m)}_{t}])  - a_{i, j}(t, z, [X^{s_1, \xi, (m)}_{t}])](\gv)- \partial^2_\mu[ a_{i, j}(t, x, [X^{s_2, \xi, (m)}_{t}])  -  a_{i, j}(t, z, [X^{s_2, \xi, (m)}_{t}])](\gv)$ can be decomposed as the sum of the following terms
 {\small
\begin{align*}
T^{1}_{i, j}   & :=  \int_{(\mathbb{R}^d)^2} \Big\{\frac{\delta^2 a_{i, j}}{\delta m^2}(t, x, [X^{s_1, \xi , (m)}_t])(z', z'') - \frac{\delta^2 a_{i, j}}{\delta m^2}(t, z, [X^{s_1, \xi , (m)}_t])(z', z'') \\
& \quad  - \Big(\frac{\delta^2 a_{i, j}}{\delta m^2}(t, x, [X^{s_2, \xi , (m)}_t])(z', z'') - \frac{\delta^2 a_{i, j}}{\delta m^2}(t, z, [X^{s_2, \xi , (m)}_t])(z', z'') \Big) \Big\}\\
& \quad \quad  \partial_{x} p_m(\mu, s_1, t, v, z')\otimes \partial_{x} p_m(\mu, s_1, t, v',z'')\, dz'\, dz'' \\
& \quad +  \int_{(\mathbb{R}^d)^2} \Big\{\frac{\delta^2 a_{i, j}}{\delta m^2}(t, x, [X^{s_2, \xi , (m)}_t])(z', z'') - \frac{\delta^2 a_{i, j}}{\delta m^2}(t, z, [X^{s_2, \xi , (m)}_t])(z', z'')  \Big\}\\
& \quad \quad  \Big[ \partial_{x} p_m(\mu, s_1, t, v, z')\otimes \partial_{x} p_m(\mu, s_1, t, v',z'') - \partial_{x} p_m(\mu, s_2, t, v, z')\otimes \partial_{x} p_m(\mu, s_2, t, v',z'')  \Big]\, dz'\, dz'',
\end{align*}
\begin{align*}
T^{2}_{i, j}  & :=  \int_{(\mathbb{R}^d)^3}  \Big\{ \frac{\delta^2 a_{i, j}}{\delta m^2}(t, x, [X^{s_1, \xi , (m)}_t])(z', z'') - \frac{\delta^2 a_{i, j}}{\delta m^2}(t, z, [X^{s_1, \xi , (m)}_t])(z', z'') \\
& \quad - \Big( \frac{\delta^2 a_{i, j}}{\delta m^2}(t, x, [X^{s_2, \xi , (m)}_t])(z', z'') - \frac{\delta^2 a_{i, j}}{\delta m^2}(t, z, [X^{s_2, \xi , (m)}_t])(z', z'') \Big)\Big\}\\
& \quad \quad  \partial_{x}  p_m(\mu, s_1, t, v, z') \otimes \partial_\mu p_m(\mu, s_1, t, x', z'')(v') \, dz'\, dz''\, \mu(dx')\\
& \quad  +   \int_{(\mathbb{R}^d)^3} \Big\{ \frac{\delta^2 a_{i, j}}{\delta m^2}(t, x, [X^{s_2, \xi , (m)}_t])(z', z'') - \frac{\delta^2 a_{i, j}}{\delta m^2}(t, z, [X^{s_2, \xi , (m)}_t])(z', z'')  \Big\}\\
& \quad  \quad   \Big[\partial_{x}  p_m(\mu, s_1, t, v, z') \otimes \partial_\mu p_m(\mu, s_1, t, x', z'')(v')  - \partial_{x}  p_m(\mu, s_2, t, v, z') \otimes \partial_\mu p_m(\mu, s_2, t, x', z'')(v') \Big]  \, dz'\, dz''\, \mu(dx'),\\
T^{3}_{i, j} & := \int_{\mathbb{R}^d}  \Big\{\frac{\delta a_{i, j}}{\delta m}(t, x, [X^{s_1, \xi , (m)}_t])(z') - \frac{\delta a_{i, j}}{\delta m}(t, z, [X^{s_1, \xi , (m)}_t])(z') \\
& \quad - \Big(\frac{\delta a_{i, j}}{\delta m}(t, x, [X^{s_2, \xi , (m)}_t])(z') - \frac{\delta a_{i, j}}{\delta m}(t, z, [X^{s_2, \xi , (m)}_t])(z')\Big) \Big\} \partial_\mu[\partial_x p_m(\mu, s_1, t, v ,z')](v') \, dz' , \\
& \quad +  \int_{\mathbb{R}^d}  \left\{ \frac{\delta a_{i, j}}{\delta m}(t, x, [X^{s_2, \xi , (m)}_t])(z') - \frac{\delta a_{i, j}}{\delta m}(t, z, [X^{s_2, \xi , (m)}_t])(z') \right\} \\
& \quad \Big[ \partial_\mu[\partial_x p_m(\mu, s_1, t, v ,z')](v') - \partial_\mu[\partial_x p_m(\mu, s_2, t, v ,z')](v')  \Big] \, dz' , \\
T^{4}_{i, j} & := \int_{\mathbb{R}^d}  \Big\{\frac{\delta a_{i, j}}{\delta m}(t, x, [X^{s_1, \xi , (m)}_t])(z') - \frac{\delta a_{i, j}}{\delta m}(t, z, [X^{s_1, \xi , (m)}_t])(z') \\
& \quad - \Big(\frac{\delta a_{i, j}}{\delta m}(t, x, [X^{s_2, \xi , (m)}_t])(z') - \frac{\delta a_{i, j}}{\delta m}(t, z, [X^{s_2, \xi , (m)}_t])(z')\Big) \Big\} \partial_x[\partial_\mu p_m(\mu, s_1, t, v' ,z')(v)] \, dz' , \\
& \quad +  \int_{\mathbb{R}^d}  \left\{ \frac{\delta a_{i, j}}{\delta m}(t, x, [X^{s_2, \xi , (m)}_t])(z') - \frac{\delta a_{i, j}}{\delta m}(t, z, [X^{s_2, \xi , (m)}_t])(z') \right\} \\
& \quad \Big[ \partial_x[\partial_\mu p_m(\mu, s_1, t, v' ,z')(v)] -\partial_x[\partial_\mu p_m(\mu, s_2, t, v' ,z')(v)]  \Big] \, dz' , \\
T^{5}_{i, j} & :=  \int_{(\mathbb{R}^d)^3} \Big\{ \frac{\delta^2 a_{i, j}}{\delta m^2}(t, x, [X^{s_1, \xi , (m)}_t])(z', z'') - \frac{\delta^2 a_{i, j}}{\delta m^2}(t, z, [X^{s_1, \xi , (m)}_t])(z', z'') \\
& \quad - \Big(\frac{\delta^2 a_{i, j}}{\delta m^2}(t, x, [X^{s_2, \xi , (m)}_t])(z', z'') - \frac{\delta^2 a_{i, j}}{\delta m^2}(t, z, [X^{s_2, \xi , (m)}_t])(z', z'')\Big) \Big\}\\
& \quad \quad  \partial_\mu p_m(\mu, s_1, t, x', z')(v) \otimes \partial_{x}p_m(\mu, s_1, t, v',z'')\, dz'\, dz''\, \mu(dx') \\
& \quad + \int_{(\mathbb{R}^d)^3} \Big\{ \frac{\delta^2 a_{i, j}}{\delta m^2}(t, x, [X^{s_2, \xi , (m)}_t])(z', z'') - \frac{\delta^2 a_{i, j}}{\delta m^2}(t, z, [X^{s_2, \xi , (m)}_t])(z', z'') \Big\} \\
& \quad  \quad  \Big[ \partial_\mu p_m(\mu, s_1, t, x', z')(v) \otimes \partial_{x}p_m(\mu, s_1, t, v',z'') - \partial_\mu p_m(\mu, s_2, t, x', z')(v) \otimes \partial_{x}p_m(\mu, s_2, t, v',z'')\Big] \, dz'\, dz''\, \mu(dx'), \\
T^{6}_{i, j} & := \int_{(\mathbb{R}^d)^4}\Big\{\frac{\delta^2 a_{i, j}}{\delta m^2}(t, x, [X^{s_1, \xi , (m)}_t])(z', z'')-\frac{\delta^2 a_{i, j}}{\delta m^2}(t, z, [X^{s_1, \xi , (m)}_t])(z', z'') \\
& \quad - \Big(\frac{\delta^2 a_{i, j}}{\delta m^2}(t, x, [X^{s_2, \xi , (m)}_t])(z', z'')-\frac{\delta^2 a_{i, j}}{\delta m^2}(t, z, [X^{s_2, \xi , (m)}_t])(z', z'') \Big) \Big\} \\
& \quad \quad  \partial_\mu p_m(\mu, s_1, t, x', z)(v) \otimes  \partial_\mu p_m(\mu, s_1, t, x'', z')(v')\,dz\, dz'\, \mu(dx'')\, \mu(dx') \\
& \quad + \int_{(\mathbb{R}^d)^4}\Big\{\frac{\delta^2 a_{i, j}}{\delta m^2}(t, x, [X^{s_2, \xi , (m)}_t])(z', z'')-\frac{\delta^2 a_{i, j}}{\delta m^2}(t, z, [X^{s_2, \xi , (m)}_t])(z', z'') \Big\} \\
& \quad \quad  \Big[ \partial_\mu p_m(\mu, s_1, t, x', z)(v) \otimes  \partial_\mu p_m(\mu, s_1, t, x'', z')(v') - \partial_\mu p_m(\mu', s_2, t, x', z)(v) \otimes  \partial_\mu p_m(\mu, s_2, t, x'', z')(v') \Big] \,dz\, dz'\, \mu(dx'')\, \mu(dx'), \\
T^{7}_{i, j}  & := \int_{(\mathbb{R}^d)^2}  \Big\{ \frac{\delta a_{i, j}}{\delta m}(t, x, [X^{s_1, \xi , (m)}_t])(z')- \frac{\delta a_{i, j}}{\delta m}(t, z, [X^{s_1, \xi , (m)}_t])(z') \\
& \quad \quad - \Big(  \frac{\delta a_{i, j}}{\delta m}(t, x, [X^{s_2, \xi , (m)}_t])(z')- \frac{\delta a_{i, j}}{\delta m}(t, z, [X^{s_2, \xi , (m)}_t])(z')\Big)\Big\}\, \partial_\mu^2 p_m(\mu, s_1, t, x', z')(\gv)\,  dz'\, \mu(dx')\\
& \quad + \int_{(\mathbb{R}^d)^2}  \Big\{ \frac{\delta a_{i, j}}{\delta m}(t, x, [X^{s_2, \xi , (m)}_t])(z')- \frac{\delta a_{i, j}}{\delta m}(t, z, [X^{s_2, \xi , (m)}_t])(z')  \\
& \quad \quad - \Big( \frac{\delta a_{i, j}}{\delta m}(t, x, [X^{s_2, \xi , (m)}_t])(x')- \frac{\delta a_{i, j}}{\delta m}(t, z, [X^{s_2, \xi , (m)}_t])(x') \Big)\Big\} \\
& \quad \quad \Big[\partial_\mu^2 p_m(\mu, s_1, t, x', z')(\gv) - \partial_\mu^2 p_m(\mu, s_2, t, x', z')(\gv)\Big]\,  dz'\, \mu(dx').
\end{align*} 
 }

 As previously done, we now quantify the contribution of each term in the above decomposition. Letting $\Theta^{(m)}_{\lambda, t}:= (1-\lambda)[X^{s_1, \xi, (m)}_t] + \lambda [X^{s_2, \xi, (m)}_t]$ and $h \in \left\{ \delta a_{i, j}/\delta m,  \delta^2 a_{i, j}/\delta m^2 \right\}$, it holds
 \begin{align*}
h(t, x, [X^{s_1, \xi, (m)}_t])(.) & - h(t, x, [X^{s_2, \xi, (m)}_t])(.) \\
& = \int_0^1\int_{\mathbb{R}^d} \frac{\delta h}{\delta m}(t, x, \Theta^{(m)}_{t, \lambda})(.)(z') (p_m(\mu, s_1, t, z') - p_{m}(\mu, s_2, t, z')) \, dz' \, d\lambda
\end{align*}

\noindent so that, from similar arguments as those used in order to derive (A.80) of Lemma A.5 in \cite{chaudruraynal:frikha}, namely, using the uniform $\eta$-H\"older regularity of the map $(x, z') \mapsto [\delta h/\delta m](t, x, m)(.)(z')$ together with \eqref{regularity:time:estimate:v1:v2:v3:decoupling:mckean} with $n=0$, we get that for any $\alpha \in [0,\eta]$ and any $\beta \in [0, 1]$
\begin{align}
| h(t, x, [X^{s_1, \xi, (m)}_t])(.)  & - h(t, x, [X^{s_2, \xi, (m)}_t])(.) - [h(t, z, [X^{s_1, \xi, (m)}_t])(.)  - h(t, z, [X^{s_2, \xi, (m)}_t])(.)]| \label{holder:property:h:delta:time} \\
& \leq K^{++}_\beta (|x-z|^\alpha \wedge 1) |s_1-s_2|^\beta \left\{\frac{1}{(t-s_1)^{\beta+\frac{\alpha-\eta}{2}}} + \frac{1}{(t-s_2)^{\beta+\frac{\alpha-\eta}{2}}}\right\}. \notag 
\end{align}
 
We now consider the first term $T^{1}_{i, j}$ which we decompose as the sum of $T^{1,1}_{i, j}$ and $T^{1,2}_{i, j}$ as written above. Then, taking $\alpha=\eta$ in \eqref{holder:property:h:delta:time} and using \eqref{bound:derivative:heat:kernel}, we get
\begin{align*}
| T^{1, 1}_{i, j} | \leq K^{++}_\beta (|x-z|^\eta \wedge 1) \frac{|s_1-s_2|^\beta}{(t-s_1\vee s_2)^{1+\beta}}. 
\end{align*}

We then take $\alpha=0$ in \eqref{holder:property:h:delta:time} so that using again \eqref{bound:derivative:heat:kernel} and the fact that $|s_1-s_2|\leq t-s_1\vee s_2$
\begin{align*}
| T^{1, 1}_{i, j} | \leq K^{++}_\beta  \frac{|s_1-s_2|^\beta}{(t-s_1\vee s_2)^{1+\beta-\frac{\eta}{2}}} 
\end{align*}
\noindent for any $\beta \in [0,1]$. 
% Otherwise, if $|s_1-s_2|\geq t-s_1\vee s_2$, we rather decompose $T^{1, 1}_{i, j}$ as follows
%\begin{align*}
%T^{1, 1}_{i, j} & :=  \int_{(\mathbb{R}^d)^2} \Big\{\frac{\delta^2 a_{i, j}}{\delta m^2}(t, x, [X^{s_1, \xi , (m)}_t])(z', z'') - \frac{\delta^2 a_{i, j}}{\delta m^2}(t, z, [X^{s_1, \xi , (m)}_t])(z', z'') \\
%& \quad - \Big( \frac{\delta^2 a_{i, j}}{\delta m^2}(t, x, [X^{s_1, \xi , (m)}_t])(z', v') - \frac{\delta^2 a_{i, j}}{\delta m^2}(t, z, [X^{s_1, \xi , (m)}_t])(z', v') \Big) \\
%& \quad  - \Big(\frac{\delta^2 a_{i, j}}{\delta m^2}(t, x, [X^{s_2, \xi , (m)}_t])(z', z'') - \frac{\delta^2 a_{i, j}}{\delta m^2}(t, z, [X^{s_2, \xi , (m)}_t])(z', z'') \Big) \\
%& \quad + \Big(\frac{\delta^2 a_{i, j}}{\delta m^2}(t, x, [X^{s_2, \xi , (m)}_t])(z', v') - \frac{\delta^2 a_{i, j}}{\delta m^2}(t, z, [X^{s_2, \xi , (m)}_t])(z', v') \Big) \Big\}\\
%& \quad \quad  \partial_{x} p_m(\mu, s_1, t, v, z')\otimes \partial_{x} p_m(\mu, s_1, t, v',z'')\, dz'\, dz'' 
%\end{align*}
%\noindent and use the uniform $\eta$-H\"older regularity of $[\delta^2 a_{i, j}/\delta m^2](t, x, m)(z', .)$ together with \eqref{bound:derivative:heat:kernel} and the space time inequality \eqref{space:time:inequality} so that
%$$
%| T^{1, 1}_{i, j} | \leq K^{+} \frac{1}{(t-s_1)^{1-\frac{\eta}{2}}} \leq K^{+} \frac{|s_1-s_2|^\beta}{(t-s_1)^{1+\beta-\frac{\eta}{2}}}.
%$$

Combining the two above estimates, we conclude that for any $\beta \in [0,1]$
$$
| T^{1, 1}_{i, j} | \leq K^{++}_\beta |s_1-s_2|^\beta \Bigg(\frac{|x-z|^{\eta} \wedge 1}{(t-s_1\vee s_2)^{1+ \beta}} \wedge \frac{1}{(t-s_1\vee s_2)^{1+\beta-\frac{\eta}{2}}}\Bigg).
$$

We deal with $T^{1, 2}_{i, j}$ by using a centering argument. Namely, we write
\begin{align*}
 \Big| T^{1,2}_{i, j} \Big| & = \Big|  \int_{(\mathbb{R}^d)^2} \Big[ \Big\{\frac{\delta^2 a_{i, j}}{\delta m^2}(t, x, [X^{s_2, \xi , (m)}_t])(z', z'') - \frac{\delta^2 a_{i, j}}{\delta m^2}(t, z, [X^{s_2, \xi , (m)}_t])(z', z'')  \Big\}  \\
 & \quad  -  \Big\{\frac{\delta^2 a_{i, j}}{\delta m^2}(t, x, [X^{s_2, \xi , (m)}_t])(z', v') - \frac{\delta^2 a_{i, j}}{\delta m^2}(t, z, [X^{s_2, \xi , (m)}_t])(z', v')  \Big\}  \Big]\\
& \quad \quad  \Big[ \partial_{x} p_m(\mu, s_1, t, v, z')\otimes \partial_{x} p_m(\mu, s_1, t, v',z'') - \partial_{x} p_m(\mu, s_2, t, v, z')\otimes \partial_{x} p_m(\mu, s_2, t, v',z'')  \Big]\, dz'\, dz'' \Big| \\
& \leq K^{+}_\beta |s_1-s_2|^\beta \left\{ \frac{|x-z|^\eta\wedge 1}{(t-s_1\vee s_2)^{1+\beta}} \wedge \frac{1}{(t-s_1\vee s_2)^{1-\frac{\eta}{2}+\beta}} \right\}
\end{align*}
\noindent for any $\beta \in [0,1]$, where we used  the uniform $\eta$-H\"older regularity of $[\delta^2 a_{i, j}/\delta m^2](t, ., m)(z',.)$, \eqref{regularity:time:estimate:v1:v2:v3:decoupling:mckean} with $n=1$ and eventually the space time inequality \eqref{space:time:inequality} together with the fact that $|s_1-s_2|\leq t-s_1\vee s_2$. Gathering the estimates on $T^{1,1}_{i, j}$ and $T^{1,2}_{i, j}$, we conclude that $T^{1}_{i, j}$ is bounded by the first term appearing on the right-hand side of \eqref{diff:time:second:deriv:mes::diff:space:a:mu:mua}.

The other terms $T^{\ell}_{i, j}$, $\ell=2, \cdots, 6$, can be treated in the same way so we will be short and omit some technical details. We first use \eqref{holder:property:h:delta:time} with $\alpha=\eta$, \eqref{bound:derivative:heat:kernel}, \eqref{first:second:estimate:induction:decoupling:mckean} to deal with $T^{2,1}_{i, j}$ and then use the boundedness and uniform $\eta$-H\"older regularity of $[\delta^2 a_{i, j}/\delta m^2](t,.,m)(z', z'')$ together with \eqref{regularity:time:estimate:v1:v2:v3:decoupling:mckean}, \eqref{regularity:time:estimate:v1:v2:decoupling:mckean}, the space time inequality \eqref{space:time:inequality} and the fact that $|s_1-s_2|\leq t-s_1\vee s_2$ to handle $T^{2, 2}_{i ,j}$. We thus obtain
$$
 \Big| T^{2}_{i, j} \Big| \leq K^{++}_\beta |s_1-s_2|^\beta \left\{ \frac{|x-z|^\eta\wedge 1}{(t-s_1\vee s_2)^{1+\beta - \frac{\eta}{2}}} \wedge \frac{1}{(t-s_1\vee s_2)^{1+\beta-\frac{\eta}{2}}} \right\}
$$
\noindent for any $\beta \in [0,(1+\eta)/2)$. By symmetry, the two terms $T^{3}_{i, j}$ and $T^{4}_{i, j}$ can be handled by similar arguments. Namely, we use \eqref{holder:property:h:delta:time} with $\alpha=\eta$, \eqref{cross:deriv:mes:space:induction:decoupling:mckean}, \eqref{sensitivity:time:deriv:cross:space:mes:pm} and the boundedness and uniform $\eta$-H\"older regularity of $[\delta a_{i, j}/\delta m](t,.,m)(z')$ so that for any $\beta \in [0,\eta/2)$
$$
\Big| T^{3}_{i, j} \Big| + \Big| T^{4}_{i, j}\Big| \leq K^{++}_\beta |s_1-s_2|^\beta \frac{|x-z|^\eta\wedge 1}{(t-s_1\vee s_2)^{1+\beta-\frac{\eta}{2}}}.
$$

Similarly, using \eqref{holder:property:h:delta:time} with $\alpha=\eta$, \eqref{bound:derivative:heat:kernel}, \eqref{first:second:estimate:induction:decoupling:mckean}, \eqref{regularity:time:estimate:v1:v2:v3:decoupling:mckean}, \eqref{regularity:time:estimate:v1:v2:decoupling:mckean} and the boundedness and uniform $\eta$-H\"older regularity of of $[\delta^2 a_{i, j}/\delta m^2](t,.,m)(z', z'')$. For any $\beta \in [0,(1+\eta)/2)$, we get
$$
\Big| T^{5}_{i, j} \Big| + \Big| T^{6}_{i, j} \Big|   \leq  K^{++}_\beta |s_1-s_2|^\beta \frac{|x-z|^\eta\wedge 1}{(t-s_1\vee s_2)^{1+\beta-\frac{\eta}{2}}}.
$$

 We deal with the last term by decomposing it as the sum of two terms $T_{i, j}^{7, 1}$ and $T_{i, j}^{7, 2}$ as written above. The first is handled using \eqref{holder:property:h:delta:time} with $\alpha=\eta$ and \eqref{second:deriv:mes:induction:decoupling:mckean} so that for any $\beta \in [0,1]$
 $$
 |T^{7, 1}_{i, j}| \leq K^{++}_\beta \frac{|x-z|^\eta\wedge 1}{(t-s_1\vee s_2)^{1+\beta-\frac{\eta}{2}}}.
 $$
 
For the second term, it follows from the boundedness and uniform $\eta$-H\"older regularity of $[\delta a_{i, j}/\delta m](t, ., m)(.)$ that
\begin{align*}
| T^{7, 2}_{i, j} |&  \leq K^{+} \int_{(\mathbb{R}^d)^2} (|z-x|^\eta \wedge |z'- x'|^\eta \wedge 1) |\partial_\mu^2 p_m(\mu, s_1, t, x', z')(\gv) - \partial_\mu^2 p_m(\mu, s_2, t, x', z')(\gv)| \,  \mu(dx') \, dz' \\
& \leq K^{+} (|z-x|^{\eta} \wedge 1)\int_{(\mathbb{R}^d)^2} | \partial_\mu^2 p_m(\mu, s_1, t, x', z')(\gv) -  \partial_\mu^2 p_m(\mu, s_2, t, x', z')(\gv) | \, \mu(dx') \,  dz'\notag\\
&\quad\quad\quad\quad\quad \wedge \int_{(\mathbb{R}^d)^2} (|z'-x'|^{\eta} \wedge 1)  | \partial_\mu^2 p_m(\mu, s_1, t, x', z')(\gv) -  \partial_\mu^2 p_m(\mu, s_2, t, x', z')(\gv) | \, \mu(dx') \, dz'.
\end{align*}

Collecting the above estimates allows to conclude the proof of \eqref{diff:time:second:deriv:mes::diff:space:a:mu:mua}.\\

\noindent \emph{Step 5: proof of the estimate \eqref{diff:time:second:L:deriv:parametrix:kernel:pmp1}.}\\

The proof is quite similar to the one of \eqref{diff:mes:L:deriv:parametrix:kernel:pmp1}, namely, we start from the decomposition of $\partial^2_\mu \mH_{m+1}(\mu, s, r ,t ,x ,z)(\gv)$ given by \eqref{derivsec:mu:parametrix:kernel:mp1} and investigate the regularity of each term with respect to the variable $s$.\\
 
 \noindent (i) \emph{Regularity of the map $s \mapsto \big(\partial_\mu {\rm I}(\mu, s)(\gv)+ \partial_\mu {\rm II}(\mu, s)(\gv)\big)\widehat{p}_{m+1}(\mu, s, r, t, x, z)$.}\\
 
 Again, we first remark that the identity \eqref{second:diff-mat-mes} combined with \eqref{diff:time:inverse:diff:matrix}, \eqref{diff:time:cross:deriv:diff:and:deriv:coeff}, \eqref{diff:time:second:deriv:mes:a:b}, \eqref{deriv:mes:a:m} and \eqref{recursive:bound:deriv:a:or:b} combined with \eqref{second:deriv:mes:induction:decoupling:mckean} (recalling that $\mathscr{C}^{1, 0}_{m}(C^+, t-s) \leq K^{+} := \lim_{m\rightarrow \infty} \mathscr{C}^{1, 0}_{m}(C^+, t-s) < \infty$) imply that for any $\beta \in [0,\eta)$
 \begin{align*}
 \Big| \partial^2_\mu & \Big[H^{i}_1\left(\int_r^{t} a(r',z, [X^{s_1, \xi, (m)}_{r'}]) dr', z-x\right)\Big](\gv) -  \partial^2_\mu  \Big[H^{i}_1\left(\int_r^{t} a(r', z, [X^{s_2, \xi, (m)}_{r'}]) dr', z-x\right)\Big](\gv) \Big| \\
 & \leq K^{++}_\beta \frac{|z-x|}{t-r}\left\{  \frac{|s_1-s_2|^{\frac{\beta}{2}}}{(r-s_1\vee s_2)^{1+\frac{\beta-\eta}{2}}}  \right. \\
 & \quad \left. +  \frac{1}{t-r} \int_r^t \int_{(\mathbb{R}^d)^2} (|y'-x'|^\eta \wedge 1) |\partial^2_\mu p_m(\mu, s_1, r', x', y')(\gv) - \partial^2_\mu p_m(\mu, s_2, r', x', y')(\gv)|  \,  \mu(dx') \, dy' \, dr' \right\}
 \end{align*}
 \noindent and
  \begin{align*}
 \Big| \partial^2_\mu & \Big[H^{i, j}_2\left(\int_r^{t} a(r',z, [X^{s_1, \xi, (m)}_{r'}]) dr', z-x\right)\Big](\gv) -  \partial^2_\mu  \Big[H^{i, j}_2\left(\int_r^{t} a(r',z, [X^{s_2, \xi, (m)}_{r'}]) dr', z-x\right)\Big](\gv) \Big| \\
 & \leq K^{++}_\beta \left\{\frac{|z-x|^2}{(t-r)^2} + \frac{1}{t-r}\right\} \left\{  \frac{|s_1-s_2|^{\frac{\beta}{2}}}{(r-s_1 \vee s_2)^{1+\frac{\beta-\eta}{2}}}  \right. \\
 & \quad \left. +  \frac{1}{t-r} \int_r^t \int_{(\mathbb{R}^d)^2} (|y'-x'|^\eta \wedge 1) |\partial^2_\mu p_m(\mu, s_1, r', x', y')(\gv) - \partial^2_\mu p_m(\mu, s_2, r', x', y')(\gv)|   \, \mu(dx')\, dy' \, dr' \right\}.
 \end{align*}
 
 Then, using the two above estimates,  \eqref{diff:time:second:deriv:mes:a:b}, \eqref{diff:time:second:deriv:mes::diff:space:a:mu:mua}, \eqref{diff:time:inverse:diff:matrix}, \eqref{diff:time:cross:deriv:diff:and:deriv:coeff}, \eqref{diff:time:first:L:deriv:diff:diff:coeff:holder:reg}, \eqref{deriv:mes:a:m}, the uniform boundedness of $b_i$, the uniform $\eta$-H\"older regularity of $a_{i, j}(t,.,m)$ and the space time inequality \eqref{space:time:inequality}, we obtain that for any $\beta \in [0,\eta)$
 \begin{align*}
 \Big| [\partial_\mu & {\rm I}(\mu, s_1)(\gv)- \partial_\mu {\rm I}(\mu, s_2)(\gv)] \widehat{p}_{m+1}(\mu, s_1, r, t, x, z) \Big| \\
 & \leq K^{++}_\beta \left\{  \frac{|s_1-s_2|^{\frac{\beta}{2}}}{(t-r)^{\frac12}(r-s_1\vee s_2)^{1+\frac{\beta-\eta}{2}}} \right.   \\
 & \quad \left. + \frac{1}{(t-r)^{\frac12}} \int_{(\mathbb{R}^d)^2} (|y'-x'|^\eta \wedge 1) |\partial^2_\mu p_m(\mu, s_1, r, x', y')(\gv) - \partial^2_\mu p_m(\mu, s_2, r, x', y')(\gv)|  \, \mu(dx')\, dy'   \right. \\
 & \quad \left. +  \frac{1}{(t-r)^{\frac32}} \int_r^t \int_{(\mathbb{R}^d)^2} (|y'-x'|^\eta \wedge 1) |\partial^2_\mu p_m(\mu, s_1, r', x', y')(\gv) - \partial^2_\mu p_m(\mu, s_2, r', x', y')(\gv)|  \, \mu(dx')  \, dy' \, dr'\right\}\\
 & \quad \quad g(c(t-r), z-x),
 \end{align*}
  \begin{align*}
 \Big| [\partial_\mu & {\rm II}(\mu, s_1)(\gv) - \partial_\mu {\rm II}(\mu, s_2)(\gv)] \widehat{p}_{m+1}(\mu, s_1, r, t, x, z) \Big| \\
 & \leq K^{++}_\beta \left\{ |s_1-s_2|^{\frac{\beta}{2}} \Bigg(\frac{1}{(t-r)^{1-\frac{\eta}{2}}(r-s_1\vee s_2)^{1+\frac{\beta}{2}}} \wedge \frac{1}{(t-r)(r-s_1\vee s_2)^{1+\frac{\beta-\eta}{2}}}\Bigg) \right. \\
 &\quad \left. + \frac{1}{(t-r)^{1-\frac{\eta}{2}}} \int_{(\mathbb{R}^d)^2}  | \partial_\mu^2 p_m(\mu, s_1, r, x', y')(\gv) -  \partial_\mu^2 p_m(\mu, s_2, r, x', y')(\gv) |  \, \mu(dx') \, dy'  \right.   \\
& \quad \quad \left.  \wedge \, \frac{1}{t-r} \int_{(\mathbb{R}^d)^2} (|y'-x'|^\eta \wedge 1)  | \partial_\mu^2 p_m(\mu, s_1, r, x', y')(\gv) -  \partial_\mu^2 p_m(\mu, s_2, r, x', y')(\gv) | \,  \mu(dx')  \, dy' \right. \\
 & \quad \left. +  \frac{1}{(t-r)^{2-\frac{\eta}{2}}} \int_r^t \int_{(\mathbb{R}^d)^2} (|y'-x'|^\eta \wedge 1) |\partial^2_\mu p_m(\mu, s_1, r', x', y')(\gv) - \partial^2_\mu p_m(\mu, s_2, r', x', y')(\gv)|  \, \mu(dx') \, dy' \, dr'\right\}\\
 & \quad \quad g(c(t-r), z-x)
 \end{align*}
  \noindent and, using \eqref{sec:mes:deriv:parametrix:kernel:s:r:t:with:beta:partii} and \eqref{sec:mes:deriv:parametrix:kernel:s:r:t:with:beta:partiii} (combined with \eqref{second:deriv:mes:induction:decoupling:mckean} and the space time inequality \eqref{space:time:inequality}) the latter being used both with $\beta=0$ and $\beta=1$, replacing therein the standard Gaussian estimate on $\widehat{p}_{m+1}$ by \eqref{gaussian:bound:diff:time:hat:pm:different:time} with $n=0$ 
 \begin{align*}
 \Big|  (\partial_\mu & {\rm I}(\mu, s_2)(\gv)+  \partial_\mu  {\rm II}(\mu, s_2)(\gv)) (\widehat{p}_{m+1}(\mu, s_1, r, t, x, z)-\widehat{p}_{m+1}(\mu, s_2, r, t, x, z)) \Big|\\
 & \leq K^+_\beta |s_1-s_2|^{\frac{\beta}{2}} \Bigg(\frac{1}{(t-r)^{1-\frac{\eta}{2}}(r-s_1\vee s_2)^{1+\frac{\beta}{2}}} \wedge \frac{1}{(t-r)(r-s_1\vee s_2)^{1+\frac{\beta-\eta}{2}}}\Bigg) \, g(c(t-r), z-x).
 \end{align*}
  
 Gathering the three previous estimates, we conclude that $\big(\partial_\mu {\rm I}(\mu, s_1)(\gv)+ \partial_\mu {\rm II}(\mu, s_1)(\gv)\big)\widehat{p}_{m+1}(\mu, s_1, r, t, x, z) - \big(\partial_\mu {\rm I}(\mu, s_2)(\gv)+ \partial_\mu {\rm II}(\mu, s_2)(\gv)\big)\widehat{p}_{m+1}(\mu, s_2, r, t, x, z)$ is bounded by the right-hand side of \eqref{diff:time:second:L:deriv:parametrix:kernel:pmp1}.\\
 
 \noindent (ii) \emph{Regularity of the maps $s \mapsto \big({\rm I}(\mu, s)(v)+  {\rm II}(\mu, s)(v)\big) \otimes \partial_\mu \widehat{p}_{m+1}(\mu, s, r, t, x, z)(v'), \,  \partial_\mu\widehat{p}_{m+1}(\mu, s, r, t, x, z)(v) \otimes \partial_\mu {\rm III }(\mu, s)(v')$.}\\
Again these two terms only involve first order L-derivative so that they can be handled using the regularity results established in \cite{chaudruraynal:frikha}. To be more specific, we first use \eqref{estimate:I:II:partial:mes:p:hat:mp1} with $\beta=1$ replacing therein the estimate on $\partial_\mu \widehat{p}_{m+1}$ by \eqref{diff:time:L:deriv:p:hat:different:time} so that 
\begin{align*}
\Big| [{\rm I}(\mu, s_2)(v) + {\rm II}(\mu, s_2)(v)] & \otimes [\partial_\mu \widehat{p}_{m+1}(\mu, s_1, r, t, x, z)(v') - \partial_\mu \widehat{p}_{m+1}(\mu, s_2, r, t, x, z)(v')]\Big| \\
 & \leq K_\beta^{+} \frac{|s_1-s_2|^{\frac{\beta}{2}}}{(t-r)^{1-\frac{\eta}{2}}(r-s_1\vee s_2)^{1+\frac{\beta-\eta}{2}}} \, g(c(t-r), z-x).
\end{align*}

 From \eqref{diff:time:cross:deriv:diff:and:deriv:coeff}, \eqref{diff:time:inverse:diff:matrix}, \eqref{diff:L:deriv:invers:diff:coeff}, \eqref{diff:time:drift:diff:coefficients} if $|s_1-s_2|\leq t-s_1\vee s_2$ and the boundedness of the coefficients otherwise, \eqref{deriv:v:deriv:mes:pm:hat}, \eqref{deriv:mes:a:m}, the space time inequality \eqref{space:time:inequality} and standard computations that we omit, we deduce
 \begin{align*}
\Big| [ {\rm I}(\mu, s_1)(v) - {\rm I}(\mu, s_2)(v)] & \otimes \partial_\mu \widehat{p}_{m+1}(\mu, s_1, r, t, x, z)(v') \Big| \\
& \leq K^{+}_\beta  \frac{|s_1-s_2|^{\frac{\beta}{2}}}{(t-r)^{\frac12}(r-s_1\vee s_2)^{1+\frac{\beta}{2}-\eta}} \, g(c(t-r), z-x).
\end{align*}

From \eqref{diff:time:first:L:deriv:diff:diff:coeff:holder:reg}, \eqref{diff:time:inverse:diff:matrix}, \eqref{diff:L:deriv:invers:diff:coeff}, \eqref{recursive:bound:deriv:mes:holder:reg:a} with $n=0$, $\beta=1$ and $\beta=0$, \eqref{diff:time:with:holder:reg:space:drift:diff:coefficients} with $\alpha=\eta$,  \eqref{deriv:mes:a:m}, the uniform $\eta$-H\"older regularity of $a_{i, j}(t, ., m)$, \eqref{deriv:v:deriv:mes:pm:hat}, the space time inequality \eqref{space:time:inequality} and standard computations that we omit, we obtain
\noindent and
\begin{align*}
 \Big| [{\rm II}(\mu, s_1)(v)\big) &-  {\rm II}(\mu, s_2)(v)]  \otimes \partial_\mu \widehat{p}_{m+1}(\mu, s_1, r, t, x, z)(v') \Big| \\
& \leq K_\beta^{+} |s_1-s_2|^{\frac{\beta}{2}} \left\{ \frac{1}{(t-r)(r-s_1\vee s_2)^{1+\frac{\beta-\eta}{2}}} \wedge \frac{1}{(t-r)^{1- \frac{\eta}{2}}(r-s_1\vee s_2)^{1+\frac{\beta}{2}}} \right\} \, g(c(t-r), z-x).
\end{align*}

Gathering the three above estimates yields
\begin{align*}
 \Big| [{\rm I}(\mu, s_1)(v)\big) &+  {\rm II}(\mu, s_1)(v)]  \otimes \partial_\mu \widehat{p}_{m+1}(\mu, s_1, r, t, x, z)(v') - [{\rm I}(\mu, s_2)(v)\big) +  {\rm II}(\mu, s_2)(v)]  \otimes \partial_\mu \widehat{p}_{m+1}(\mu, s_2, r, t, x, z)(v')  \Big| \\
& \leq K_\beta^{+} |s_1-s_2|^{\frac{\beta}{2}} \left\{ \frac{1}{(t-r)(r-s_1\vee s_2)^{1+\frac{\beta-\eta}{2}}} \wedge \frac{1}{(t-r)^{1- \frac{\eta}{2}}(r-s_1\vee s_2)^{1+\frac{\beta}{2}}} \right\} \, g(c(t-r), z-x).
\end{align*}

Moreover, from the symmetry identity \eqref{identity:partial:mes:pmp1:convolution:partial:mes:parametrix:kernel} and similar arguments, we also deduce
\begin{align*}
\Big| \partial_\mu & \widehat{p}_{m+1}(\mu, s_1, r, t, x, z)(v)  \otimes \partial_\mu {\rm III }(\mu, s_1)(v') - \partial_\mu\widehat{p}_{m+1}(\mu, s_2, r, t, x, z)(v) \otimes \partial_\mu {\rm III }(\mu, s_2)(v') \Big|\\
& \leq K_\beta^{+} |s_1-s_2|^{\frac{\beta}{2}} \left\{ \frac{1}{(t-r)(r-s_1\vee s_2)^{1+\frac{\beta-\eta}{2}}} \wedge \frac{1}{(t-r)^{1- \frac{\eta}{2}}(r-s_1\vee s_2)^{1+\frac{\beta}{2}}} \right\} \, g(c(t-r), z-x).
\end{align*}

\noindent (iii) \emph{Regularity of the maps $s \mapsto {\rm III }(\mu, s)\, \partial_\mu^2\widehat{p}_{m+1}(\mu, s, r, t, x, z)(\gv)$.}\\

From \eqref{diff:time:with:holder:reg:space:drift:diff:coefficients} with $\alpha = \eta$, \eqref{diff:time:drift:diff:coefficients}, \eqref{bound:second:deriv:mes:hat:p} combined with \eqref{second:deriv:mes:induction:decoupling:mckean} and the space time inequality \eqref{space:time:inequality} recalling that $\mathscr{C}^{1,0}_m(C^+, t-s)\leq K^+:=\lim_{m\rightarrow \infty} \mathscr{C}^{1,0}_m(C^+, t-s)$, \eqref{diff:time:inverse:diff:matrix}, the uniform boundedness of $b_i$ and the uniform $\eta$-H\"older regularity of $ a_{i, j}(t, ., m)$ and again the space time inequality \eqref{space:time:inequality}, we deduce that for any $\beta \in [0,\eta)$
\begin{align*}
\Big| \Big({\rm III }(\mu, s_1) -{\rm III}(\mu, s_2) \Big)  \partial_\mu^2\widehat{p}_{m+1}(\mu, s_1, r, t, x, z)(\gv)\Big| \leq K^+  \frac{|s_1-s_2|^{\frac{\beta}{2}}}{(t-r)^{1-\frac{\eta}{2}}(r-s_1\vee s_2)^{1+\frac{\beta-\eta}{2}}} \, g(c(t-r), z-x)
\end{align*} 
  
\noindent and, again from the uniform boundedness of $b_i$ and $\eta$-H\"older regularity of $a_{i, j}(t, ., m)$, \eqref{diff:time:second:L:deriv:p:hat} and the space time inequality \eqref{space:time:inequality}
\begin{align*}
\Big| {\rm III}(\mu, s_2) & \Big(\partial_\mu^2\widehat{p}_{m+1}(\mu, s_1, r, t, x, z)(\gv) - \partial_\mu^2\widehat{p}_{m+1}(\mu, s_2, r, t, x, z)(\gv) \Big) \Big| \\
& \leq K^{++}_\beta  \left\{ \frac{|s_1-s_2|^{\frac{\beta}{2}}}{(t-r)^{1-\frac{\eta}{2}}(r-s_1\vee s_2)^{1+\frac{\beta-\eta}{2}}} \right. \\ 
& \quad \left. +\frac{1}{(t-r)^{2-\frac{\eta}{2}}}   \int_r^t \int_{(\mathbb{R}^d)^2} (|y'-x'|^{\eta} \wedge1) | \partial_\mu^2 p_m(\mu, s_1, r', x', y')(\gv) -  \partial_\mu^2 p_m(\mu, s_2, r', x', y')(\gv) | ' \, \mu'(dx') \, dy \, dr' \right\} \\
& \quad \quad \times g(c(t-r), z-x).
\end{align*}
 
 Gathering the two above estimates, we thus obtain
 \begin{align*}
\Big| {\rm III }(\mu, s_1)\, & \partial_\mu^2\widehat{p}_{m+1}(\mu, s_1, r, t, x, z)(\gv) - {\rm III }(\mu, s_2)\, \partial_\mu^2\widehat{p}_{m+1}(\mu, s_2, r, t, x, z)(\gv) \Big) \Big| \\
& \leq K^{++}_\beta  \left\{ \frac{|s_1-s_2|^{\frac{\beta}{2}}}{(t-r)^{1-\frac{\eta}{2}}(r-s_1\vee s_2)^{1+\frac{\beta-\eta}{2}}} \right. \\ 
& \quad \left. +\frac{1}{(t-r)^{2-\frac{\eta}{2}}}   \int_r^t \int_{(\mathbb{R}^d)^2} (|y'-x'|^{\eta} \wedge1) | \partial_\mu^2 p_m(\mu, s_1, r', x', y')(\gv) -  \partial_\mu^2 p_m(\mu, s_2, r', x', y')(\gv) | \, \mu'(dx') \, dy' \, dr' \right\} \\
& \quad \quad \times g(c(t-r), z-x).
\end{align*}

Collecting the above estimates concludes the proof of \eqref{diff:time:second:L:deriv:parametrix:kernel:pmp1}.

\bibliographystyle{alpha}
\bibliography{bibli}
\end{document}